\documentclass[10pt,a4paper]{article}
\usepackage[margin=1in]{geometry}
\usepackage{tabularx, enumerate, authblk,soul}
\usepackage{algorithm}
\usepackage{algpseudocode} 
\usepackage[dvipsnames]{xcolor}
\usepackage{amsmath,amssymb,dsfont,caption,subcaption}
\usepackage{mathtools, graphicx,color,dsfont,amsthm,bbm}
\usepackage{lipsum,bm,comment}
\usepackage[round]{natbib}
\usepackage{longtable}
\usepackage{array}
\usepackage{url}
\usepackage[colorlinks=true,linkcolor=blue,anchorcolor=blue,citecolor=blue,filecolor=black,menucolor=black,runcolor=black,urlcolor=black]{hyperref}
\usepackage{xurl}
\hypersetup{breaklinks=true}
\usepackage{cleveref}
\usepackage{float}
\usepackage{tablefootnote}

\usepackage{tikz}
\usepackage{pgfplots}
\pgfplotsset{compat=1.10}
\usepgfplotslibrary{fillbetween}
\usetikzlibrary{patterns}

\newcommand{\RN}[1]{%
  (\textup{\uppercase\expandafter{\romannumeral#1}})%
}
\newcolumntype{H}{>{\setbox0=\hbox\bgroup}c<{\egroup}@{}}

\theoremstyle{plain}
\newtheorem{theorem}{Theorem}
\newtheorem{lemma}[theorem]{Lemma}
\newtheorem{assumption}{Assumption}
\newtheorem{proposition}[theorem]{Proposition}
\newtheorem{corollary}[theorem]{Corollary}

\theoremstyle{definition}

\newtheorem{example}{Example}
\newtheorem{remark}{Remark}

\DeclareMathOperator*{\argmax}{arg\,max}
\DeclareMathOperator*{\argmin}{arg\,min}

\newcommand*\diff{\mathop{}\!\mathrm{d}}

\allowdisplaybreaks

\title{Rate Optimality and Phase Transition for \\
User-Level Local Differential Privacy}
\author{Alexander Kent}
\author{Thomas B.~Berrett}
\author{Yi Yu}
\affil{Department of Statistics, University of Warwick}
\date{\today}

\begin{document}

\maketitle

\begin{abstract}
Most of the literature on differential privacy considers the \textit{item-level} case where each user has a single observation, but a growing field of interest is that of \textit{user-level} privacy where each of the $n$ users holds $T$ observations and wishes to maintain the privacy of their entire collection.

In this paper, we derive a general minimax lower bound, which shows that, for locally private user-level estimation problems, the risk cannot, in general, be made to vanish for a fixed number of users even when each user holds an arbitrarily large number of observations. We then derive matching, up to logarithmic factors, lower and upper bounds for univariate and multidimensional mean estimation, sparse mean estimation and non-parametric density estimation. In particular, with other model parameters held fixed, we observe phase transition phenomena in the minimax rates as~$T$ the number of observations each user holds varies.

In the case of (non-sparse) mean estimation and density estimation, we see that, for~$T$ below a phase transition boundary, the rate is the same as having $nT$ users in the item-level setting. Different behaviour is however observed in the case of $s$-sparse $d$-dimensional mean estimation, wherein consistent estimation is impossible when $d$ exceeds the number of observations in the item-level setting, but is possible in the user-level setting when $T \gtrsim s \log (d)$, up to logarithmic factors. This may be of independent interest for applications as an example of a high-dimensional problem that is feasible under local privacy constraints.

\end{abstract}

\section{Introduction}\label{sec1}

Advances in data science and machine learning have demonstrated the immense statistical utility of large data sets when used as training data for a range of problems. Growing user concerns and regulatory demands \citep[e.g.][]{RaziehNokhbeh:2020}, however, have raised issues of protecting the privacy of individuals whose data are collected and analysed. The desirable outcome of protecting an individual's privacy in a quantifiable manner whilst maintaining the statistical utility of the data has led to the development of formal frameworks for the notion of privacy, with \emph{differential privacy} \citep{Dwork:2006} emerging as the ``gold standard''. These ideas have since been developed and implemented by technology companies such as Google \citep{Erlingsson:2014} and Meta \citep{yousefpour:2022}, and public bodies such as the United States Census Bureau \citep{USCensus}.

First formalised in what is called the \emph{central model} of differential privacy, where there is a trusted data aggregator who has access to the original data before producing a private output statistic, there is also the \emph{local model}, where data are privatised before being sent to a data aggregator. This is a more stringent requirement but results in greater privacy protection. The local model, whilst at the time not yet formalised under the differential privacy framework, is one of the older instances of privacy, dating back to the randomised response mechanism of \cite{Warner:1965}, designed to avoid the issue of evasive responses in surveys where an individual lies when answering a potentially incriminating question.

Originating from computer science and cryptography, differential privacy has more recently attracted interest from the statistics community. A rich field of questions can be generated by aiming to introduce optimal private protocols for familiar statistical problems and thus to quantify the cost of privacy; see, for example, work on mean estimation \citep[e.g.][]{Duchi:2018}, density estimation \citep[e.g.][]{Butucea:2020,Sart:2023}, hypothesis testing \citep[e.g.][]{Canonne:2019, Gopi:2020, Berrett:2020,Lam-Weil:2022,Pensia:2023} and change point analysis \cite[e.g.][]{Berrett:2021}. There is also a growing appreciation of the connections between differential privacy and robust statistics, both in the central~\citep{Dwork:2009, liu2022differential, Asi:2023, hopkins2023robustness} and local models~\citep{Li:2022b,chhor2023robust}, and techniques from the robust statistics literature can often be leveraged to construct private estimators. Questions of optimality are usually answered by deriving information-theoretic lower bounds, for which there are now general theoretical tools developed for both the central~\citep[e.g.][]{ Wasserman:2010, Bun:2018,Cai:2021, Cai:2023} and local models \citep[e.g.][]{Duchi:2018, Acharya:2022}.

Whilst traditionally both the central and local models of differential privacy assume each user possesses a single observation, this is not always the case in practice. This can be the case when, for example, training a language model on a corpus of text each user holds, which may contain sensitive information \citep[e.g.][]{McMahan:2018}, or training a recommender system where users have rated a large collection of films \citep[e.g.][]{McSherry:2009}. More generally, such situations arise in federated learning settings \citep[e.g.][]{Wang:2019}. When a user possesses multiple observations, two frameworks are possible, the event-level setting and the stronger user-level setting. In the event-level setting, each user is subject to a level of privacy which protects against an inference attempt for any one item of the sample. In the user-level setting, each user's data must be sufficiently privatised to protect against an inference attempt on the entire sample at once. Naively applying existing techniques used in the event-level setting to the user-level setting under the local model either leads to a reduction in the privacy guarantee proportional to the number of data points a user holds, or results in estimators with greater error introduced for privatisation than the benefits provided by the additional data. As such, new methods are in demand to protect the user's entire collection of data, whilst simultaneously leveraging the additional data to provide improved performance over using only a single observation per user.

The user-level setup has been explored in the central model, where each user provides their entire collection, before privatisation, to the data aggregator \citep[e.g.][]{Yuhan:2020, Levy:2021, Ghazi:2023, Ghazi:2023:user, Bun:2023}. There also exists a line of work on user-level guarantees for data observed only sequentially \citep[e.g.][]{Dong:2023, George:2024}. There are comparatively fewer studies on the local model of user-level privacy, particularly in a statistical context. The local model of user-level privacy has, so far, been considered for a few problems such as mean estimation \citep{Girgis:2022, Bassily:2023}, stochastic convex optimisation \citep{Bassily:2023}, empirical risk minimisation \citep{Girgis:2022}, discrete density estimation \citep{Acharya:2023} and recently, beyond estimation, in testing problems \citep{Canonne:2025}. To the best of our knowledge, however, minimax optimality of existing estimators for a problem as fundamental as univariate mean estimation is unproven, and the question of what happens as the number of observations a single user holds grows to infinity has not been addressed in detail. We, in this paper, establish the minimax rates in a range of canonical statistical problems and discover interesting phase transition phenomena in these rates. After the initial version of this preprint was made available on arXiv, the works of \cite{Zhao:2024} and \cite{Ma:2024} appeared online, considering similar estimation problems to ours with some overlapping findings.

Lastly, we remark on connections between user-level local differential privacy and other decentralised learning problems. One can view local privacy as having each user's data pass through a channel that injects noise to privatise the data. Statistical inference under local information constraints \citep[see e.g.][and the references therein]{Cai:2024}, for instance, can be viewed as having each user's data pass through a channel which compresses their data into a specified number of binary digits before being passed to the data aggregator, and there exist connections in how minimax lower bounds can be obtained in both these frameworks \citep[e.g.][]{Acharya:2022}. Another similar topic is that of federated learning, introduced in \cite{McMahan:2017} (see e.g.~\citealt[]{Kairouz:2021} for a more general overview), where a central server aims to fit a model using the data held by a collection of clients, without the central server having access to all the data, due to privacy requirements and storage constraints, among other reasons. In particular, we note the similarity of our setting to that of \emph{federated differential privacy}, sharing similar motives and possible applications. User-level local differential privacy as we consider is an application of differential privacy to a distributed data setting. Federated differential privacy considers this distributed data setting also, albeit with different privacy constraints, and has seen use in problems such as non-parametric regression \citep{Cai:2024b}, transfer learning \citep{Li:2024} and bandit problems \cite{Zhou:2023}.

\subsection{General Setup}\label{sec-general-setup}
We first introduce the framework of differential privacy (DP). Given data $\{X^{(i)}\}_{i = 1}^n \subset \mathcal{Y}$, we consider a conditional distribution $Q: \sigma(\mathcal{Z}) \times \mathcal{Y}^n \to [0,1]$ induced via a random mapping from~$\mathcal{Y}^n$ to a (potentially different) separable metric space $\mathcal{Z}$, where $\sigma(\mathcal{Z})$ denotes the Borel $\sigma$-algebra on $\mathcal{Z}$. Writing $X^{(k:l)} = (X^{(k)}, X^{(k+1)}, \ldots, X^{(l)})$ and fixing $\alpha \geq 0$, the mechanism $Q$ is said to be $\alpha$-DP if
\begin{equation*}
    \sup_{S \in \sigma(\mathcal{Z})} \frac{Q(Z \in S \mid X^{(1:n)} = x^{(1:n)})}{Q(Z \in S \mid X^{(1:n)} = x'^{(1:n)})} \leq e^{\alpha}, \quad \forall x^{(1:n)}, \, x'^{(1:n)} \in \mathcal{Y}^n \quad \mbox{with} \quad \sum_{i = 1}^n \mathbbm{1}\{x^{(i)} \neq x'^{(i)}\} \leq 1.
\end{equation*} 

The value of $\alpha$ controls the strength of the privacy constraint, with smaller values imposing more restrictive conditions. In this work, we primarily consider the case $\alpha \lesssim 1$, known as the high-privacy regime. We exclude the degenerate case of $\alpha = 0$ in which no inference is possible.

For the more stringent local differential privacy (LDP) condition, we consider a family of conditional distributions $\{Q_i\}_{i = 1}^n$ where $Q_1:\sigma(\mathcal{Z}) \times \mathcal{Y} \rightarrow [0,1]$ and $Q_i:\sigma(\mathcal{Z}) \times \mathcal{Y} \times \mathcal{Z}^{i-1} \rightarrow [0,1]$ for $i \in \{2, \ldots, n\}$. Here the $i$-th individual privatises their data $X^{(i)}$ according to $Q_i$ to produce a private observation~$Z^{(i)}$. For $\alpha \geq 0$, we say that such a collection of conditional distributions $\{Q_i\}_{i=1}^n$ satisfies $\alpha$-LDP if, for all $i \in \{1, 2, \ldots, n\}$,
\begin{equation}\label{sec1:eq:alphaLDP}
    \sup_{S \in \sigma(\mathcal{Z})} \frac{Q_i(Z^{(i)} \in S \mid X^{(i)} = x^{(i)}, Z^{(1:i-1)} = z^{(1:i-1)})}{Q_i(Z^{(i)} \in S \mid X^{(i)} = x'^{(i)}, Z^{(1:i-1)} = z^{(1:i-1)})} \leq e^{\alpha}, \quad \forall x^{(i)}, \, x'^{(i)} \in \mathcal{Y} \mbox{ and } \forall z^{(1:i-1)} \in \mathcal{Z}^{i-1}.
\end{equation}
Such a collection $\{Q_i\}_{i=1}^n$ is \emph{sequentially interactive}. In the case that, for each $i \in \{1, 2, \ldots, n\}$, $Z^{(i)}$ only depends on $X^{(i)}$, we say that the collection is instead \emph{non-interactive}, and the conditional distributions take a simpler form $Q_i(Z^{(i)} \in S \mid X^{(i)} = x^{(i)})$.

We suppose that each user holds $T \geq 1$ independent and identically distributed (i.i.d.)~data points, denoted by $X_{1:T}^{(i)}$ and taking values in $\mathcal{X} = \mathcal{Y}^T$. Most of the existing differential privacy literature assumes $T=1$ whereas in this work we will consider general $T$. Let $\mathcal{Z}$ denote the output space, which may depend on $T$. We say that, for $\alpha \geq 0$, a collection of conditional distributions $\{Q_i\}_{i=1}^n$ constitutes an $\alpha$-LDP mechanism at the \emph{user-level} if, for all $i \in \{1, \ldots, n\}$, the following inequalities hold
\begin{equation}\label{sec1:eq:alphaLDP_user}
    \sup_{S \in \sigma(\mathcal{Z})} \frac{Q_i(Z^{(i)} \in S \mid X^{(i)}_{1:T} = x^{(i)}_{1:T}, Z^{(1:i-1)} = z^{(1:i-1)})}{Q_i(Z^{(i)} \in S \mid X'^{(i)}_{1:T} = x'^{(i)}_{1:T}, Z^{(1:i-1)} = z^{(1:i-1)})} \leq e^{\alpha}, \quad \forall x_{1:T}^{(i)}, x_{1:T}'^{(i)} \in \mathcal{X}, \quad z^{(1:i-1)} \in \mathcal{Z}^{i-1}.
\end{equation}
If we relax this requirement so that the inequalities only need to holds when $x^{(i)}_t = x'^{(i)}_t$ for all except at most one $t \in \{1, \ldots, T\}$, then we obtain the definition of LDP at the \emph{event-level}. Note, however, that this relaxation would not prevent privacy leakage due to repeated measures of the same individual's data. Writing $Q = \{Q_i\}_{i=1}^n$ when there is no ambiguity, we refer to any conditional distribution $Q$ satisfying~\eqref{sec1:eq:alphaLDP_user} as a user-level $\alpha$-LDP mechanism, and denote the set of all such mechanisms by $\mathcal{Q}_\alpha$. As we focus on this model of privacy in this work, we refer to user-level local differential privacy as \emph{user-level privacy} for brevity, specifying the distinction between the local and central models only when relevant.

For the minimax framework, we denote by $\mathcal{P}$ the family of distributions that an observation $X_t^{(i)}$ is generated from. To quantify the best possible performance of estimators based on the privatised data generated by mechanisms $Q$ satisfying~\eqref{sec1:eq:alphaLDP_user}, we consider the user-level $\alpha$-LDP minimax risk, that is,
\begin{equation}\label{sec2:eq:UserMinimax}
    \mathcal{R}_{n,T,\alpha}(\theta(\mathcal{P}), \Phi \circ \rho) = \inf_{Q \in \mathcal{Q}_\alpha} \inf_{\hat{\theta}} \sup_{P \in \mathcal{P}} \mathbb{E}_{P,Q}\left\{\Phi \circ \rho\left(\hat{\theta}, \, \theta\left(P\right)\right)\right\},
\end{equation}
where
\begin{itemize}
    \item for a given $P \in \mathcal{P}$, the $n$ users draw data as $X_{1:T}^{(1)}, \hdots, X_{1:T}^{(n)} \overset{\mathrm{i.i.d.}}{\sim} P^{\otimes T}$;
    \item the quantity of interest to be estimated is $\theta(P) \in \theta(\mathcal{P})$, denoting a functional on the space $\mathcal{P}$;
    \item the function $\rho$ is a metric on the space $\theta(\mathcal{P})$ and $\Phi: \mathbb{R}_{\geq 0} \rightarrow \mathbb{R}_{\geq 0}$ is a non-decreasing function with $\Phi(0) = 0$;
    \item the outermost infimum is taken over all user-level $\alpha$-LDP privacy mechanisms generating the privatised data $Z^{(i)}$;
    \item the inner infimum is taken over all measurable functions $\hat{\theta} = \hat{\theta}(Z^{(1)}, \dotsc, Z^{(n)})$ of the privatised data generated by the privacy mechanism $Q$, conditional on the raw data generated from $P \in \mathcal{P}$; and
    \item the risk is the expectation with respect to both the distribution $P$ of the data and the privacy mechanism $Q$.
\end{itemize}

We note that the usual $\alpha$-LDP risk where each user has a single item, as considered in, for example, \cite{Duchi:2018}, can be recovered as a special case of \eqref{sec2:eq:UserMinimax} by considering $\mathcal{R}_{n,1,\alpha}(\theta(\mathcal{P}), \Phi \circ \rho)$. In an abuse of notation, denote the minimax risk without privacy by $\mathcal{R}_{n,1,\infty}(\theta(\mathcal{P}), \Phi \circ \rho)$, with $\alpha = \infty$.

\begin{remark}\label{sec1:rem:Equivalence}
    As each observation a user holds is distributed according to a distribution $P$, we see that user-level privacy is the natural extension of \eqref{sec1:eq:alphaLDP} where we view each user's single data point as from the product distribution $P' = P^{\otimes T}$ and where the quantity being estimated $\theta(P') = \theta(P)$ regardless of $T$.
\end{remark}

Being a stronger notion of privacy, user-level privacy provides a stronger guarantee than event-level privacy at the cost of reduced statistical utility of the privatised output. What is not so clear however, is whether the user-level framework with $n$ users each possessing an independent sample of size $T$ is preferable in terms of statistical utility to the traditional LDP framework having $nT$ many users each with one observation. On the one hand, under user-level privacy a greater amount of noise may be needed to privatise the user's sample, suggesting the user-level setup is at a disadvantage. On the other hand, it is possible that under the user-level setup, a user can discern useful information locally from their collection of items before submitting a privatised value, potentially providing an advantage. To distinguish between these two setups, we refer to the case where each user holds a single data point as the item-level case, so we will be interested in comparing the performance of item-level $\alpha$-LDP with $nT$ users against user-level $\alpha$-LDP with $n$ users, each with a sample of size $T$. Note, however, that with the same~$\alpha$, this results in a stricter privacy protection for the user-level. This is because in the user-level setting, every $T$ observations of a particular user share $\alpha$ privacy budget, while every one observation enjoys $\alpha$ privacy budget in the item-level setting. Lacking a widely-accepted strategy for a fairer comparison, we stick to this comparison while acknowledging the stricter privacy protection imposed in user-level privacy.

\subsection{Contributions}\label{sec1:contributions}

In this paper, we aim to investigate the aforementioned question. The list of contributions is summarised below.

\begin{itemize}
    \item In Section \ref{sec2}, we discuss the case where the number of items a user holds diverges whilst the number of users is kept fixed. As a user would be able to estimate any functional of the distribution arbitrarily well given an infinite and independent sample, we consider the minimax risk of estimating a quantity where each user knows the quantity exactly, but cannot report this value without randomising it to satisfy $\alpha$-LDP. We develop general upper and lower bounds on the minimax risk for any estimation problem. The lower bound in particular has relevance beyond this thought experiment as it also applies to the user-level setting that we consider in specific problems, and will be used to characterise when phase transitions occur in the fundamental limits.
    
    \item In Section \ref{sec3}, for $d \in \mathbb{N}$, we consider $d$-dimensional mean estimation for distributions supported on unit $\ell_2$- and $\ell_\infty$-balls. Mean estimation procedures and proofs of their minimax optimality in the item-level case have been previously developed in \cite{Duchi:2018}. The problem has also been considered in the central user-level setting in \cite{Levy:2021}, and in the local setting, without showing minimax optimality, for estimators for distributions supported on the $\ell_\infty$-ball with some further assumptions in \cite{Girgis:2022} and \cite{Bassily:2023}. We build upon this prior work by only imposing a restriction on the support of data without further assumptions, considering both $\ell_2$- and $\ell_\infty$-balls, and deriving minimax rates. In particular, we show that a phase transition occurs when $T$ is larger than a quantity exponential in $n\alpha^2$ and further increases in $T$ no longer improve estimation.
    
    \item In Section \ref{sec5}, we consider sparse mean estimation. Unlike the mean estimation problems in \Cref{sec3} where, for small $T$ - up until the phase transition, the optimal minimax rate is the same (up to logarithmic factors) as having $nT$ users in the item-level setup, in \Cref{sec5} we see a prominent difference between minimax risks of the user- and item-level setups. For comparison, without privacy restrictions, assuming that the mean has at most one non-zero entry enables consistent estimators as long as the dimension is not exponential in the number of users. \cite{Duchi:2018}, however, shows that under item-level LDP, no consistent estimator exists if $n\alpha^2 = O(d)$. We show that in the user-level setting, consistent estimators exist even if $d$ is polynomial in $n\alpha^2$ as long as $\log(d) = O(T)$. This shows that there exist estimation problems that are significantly easier in the user-level setting compared to the item-level setting.
    
    \item In Section \ref{sec4}, we consider non-parametric density estimation. The estimation of discrete densities under user-level local privacy has been considered in \cite{Acharya:2023} where an estimator is developed and its minimax optimality established in a regime where the value of $T$ is not too large. We consider the case of continuous densities lying in a function family defined through a Sobolev ellipsoid condition on the coefficients of its basis expansion. As well as constructing an estimator and proving its near-minimax optimality in all regimes of $T$, we will see that unlike the two previously considered parametric problems where the phase transition boundary is exponential in $n\alpha^2$, the phase transition boundary for this non-parametric problem is polynomial in $n\alpha^2$.

    \item In Section \ref{sec:sim} we provide numerical simulations on synthetic and real-world data where the user-level setting naturally arises. We empirically validate our theoretical results and show the effectiveness of our methods, before demonstrating the use of our methods on real-world data.

    \item In the appendices, we extend the results of Section \ref{sec2} to the setting of approximate LDP. We further consider the problem of mean estimation under a generalisation of the user-level setting, where users draw their sample from differing components of a mixture distribution, instead of all being identically distributed. Finally, further numerical simulations are presented, including validating the phase transition phenomena and carrying out sensitivity analyses to inform the implementation of the user-level techniques in practice.
\end{itemize}

\subsection{Notation}
For $n \in \mathbb{N}$, we write $[n] = \{1, \hdots, n\}$. For a collection of vectors $\{x^{(i)}\}_{i \in [n]} \subset \mathbb{R}^d$, denote $x^{(1:n)} = \{x^{(1)}, \hdots, x^{(n)}\}$. For $a,b \in \mathbb{R}$, let $a \wedge b = \min\{a, b\}$ and $a\vee b = \max\{a, b\}$. For non-negative real sequences $\{a_n\}_{n \in \mathbb{N}}$ and $\{b_n\}_{n \in \mathbb{N}}$, $a_n \lesssim b_n$ denotes the existence of a constant $C>0$ such that $\limsup_{n \rightarrow \infty} a_n/b_n \leq C$, $a_n \gtrsim b_n$ denotes that $b_n \lesssim a_n$ and $a_n \asymp b_n$ denotes that $a_n \lesssim b_n \lesssim a_n$. For $x > 0$, we denote $\omega_x = e^x/(1+e^x)$. We say a random variable $X$ is $\sigma^2$-sub-Gaussian if $\mathbb{P}(|X|>\varepsilon) \leq 2e^{-\varepsilon^2/(2\sigma^2)}$ for any $\varepsilon \geq 0$. For a distribution $P$, denote the $n$-fold product distribution arising from $n$ i.i.d.~observations as $P^{\otimes n}$. For two distributions $P,Q$, let $D_{\mathrm{KL}}(P \| Q)$, $D_{\mathrm{H}}(P,Q)$ and $D_{\mathrm{TV}}(P,Q)$ denote the Kullback--Leibler divergence, Hellinger distance and total variation distance between $P$ and $Q$ respectively. Given a convex set $\mathcal{C} \subset \mathbb{R}^d$, denote the projection of $x \in \mathbb{R}^d$ onto $\mathcal{C}$ as $\Pi_\mathcal{C}(x)$. We define the Laplace distribution with scale $\lambda > 0$ to be the distribution with the the density $(2\lambda)^{-1}\exp(-|x|/\lambda)$, and refer to the case where $\lambda = 1$ as the standard Laplace distribution. For a vector $v \in \mathbb{R}^d$ denote the $j \in [d]$ co-ordinate of $v$ as $v_j$, and write $\|v\|_0 = \sum_{i=1}^d \mathbbm{1} \{ v_i \neq 0 \}$, $\|v\|_2^2 = \sum_{i=1}^d v_i^2$ and $\|v\|_\infty = \max_{i \in [d]}\{|v_i|\}$. For two vectors $v,v' \in \mathbb{R}^d$, denote the Hamming distance between $v$ and $v'$ by $H(v,v') = \|v-v'\|_0$. For any set $S$, let $|S|$ be its cardinality. For a given radius $r > 0$ and dimension $d \in \mathbb{N}$, we write $\mathbb{B}_2(r) = \{x \in \mathbb{R}^d : \|x\|_2 \leq r\}$ and $\mathbb{B}_\infty(r) = \{x \in \mathbb{R}^d : \|x\|_\infty \leq r\}$. Throughout this work, for simplicity we may omit the taking the floor and/or ceiling of quantities for which an integer value is expected. 

\section{Infinite Observations per User\texorpdfstring{: $T = \infty$}{}}\label{sec2}

In this section, we will consider the limiting behaviour as $T$ diverges to infinity. Whilst without any privacy constraints the error typically vanishes as $T \rightarrow \infty$, with privacy constraints, we will show that this is no longer the case. To understand what happens as $T$ diverges, we consider the case where each user knows exactly the functional to estimate. Given a family of distributions $\mathcal{P}$, we consider the family
\begin{equation}
    \mathcal{P}^{\infty} = \{\delta_\theta: \, \theta \in \theta(\mathcal{P})\}, \quad \text{where } \theta(\mathcal{P}) = \{\theta(P):\, P \in \mathcal{P}\}, \label{sec2:eq:infFamily}
\end{equation}
and $\delta_\theta$ is the point mass at $\theta$. The sample space is therefore $\theta(\mathcal{P})$. For any $\theta \in \theta(\mathcal{P})$, the corresponding functional to estimate for $\delta_\theta \in \mathcal{P}^\infty$ is $\theta(\delta_\theta) = \theta$. 

Our motivation is that if each user has an infinite sample size, then they would be able to estimate the functional perfectly. The problem is therefore equivalent to the case where the data distribution is a point mass on the value of the functional. With some abuse of notation and the user-level $\alpha$-LDP minimax risk defined in \eqref{sec2:eq:UserMinimax}, we define
\begin{equation}
    \mathcal{R}_{n, \infty, \alpha}(\theta(\mathcal{P}), \Phi \circ \rho) = \mathcal{R}_{n, 1, \alpha}(\theta(\mathcal{P}^{\infty}), \Phi \circ \rho), \label{sec2:eq:InfMinimaxDef}
\end{equation}
and refer to this estimation problem as the \emph{infinite-$T$} problem giving rise to the infinite-$T$ minimax rate. Of particular interest is whether $\mathcal{R}_{n, T, \alpha}(\theta(\mathcal{P}), \Phi \circ \rho)$ is of the same rate as $\mathcal{R}_{n, \infty, \alpha}(\theta(\mathcal{P}), \Phi \circ \rho)$ as $T \rightarrow \infty$, and how the user-level rate $\mathcal{R}_{n, T, \alpha}(\theta(\mathcal{P}), \Phi \circ \rho)$ compares to $\mathcal{R}_{nT, 1, \alpha}(\theta(\mathcal{P}), \Phi \circ \rho)$ - the item-level rate with an equivalent number of observations.

We first demonstrate in \Cref{sec2:ex:example} that it is wrong to take the lower bounds from the item-level case and to assume that $\mathcal{R}_{n, T, \alpha}(\theta(\mathcal{P}), \Phi \circ \rho)$ and $\mathcal{R}_{nT, 1, \alpha}(\theta(\mathcal{P}), \Phi \circ \rho)$ are of the same order. In fact, as our results in~\Cref{sec5} show, it is generally not even true that $\mathcal{R}_{n,T,\alpha}$ is lower bounded by a constant multiple of $\mathcal{R}_{nT,1,\alpha}$, so these two risks are generally incomparable.

\begin{example} \label{sec2:ex:example}
    Estimating the mean of a distribution from the family $\mathcal{P} = \{P : \mathbb{E}_P(X) \in [-1,1]\}$, we have that the user-level LDP minimax risk with respect to the squared error loss is lower bounded as
    \begin{equation}
        \mathcal{R}_{n,T,\alpha}(\theta(\mathcal{P}), (\cdot)^2) \gtrsim \min\left\{1, \frac{1}{nT\alpha^2} \right\}. \label{sec2:eq:examplefinite}
    \end{equation}
    When $T=1$, \eqref{sec2:eq:examplefinite} coincides with the item-level lower bound on $\mathcal{R}_{n, 1, \alpha}(\theta(\mathcal{P}), (\cdot)^2)$ \citep[Corollary~1 in][]{Duchi:2018}. We will show in Section \ref{sec3} that, for a certain range of $T$, \eqref{sec2:eq:examplefinite} is tight up to logarithmic factors, but not for the whole range of $T$. Indeed, we will also show that
    \begin{equation}
        \mathcal{R}_{n,T,\alpha}(\theta(\mathcal{P}), (\cdot)^2) \gtrsim e^{-12n\alpha^2}. \label{sec2:eq:exampleinfinite}          
    \end{equation}
    This second lower bound, independent of $T$, demonstrates that for a fixed $n$, even as $T \rightarrow \infty$, the risk is bounded away from $0$. The proofs of \eqref{sec2:eq:examplefinite} and \eqref{sec2:eq:exampleinfinite} are provided in \Cref{Appendix_sec2}.
\end{example}

In the rest of this section, we will show that this non-vanishing phenomenon is not specific to the estimation problem in \Cref{sec2:ex:example}, but general for user-level LDP problems.

\subsection{General Minimax Rates} 
We now introduce general upper and lower bounds for infinite-$T$ LDP estimation problems, covering both the high- and low-privacy regimes of $\alpha \lesssim 1$ and $\alpha \gtrsim 1$ respectively. The lower bound developed here will also apply to the finite-$T$ case. We also construct a general method, providing tight upper bounds for the infinite-$T$ case for a range of illustrative problems.

\begin{theorem}[General infinite-$T$ rates] \label{sec2:thm:GeneralBound}
    Given a family of distributions $\mathcal{P}$, let $N(\Delta)$ be the $\Delta$-covering number of the metric space $(\Theta, \rho)$ with $\Theta = \theta(\mathcal{P})$. The LDP minimax risk satisfies
    \begin{align}
        & \sup_{\Delta : N(2 \Delta)>1}\bigg\{ \frac{\Phi(\Delta)}{2}\bigg(1 - \frac{12 n \min\{\alpha, \alpha^2\} + \log(2)}{\log(N(2\Delta))} \bigg) \bigg\}
        \leq \mathcal{R}_{n, \infty, \alpha}(\theta(\mathcal{P}), \Phi \circ \rho) \nonumber \\
        & \hspace{5cm} \leq \inf_{\Delta : N(2 \Delta)>0} \Big\{ \Phi(\Delta) + \Phi \bigl( \mathrm{diam}(\Theta) \bigr) N(\Delta) e^{-n\min\{\alpha, \alpha^2\}/20} \Big\}, \label{sec2:eq:GeneralLB}
    \end{align}
    where $\mathrm{diam}(\Theta) = \sup_{\theta, \theta' \in \Theta} \rho(\theta, \theta')$.
\end{theorem}

The proof of \Cref{sec2:thm:GeneralBound} is provided in \Cref{Appendix_sec2}. We note the same bounds hold, up to constants, in the case of \emph{approximate} $(\alpha, \delta)$-privacy, provided $\delta \lesssim \alpha/\{n\log(n)\}$. We note that in, for example, \cite{Bassily:2015} and \cite{Bun:2019}, it is shown that once $\delta$ is less than the reciprocal of the sample size, up to logarithmic factors, approximate privacy provides no improvement in the minimax rates. In the context of user-level LDP, it is not immediately apparent whether ``sample size'' should be interpreted as $nT$ or $n$ in this context (or something entirely different). If it were the former, it would suggest that one could bypass the limiting behaviour as $T \rightarrow \infty$ by utilising approximate privacy. Our results for approximate privacy answer this in the negative. The definition and results for approximate privacy may be found in \Cref{app:sec:approxLDP}.

The upper bound in \Cref{sec2:thm:GeneralBound} is due to a non-interactive procedure to be presented and the lower bound follows from an application of Fano's inequality along with a bound on the private Kullback--Leibler divergence \citep[][Corollary~3]{Duchi:2018}. Of importance is the fact that the proof of the lower bound uses no properties of the family of distributions other than the metric entropy of the space $\theta(\mathcal{P})$. As the parameter space $\theta(\mathcal{P})$ remains the same across the item-level, user-level and infinite-$T$ setups, the lower bound in \Cref{sec2:thm:GeneralBound} is applicable across all three, leading immediately to the following result.

\begin{corollary} \label{sec2:thm:TIndepLB}
With the same conditions and notation as \Cref{sec2:thm:GeneralBound}, we have that, for all $T \geq 1$, the user-level private minimax risk is lower bounded as
    \begin{equation}
        \mathcal{R}_{n, T, \alpha}(\theta(\mathcal{P}), \Phi \circ \rho)
        \geq \sup_{\Delta : N(2 \Delta)>1}\bigg\{ \frac{\Phi(\Delta)}{2}\bigg(1 - \frac{12 n \min\{\alpha, \alpha^2\} + \log(2)}{\log(N(2\Delta))} \bigg)\bigg\}. \label{sec2:eq:TIndepLB}
    \end{equation}
\end{corollary}
In particular, letting $\Delta_{\mathrm{LB}}$ and $\Delta_{\mathrm{UB}}$ satisfy
\begin{equation*}
    N(2\Delta_{\mathrm{LB}}) \geq \exp(\lceil 24n \min\{\alpha, \alpha^2\} + 2\log(2) \rceil) \quad \mbox{and} \quad \Phi(\Delta_{\mathrm{UB}}) \geq \Phi\{\mathrm{diam}(\Theta)\} N(\Delta_{\mathrm{UB}}) e^{-n\min\{\alpha, \alpha^2\}/20},
\end{equation*}
we have that
\begin{equation*}
    \Phi(\Delta_{\mathrm{LB}})
    \lesssim \mathcal{R}_{n, \infty, \alpha}(\theta(\mathcal{P}), \Phi \circ \rho)
    \lesssim \Phi(\Delta_{\mathrm{UB}}). 
\end{equation*}

\begin{remark} \label{sec2:rem:lower}
    As the lower bound in \eqref{sec2:eq:TIndepLB} is independent of $T$, an important consequence of \Cref{sec2:thm:TIndepLB} is that for any user-level LDP estimation problem it is impossible for the minimax risk to vanish as $T$ diverges, provided that there exists $\Delta>0$ such that $12 n \min\{\alpha, \alpha^2\} + \log(2) < \log(N(2\Delta))$.
\end{remark}

We now construct a general estimator which attains the upper bound in \eqref{sec2:eq:GeneralLB}. The estimator involves constructing a $\Delta$-covering of the $\rho$-metric space of interest, and uses a private voting procedure to select the ball believed to be closest to the unknown parameter. We note that this voting-based strategy is commonly used in the privacy literature, sometimes under the name of unary encoding. See e.g.~\cite{Wang:2017} for an overview of this technique, and e.g.~\cite{Girgis:2022} for an application in user-level privacy. The procedure is formalised in \Cref{alg:CoverSelect}. Note that the privatisation in \eqref{eq-general-upper-bound-privacy} is non-interactive and satisfies $\alpha$-LDP by \Cref{app:lem:RR}.

\begin{algorithm}
    \caption{\textsc{CoverSelect}\,$(\{X^{(i)}\}_{i\in[n]},\,n,\,\alpha,\,\Delta)$}
    \begin{algorithmic}[1]
        \State Denote $N(\Delta)$ covering number, let $\Delta$-cover $\mathcal{B}=\{B_j\}_{j=1}^{N(\Delta)}$ with centres $\theta_j$.
        \State Initialise $B'_1 = B_1$ and let $B'_j = B_j \setminus \bigcup_{k=1}^{j-1} B_k$ for $j \geq 1$. \refstepcounter{equation}\label{sec2:eq:nonoverlapping} \hfill(\theequation)
        
        \For{$i=1,\dots,n$}
          \For{$j=1,\dots,N(\Delta)$}
            \State Let $V^{(i)}_j = \mathbbm{1}\{X^{(i)}\in B'_j\}$ and sample $U_{i,j}\sim\mathrm{Unif}[0,1]$.
            \State Let $\widetilde{V}^{(i)}_j = V^{(i)}_j \mathbbm{1}\{U_{i,j} \leq \omega_{\alpha/2}\} + (1 - V^{(i)}_j)\mathbbm{1}\{U_{i,j} > \omega_{\alpha/2}\}$ .\refstepcounter{equation}\label{eq-general-upper-bound-privacy} \hfill(\theequation)
          \EndFor
        \EndFor
        
        \State Denote $j^\ast = \min\arg\max_{j\in[N(\Delta)]}\left\{ \sum_{i=1}^n \widetilde{V}^{(i)}_j \right\}$, set $\hat{\theta} = \theta_{j^\ast}$ and \Return $\hat{\theta}$.
    \end{algorithmic}
    \label{alg:CoverSelect}
\end{algorithm}
\noindent

With this algorithm in hand, we focus on the form of the upper bound in \eqref{sec2:eq:GeneralLB}, consisting of two terms. The term $\Phi(\Delta)$ is the error that arises when the correct ball is chosen. In this case, the error is at most~$\Phi(\Delta)$ by the definition of the constructed covering of the space. The other term is the error incurred when a different ball is chosen, in which case the error varies depending on the distance between the parameter and the centre of the chosen ball. This error is therefore upper bounded by the worst case distance $\Phi(\mathrm{diam}(\Theta))$. This adverse outcome occurs when, after privatising the votes, the number of votes for the correct ball is less than that of some other ball, and the probability of this adverse outcome is upper bounded by $N(\Delta) e^{-n\min\{\alpha, \alpha^2\}/20}$. The $\alpha^2$ term can be obtained by Hoeffding's inequality when considering the privatised votes in the case $\alpha \lesssim 1$. For $\alpha \gtrsim 1$ however, Hoeffding's inequality is loose, resulting in a term independent in $\alpha$, whereas using the stronger Chernoff--Hoeffding bound \citep[Theorem~1]{Hoeffding:1963} recovers a term of $\alpha$. In particular, we see the error vanishes as $\alpha \rightarrow \infty$, provided $\Delta$ is taken to zero suitably.

There is a trade-off between taking $\Delta$ smaller to obtain a finer covering of $\Theta$, which would result in a smaller $\Phi(\Delta)$, and the difficulty in identifying the ball the users are voting for after privatisation as~$N(\Delta)$ increases. Loosely speaking, as $\Delta$ decreases, there are more candidates to vote for, and so more opportunities for an incorrect ball to obtain more votes after privatisation by random chance. When obtaining upper bounds, we set $\Delta$ to balance these two terms to obtain a sharp upper bound.

\subsection{Application of Theorem \ref{sec2:thm:GeneralBound}} \label{sec2:sec:applicationofgeneralbound}

In this subsection we apply \Cref{sec2:thm:GeneralBound} in a range of problems. For suitably chosen values of $\Delta$ we obtain minimax rates up to constants, and constants in the exponent in the case of exponential bounds. The proofs are contained in \Cref{Appendix_sec2}.

\subsubsection{Mean Estimation}
Consider the mean estimation problems for the classes of distributions
\begin{equation}
    \mathcal{P}_d = \left\{ P : \mathrm{supp}(P) \subseteq \mathbb{B}_\infty(1) \subset \mathbb{R}^d  \right\} \quad \mbox{and} \quad \mathcal{P}_d' = \left\{ P : \mathrm{supp}(P) \subseteq \mathbb{B}_2(1) \subset \mathbb{R}^d  \right\}, \label{sec2:eq:MeanClasses}
\end{equation}
with the mean of a distribution $P$ denoted as $\theta(P) = \mathbb{E}_P(X)$.

\begin{proposition}\label{sec2:thm:Means} Let $\mathcal{R}_{n, \infty, \alpha}$ be the infinite-$T$ minimax risk rate defined in \eqref{sec2:eq:InfMinimaxDef}.

    (i) Assuming that $n\min\{\alpha, \alpha^2\} > 60d\log(6d)$, it holds that
    \begin{equation}
        e^{-C n \min\{\alpha, \alpha^2\} / d}
        \lesssim \mathcal{R}_{n, \infty, \alpha}(\theta(\mathcal{P}_d), \|\cdot\|_2^2)
        \lesssim e^{-c n \min\{\alpha, \alpha^2\} / d}, \label{sec2:eq:linfMeanStatement}
    \end{equation}
    where $\mathcal{P}_d$ is denoted in \eqref{sec2:eq:MeanClasses}, and $C > c > 0$ are absolute constants.

    (ii) Assuming that $n\min\{\alpha, \alpha^2\} > 60d$, it holds that
    \begin{equation}
        e^{-C' n \min\{\alpha, \alpha^2\} / d}
        \lesssim \mathcal{R}_{n, \infty, \alpha}(\theta(\mathcal{P}_d'), \|\cdot\|_2^2)
        \lesssim e^{-c' n \min\{\alpha, \alpha^2\} / d}, \label{sec2:eq:l2MeanStatement}
    \end{equation}
    where $\mathcal{P}'_d$ is denoted in \eqref{sec2:eq:MeanClasses}, and $C' > c' > 0$ are absolute constants.

\end{proposition}

The primary component of the proof of \Cref{sec2:thm:Means} is to upper and lower bound the covering numbers of the spaces $\theta(\mathcal{P}_d)$ and $\theta(\mathcal{P}_d')$, before applying the general results \Cref{sec2:thm:GeneralBound}. 

For $\alpha \lesssim 1$, the item-level counterpart of \Cref{sec2:thm:Means} is shown in \cite{Duchi:2018}, where the minimax rates are
\begin{equation}\label{eq-duchi-mean-est-rates}
    \mathcal{R}_{n, 1, \alpha}(\theta(\mathcal{P}_d), \|\cdot\|_2^2) \asymp d^2/(n\alpha^2) \quad \mbox{and} \quad \mathcal{R}_{n, 1, \alpha}(\theta(\mathcal{P}'_d), \|\cdot\|_2^2) \asymp d/(n\alpha^2).
\end{equation}
Compared to these polynomial rates, \Cref{sec2:thm:Means} exhibits exponentially decaying rates. In order to achieve vanishing minimax rates, in the $\ell_2$-ball scenarios ($\mathcal{P}'_d$), the same dependence on $d$ is observed in both the infinite-$T$ and item-level cases, i.e.~$d/(n\alpha^2) \to 0$. On the other hand, such dependence is substantially weakened in the infinite-$T$ case when considering the $\ell_{\infty}$-ball scenario ($\mathcal{P}_d$). While in both setups the $\ell_{\infty}$-ball scenario suffers from a worse dependence on the dimension $d$ compared to the $\ell_{2}$-ball one, the infinite-$T$ setup only requires $n\alpha^2 \gtrsim d\log(d)$, as opposed to $n\alpha^2 \gtrsim d^2$ for the item-level rate.

\subsubsection{Sparse Mean Estimation}
Consider the  high-dimensional sparse mean estimation problem where the family of distributions is
\begin{equation}
    \mathcal{P}_{d, s} = \left\{ P : \mathrm{supp}(P) \subseteq \mathbb{B}_\infty(1) \subset \mathbb{R}^d,\ \|\mathbb{E}_P(X)\|_0 \leq s \right\}, \label{sec2:eq:SparseFamily}
\end{equation}
with the mean of distribution $P$ denoted as $\theta(P) = \mathbb{E}_P(X)$.
\begin{proposition} \label{sec2:thm:Sparse}
    Assuming that $n\min\{\alpha, \alpha^2\} > 120s\log(6d)$, it holds that
    \begin{equation}
        e^{-C n \min\{\alpha, \alpha^2\}/s}
        \lesssim \mathcal{R}_{n, \infty, \alpha}(\theta(\mathcal{P}_{d, s}), \|\cdot\|_2^2)
        \lesssim e^{-cn \min\{\alpha, \alpha^2\} / s}, \label{sec2:eq:sparseMeanStatement}
    \end{equation}
    where $\mathcal{P}_{d,s}$ is denoted in \eqref{sec2:eq:SparseFamily}, $\mathcal{R}_{n, \infty, \alpha}$ is the infinite-T minimax risk rate defined in \eqref{sec2:eq:InfMinimaxDef} and $C > c > 0$ are absolute constants.
\end{proposition}

\Cref{sec2:thm:Sparse} is shown by obtaining suitable bounds on the covering number of the metric space $(\theta(\mathcal{P}_{d, s}), \|\cdot\|_2)$, and by applying \Cref{sec2:thm:GeneralBound}. The upper bound on the covering number of $\theta(\mathcal{P}_{d, s})$ comes from considering a union of $s$-dimensional subspaces which cover the space $\theta(\mathcal{P}_{d,r})$. For the lower bound we explicitly construct a packing of $\theta(\mathcal{P}_{d, s})$, described in the proof of \Cref{app:lem:sparsecoverlower} in \Cref{app:sec:misc}. This is an interesting by-product on its own and gives a lower bound on the covering number.

Without privacy constraints, one can show that the minimax risk for $s$-sparse $d$-dimensional Gaussian distributions (with $O(1)$ variance) is of order $s\log(ed/s)/n$ \citep[e.g.][Theorem~8.2.6]{hanfang}, showing that the dimension $d$ can grow exponentially large in $n$. On the other hand, for $\alpha \lesssim 1$, it is shown in Theorem 3 in \cite{Acharya:2022} that for $s > 4\log(d)$, $\mathcal{R}_{n, 1, \alpha}(\theta(\mathcal{P}_{d, s}), \|\cdot\|_2^2) \asymp sd/(n\alpha^2)$.\footnote{For $s = 1$ and $\alpha \lesssim 1$, \cite{Duchi:2018} shows that the non-interactive minimax risk, as opposed to the interactive risks considered throughout this work and in \cite{Acharya:2022}, is of order $d\log(2d)/(n\alpha^2)$.}

We see that, in the high privacy setting $\alpha \lesssim 1$, item-level LDP constraints are disastrous for sparse mean estimation, where the dimension cannot be linear in $n\alpha^2$ for consistent estimation, let alone exponential. On the other hand, we see that in the infinite-$T$ case, the rates in \Cref{sec2:thm:Sparse} and \eqref{sec2:eq:linfMeanStatement} match when replacing $d$ by $s$. 
This shows that in the infinite-$T$ case, the fundamental limits of the $s$-dimensional mean and $s$-sparse $d$-dimensional mean problems are the same, up to constants in exponential. This hints at potential benefits of user-level LDP for finite-$T$ in the context of sparse mean estimation to be explored further in \Cref{sec5}.

\subsubsection{Non-Parametric Density Estimation}
For the non-parametric density estimation problem we consider the class of distributions whose densities satisfy a Sobolev regularity condition, in particular that given by an ellipsoid in a Sobolev
space \citep[e.g.~equation~(1.91) in][]{Tsybakov:2009} defined below. We define the trigonometric basis
\begin{equation}
    \varphi_1(x) = 1, \quad \varphi_{2j}(x) = \sqrt{2}\cos(2 \pi j x) \quad \mbox{and} \quad \varphi_{2j + 1}(x) = \sqrt{2}\sin(2 \pi j x), \quad j \in \mathbb{N}, \label{sec2:eq:TrigBasis}
\end{equation}
Given a smoothness parameter $\beta \in \mathbb{N}$ and a radius $r > 0$, the Sobolev class of functions of smoothness~$\beta$ and radius $r$ is given by
\begin{equation} \label{sec2:eq:Sobolev}
    \mathcal{S}_{\beta, r} = \bigg\{ f \in L^2([0,1]) : \, f = \sum_{j=1}^{\infty} \theta_j \varphi_j,\, \sum_{j=1}^{\infty} j^{2\beta} \theta_j^2 \leq r^2 \bigg\}.
\end{equation}
This leads to the definition of the sub-class of Sobolev densities
\begin{align} \label{sec2:eq:SobolevDensity}
    \mathcal{F}_{\beta, r} = \bigg\{ f \in \mathcal{S}_{\beta, r} : \, f > 0, \int f(x) \diff x = 1 \bigg\}.
\end{align}
We denote by $P_f$ the distribution induced by the density $f$ and define the class of all such distributions
\begin{equation}
    \mathcal{P}_{\beta, r} = \left\{ P_f : f \in \mathcal{F}_{\beta, r} \right\}, \label{sec2:eq:InducedDensityFamily}
\end{equation}
where the functional to estimate is $\theta(P_f) = f$ and the error is measured by the squared $L^2([0,1])$ norm denoted by $\|\cdot\|_2^2$ where
\begin{equation}
    \|f\|_2^2 = \int\{f(x)\}^2\,\diff{x}. \label{sec2:eq:L2FuncNormDef}
\end{equation}

\begin{proposition} \label{sec2:thm:Density}
There exists a constant $C_\beta'>0$ depending only on $\beta$ such that, whenever $n\min\{\alpha, \alpha^2\} \geq C_\beta'$, it holds that
    \begin{equation}
        \frac{c_\beta}{(n \min\{\alpha, \alpha^2\})^{2 \beta}}
        \lesssim \mathcal{R}_{n, \infty, \alpha}(\theta(\mathcal{P}_{\beta, 1}), \|\cdot\|_2^2)
        \lesssim \frac{C_\beta}{(n \min\{\alpha, \alpha^2\})^{2 \beta}}, \label{sec2:eq:densityStatement}
    \end{equation}
    where $\mathcal{P}_{\beta, 1}$ is denoted in \eqref{sec2:eq:InducedDensityFamily}, $\mathcal{R}_{n, \infty, \alpha}$ is the infinite-$T$ minimax risk rate defined in \eqref{sec2:eq:InfMinimaxDef} and $C_\beta > c_\beta > 0$ are constants depending only on $\beta$.
\end{proposition}
The upper and lower bounds of \Cref{sec2:thm:Density} are derived by obtaining suitable upper and lower bounds on the covering number of the space $\mathcal{F}_{\beta, 1}$ and by applying \Cref{sec2:thm:GeneralBound}. For the upper bound on the complexity of $\mathcal{F}_{\beta, 1}$, it suffices to use the standard upper bound on the covering number of $\mathcal{S}_{\beta, r}$ \citep[e.g.][Example 5.12]{Wainwright:2019}; for the lower bound, we explicitly construct a packing of $\mathcal{F}_{\beta, 1}$, described in the proof of \Cref{app:lem:densitycoverlower} in \Cref{app:sec:misc}, which is again a by-product of its own interest.

Compared to \Cref{sec2:thm:Means}, we see that the minimax risk is polynomial in $n\min\{\alpha, \alpha^2\}$ rather than exponential. This occurs due to the fact that the covering number of the $\mathcal{F}_{\beta, 1}$ is exponential in $1/\Delta$, rather than polynomial as is the case for the spaces arising from the classes in \eqref{sec2:eq:MeanClasses} and \eqref{sec2:eq:SparseFamily}. This also shows a significant difference between parametric and non-parametric rates in user-level LDP settings.

\subsection{Discussion}
In this section we have discussed the minimax rates for the infinite-$T$ case, where the data are drawn from distributions in $\mathcal{P}^{\infty}$, defined in \eqref{sec2:eq:infFamily}. Comparisons of the minimax rates corresponding to $\mathcal{P}^{\infty}$ and those to $\mathcal{P}$, in the three problems we discussed in this section are collected in \Cref{tab-1}, where we suppress multiplicative constants and constants in exponents.

\begin{table}[H]
\centering
\begin{tabular}{lccc}
\hline\hline
& $d$-dim.~mean ($\ell_2$-ball) & Sparse mean & Density (Sobolev $\beta$-smooth) \\ \hline
Distribution class $\mathcal{P}$   & $ d/(n\alpha^2)$  & $sd/(n\alpha^2)$  & $(n\alpha^2)^{-2\beta/(2\beta + 2)}$     \\ 
Distribution class $\mathcal{P}^\infty$ &  $e^{-n\alpha^2/d}$ & $e^{-n\alpha^2/s}$ & $(n\alpha^2)^{-2\beta}$ \\ \hline
\end{tabular}
\caption{Comparison, for $\alpha \lesssim 1$, of infinite-$T$ user-level rates, and item-level mean estimation and density rates derived in \cite{Duchi:2018}, and sparse mean estimation rate derived in \cite{Acharya:2022}. Multiplicative constants and constants in exponents are suppressed for illustration convenience.} \label{tab-1}
\end{table}

We note the significant difference between the two cases, for example in the user-level setting with the exponential rate in mean estimation, or the lack of dependence on $d$ in the sparse rate except through the required assumption $n\min\{\alpha, \alpha^2\} \gtrsim 120\log(6d)$, which is itself no stronger (up to constants) than what is assumed in the non-private analogue. We also note that the rate is not always exponential, as seen in the non-parametric density estimation, where a polynomial rate is still observed, but with a different exponent from that in the item-level case.

We observe similar behaviour to \Cref{tab-1} in the low-privacy regime. Comparing to known rates from \citet[Theorem~3]{Acharya:2022}, for $1 \lesssim \alpha \lesssim d$, the lower bound for mean estimation takes the form $d^2/(n\alpha)$. Hence, the effective sample size is $n\alpha$, which matches that observed in the exponential of our $T$-independent lower bound for the $\alpha \gtrsim 1$ setting.

With these results in hand for the infinite-$T$ case we will next consider the finite-$T$ user-level LDP case. By virtue of \Cref{sec2:thm:TIndepLB}, we have that all the lower bounds developed thus far also hold for the minimax risk for finite-$T$. We note the parallels to \cite{Levy:2021}, where it was shown that for the problems considered therein, the minimax risk for the central user-level framework does not vanish even as $T \rightarrow \infty$, answering a conjecture of \cite{Amin:2019}. We thus resolve the local model analogue of this conjecture where one asks what these $T$-independent lower bounds will be in the local model. Interestingly, though the $T$-independent lower bounds differ in the high-privacy setting where $\alpha \lesssim 1$, being $e^{-n\alpha^2}$ and $e^{-n\alpha}$ in the local and central settings respectively for mean estimation, in the low-privacy setting of $\alpha \gtrsim 1$, the $T$-independent lower bounds coincide, both taking the form $e^{-n\alpha}$. We recall we also obtain generalisations of the results of this section under the setting of approximate privacy which may be found in \Cref{app:sec:approxLDP}.

Lastly, intuition suggests that we might expect that as $T \rightarrow \infty$ with other quantities held fixed, the minimax risk for user-level LDP will demonstrate similar behaviour in the limit as in the infinite-$T$ cases above as each user can estimate the functional increasingly well. We will shortly see that this is indeed the case, with the user-level rates demonstrating a phase transition once $T$ is sufficiently large, matching the rates we obtained in this section.

\section{Mean Estimation}\label{sec3}
Recalling the families~$\mathcal{P}_{d}$ and $\mathcal{P}_{d}'$ defined in \eqref{sec2:eq:MeanClasses}, we estimate the mean of a distribution $P$ so that $\theta(P) = \mathbb{E}_P(X)$ and measure the error by the squared $\ell_2$-loss in both the high- and low-privacy regimes. In what follows, denote $z_{\alpha, d} = \min\{\alpha, \alpha^2, d\}$. 

\begin{theorem}\label{sec3:thm:main} Let $\mathcal{R}_{n, T, \alpha}$ be the user-level minimax risk rate defined in \eqref{sec2:eq:UserMinimax}.

    (i) Assuming that $n\min\{\alpha, \alpha^2\} > \widetilde{C}d\log(ed)$, where $\widetilde{C} > 0$ is an absolute constant, we have that
    \begin{equation}
        \frac{d^2}{nTz_{\alpha, d}} \vee e^{-Cn\min\{\alpha, \alpha^2\}/d}
        \lesssim \mathcal{R}_{n, T, \alpha}(\theta(\mathcal{P}_{d}), \|\cdot\|_2^2)
        \lesssim
        \frac{d^2 \log(nTz_{\alpha, d}/d)}{nTz_{\alpha, d}} + e^{-cn\min\{\alpha, \alpha^2\}/d}, \label{sec3:eq:linfMeanStatement} 
    \end{equation}
    where $\mathcal{P}_d$ is denoted in \eqref{sec2:eq:MeanClasses} for the $\ell_{\infty}$ case,  and $C > c > 0$ are absolute constants.

    (ii) Assuming that $n\min\{\alpha, \alpha^2\} > \widetilde{C}'d\log(ed)$, and $n\alpha^2 > \widetilde{C}'d^2 \mathbbm{1}\{\alpha > 1\}$, where $\widetilde{C}' > 0$ is an absolute constant, we have that
    \begin{equation*}
        \frac{d}{nTz_{\alpha, d}} \vee e^{-C'n\min\{\alpha, \alpha^2\}/d}
        \lesssim \mathcal{R}_{n, T, \alpha}(\theta(\mathcal{P}_{d}'), \|\cdot\|_2^2)
        \lesssim
        \min\left\{1, \frac{d \{\log(nTz_{\alpha, d})\}^2}{nTz_{\alpha, d}} + e^{-c'n\min\{\alpha, \alpha^2\}/d} \right\},
    \end{equation*}
    where $\mathcal{P}'_d$ is denoted in \eqref{sec2:eq:MeanClasses} for the $\ell_2$ case, and $C' > c' > 0$ are absolute constants.
\end{theorem}
In both cases, we see that lower bounds are the maxima of two parts, the polynomial term analogous to the standard item-level rate \cite[Corollary 4]{Duchi:2018}, and the exponential term inherited from \Cref{sec2:thm:TIndepLB}.

The estimation methods in deriving upper bounds are similar to that in \cite{Girgis:2022} and \cite{Bassily:2023}, and repeat standard user-level estimation techniques of first localising to a small region where the mean is believed to lie, and then obtaining a refined estimator on the chosen region. 

The exponential-decay minimax rate once $T$ is sufficiently large is a phenomenon new to the literature, and confirms the intuition of \Cref{sec2} that the infinite-$T$ rate, for this problem given by~\Cref{sec2:thm:Means}, should coincide with the finite-$T$ rate for large $T$. 

Regarding the low-privacy regime, the condition $n\alpha^2 \gtrsim d^2 \mathbbm{1}\{\alpha > 1\}$ is necessary for our lower bounds to be valid when $\alpha \gtrsim 1$, but the upper bounds can hold without it. The necessity of this condition for valid lower bounds when $\alpha \gtrsim 1$ has similarly been observed in discrete density estimation under user-level LDP in \cite{Acharya:2023}.

In the rest of this section, we present the estimation methods for the univariate case in \Cref{sec3:univariate}, the multivariate case with the $\ell_{\infty}$-ball in \Cref{sec3:linfball} and the multivariate case with the $\ell_2$-ball in \Cref{sec-multi-ell2}, with discussion in \Cref{sec-3.4}. The upper bounds follow from the analyses, in \Cref{app:sec:Upper}, of the constructed estimators, with the lower bounds proven in \Cref{app:sec:Lower}. We discuss the computational complexity of these procedures in \Cref{app:sec:CompComplex}. For brevity, we provide the algorithms for the high-privacy case $\alpha \lesssim 1$, with the modifications requires to handle the low-privacy setting detailed in \Cref{app:sec:meanlowpriv}

\begin{remark}
    The special case of Bernoulli mean estimation is important due to its ubiquity, and it is common to focus on this setting \citep[e.g.][]{Cummings:2022, George:2024}. The constructions used to obtain the lower bounds in \Cref{sec3:thm:main} are in fact (rescaled) Bernoulli distributions, and so the rates are tight even when restricting to this class. 
\end{remark}

\subsection{Univariate Procedure} \label{sec3:univariate}
The estimation procedure developed for the univariate case will be used as a building block for both the $\ell_2$- and $\ell_\infty$-ball cases of the multivariate problem. The intuition for the univariate case is as follows. As each user has an i.i.d.~sample of size $T$, each user is able to estimate the mean within $O(1/T)$ mean-squared-error accuracy. For $\alpha \lesssim 1$, if we were to naively privatise this procedure via, e.g.~the Laplace mechanism, to ensure $\alpha$-LDP the added noise would need to have $O(1/\alpha^2)$ variance, dominating $O(1/T)$.
    
To avoid this problem, the procedure splits the users into two halves. Without loss of generality, assume that the number of the users $n$ is even. We partition $[-1, 1]$ into non-overlapping sub-intervals of width $2\Delta$, with $\Delta$ to be specified. The first half of users calculate their own sample means, identify which sub-intervals contain their sample means and publicise indicators. The central statistician then produces an interval of width $O(\Delta)$ that will contain the mean with high probability.

The second half calculate their own sample means, project the sample means onto the interval output from the first half and apply the Laplace mechanism. As the width of the initially estimated interval is $O(\Delta)$ rather than $O(1)$, it suffices to perturb with noise with variance (for $\alpha \lesssim 1$) of order $O(\Delta^2/\alpha^2)$ rather than $O(1/\alpha^2)$. By choosing $\Delta = O(1/\sqrt{T})$, up to logarithmic factors, we are able to take advantage of the $T$ observations each user holds. 

We note that for large values of $T$ we do not make use of all $T$ observations, but instead cap the value of $T$ at $T^\ast$, a value exponential in $n\alpha^2$, making users discard excess data. This step is taken to prevent the value of $\Delta$ becoming too small, which would result in the total number of sub-intervals becoming too large, and hence impossible to reliably select the correct interval from the private votes. Further, discarding excess observations does not reduce the utility of our final estimator. To be specific, using only $T^*$ observations in calculating sample means in \eqref{app:eq:unilocalisationestimator} and \eqref{sec3:eq:UniRefinedEstimator} does not reduce the accuracy of the final estimator as if all $T$ observations were used. This is because the estimation error rate is dominated by either $\Delta$ or the added Laplace noise required for privacy. As a by-product, this also results in a lower storage and computational cost. 

We first introduce our main sub-routine in \Cref{alg:UserLevelMean}, followed by the univariate mean estimation procedure in \Cref{alg:PrivateUniMean}. The privatisation in \eqref{sec3:eq:uniGRR} satisfies $\alpha$-LDP by \Cref{app:lem:RR}, as does the estimator $\hat{\theta}^{(i)}$ in \eqref{sec3:eq:UniRefinedEstimator} due to being an instance of the Laplace mechanism.

\begin{algorithm}
    \caption{\textsc{UserLevelMean}\,$(\{X^{(i)}_{t}\}_{i\in[n],t\in[T]},\,n,\,T,\,\alpha,\,\Delta)$}
    \begin{algorithmic}[1]
        \State Let $N(\Delta) = \lceil 1/\Delta \rceil$;\; $I_j = [-1 + 2(j-1)\Delta,\; -1 + 2j\Delta)$ for all $j \in [N(\Delta)]$;\; and $I_{N(\Delta)} = [-1 + 2\{N(\Delta)-1\}\Delta,\; -1 + 2N(\Delta)\Delta]$.
        \For{$i=1,\dots,n/2$}
          \State Calculate $\hat{\theta}^{(i)} = (T^{\ast})^{-1}\sum_{t=1}^{T^{\ast}} X^{(i)}_{t}$.
          \refstepcounter{equation}\label{app:eq:unilocalisationestimator} \hfill(\theequation)
          \For{$j=1,\dots,N(\Delta)$}
            \State Let $V^{(i)}_{j} = \mathbbm{1}\{\hat{\theta}^{(i)}\in I_j\}$ and sample $U_{i,j}\sim\mathrm{Unif}[0,1]$.
            \State $\widetilde{V}^{(i)}_j = V^{(i)}_j \mathbbm{1}\{U_{i,j} \leq \omega_{\alpha/2}\} + (1 - V^{(i)}_j)\mathbbm{1}\{U_{i,j} > \omega_{\alpha/2}\}$.
            \refstepcounter{equation}\label{sec3:eq:uniGRR} \hfill(\theequation)
          \EndFor
        \EndFor
        \State Let $j^{\ast} = \min\arg\max_{j\in[N(\Delta)]}\sum_{i=1}^{n/2}\widetilde{V}^{(i)}_{j}$. \refstepcounter{equation}\label{sec3:eq:unichosenindex} \hfill(\theequation) 
        \State Denote $L,U$ the lower/upper endpoints of $I_{j^{\ast}}$ and let $\tilde{L} = L-2\Delta,  \tilde{U} = U+2\Delta$, $\tilde{I}_{j^{\ast}} = [\tilde{L},\tilde{U}]$.
        \refstepcounter{equation}\label{sec3:eq:uniMeanIntervalExpand}\hfill(\theequation)

        \For{$i=n/2+1,\dots,n$}
            \State Sample $\ell_i \sim \mathrm{Laplace}(1)$ and let $\hat{\theta}^{(i)} = \Pi_{\tilde{I}_{j^{\ast}}}\{(T^\ast)^{-1}\sum_{t=1}^{T^{\ast}} X^{(i)}_{t}\} + (6\Delta/\alpha)\,\ell_i$. \refstepcounter{equation}\label{sec3:eq:UniRefinedEstimator}\hfill(\theequation)
        \EndFor
        \State Let $\hat{\theta} = (n/2)^{-1}\sum_{i=n/2+1}^{n}\hat{\theta}^{(i)}$ and
        \Return $\hat{\theta}$.
    \end{algorithmic}
    \label{alg:UserLevelMean}
\end{algorithm}

\begin{algorithm}
    \caption{\textsc{PrivateUniMean}\,$(\{X^{(i)}_{t}\}_{i\in[n],t\in[T]},\,n,\,T,\,\alpha,\,K)$}
    \begin{algorithmic}[1]
        \State Let $T^{\ast} = T \wedge \exp(n\alpha^{2}/K)$ and $\Delta = \{2\log(nT^{\ast}\alpha^{2})/T^{\ast}\}^{1/2}$ \refstepcounter{equation}\label{sec3:eq:deltaval}\hfill(\theequation)
        \State Let $\hat{\theta} = \textsc{UserLevelMean}(\{X_t^{(i)}\}_{i \in [n], t \in [T^\ast]}, n, T^\ast, \alpha, \Delta)$ and
        \Return $\hat{\theta}$.
    \end{algorithmic}
    \label{alg:PrivateUniMean}
\end{algorithm}

With this estimator for the univariate case, we now consider the multivariate generalisations, where the univariate procedure is applied co-ordinate-wise. For the $\ell_\infty$-ball this can be done to the original data, but for the $\ell_2$-ball we first apply a transformation.

\subsection{Multivariate Procedure for \texorpdfstring{$\ell_\infty$}{Sup-Norm}-ball} \label{sec3:linfball}

For the multivariate procedure for the $\ell_\infty$-ball, we split the $n$ users into $d$ many equally-sized folds, each of which is assigned to estimate a co-ordinate using the univariate procedure detailed in \Cref{sec3:univariate}.  Since we assume that $n \alpha^2 \gtrsim d \log(ed)$, we may assume without loss of generality that $n$ is a multiple of $2d$. The procedure is formalised in \Cref{alg:PrivateMultiDimMean}, and the procedure is $\alpha$-LDP by virtue of the privacy of the univariate sub-procedure.
\begin{algorithm}
    \caption{\textsc{PrivateMultiDimMean}\,$(\{X^{(i)}_{t}\}_{i\in[n],t\in[T]},\,n,\,T,\,\alpha,\,K)$}
    \begin{algorithmic}[1]
        \State Let $T^\ast = T \wedge \exp\{n\alpha^2/Kd\}$ and $\Delta = \{2\log(nT^{\ast}\alpha^{2}/d)/T^{\ast}\}^{1/2}$. Let $N_j = N_{j, 1} \cup N_{j, 2}$ where
        \newline
        $N_{j,1} = \left\{(j-1)\frac{n}{2d} + 1, \hdots, j\frac{n}{2d} \right\}$, \,
        $N_{j,2} = \left\{(d+j-1)\frac{n}{2d} + 1, \hdots, (d+j)\frac{n}{2d} \right\}$, for $j \in [d]$.
        \refstepcounter{equation}\label{sec3:eq:linfGroups}\hfill(\theequation)
        \For{$j = 1, \hdots, d$}
          \State Calculate $\hat{\theta}_j = \textsc{UserLevelMean}(\{X^{(i)}_{t}\}_{i \in N_j, t \in [T^\ast]},\, n,\, T^\ast,\,\alpha,\,\Delta)$.
        \EndFor
    
        \State Let $\hat{\theta} = \sum_{j=1}^d e_j \hat{\theta}_{j}$ where $e_j$ is the $j$-th standard basis vector of $\mathbb{R}^d$, and
        \Return $\hat{\theta}$. \refstepcounter{equation}\label{sec3:eq:linfFinalEstimator}\hfill(\theequation)
    \end{algorithmic}
    \label{alg:PrivateMultiDimMean}
\end{algorithm}

\subsection{Multivariate Procedure for \texorpdfstring{$\ell_2$}{Euclidean}-ball}
\label{sec-multi-ell2}  

For the $\ell_\infty$-ball, splitting the sample and applying the univariate estimator to each co-ordinate attains the optimal rate.  This procedure however does not attain the optimal rate for estimation on the $\ell_2$-ball. As seen in the lower bound in \Cref{sec3:thm:main}, for estimation on the $\ell_2$-ball, the dependence on the dimension is linear rather than quadratic, but applying the same sample splitting method with no modification yields a quadratic rate.  We instead carry out an encoding step which, loosely speaking, improves the concentration of data by a factor of $O(1/\sqrt{d})$, resulting in a linear rate in~$d$. Such encoding has been implemented in the context of differential privacy, including Kashin's representation (see e.g.~\citealt{Lyubarskii:2010} for an introduction and e.g.~\citealt{Chen:2023} for an application to differential privacy) and random rotations. The latter method has been used for private mean estimation in both central \citep[e.g.][]{Levy:2021} and local \citep[e.g.][]{Girgis:2022} frameworks.
    
Without loss of generality, we assume that the dimension is a power of $2$, as we can append auxiliary values of $0$ to all the data. This dyadic design does not affect the final rate up to constants.

We first define the Hadamard matrices, denoting $H_{k}$ for $k$ a power of $2$, as the family of matrices which can be constructed inductively from
    \begin{align*}
        H_1 = 
            \begin{pmatrix}
                1 
            \end{pmatrix}
        \quad
        \text{and}
        \quad
        H_{k} =
            \begin{pmatrix}
                H_{k-1} & H_{k-1} \\
                H_{k-1} & -H_{k-1} 
            \end{pmatrix},
            \quad
        k > 1.
    \end{align*}
    We then denote by $W_k$ a diagonal matrix with diagonal entries being i.i.d.~Rademacher random variables.  Construct the rotation matrix 
    \begin{equation}\label{eq-rotation-matrix}
        R_k = H_k W_k/\sqrt{k}.    
    \end{equation}
    Rotating data using the rotation matrix $R_k$, with high probability, ensures that the resulting data concentrate better in $\ell_\infty$-norm. This is made precise in \Cref{sec3:lem:RandomRotation}, derived from the proof of Lemma~2 in \cite{Levy:2021}.

    Applying this random transformation to the data before using the univariate estimation procedure to each co-ordinate yields \Cref{alg:PrivateMultiDimMean2}.

    \begin{algorithm}
        \caption{\textsc{PrivateMultiDimMean2}\,$(\{X^{(i)}_{t}\}_{i\in[n],t\in[T]},\,n,\,T,\,\alpha,\,K,\,C)$}
        \begin{algorithmic}[1]
            \State Let $T^\ast = T \wedge \exp\{n\alpha^2/Kd\}$ and $\Delta = [C\{\log(n T^\ast \alpha^2)\}^2/(dT^\ast)]^{1/2}$. \refstepcounter{equation}\label{sec3:eq:l2deltaval}\hfill(\theequation)
            \State Let $N_j$ for $j \in [d]$ as defined in \eqref{sec3:eq:linfGroups} and sample random rotation matrix $R_d$ as in \eqref{eq-rotation-matrix}.
            \For{$j = 1, \hdots, d$}
              \State Calculate $\hat{\theta}_j = \textsc{UserLevelMean}(\{(R_d X^{(i)}_{t})_j\}_{i \in N_j, t \in [T^\ast]},\, n,\, T^\ast,\, \alpha,\, \Delta)$
            \EndFor
        
            \State Let $\hat{\theta} = R_d^{-1}\sum_{j=1}^d e_j \hat{\theta}_{j}$ where $e_j$ is the $j$-th standard basis vector of $\mathbb{R}^d$, and
            \Return $\hat{\theta}$. \refstepcounter{equation}\label{sec3:eq:Finall2Estimator}\hfill(\theequation)
        \end{algorithmic}
        \label{alg:PrivateMultiDimMean2}
    \end{algorithm}

    In addition to the random rotation technique that is used explicitly in the above algorithm, we also use the fact that the $\ell_2$-norm of an average of random variables that are supported on an $\ell_2$-ball concentrates better than only using the fact that the random variables are sub-Gaussian due to having bounded support. This is because, due to being supported on an $\ell_2$-ball, the $\ell_2$-norms of the random variables are themselves bounded random variables, and hence satisfy the sub-Gaussian tail bound; see \cite{Jin:2019} for a detailed discussion.

\subsection{Discussion}\label{sec-3.4}

In this section we derived the user-level LDP minimax rates for mean estimation problems.  One interesting observation is that the user-level rate for $n$ users each with sample size $T$ is equivalent to the item-level rate with $nT$ users, until $T$ is sufficiently large that the infinite-$T$ rate from \Cref{sec2:thm:Means} dominates. Similar limiting behaviour is observed in the central user-level privacy setting in \cite{Levy:2021} where the minimax risk of mean estimation (on both the $\ell_2$-~and $\ell_\infty$-balls) is lower bounded independently of $T$.

We consider \Cref{tab-2} which compares the minimax rates for a range of differential privacy settings with $\alpha \lesssim 1$, and the non-private case as a baseline, all with a total sample size of $nT$. We note that despite using the strictest notion of privacy of all the settings, our rates still compare favourably, matching those of item-level LDP with an equivalent number of samples when $T$ is smaller than some quantity exponential in $n\alpha^2$, and with an exponentially decaying rate otherwise.

\begin{table}[H]
\centering
    \begin{tabular}{lcc}
    \hline\hline
    Setting & Minimax rate & Reference\\ \hline
    Non-Private   & $d/(nT)$ & e.g.~\citet{Wu:2016} \\
    Local Item-Level   & $d/(nT\alpha^2)$ & \citet{Duchi:2018} \\ 
    Local User-Level ($T$-small)   & $d/(nT\alpha^2)$ & \Cref{sec3:thm:main} \\ 
    Local User-Level ($T$-large)   & $e^{-n\alpha^2/d}$ & \Cref{sec3:thm:main} \\ 
    Central Item-level   & $ d/(nT) + d/(n^2T^2\alpha^2)$ & e.g.~\citet{Levy:2021}\\ 
    Central User-level ($T$-small) &  $d/(nT) + d/(n^2T\alpha^2)$ & \citet{Levy:2021}\\ \hline
    \end{tabular}
    \caption{Comparison of minimax optimal rates for $d$-dimensional mean estimation of distributions supported on the unit $\ell_2$-ball with total sample size of $nT$ in all. Dependence on multiplicative constants, constants in exponents, and logarithmic factors are suppressed for brevity.} \label{tab-2}
\end{table}

One notable difference between the central and local settings however, taking estimation on the $\ell_2$-ball as an example, is the linear rather than quadratic factor of $1/T$ in the second term for the central user-level rate in \Cref{tab-2}.  In particular, this shows that under central privacy constraints it is not the case that the user-level rate with $n$ users and $T$ independent observations is the same as the item-level rate with $nT$ users. Hence, a key difference between the behaviour of central and local user-level privacy for mean estimation is that in the local case, as long as $T$ is not too large, there is no difference in the minimax error rate between the item-level and user-level problems with an equivalent total sample size, whereas in the central case the user-level problem is strictly harder for all values of $T$.

Comparing our results to existing results on mean estimation under user-level local differential privacy in \cite{Girgis:2022} and \cite{Bassily:2023}, these previous works only consider the case of the $\ell_\infty$-ball, that is, the class $\mathcal{P}_{d}$, and impose further sub-Gaussianity assumptions beyond that which is already implied by being supported on an $\ell_\infty$-ball. On the other hand, in \Cref{sec3:thm:main}, we make no such assumptions. This leads to a different optimal rate for the family $\mathcal{P}_{d}$, where a quadratic rather than linear dependence on the dimension is observed. We further consider the case of distributions supported on $\ell_2$-balls in the family $\mathcal{P}'_{d}$. To the best of our knowledge, this is the first consideration of distributions compactly supported on $\ell_2$-balls under user-level LDP constraints as opposed to $\ell_\infty$-balls, and the resulting different dependence on the dimension $d$.  We consider the full range of possible values of $T$ whereas these prior works only consider the case that $T$ is not larger than a quantity exponential in $n\alpha^2/d$.

Lastly, regarding the assumptions on the support of the distributions in \Cref{sec3:thm:main}, we remark that it would be sufficient to assume that the data are sub-Gaussian rather than bounded. This is because the localisation step only requires the parameter space to be compact, though some minor adjustments would be needed for data lying outside of candidate sub-intervals, and the compact support is used only through its implication that the data are sub-Gaussian.

\section{Sparse Mean Estimation}\label{sec5}
In this section, we consider the problem of user-level LDP sparse mean estimation. Recalling the family~$\mathcal{P}_{d, s}$ defined in \eqref{sec2:eq:SparseFamily}, we estimate the mean of a distribution $P \in \mathcal{P}_{d, s}$ so that $\theta(P) = \mathbb{E}_P(X)$ and measure the error by the squared $\ell_2$-loss.

\begin{theorem}\label{sec5:thm:main}
For $s$ such that $16\log(ed)/3 \leq s \leq d$, assuming that  $n\alpha^2 > \widetilde{C}s\log(ed)$, where $\widetilde{C} > 0$ is an absolute constant, we then have that
    \begin{equation} \label{sec5:eq:SparseMeanStatement}
        \begin{aligned}
            & \left[ \frac{s}{T} \wedge s\left\{ \left( 1 + \frac{d}{n\alpha^2} \right)^{1/T} - 1\right\}\right] \vee e^{-Cn\alpha^2/s}
            \lesssim \mathcal{R}_{n, T, \alpha}(\theta(\mathcal{P}_{d, s}), \|\cdot\|_2^2) \\
            & \hspace{100pt} \lesssim 
            \min \left\{ s , \frac{s\log(dnT\alpha^2)}{T} + e^{-cn\alpha^2/s}, \, \frac{sd\{\log(dnT\alpha^2)\}^2}{nT\alpha^2} + e^{-cn\alpha^2/d }\right\},
        \end{aligned}
    \end{equation}
    where $\mathcal{P}_{d,s}$ is the family of distributions as in \eqref{sec2:eq:SparseFamily}, $\mathcal{R}_{n, T, \alpha}$ is the user-level minimax risk rate defined in~\eqref{sec2:eq:UserMinimax} and $C > c > 0$ are absolute constants.
\end{theorem}

The above result, whilst involved, simplifies greatly in two primary regimes of interest, with mild extra conditions controlling the logarithmic terms, though we briefly note that in the case that $T = 1$ we obtain the same item-level lower bound to that in \cite{Acharya:2022}. The first corresponds to the usual regime of interest for sparse problems, where the effective sample size is smaller than the ambient dimension. The second occurs when the effective sample size is greater, up to some logarithmic term.

\begin{corollary}\label{sec5:cor:mainthmsimplified}
    Follow the notation and setting of \Cref{sec5:thm:main}.

    \noindent (i) Assume $n\alpha^2 \lesssim d^{\gamma}$ for some $0 < \gamma \leq 1$, and $s\log(dnT\alpha^2)/T + e^{-cn\alpha^2/s} \lesssim s$. It holds that
    \begin{equation*}
        \frac{s}{T} \vee e^{-Cn\alpha^2/s}
        \lesssim \mathcal{R}_{n, T, \alpha}(\theta(\mathcal{P}_{d, s}), \|\cdot\|_2^2)
        \lesssim \frac{s\log(dnT\alpha^2)}{T} + e^{-Cn\alpha^2/s}.
    \end{equation*}

    \noindent (ii) Assume $n\alpha^2 \geq \widetilde{C}d\log(dnT\alpha^2)$ where $\widetilde{C}$ is as in \Cref{sec5:thm:main}, and $d\{\log(dnT\alpha^2)\}^2/(nT\alpha^2) \lesssim 1$. It holds that
    \begin{equation*}
        \frac{sd}{nT\alpha^2}
        \lesssim \mathcal{R}_{n, T, \alpha}(\theta(\mathcal{P}_{d, s}), \|\cdot\|_2^2)
        \lesssim \frac{sd\{\log(dnT\alpha^2)\}^2}{nT\alpha^2}.
    \end{equation*}
\end{corollary}

There are two terms in the lower bound in \Cref{sec5:thm:main}. The second is the infinite-$T$ rate inherited from \Cref{sec2:thm:TIndepLB}, and the first is by an application of a strengthened form of Assouad's lemma. These lower bounds are proven in \Cref{app:sec:Lower} and the upper bounds follow from the analysis, in \Cref{app:sec:Upper}, of two constructed estimators to be introduced shortly. The proof of the simplified bounds of \Cref{sec5:cor:mainthmsimplified} can be found in \Cref{app:sec:misc}, and we discuss the computational complexity of these procedures in \Cref{app:sec:CompComplex}.

The upper bound in \Cref{sec5:thm:main} is obtained by a two-component procedure detailed in \Cref{sec-sparse-mean-upper-bound}.  We note that the procedure requires the knowledge of the true sparsity level $s$, the consequences of the misspecification will be discussed in \Cref{sec5:sec:disc}.

\begin{remark}
The estimation procedures constructed to derive the upper bounds hold for any $s \in [d]$.  The assumption that $s \geq 16\log(ed)/3$ is only required for the lower bound.  For the regime that $s < 16\log(ed)/3$ or the case that $16\log(ed)/3 > d$, more refined analysis might be required and we will leave this for future work.  We also remark that in $s\log(dnT\alpha^2)/T$, the term $d$ is unavoidable based on the current proof.  The term $sd\{\log(dnT\alpha^2)\}^2/(nT\alpha^2)$, however, can be sharpened to $sd\{\log(nT\alpha^2)\}^2/(nT\alpha^2)$, which in fact is the upper bound we derived in the proof. We keep the current presentation to facilitate comparison between the two terms.
\end{remark}

\subsection{Estimation Procedures}\label{sec-sparse-mean-upper-bound}
Depending on the value of $T$, we employ one of the following two estimation procedures. The first is used when $T$ is sufficiently large so that each user can locally identify the $s$-many non-zero entries of the mean to reduce the problem to an $s$-dimensional mean estimation problem. When $T$ is small, we threshold the estimator obtained by naively applying the $\ell_\infty$-ball estimator to improve its performance under the sparsity assumption.  

In particular, the first procedure is conducted when $n\alpha^2 < \widetilde{C}d\log(dnT\alpha^2)$ for $\widetilde{C}$ as in \Cref{sec5:thm:main}, and the second procedure otherwise. It is easy to verify that with this choice we pick the estimator with the smaller error of the two where the first and second procedures have errors corresponding to the second and third terms respectively in the minimum in the upper bound in \eqref{sec5:eq:SparseMeanStatement}.

\subsubsection{Variable Selection-Based Procedure} \label{sec:sparse:variableselection}
For the first procedure, deployed when $n\alpha^2 \leq \widetilde{C}d\log(dnT\alpha^2)$, the first half of users carry out variable selection locally with their own data to identify the $s$-many non-zero entries of the mean. Then, using a hashing-based voting method inspired by \cite{Zhou:2022}, the users vote for the co-ordinates they believe containing the non-zero entries. The other half of users then conduct $s$-dimensional mean estimation on the resulting $\ell_\infty$-ball. The procedure is formalised in \Cref{alg:PrivateSparseSelection}.

\begin{algorithm}
    \caption{\textsc{PrivateSparseSelection}\,$(\{X^{(i)}_{t}\}_{i\in[n],t\in[T]},\,n,\,T,\,\alpha,\,K)$}
    \begin{algorithmic}[1]
        \State Let $T^\ast = T \wedge \exp\{n\alpha^2/(Ks)\}$, $\varepsilon = \left\{2\log(dnT^\ast\alpha^2)/T^\ast \right\}^{1/2}$, and $\eta = (30s)^{1/2}$. \refstepcounter{equation}\label{sec5:eq:thresholdboundary} \hfill(\theequation)
        \For{$i = 1, \hdots, n/2$}
            \State Construct vector $\hat{\bm{\omega}}^{(i)}$ where $\hat{\omega}_j^{(i)} = \mathbbm{1}\{ | (T^\ast)^{-1} \sum_{t=1}^{T^\ast} X_{t,j}^{(i)} | \geq \varepsilon \}$ for $j \in [d]$. \refstepcounter{equation}\label{sec5:eq:HTestimator} \hfill(\theequation)
            \State Sample $r_{i, 1}, \hdots, r_{i, d}$ i.i.d.~Rademacher random variables. Sample $\ell_i \sim \mathrm{Laplace}(1)$.
            \State Let $Z_i = \Pi_{[-\eta, \eta]} \left( \sum_{j = 1}^d r_{i, j} \hat{\omega}_j^{(i)} \right) + 2\eta\ell_i/\alpha$. \refstepcounter{equation}\label{sec5:eq:truncwidtheq} \hfill(\theequation)
        \EndFor
        \State Denote $\mathcal{I} = \{ j \in [d] : (2/n)\sum_{i = 1}^{n/2} r_{i, j} Z_i \geq 1/2 \}$. \refstepcounter{equation}\label{eq-hashing-voting} \hfill(\theequation)
        \State Let $\mathcal I' = \{\iota_1,\dots,\iota_{s'}\}$ where $(\iota_1<\cdots<\iota_{s'})$ are the first $s' = \min\{|\mathcal I|, s\}$ elements of $\mathcal{I}$
        \State Denote $\bm{Y}_t^{(i)} = (X_{t, \iota_1}^{(i)}, \hdots, X_{t, \iota_{s'}}^{(i)})$ for $i \in [n] \setminus [n/2]$ and $t \in [T]$, 
        \State Calculate $\tilde{\bm{\theta}} = \textsc{PrivateMultiDimMean}(\{\bm{Y}_t^{(i)}\}_{i \in [n]\setminus[n/2], t \in [T]}, n, T, \alpha, K) \in \mathbb{R}^{s'}$
        \State Let $\hat{\bm{\theta}} \in \mathbb{R}^d$ defined via its co-ordinates $\hat{\theta}_j = \tilde{\theta}_j \mathbbm{1}\{j \in \mathcal{I}_{s'}\}$ for $j \in [d]$ and \Return $\hat{\bm{\theta}}$. \refstepcounter{equation}\label{sec5:eq:sparseestimator1} \hfill(\theequation)
    \end{algorithmic}
    \label{alg:PrivateSparseSelection}
\end{algorithm}

The strength of \Cref{alg:PrivateSparseSelection} comes from the effectiveness of the hashing-based voting method to identify the non-zero entries of the mean vector.  To identify these non-zero entries, one could naively use unary-encoding as in, for example, \Cref{alg:CoverSelect}, which involves either (1) splitting users to estimate each co-ordinate separately or (2) privatising a $d$-dimensional binary vector with $s$-many non-zero entries (c.f.~\Cref{app:lem:RR}).  These two strategies result in factors $d$ and $s^2$ in the final rate, respectively.  Instead, the hashing-based method compresses the votes of the users to a one-dimensional quantity without splitting users, by exploiting Rademacher random variables.  The votes for a specific co-ordinate can then be extracted from this univariate quantity and the final rate is merely inflated by a factor of $\log(d)$.

\subsubsection{Thresholding-Based Procedure}
\label{sec:sparse:thresholding}
For the second procedure, conducted when $n\alpha^2 > \widetilde{C}d\log(nT\alpha^2)$, we apply the $\ell_\infty$-ball procedure of \Cref{alg:PrivateMultiDimMean} with an additional thresholding step applied to the estimator thus obtained. This reduces the $d^2$ factor in the polynomial term of the error for the $\ell_\infty$-ball procedure in \eqref{sec3:eq:linfMeanStatement} to a factor of~$sd$. The procedure is formalised in \Cref{alg:PrivateSparseThreshold}

\begin{algorithm}
    \caption{\textsc{PrivateSparseThreshold}\,$(\{X^{(i)}_{t}\}_{i\in[n],t\in[T]},\,n,\,T,\,\alpha,\,K,\,C')$}
    \begin{algorithmic}[1]
        \State Let $N_j$ for $j \in [d]$ as in \eqref{sec3:eq:linfGroups} and $T^\ast = T \wedge \exp\{n\alpha^2/(Kd)\}$.
        \State Let $\varepsilon = C'\Delta\left\{d\log(nT^\ast\alpha^2)/(n\alpha^2) \right\}^{1/2}$, and $\Delta = \left\{ 4\log(n T^\ast\alpha^2)/{T^\ast} \right\}^{1/2}$  .\refstepcounter{equation}\label{sec5:eq:Deltaval2nd} \hfill(\theequation)
        \For{$j = 1, \hdots, d$}
          \State Calculate $\hat{\theta}_j = \textsc{UserLevelMean}(\{X_{t, j}^{(i)}\}_{i \in [N_j], t \in [T^\ast]}, n, T^\ast, \alpha, \Delta)$.
          \State Let $\tilde{\theta}^{\mathrm{thresh}}_j = \tilde{\theta}_j \mathbbm{1}\{ |\tilde{\theta}_j| > 3\varepsilon \}$. \refstepcounter{equation}\label{sec5:eq:thresholding} \hfill(\theequation)
        \EndFor
        \State Denote $s$-many largest co-ordinates $\mathcal{I} = \{\iota_1, \hdots, \iota_s\}$ where $|\tilde{\theta}^{\mathrm{thresh}}_{\iota_j}| \geq |\tilde{\theta}^{\mathrm{thresh}}_{\iota_{j-1}}|$ for $1 < j \leq s$.
        \State Let $\hat{\bm{\theta}} = \sum_{j=1}^d (\bm{e}_j \tilde{\theta}^{\mathrm{thresh}}_j\mathbbm{1}\{j \in \mathcal{I}\})$ for $\bm{e}_j$ the $j$-th standard basis vector of $\mathbb{R}^d$, and
        \Return $\hat{\bm{\theta}}$. \refstepcounter{equation}\label{sec5:eq:finalThresholdedEstimator} \hfill(\theequation)
    \end{algorithmic}
    \label{alg:PrivateSparseThreshold}
\end{algorithm}

\subsection{Discussion} \label{sec5:sec:disc}

This section demonstrates, through sparse mean estimation, an important example where in some regimes the user-level framework provides a decidedly superior minimax rate over the item-level framework with an equivalent number of total data points, in contrast to the findings of \Cref{sec3}. In particular, even as $d/(n\alpha^2) \rightarrow \infty$ we saw that, as long as $\log\{d/(n\alpha^2)\}/T \rightarrow 0$, the rate is improved to $s\log(d)/T$ up to further logarithmic factors compared to the item-level rate with $nT$ observations of $sd/(nT\alpha^2)$ \citep{Acharya:2022}. The enabling factor for this is that with more than $\log\{d/(n\alpha^2)\}$-many observations a user can, with high probability, successfully identify the relevant non-zero co-ordinates locally, without the need for any communication with the data aggregator. Independent work by \cite{Ma:2024} consider a subset of this setting, and obtain similar findings. This suggests that similar phenomena of superior user-level performance may exist in other problems where there is again potential for localisation. For example, \cite{Zhao:2024} note that the user-level rate for heavy-tailed mean estimation similarly improves over the corresponding item-level rate.

The following table includes comparisons of the minimax rates for a range of differential privacy settings, and the non-private case as a baseline, all with an effective sample size of $nT$.   We note, similarly to \Cref{sec3}, our rates still compare favourably, in fact out-performing local item-level in some regimes, despite using the strictest notion of privacy.

\begin{table}[H]
\centering
    \begin{tabular}{lcc}
    \hline\hline
    Setting & Minimax rate & Reference\\ \hline
    Non-Private   & $s\log(ed/s)/(nT)$ & e.g.~\citet{hanfang} \\
    Local Item-Level   & $sd/(nT\alpha^2)$ & \citet{Acharya:2022} \\ 
    Local User-Level ($n,T$-small)   & $s/T$ & \Cref{sec5:thm:main} \\ 
    Local User-Level ($n$-large)   & $sd/(nT\alpha^2)$ & \Cref{sec5:thm:main} \\
    Local User-Level ($T$-large)   & $e^{-n\alpha^2/s}$ & \Cref{sec5:thm:main} \\
    Central Item-level\tablefootnote{In the absence of existing results for central item-level privacy for estimation of sparse means on the $\ell_\infty$-ball, we compare to results obtained under Gaussian distribution assumptions.} & $ s\log(d)/(nT\alpha) + \{s\log(d)\}^2/(n^2T^2\alpha^2)$ & \cite{Cai:2021}\\  \hline
    \end{tabular}
    \caption{Comparison of minimax optimal rates for $d$-dimensional $s$-sparse mean estimation of distributions supported on the unit $\ell_\infty$-ball, all with effective sample size of $nT$. Dependence on multiplicative constants, constants in exponents, and logarithmic factors are suppressed for brevity.} \label{tab-20}
\end{table}

\Cref{fig} illustrates the difference in performance between the user-level setup, and that of item-level with an equivalent number of samples, depending on the relationship between the values of $n\alpha^2/d$ and $T$. We note again that this is a compromised comparison for the user-level framework as the $\alpha$ privacy budget has to simultaneously protect across all $T$ observations a user holds, as opposed to the item-level setting where each observation enjoys the entire $\alpha$ budget. In \Cref{fig:sub1}, we see that for estimation on an $\ell_2$-ball, user-level is never better than item-level, with equivalent performance as long as $T$ is not too large. In contrast, in \Cref{fig:sub2}, the behaviour where user-level outperforms item-level in certain regimes is apparent.

\begin{figure}[H]
\centering
\begin{subfigure}{.45\textwidth}
    \centering
    \begin{tikzpicture}[scale=0.8,transform shape]
        \begin{axis}[axis lines=middle,
                    xlabel=$T$,
                    ylabel=$\frac{n\alpha^2}{d}$,
                    ytick={1},
                    hide obscured x ticks=false,
                    xtick={0},
                    xticklabels={$1$},
                    xmin=0,
                    xmax=5,
                    ymin=0,
                    ymax=5,
                    xlabel style={xshift=-100pt, yshift=-10pt, anchor=north},
                    ylabel style={yshift=-80pt, xshift=-10pt, anchor=east}]
    
            \addplot[name path=F,black,domain={0:5}] {1/(x+1)} node[pos=0.7, above]{$T = \frac{d}{n\alpha^2}$};
            
            \addplot[name path=G,black,domain={0:5}] {ln(e*(x+1)}node[pos=0.9, below]{$T = e^\frac{n\alpha^2}{d}$};
    
            \path[name path=baxis] (axis cs:0,0) -- (axis cs:5,0);
            \path[name path=taxis] (axis cs:0,5) -- (axis cs:5,5);
            \addplot[pattern=north west lines, pattern color=red!30] fill between[of=F and baxis];
            \addplot[pattern=grid, pattern color=blue!30]fill between[of=F and G];
            \addplot[pattern=north east lines, pattern color=gray!30]fill between[of=taxis and G];
    
            \node at (axis cs:1,0.21) {Both impossible};
            \node at (axis cs:2,1) {User-level harder};
            \node[align=center] at (axis cs:2.5,3.5) {$nT$-item level equiv.\\to user-level};
        \end{axis}
    \end{tikzpicture}
    \caption{Estimation on $\ell_2$-ball}
    \label{fig:sub1}
    \end{subfigure}
    \hfill
    \begin{subfigure}{.45\textwidth}
    \centering
    \begin{tikzpicture}[scale=0.8,transform shape]
        \begin{axis}[axis lines=middle,
                    xlabel=$T$,
                    ylabel=$\frac{n\alpha^2}{d}$,
                    ytick={1},
                    hide obscured x ticks=false,
                    xtick={0},
                    xticklabels={$1$},
                    xmin=0,
                    xmax=5,
                    ymin=0,
                    ymax=5,
                    xlabel style={xshift=-100pt, yshift=-10pt, anchor=north},
                    ylabel style={yshift=-80pt, xshift=-10pt, anchor=east}]
    
            \addplot[name path=F,black,domain={0:5}] {e^(-(x)} node[pos=0.7, above]{$T = \log\left(\frac{d}{n\alpha^2}\right)$};
        
            \addplot[name path=L,black,domain={0:3}] {1}node[pos=0.9, below]{};
        
            \addplot[name path=C,black,domain={3:5}] {4/(x+1)}node[pos=0.7, above]{$T = \frac{d}{n\alpha^2} e^{n\alpha^2}$};
            
            \addplot[name path=G,black,domain={0:5}] {ln(e*(x+1)}node[pos=0.9, below]{$T = e^\frac{n\alpha^2}{d}$};
        
            \path[name path=baxis] (axis cs:0,0) -- (axis cs:5,0);
            \path[name path=taxis] (axis cs:0,5) -- (axis cs:5,5);
            
            \addplot[pattern=north west lines, pattern color=red!30] fill between[of=F and baxis];
            \addplot[pattern=crosshatch, pattern color=green!30]fill between[of=F and L, soft clip = {domain = 0:3}];
            \addplot[pattern=crosshatch, pattern color=green!30]fill between[of=F and C, soft clip = {domain = 3:5}];
            \addplot[pattern=north east lines, pattern color=gray!30]fill between[of=taxis and G];
            \addplot[pattern=grid, pattern color=blue!30]fill between[of=G and L, soft clip = {domain = 0:3}];
            \addplot[pattern=grid, pattern color=blue!30]fill between[of=G and C, soft clip = {domain = 3:5}];
        
            \node[align=left] at (axis cs:0.62,0.29) {Both\\impossible};
            \node at (axis cs:2.5,0.6) {User-level easier};
            \node at (axis cs:2.5,1.5) {User-level harder};
            \node[align=center] at (axis cs:2.5,3.5) {$nT$-item level equiv.\\to user-level};
        \end{axis}
    \end{tikzpicture}
    \caption{Sparse mean estimation $(s=1)$}
    \label{fig:sub2}
\end{subfigure}
\caption{Comparison of performance of user-level LDP and $nT$-item level LDP estimator performance. A point in a region of the graph indicates the regime for a value of the pair $(T, n\alpha^2/d)$. Constants, constants in exponents, and logarithmic factors are omitted for ease of comparison.}
\label{fig}
\end{figure}

\medskip
\noindent\textbf{Misspecification of $s$.}  Recalling that both of procedures introduced in \Cref{sec-sparse-mean-upper-bound} require the true sparsity level $s$ as an input, we remark on the consequences of misspecification. If underestimated, i.e.~the input $s_1 < s$, then each method suffers an additional bias of at most $s-s_1 \leq s$. In the selection-based procedure, this arises due to failing to select a non-zero co-ordinate. For the thresholding-based procedure, this arises at the final step when only the largest $s_1 < s$ co-ordinates of the estimator are kept. This gives an upper bound on the risk of the estimator of
\begin{equation*}
    \mathbb{E}\left[\| \hat{\theta}-\theta \|_2^2 \right]
    \lesssim  s - s_1 + \min\left\{ s , \frac{s_1\log(dnT\alpha^2)}{T} + e^{-cn\alpha^2/s_1}, \, \frac{s_1d\{\log(dnT\alpha^2)\}^2}{dnT\alpha^2} + e^{-cn\alpha^2/d } \right\}.
\end{equation*}

If overestimated, i.e.~$s_2 > s$, then no extra bias will be introduced but larger variance will be due to estimating more co-ordinates.  The risk of both estimation procedures see $s$ replaced with $s_2$ accordingly, giving the upper bound
\begin{equation*}
    \mathbb{E}\left[\| \hat{\theta}-\theta \|_2^2 \right]
    \lesssim \min\left\{ s , \frac{s_2\log(dnT\alpha^2)}{T} + e^{-cn\alpha^2/s_2}, \, \frac{s_2d\{\log(dnT\alpha^2)\}^2}{dnT\alpha^2} + e^{-cn\alpha^2/d } \right\}.
\end{equation*}

As such, we see that in practice, it is generally safer to choose a relatively large $s$ when the truth is unknown.

\medskip
\noindent\textbf{Choice of the Estimation Procedure.}  In \Cref{sec-sparse-mean-upper-bound} we recommended to use the first procedure -- the variable selection-based method when $n\alpha^2 \leq \widetilde{C}d\log(dnT\alpha^2)$, and the second procedure -- the thresholding-based method otherwise.  We now discuss the estimation error upper bound if an estimator is chosen against recommendation.

As we mentioned at the beginning of \Cref{sec-sparse-mean-upper-bound}, the error control of the first procedure corresponds to the second term in the upper bound in \eqref{sec5:eq:SparseMeanStatement}, and the second procedure corresponds to the third term.  When the choice is made against our recommendation, i.e.~chosen when  $n\alpha^2 > \widetilde{C}d\log(dnT\alpha^2)$, the error control of the first procedure still holds.   Whereas for the second procedure, the error control will change.  To be specific, when $n\alpha^2 \asymp d$, the corresponding error control will always be bounded below from a positive constant.  This is due to the term $e^{-n\alpha^2/(Kd)}$ arising from an application of the $d$-dimensional $\ell_\infty$-ball estimator of \Cref{sec3}. From this, we see that if $n\alpha^2 \asymp d$, the error contains a constant term bounding the error away from zero.

\medskip
\noindent\textbf{Previous Work With the Same Terminology.}  \cite{Zhou:2022} share the term of user-level privacy for sparse mean estimation with us, but the setups considered differ. In \cite{Zhou:2022}, it is assumed that each user holds a single, possibly high-dimensional, vector. The item-level privacy in \cite{Zhou:2022} is to protect against a perturbation of a single entry of this vector, whereas user-level privacy protects against any two different vectors the user may hold. In our setup, we assume that each user holds multiple, possibly high-dimensional, vectors. The item-level considered in our paper corresponds to protecting against the perturbation of one observation from a user's collection, regardless of how many entries are perturbed, whilst user-level corresponds to the entire collection, perturbing both every vector and every co-ordinate of any vector. Further, in \cite{Zhou:2022} sparsity is defined to mean that at most $s$-many entries of a user's vector will be non-zero, whereas in our work, the sparsity assumption is applied to the mean vector, whilst all the entries of a user's data may be non-zero.

We also note that whilst we implement the same hashing method as \cite{Zhou:2022}, our analysis of its performance is different.  In terms of the inference goal, \cite{Zhou:2022} aim at deriving a  high-probability $\ell_\infty$-upper bound on the estimation error, whilst we conduct variable selection based on binary vectors.  In terms of the technical details, the truncation intervals in our paper have width $O(s^{1/2})$, compared to $O(s^{1/2}\log(n/\beta))$ in \cite{Zhou:2022}.  This is achieved through a tighter analysis and different inference goals, and allows us to avoid further logarithmic factors in the error controls.

\section{Non-Parametric Density Estimation}\label{sec4}
    In this section, we consider the user-level LDP non-parametric density estimation problem.  Recalling the Sobolev class of densities $\mathcal{F}_{\beta, r}$ defined in \eqref{sec2:eq:SobolevDensity}, for a density $f \in \mathcal{F}_{\beta, r}$, we denote by $P_f$ the induced distribution the data are drawn from. Our aim is to estimate the density $\theta(P_f) = f$ and we measure error through the $L_2$-norm as defined in \eqref{sec2:eq:L2FuncNormDef}.  We have the following result.

    \begin{theorem}\label{sec4:thm:main}
        Assuming that $n\alpha^2 \geq \widetilde{C}_\beta$, where $\widetilde{C}_\beta > 0$ is an absolute constant depending only on $\beta$, we then have that
        \begin{align*}
            c_\beta \left\{(nT\alpha^2)^{-\frac{2\beta}{2\beta + 2}} \vee  (n\alpha^2)^{-2\beta} \right\}
            \lesssim \mathcal{R}_{n, T, \alpha}&(\theta(\mathcal{P}_{\beta, 1}), \|\cdot\|_2^2)
            \lesssim C_\beta \Big\{ (nT\alpha^2)^{-\frac{2\beta}{2\beta + 2}} \log(nT\alpha^2)  \\
            & \hspace{100pt} + (n\alpha^2)^{-2\beta} \{\log(n\alpha^2)\}^{2\beta + 3}\Big\},
        \end{align*}
        where $\mathcal{P}_{\beta, 1}$ is the family of distributions as in \eqref{sec2:eq:InducedDensityFamily}, $\mathcal{R}_{n, T, \alpha}$ is the user-level minimax risk rate defined in~\eqref{sec2:eq:UserMinimax}, and $C_\beta > c_\beta > 0$ are absolute constants depending only on $\beta$.
    \end{theorem}

    We see that the lower bound in \Cref{sec4:thm:main} consists of two parts, the first analogous to the standard item-level rate \cite[Corollary 4]{Duchi:2018}, and the second inherited from \Cref{sec2:thm:TIndepLB}. Unlike in the mean estimation problems we considered, the rate in the large $T$ regime is polynomial rather than exponential, appearing to be a non-parametric rate. Nevertheless, this again demonstrates the intuition of \Cref{sec2} wherein we see that the rate for sufficiently large $T$ coincides with the infinite-$T$ rate in \Cref{sec2:thm:Density}. The lower bound is proven in \Cref{app:sec:Lower}.  An estimation procedure achieving the upper bound is presented below, the performance of which is analysed in \Cref{app:sec:Upper} and the computational complexity discussed \Cref{app:sec:CompComplex}.

\subsection{Estimation Procedure} \label{sec4:sec:procedure}

    The trigonometric basis \eqref{sec2:eq:TrigBasis} allows us to translate a density estimation problem into mean estimation, where we estimate a truncation of the sequence $\{\theta_j\}_{j \geq 1}$ of basis coefficients given in \eqref{sec2:eq:Sobolev}. In particular, we use the fact that for any $f \in \mathcal{F}_{\beta, r}$ and $j \in \mathbb{N}$, we have $\int f(x) \varphi_j(x) \diff x = \theta_j$ by the orthonormality of the trigonometric basis. Hence, if the data $\{X_i\}_{i \in [n]}$ are i.i.d.~from the distribution~$P_f$, then the sample mean of $\{\varphi_j(X_i)\}_{i \in [n]}$ is a consistent estimator of $\theta_j$. Importantly, for all $j \in \mathbb{N}$, we have that $\sup_{x}|\varphi_j(x)| \leq \sqrt{2}$, allowing us (after rescaling) to invoke the results of \Cref{sec3} for estimation on a bounded $\ell_\infty$-ball. The procedure is formalised in \Cref{alg:PrivateDensity}.
    \begin{algorithm}
        \caption{\textsc{PrivateDensity}\,$(\{X^{(i)}_{t}\}_{i\in[n],t\in[T]},\,n,\,T,\,\alpha,\,K, \beta, \,C_\beta)$}
        \begin{algorithmic}[1]
            \State Let $T_0=(n\alpha^2)^{2\beta+1}/\{C_\beta\log(n\alpha^2)\}^{2\beta + 2}$ and $M=\{n \alpha^2 \min(T,T_0)\}^{1/(2\beta+2)}$. \refstepcounter{equation}\label{sec4:eq:Mval}\hfill(\theequation)
            \State Calculate $\bm{Y}_t^{(i)} = (\varphi_j(X_t^{(i)}), \, j \in [M])^{\top}$.
            \State Calculate $\hat{\bm{\theta}} = \textsc{PrivateMultiDimMean}(\{\bm{Y}_t^{(i)}\}_{i \in [n], t \in [T]}, n, T, \alpha, K)$.
            \State \Return $\hat{f} = \sum_{j=1}^M \hat{\theta}_{j} \varphi_j$. \refstepcounter{equation}\label{sec4:eq:finalestimator}\hfill(\theequation)
        \end{algorithmic}
        \label{alg:PrivateDensity}
    \end{algorithm}
    
    In a standard non-parametric projection estimation procedure, one estimates the density by obtaining estimates of some finite number of coefficients through calculating the sample mean of the basis functions applied to data \cite[e.g.][Chapter 1]{Tsybakov:2009}. The item-level LDP non-parametric estimation procedure in \cite{Duchi:2018} is a private analogue, where the mean estimation step is carried out using a suitable (item-level) LDP mean estimation technique. We see that in the user-level setting, our estimator $\hat{f}$ in \eqref{sec4:eq:finalestimator} operates in a similar fashion. The only adjustment is to use a user-level LDP mean estimation technique.

    As for the cut-off level $M$, the larger it is, the less approximation bias is introduced but the greater the variance associated with a higher dimensional mean estimation problem is. As a sketch, one can see that for $T$ sufficiently large, the error introduced by the mean estimation component is $O\{\exp(n\alpha^2/M)\}$ and the error from truncation is $O(1/M^{2\beta})$. Balancing these two terms yields the choice of $M$ and consequently the upper bound in \Cref{sec4:thm:main}.

\subsection{Discussion}
In this section, we saw that the user-level rate for $n$ users with sample size $T$ is equivalent up to the item-level rate with $nT$ users, until $T$ is sufficiently large so that the infinite-$T$ rate dominates, as was the case with the mean estimation problems of \Cref{sec3} and unlike sparse mean estimation as in \Cref{sec5}.

The comparisons of the minimax rate for a range of differential privacy settings, and the non-private case as a baseline, all with a total sample size of $nT$, are presented in \Cref{tab-5}.  We start by noting that, the central item-level rate consists of the sum of two terms, the non-private rate $(nT)^{-2\beta/(2\beta+1)}$ and the cost for privacy $(nT\alpha)^{-2\beta/(\beta+1)}$.  As is usual, provided $\alpha$ is not too small, the privacy cost is dominated by the non-private rate. 

On the other hand, for the local item-level setting, we see that the privacy cost dominates the non-private rate. This behaviour is again seen in the local user-level setting when $T$ is small. Further, the rate for $T$ sufficiently large is polynomial, albeit smaller than the small-$T$ case, owing to the larger metric entropy of the function space and the result of \Cref{sec2:thm:TIndepLB}.

\begin{table}[H]
\centering
    \begin{tabular}{lcc}
    \hline\hline
    Setting & Minimax rate & Reference\\ \hline
    Non-Private   & $(nT)^{-2\beta/(2\beta+1)}$ & e.g.~\citet{Tsybakov:2009} \\
    Local Item-Level   & $(nT\alpha^2)^{-2\beta/(2\beta+2)}$ & \citet{Duchi:2018} \\ 
    Local User-Level ($T$-small)   & $(nT\alpha^2)^{-2\beta/(2\beta+2)}$ & \Cref{sec4:thm:main} \\ 
    Local User-Level ($T$-large)   & $(n\alpha^2)^{-2\beta}$ & \Cref{sec4:thm:main} \\
    Central Item-Level   & $(nT)^{-2\beta/(2\beta+1)} + (nT\alpha)^{-2\beta/(\beta+1)}$ & \cite{Cai:2021} \\
    \hline
    \end{tabular}
    \caption{Comparison of minimax optimal rates for non-parametric density estimation for $\beta$-Sobolev smooth densities supported on $[0,1]$ with effective sample size of $nT$ in all. Dependence on multiplicative constants, constants in exponents, and logarithmic factors are suppressed for brevity.} \label{tab-5}
\end{table}

Regarding existing results on user-level privacy, \cite{Acharya:2023} show that the minimax risk of discrete density estimation (under the high privacy regime $\alpha \leq 1$) for a distribution with $N$ outcomes is $O\{N/(nT\alpha^2)\}$. However, this holds only for the regime $n\alpha^2 \leq CN\log(T)$ for some constant $C > 0$, and what happens outside of this regime is not considered. In light of our results on mean estimation, this regime suggests a phase transition may occur once $T$ is exponential in $n\alpha^2/N$, akin to $d$-dimensional mean estimation on the $\ell_2$-ball. This is perhaps to be expected as discrete density estimation is similarly a parametric problem. For a sketch for the lower bound, suppose that $\Delta > 0$ is sufficiently small and is such that $1/\Delta$ is an integer. One can construct a $2\Delta$-separated packing with respect to the $L_2$-norm for densities (or equivalently, the $\ell_2$-norm for the vector of probabilities of outcomes) of the space of discrete densities on $N$ outcomes, indexed by the probability vector $p \in [0,1]^N$ denoting the probability of each outcome, as
\begin{align*}
    \mathcal{V} &= \{p \in \{0, \Delta, \hdots, (1/\Delta - 1)\Delta, 1\}^N : p_1 + \hdots + p_N = 1\} \\
    &= \{\Delta x: x \in \{0, 1, \hdots, 1/\Delta\}^N, \, \Delta x_1 + \hdots + \Delta x_N = 1\}.
\end{align*}
In particular, by a standard stars-and-bars counting method \citep[e.g.][Chapter~3]{Feller:1950}, it can be shown that
\begin{align*}
    |\mathcal{V}| = \binom{1/\Delta - 1}{N - 1} \geq \left( \frac{1/\Delta - 1}{N - 1} \right)^{N-1}.
\end{align*}
Assuming that $n\alpha^2 \geq C(N-1)\log(e(N-1))$ for some sufficiently large absolute constant $C > 0$ and setting 
\begin{equation*}
    \Delta =  \left\{ 1 + (N-1)\exp\left( \frac{24n\alpha^2 + 2\log(2)}{N-1} \right)\right\}^{-1},
\end{equation*}
we have that
\begin{align*}
    \mathcal{R}_{n, \infty, \alpha} \big(\theta(\mathcal{P}), \| \cdot \|_2^2 \big)
    &\gtrsim \left\{ 1 + (N-1) \exp\left( \frac{24n\alpha^2 + 2\log(2)}{N-1} \right)\right\}^{-2} \\
    &\gtrsim \frac{1}{(N-1)^2} \exp\left( \frac{-24n\alpha^2 - 2\log(2)}{N-1} \right) 
    \gtrsim e^{-cn\alpha^2/(N-1)}
\end{align*}
where $c > 0$ is an absolute constant, the first inequality is by an application of \Cref{sec2:thm:GeneralBound}, and the last two both use the assumption $n\alpha^2 \geq C(N-1)\log(e(N-1))$.

Hence, we see that the infinite-$T$ minimax lower bound for discrete density estimation is exponential, matching that for mean estimation rather than the polynomial rate for non-parametric density estimation.

\section{Numerical Experiments} \label{sec:sim}
    In this section, we implement our developed procedures to demonstrate their feasibility, and validate the theoretical rates. We consider synthetic and real data in Sections~\ref{sec:sim:synth} and \ref{sec:sim:real} respectively. Details regarding tuning parameters are deferred to \Cref{sec-sens}. Finally, we consider further simulations investigating the phase transition phenomena in \Cref{sec:phasetransition}, and simulations in \Cref{sec:sim:mixture} for mean estimation under the data heterogeneity setting we consider in \Cref{secmix}.
    
    \subsection{Synthetic Data} \label{sec:sim:synth}
        Throughout, we consider the values $\alpha \in \{0.5, 1, 2, 4\}$. These values are commonly chosen in practice \citep[e.g.][]{Apple:2023}. We use the notation $\mathrm{Rad}(p)$ to denote the distribution where, for $X \sim \mathrm{Rad}(p)$, $\mathbb{P}(X = 1) = 1 - \mathbb{P}(X = -1) = p$, and refer to $\mathrm{Rad}(1/2)$ as the Rademacher distribution.

        \subsubsection{Mean Estimation} \label{sec:meanestsim}
            To investigate the performance of our mean estimation procedures, we start with the univariate procedure. Briefly, we note the localisation stage of our developed procedures involves dividing the interval into sub-intervals. In particular, if the mean of a distribution lies close to the boundary of two sub-intervals as opposed to being close to the centre, the performance will decrease slightly. Without adjustments, this may result in unstable numerical performances. Hence, we will specify random ``shifts'' within our experiments to average out this behaviour as follows. Sample a shift $U \sim \mathrm{Unif}[-0.3, 0.3]$ and generate $X_{1:T}^{(1)}, \hdots, X_{1:T}^{(n)} \mid U \overset{\mathrm{i.i.d}}{\sim} \mathrm{Unif}[0,1]^{\otimes T} + U\mathbf{1}$ for $\mathbf{1}$ being the $T$-dimensional vector of ones. Sample $\{\ell_{i,t}\}_{i \in [n], t \in [T]}$ i.i.d.~standard Laplace random variables. We fix $n = 500$ and vary $T \in \{\lfloor 10^m \rfloor: m = 2, 2.1, \hdots, 4\}$. We estimate the mean-squared-error across 500 repetitions using the following four estimation schemes for the mean $\theta$:
            
            \noindent
            \textbf{Full Item Level}: This scheme corresponds to item-level with a sample size of $nT$, which serves as a lower bound on the performance of our estimator. Here, $\hat{\theta} = (nT)^{-1}\sum_{i = 1}^n \sum_{t = 1}^T \{X_t^{(i)} + (2/\alpha) \ell_{i,t}\}$.
            
            \noindent
            \textbf{Semi User Level}: This scheme corresponds to each user locally calculating the sample mean of their data, and publishing this after privatisation. Hence, a user attempts to utilise their entire sample of $T$-many observations, but too naively to sufficiently take advantage of the user-level framework. Here, $\hat{\theta} = n^{-1}\sum_{i = 1}^n \{(T^{-1}\sum_{t = 1}^T X_t^{(i)}) + (2/\alpha) \ell_{i,1}\}$.

            \noindent
            \textbf{User Level}: This is the user-level mean estimation procedure of \Cref{sec3:univariate} with $\hat{\theta}$ therein.

            \noindent
            \textbf{Split User}: This scheme arises from naively splitting the privacy budget across a user's $T$-many observations in the user-level setting, yielding $\hat{\theta} = (nT)^{-1}\sum_{i = 1}^n \sum_{t = 1}^T \{X_t^{(i)} + (2T/\alpha) \ell_{i,t}\}$. 

            The results of the comparison of these these four estimation schemes are given in \Cref{fig:synthmean}. Plotting the ratio of the mean-squared-error of the full item level scheme and the remaining three, we see that the relative error grows in $T$ for the semi user level and split user schemes, showing that such naive approaches cannot take advantage of the user-level framework. On the other hand, the relative error remains constant in $T$ for the user-level scheme for all values of $\alpha$ except $0.5$, where the privacy constraint is too strict for the localisation step to succeed and the phase-transition phenomenon occurs. This demonstrates the ability of our developed procedure to take advantage of the multiple observations each user has in the user-level framework, and validates the theoretical results of \Cref{sec3:thm:main}. Lastly, we emphasise that these comparisons are harsh to the user-level setting, as user-level privacy provides as stronger privacy guarantee, sharing $\alpha$ budget for all $T$ observations, whereas in the item level settings, the entire $\alpha$ budget is enjoyed by each data point.

            \begin{figure}[H]
                \centering
                \includegraphics[height=0.162\textheight]{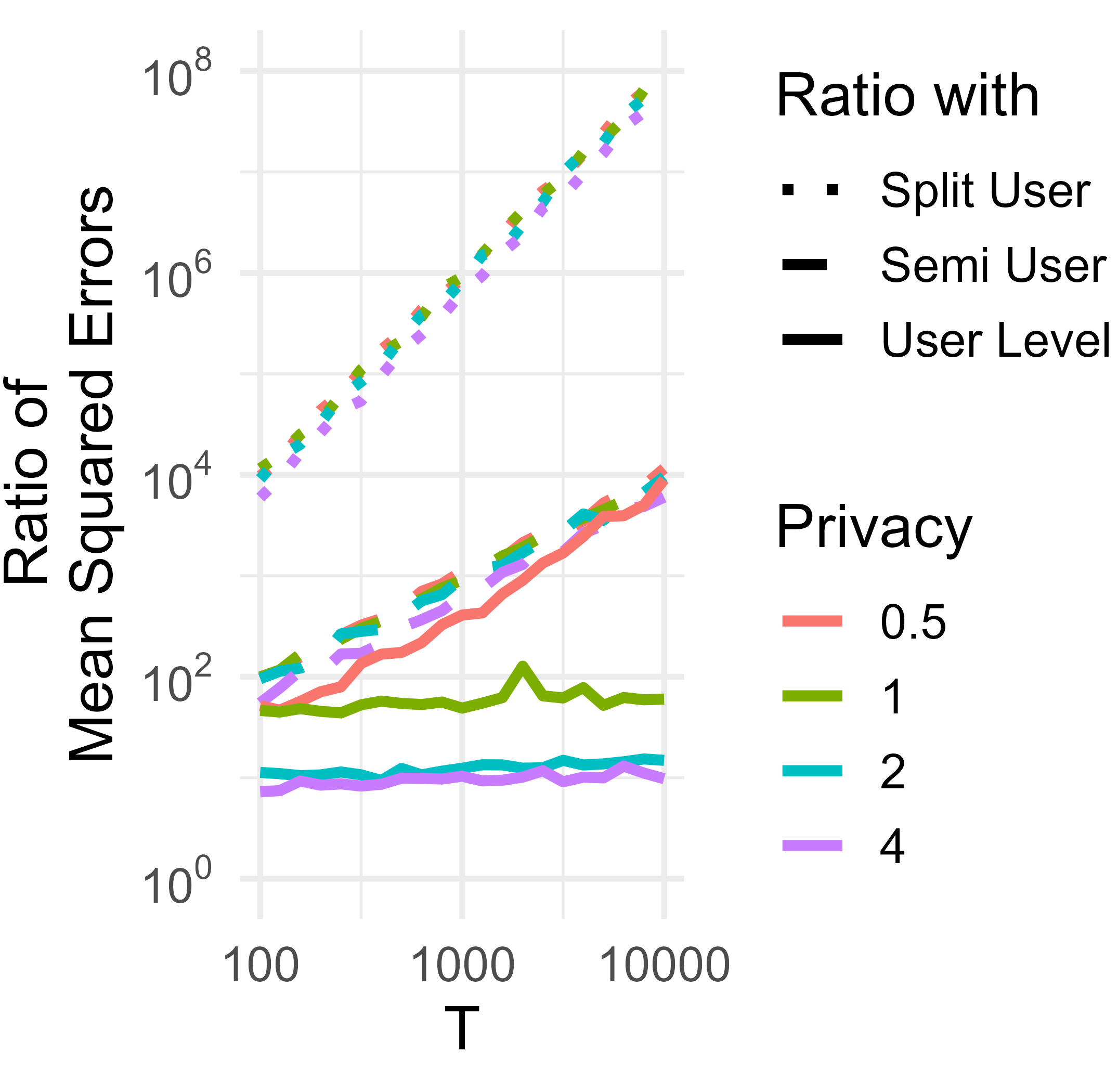}
                \includegraphics[height=0.162\textheight]{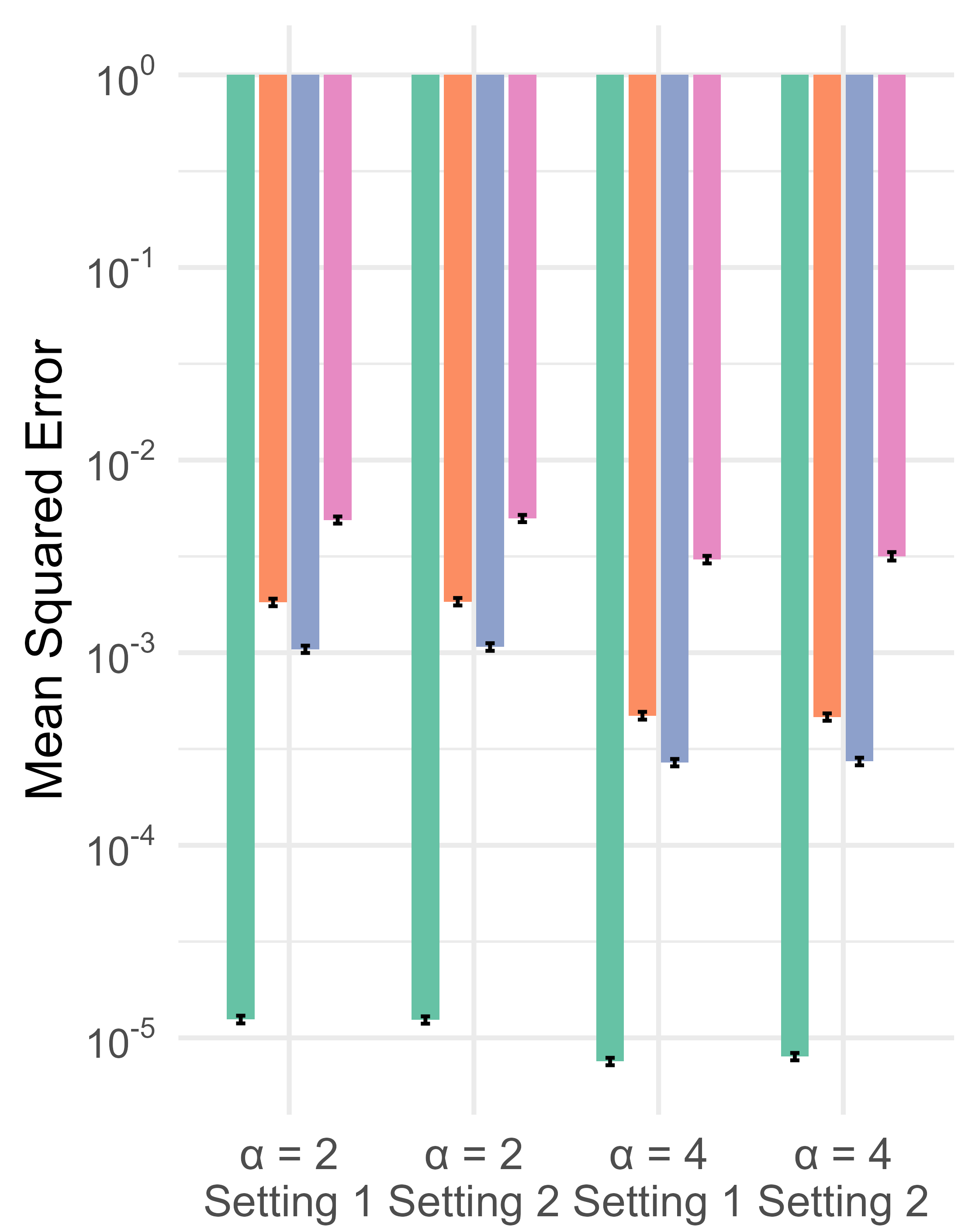}
                \includegraphics[height=0.162\textheight]{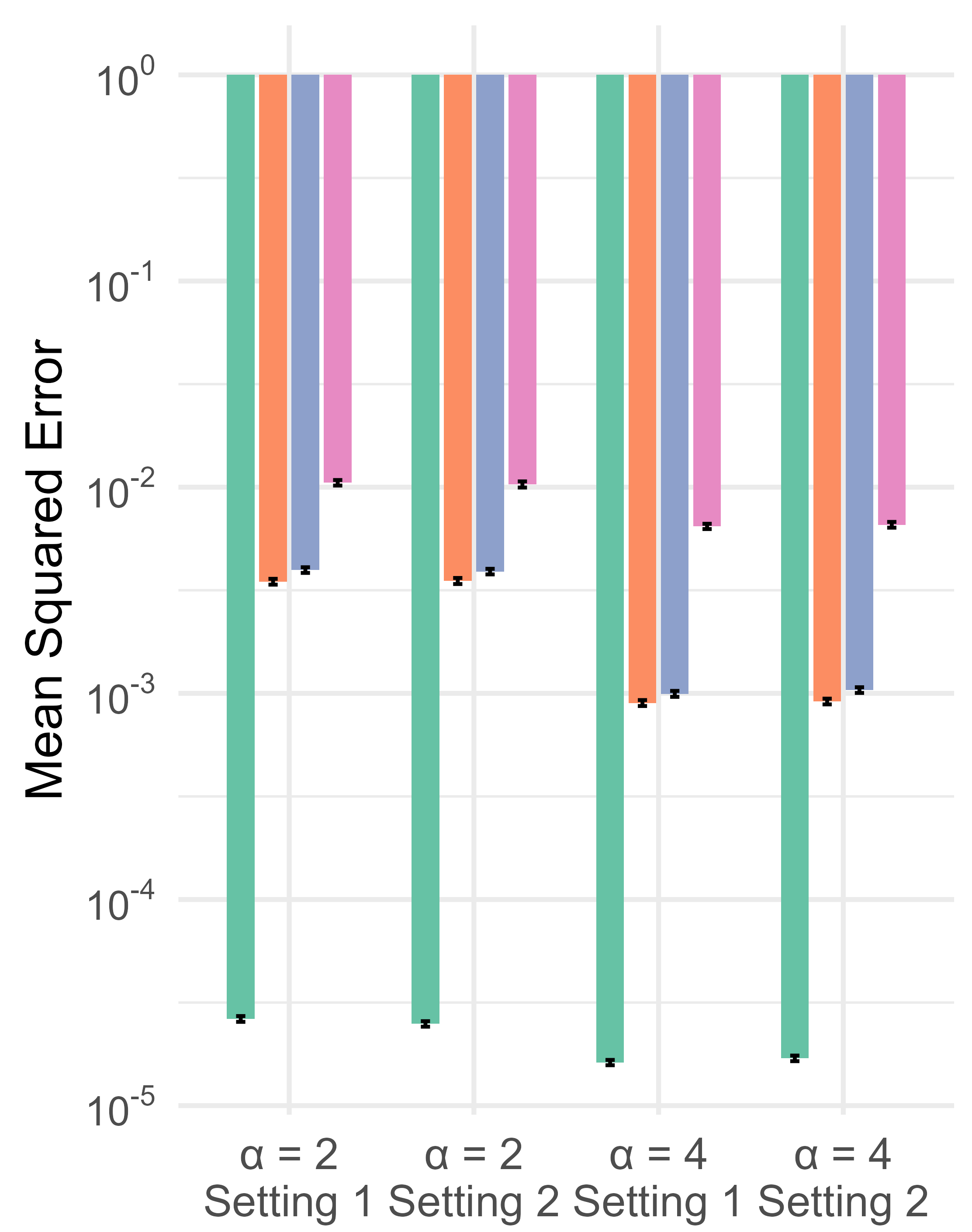}
                \includegraphics[height=0.162\textheight]{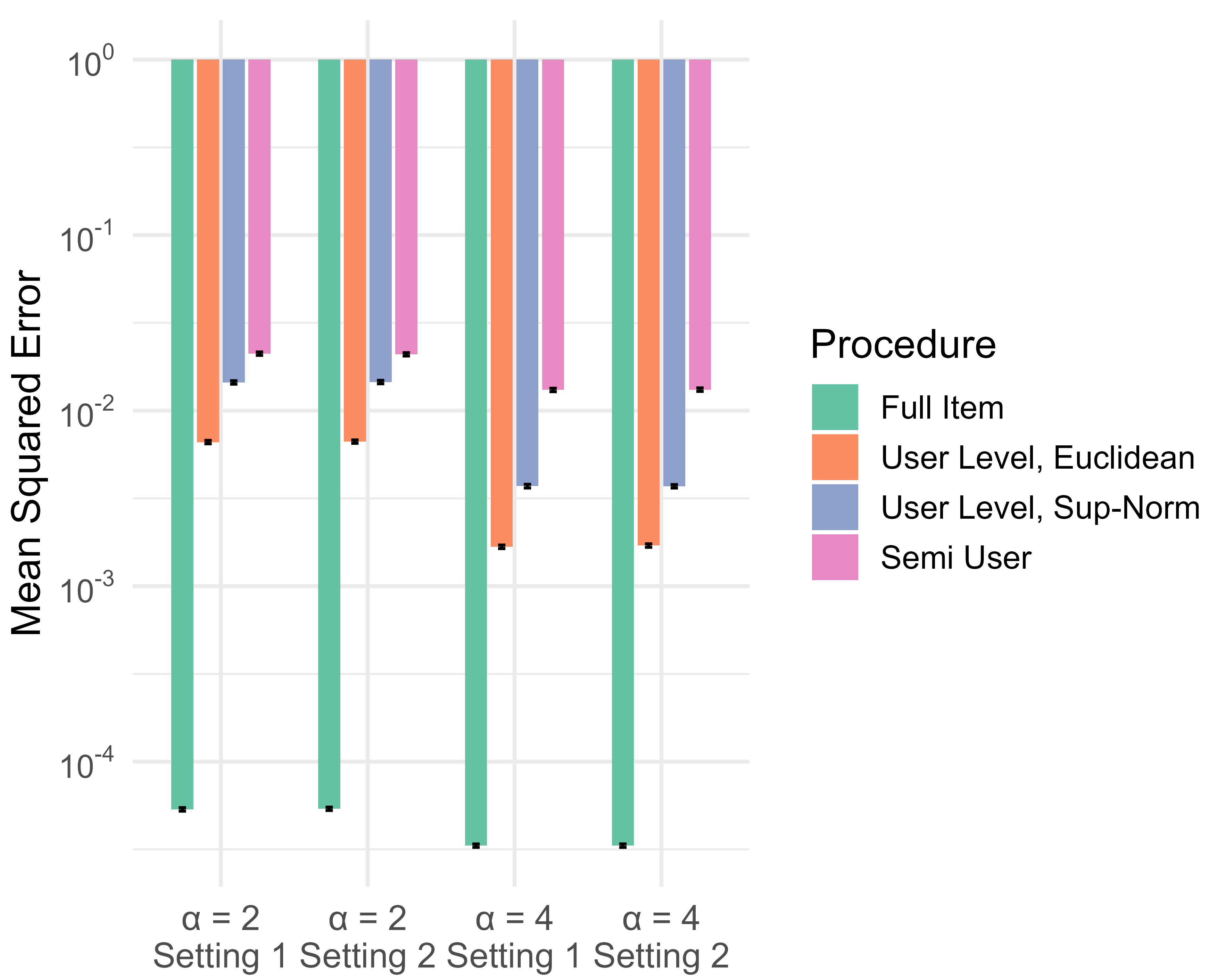}
                \caption{Ratio of mean-squared-error between full item level and other estimation schemes in univariate setting (left), and mean-squared-error of estimator in multivariate setting with $d \in \{8, 16, 32\}$ (centre-left, centre-right, right). Error bars denote two standard errors across 500 repetitions.}
                \label{fig:synthmean}
            \end{figure}
            
            We now consider the multivariate problem in the following two settings:
            
            \noindent
            \textbf{Setting 1}: Sample $X_{1:T}^{(1)}, \hdots, X_{1:T}^{(n)} \overset{\mathrm{i.i.d.}}{\sim} \mathrm{Rad}(0.9) \otimes \delta_0^{\otimes(d-1)}$. The data lie on the corner of an $\ell_1$-ball.
            
            \noindent
            \textbf{Setting 2}: Sample $\widetilde{X}_{1:T}^{(1)}, \hdots, \widetilde{X}_{1:T}^{(n)} \overset{\mathrm{i.i.d.}}{\sim} \mathcal{N}(0, 1)^{\otimes d}$, and let $X_{1:T}^{(i)} = \widetilde{X}_{1:T}^{(i)}/\|\widetilde{X}_{1:T}^{(i)}\|_2$ for $i \in [n]$. The data are uniformly distributed over the unit sphere in $d$ dimensions.

            In both settings, we fix $n = 4000$, $T = 400$, varying $\alpha \in \{2,4\}$ and $d \in \{8, 16, 32\}$ and calculate the mean-squared-error over 500 repetitions. As both settings satisfy that the data are supported on an $\ell_2$-ball, we are able to apply both our $\ell_2$- and $\ell_\infty$-ball estimation procedures and compare their performance. We also apply an optimal item-level procedure for the $\ell_2$-ball \cite[][Equation~25]{Duchi:2018} applied to each observation in full item level, and the locally calculated sample mean of each user's sample for the semi-user-level setting. The results of the simulations are in \Cref{fig:synthmean}. Firstly, comparing our two user-level procedures, we see the $\ell_2$-ball procedure outperforms the $\ell_\infty$-ball for sufficiently large $d$, validating the improvement we expect from \Cref{sec3:thm:main}. Comparing user-level to the other procedures, as expected full-item-level performs best, and semi-user-level worst, with the user-level methods always outperforming semi-user-level, at times by an order of magnitude. This again demonstrates the improvements attainable when taking advantage of the user-level setting. In fact, the user-level $\ell_\infty$-ball procedure, though not specialized for the $\ell_2$-ball setting, still outperforms the semi-user-level.
                
            \subsubsection{Sparse Mean Estimation}
               To investigate the performance of the sparse mean estimation procedure of \Cref{sec5}. Consider the following $s$-sparse distribution family
                \begin{equation} \label{secsim:eq:sparsefamily}
                    P_{d,s,p}^{\mathrm{sparse}} = \mathrm{Rad}(p)^{\otimes s} \otimes \mathrm{Rad}(0.5)^{\otimes (d-s)}.
                \end{equation}

                We compare with the item-level estimator of \cite{Duchi:2018}, the results of which state that under (non-interactive) item-level constraints, the minimax risk is $O(d\log(d)/(n\alpha^2))$. We vary the dimension $d$, setting $n = T= \lfloor \{10d\log(d)\}^{1/2} \rfloor$, and calculate the mean-squared-error over 500 repetitions. As the total number of samples $nT = O(d\log(d))$, consistent estimation should be impossible under the item-level setting. From the results in \Cref{fig:sparseMSE}, as $d$ grows the error of the item-level procedure stays bounded away from zero, whereas the user-level procedure estimator exhibits decreasing error, demonstrating that the user-level procedure can surpass the fundamental limitations of the item level setting. On the logarithmic scale, we see approximately linear decay for large $\alpha$, validating the minimax rate of $O(s/T)$ up to logarithmic factors in \eqref{sec5:eq:SparseMeanStatement}.

                \begin{figure}[htbp]
                    \centering
                    
                  \begin{subfigure}[t]{0.235\textwidth}
                    \centering
                    \includegraphics[height=0.135\textheight]{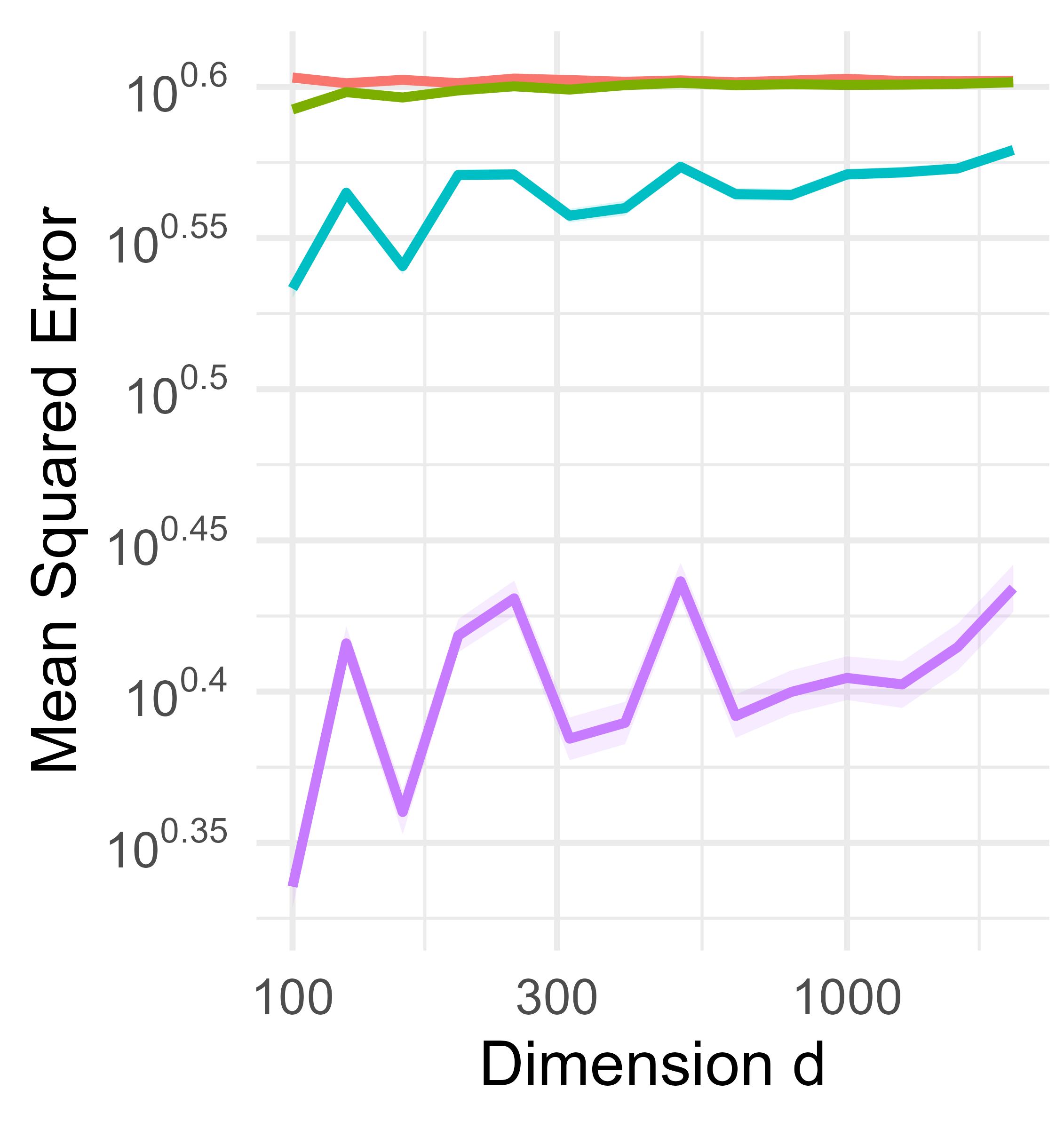}
                    \caption{\mbox{Item-level, $P_{100,4,0.75}^{\mathrm{sparse}}$.}}
                    \label{fig:sparseMSE:item:075}
                  \end{subfigure}%
                  \begin{subfigure}[t]{0.235\textwidth}
                    \centering
                    \includegraphics[height=0.135\textheight]{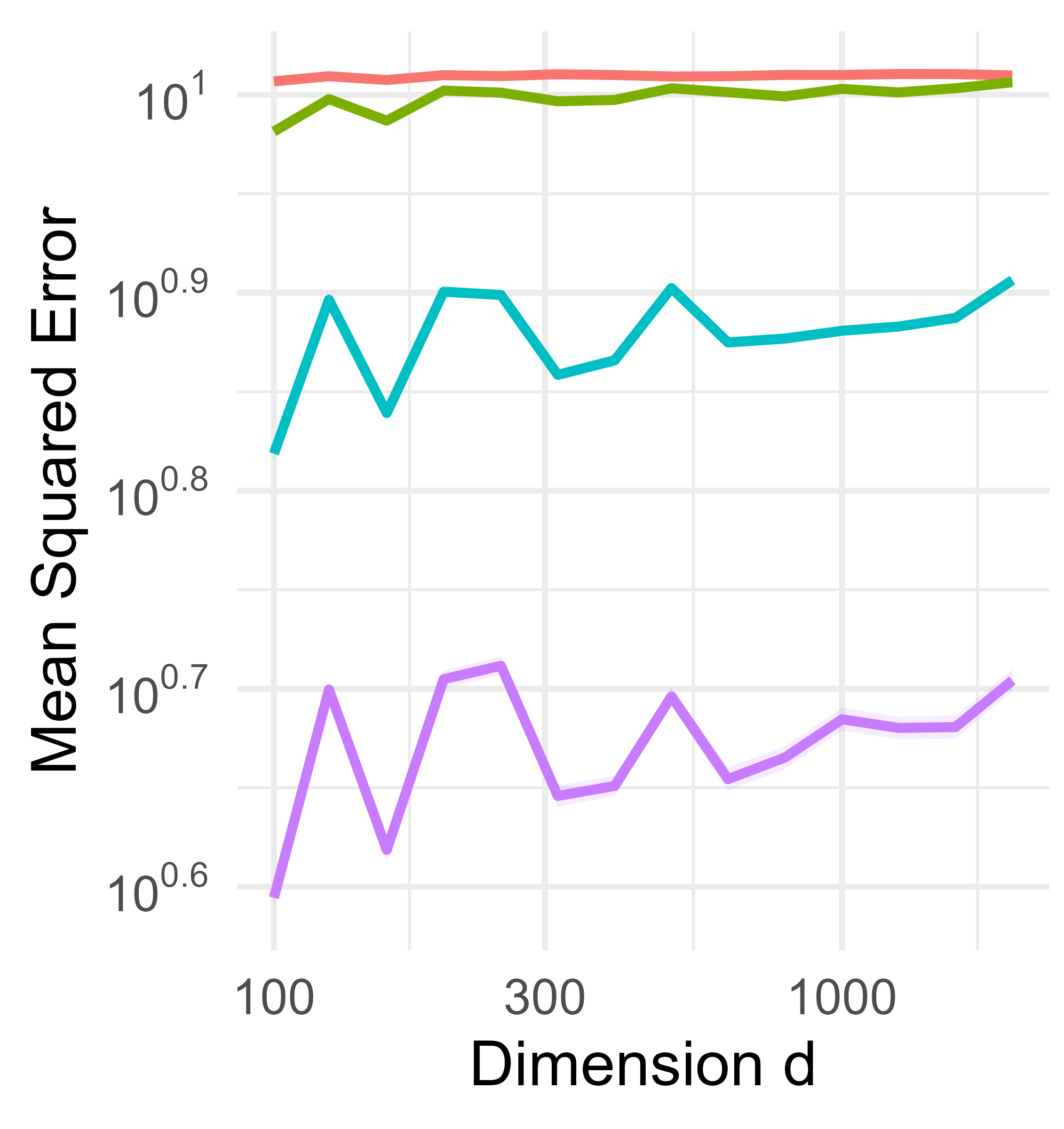}
                    \caption{\mbox{Item-level, $P_{100,4,0.9}^{\mathrm{sparse}}$.}}
                    \label{fig:sparseMSE:item:09}
                  \end{subfigure}%
                  \begin{subfigure}[t]{0.235\textwidth}
                    \centering
                    \includegraphics[height=0.135\textheight]{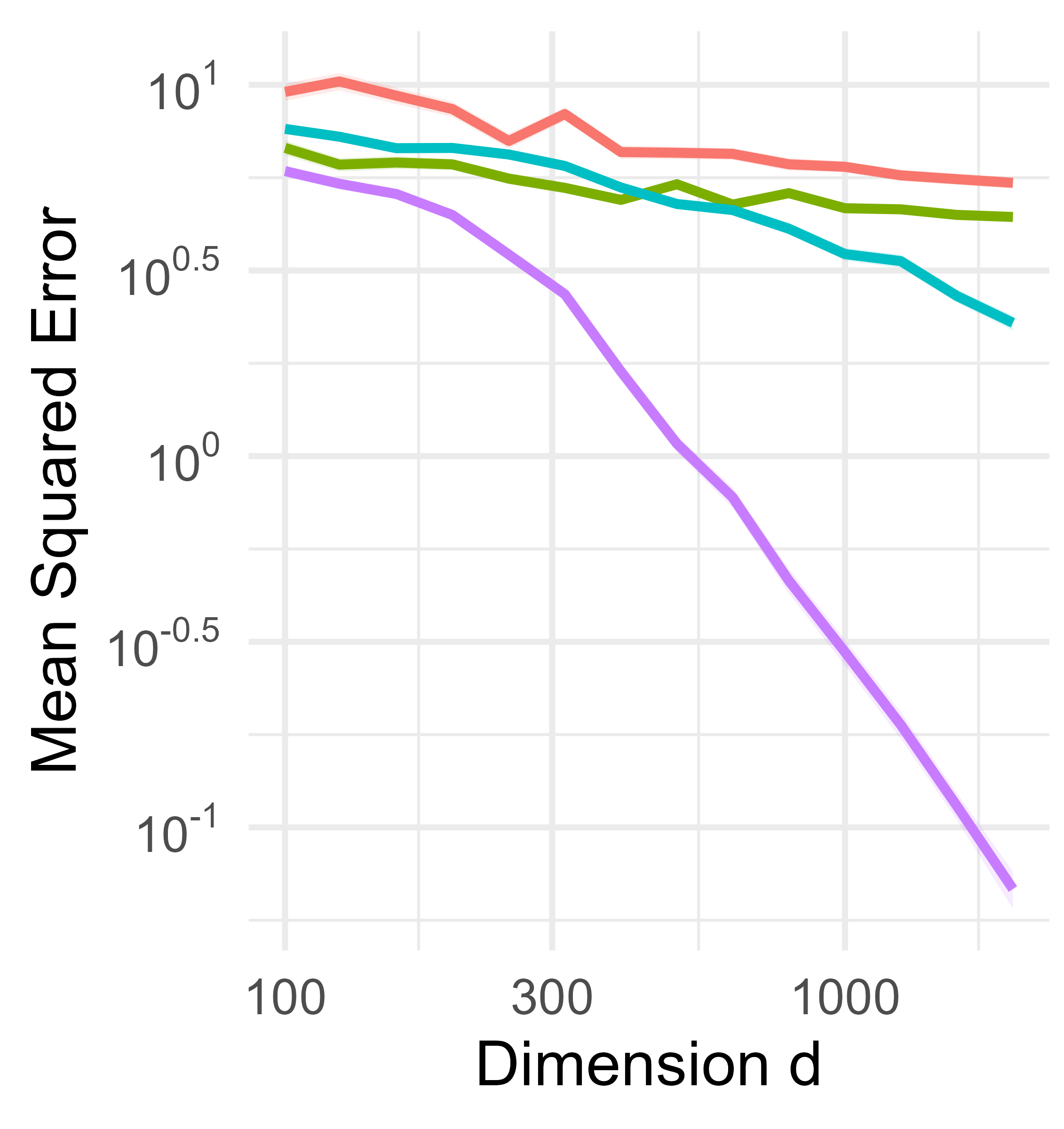}
                    \caption{\mbox{User-level, $P_{100,4,0.75}^{\mathrm{sparse}}$.}}
                    \label{fig:sparseMSE:user:075}
                  \end{subfigure}%
                  \begin{subfigure}[t]{0.235\textwidth}
                    \centering
                    \includegraphics[height=0.135\textheight]{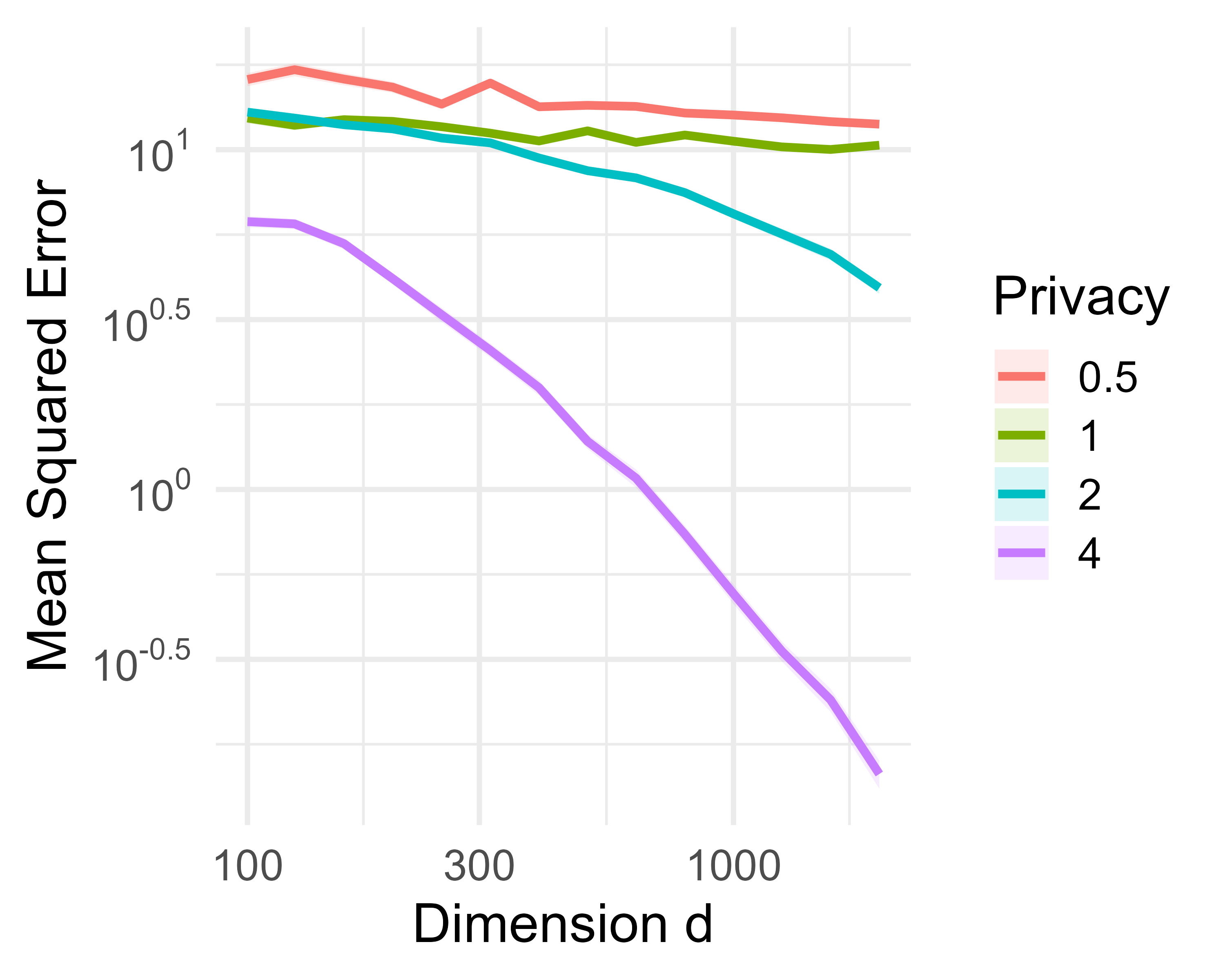}
                    \caption{\mbox{User-level, $P_{100,4,0.9}^{\mathrm{sparse}}$.}}
                    \label{fig:sparseMSE:user:09}
                  \end{subfigure}

                    \caption{MSE for the sparse mean-estimation procedures. Ribbons indicate one standard error across 500 repetitions.}
                    \label{fig:sparseMSE}
                \end{figure}
                
        \subsection{Real-World Data} \label{sec:sim:real}
            We now verify the applicability of our methods on real world data where the framework of user-level privacy naturally arises. We consider the Stanford Open Policing Project dataset \citep{Pierson:2020}, which contains information on traffic stops by police in the United States. In particular, we focus on the Texas State Patrol dataset which we denote \texttt{TX-Statewide} and which covers a total of 27,426,840 traffic stops across the state of Texas. Each row has an officer hash entry, allowing us to restrict to traffic stops conducted by an individual officer. We suppose a situation where each officer is obligated to provide information of the stops they have conducted, but are allowed to privatise the data before sending it. This naturally fits the user-level locally private setting as each officer will conduct multiple stops over a sufficiently long period of time and will have a local sample of observations.

            First, we filter out stops with missing values, then restrict data to the years $2015$ and $2016$ and for data recorded in district \texttt{A} yielding 1,153,346 traffic stops before grouping by officer hash. We then filter out officers with fewer than $T = 300$ observations, resulting in a total of $n = 822$ officers. We permute the data within each officer to minimise temporal dependency between stops by the same officer which are adjacent in the dataset, before finally truncating each officer's sample to be exactly $T = 300$ observations.

            To test the univariate mean estimation procedure, we focus on the \texttt{search\_conducted} variable, which indicates the binary outcome of whether a traffic stop led to a search of, for example, the vehicle or person. We then seek to estimate the proportion of stops that resulted in a search using the univariate mean estimation procedure of \Cref{sec3:univariate}. For semi-user-level, we calculate the sample mean within a user and add suitably scaled Laplace noise to privatise. We similarly use Laplace noise for the full-item-level setting. Lastly, we also consider a setting termed one-item-level, where each user contributes only a single observation privatised via the randomised-response mechanism. As this method only uses one observation per user, regardless of $T$, it will not demonstrate improved performance as $T$ grows.

            Taking the truth to be the proportion of stops over all officers, the mean-squared-error, scaled by $\alpha^2$, where we use an increasing share of each officer's sample is given in \Cref{fig:police}. Here, we vary $T \in \{100, 125, \hdots, 300\}$ and $\alpha \in \{0.5, 1, 2, 4\}$. We observe that, as $T$ grows, the performance of the full-item-level and user-level method improve, whilst that for one-item-level and semi-user-level remain constant, as expected. For $\alpha \geq 1$, the user-level method exhibits superior performance excepting the full-item-level baseline as expected.
            
            \begin{figure}[htbp]
                \centering
                \includegraphics[width=0.45\linewidth]{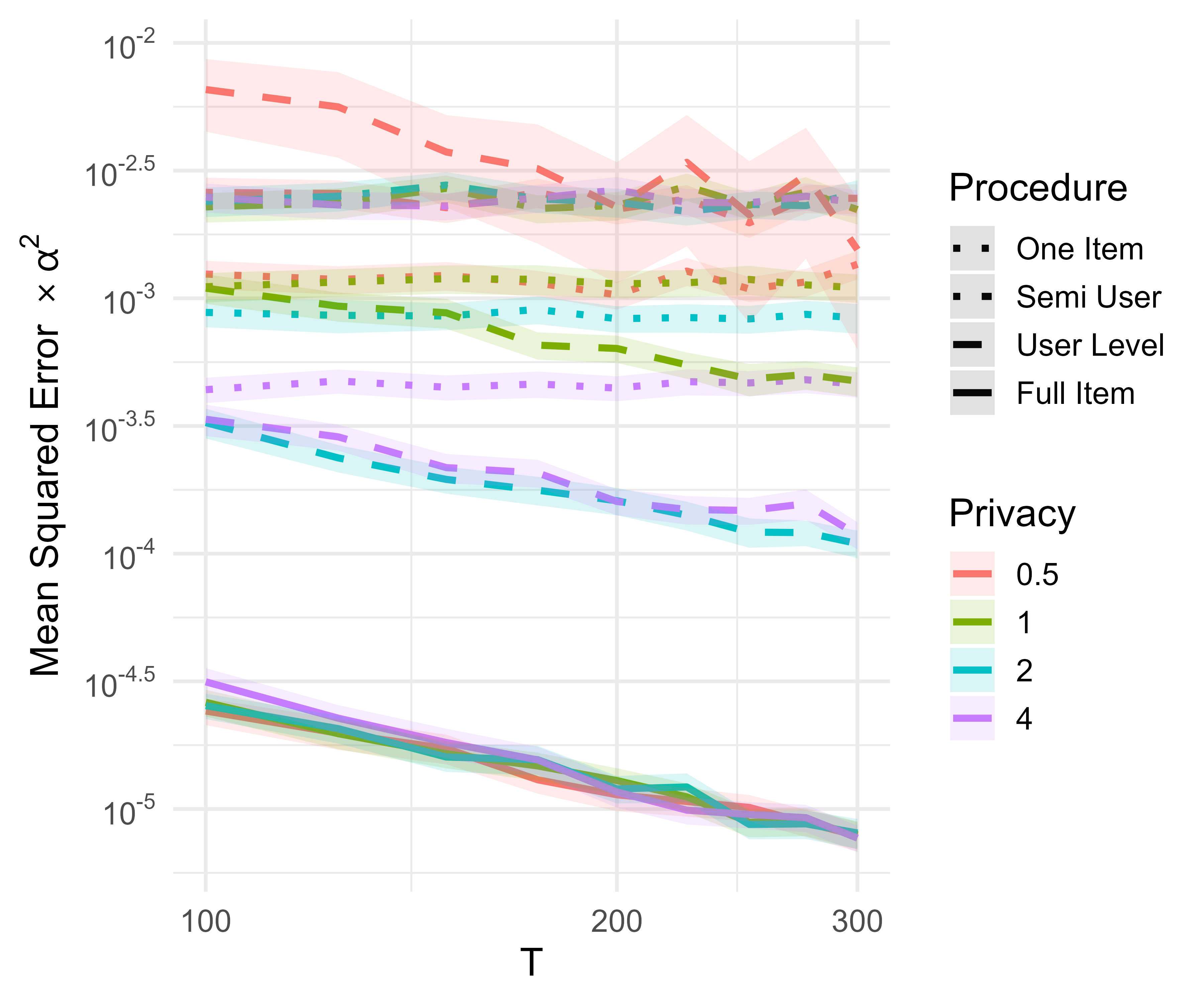}
                \includegraphics[width=0.45\linewidth]{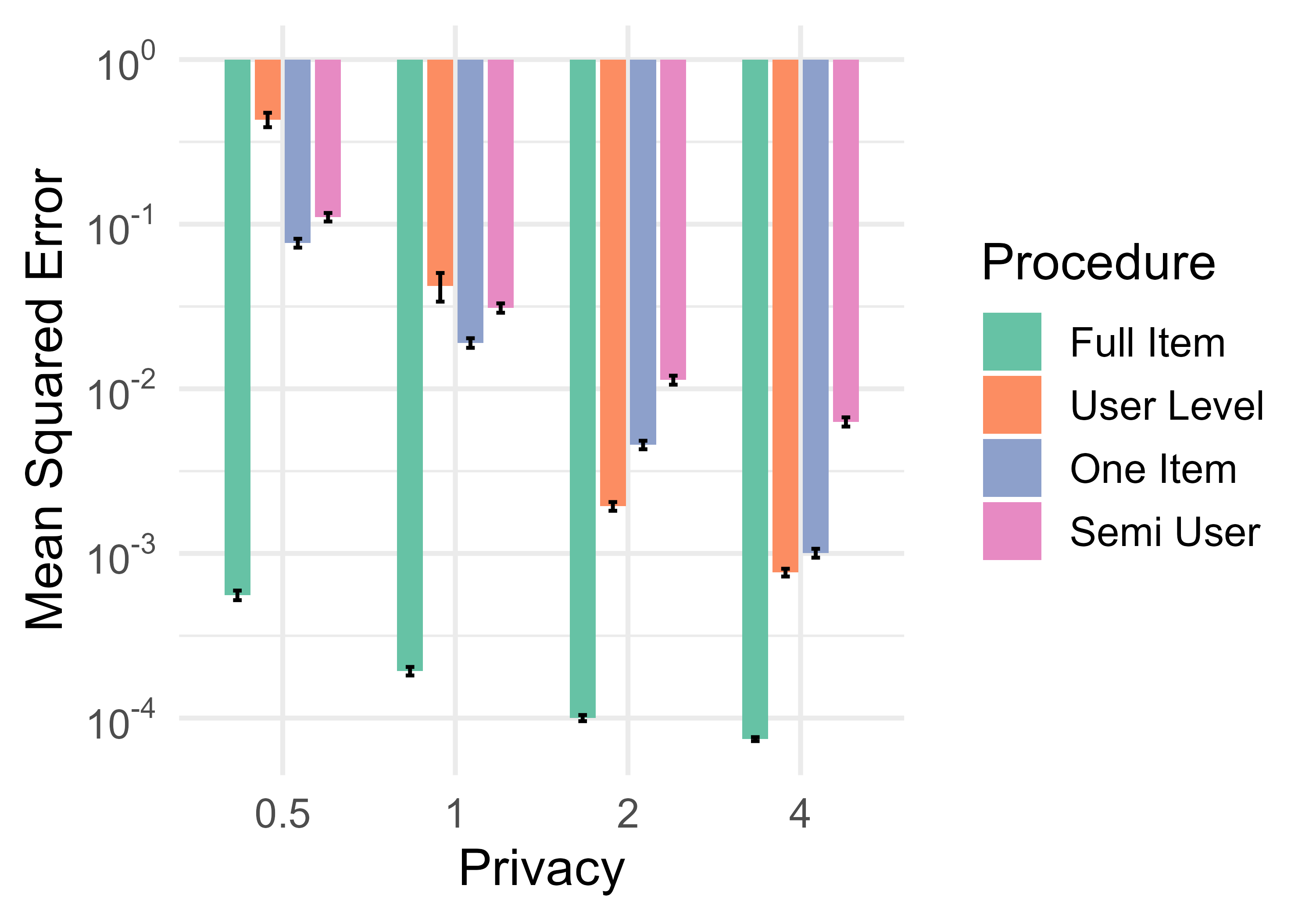}
                \caption{Mean-squared-error of estimator of proportion of traffic stops in the \emph{TX-Statewide} dataset leading to a search (left), and estimator of proportion of races of subjects of stops (right), with privacy parameters $\alpha \in \{0.5, 1, 2, 4\}$. Mean-squared-error in univariate setting is scaled by $\alpha^2$ to aid comparison across privacy levels. Ribbons and error bars denote two standard errors across 1,000 repetitions.}
                \label{fig:police}
            \end{figure}

            We next consider a multivariate problem for the same dataset. The \texttt{subject\_race} variable denotes the race of the traffic stop subject and takes one of the values \{\texttt{white, black, hispanic, asian/pacific islander, other, unknown}\}. We group the latter three outcomes into \texttt{other}, leaving four possible values. The category is then encoded as a four-dimensional one-hot vector, and we seek to estimate the proportion of stops that each group constituted. We consider the following estimation schemes. For full-item-level, the users are split evenly between the groups use the Laplace mechanism to privatise. We employ our $\ell_\infty$-ball method, \Cref{alg:PrivateMultiDimMean} for user-level, the $\ell_2$-ball method of (Duchi et al., 2018, Equation 25) applied to the locally calculated sample mean for semi-user-level, and lastly use the unary-encoding mechanism for one-item-level. The results are in \Cref{fig:police}, where we again see full-item-level performs best, with our user-level method outperforming one-item and semi-user-level for $\alpha \geq 2$.
    
\section{Conclusions}\label{sec-conclusion}

In this paper, we considered a range of canonical statistical problems under the user-level local differential privacy framework. We provided, up to poly-logarithmic factors and constants in exponents when exponential rates occur, matching upper and lower bounds on the minimax rates in all regimes of $T$ for non-sparse mean estimation and non-parametric density estimation, and similarly for sparse mean estimation up to a logarithmic gap in the regime. We further characterise phase transition phenomena as the number of samples held by each user varies.

\begin{table}[H]
    \centering
    \begin{tabular}{lccc}
    \hline\hline
    & $d$-dim.~mean ($\ell_2$-ball) & $s$-sparse mean & Density \\ \hline
    Small-$T$ rate   & $ d/(nT\alpha^2)$ & $\min\{s/T, sd/(nT\alpha^2)\}$ & $(nT\alpha^2)^{-2\beta/(2\beta + 2)}$       \\ 
    Large-$T$ rate &  $e^{-n\alpha^2/d}$& $e^{-n\alpha^2/s}$ & $(n\alpha^2)^{-2\beta}$\\
    Phase transition $(T = \cdot)$ &  $O\big(e^{n\alpha^2/d}\big)$& $\begin{cases}
        O\big(e^{n\alpha^2/s}\big), & d/(n\alpha^2) \gtrsim 1\\
        O\big(e^{n\alpha^2/d}\big), & d/(n\alpha^2) \lesssim 1
    \end{cases}$& $O\left((n\alpha^2)^{2\beta+1}\right)$\\ \hline
    \end{tabular}
    \caption{Comparison of user-level rates and phase transitions. Dependence on multiplicative constants, constants in exponents, and logarithmic factors are suppressed for brevity.}  \label{tab-3}
\end{table}
 
Minimax optimal rates for user-level LDP problems, in particular proofs of lower bounds, have not been studied systematically before. \cite{Acharya:2023} prove matching upper and lower bounds, but restricts to a certain range of $T$, not considering the behaviour as $T$ diverges. We on the other hand consider all regimes of $T$. This provides insight on the fundamental limits of these problems in a way that appears to be intrinsically linked to the metric entropy of the parameter space. As observed in \cite{Acharya:2023}, the user-level rate matches the item-level rate with an equivalent number of users in the ``small-$T$'' regime in discrete distribution estimation problems. It was not known whether this was a special case or more general. We observe different behaviour for sparse mean estimation problems, bucking this trend and showing it is not universal behaviour for user-level LDP estimation. This is also particularly striking as the high-dimensional sparse mean estimation problem is intractable under item-level privacy, but becomes feasible under the user-level framework. This itself may have useful applications in practice, but could also motivate further exploration to identify other statistical problems where item-level privacy is pessimistic, especially where there is an indication of the ability for users to localise onto relevant factors of the data locally.

Considering the method used to derive the upper bound in \Cref{sec2:thm:GeneralBound}, whilst the method only works in the infinite-$T$ case, it suggests a general method that might be applied to the finite-$T$ case. For example, as can be seen from the univariate estimation method in \Cref{sec3}, one proceeds by covering the parameter space with intervals (resp.~balls) of a certain width (resp.~radius), using a private voting procedure to select a candidate, and then localising on this candidate to extract the benefits of each user having multiple samples. However, difficulties arise in the finite-$T$ case when trying to vote on a candidate as the error in the estimate of the functional can result in the votes being split over neighbouring candidates, with the number of neighbouring candidates increasing as the dimension increases. If some technique to circumvent these issues can be developed, it would provide a general procedure for finite-$T$ that would provide an upper bound that intrinsically depends on the metric entropy of the space. Combined with the general lower bound of Theorem \ref{sec2:eq:GeneralLB}, this would show that the fundamental limits of user-level LDP estimation are entirely dependent on, and can be deduced from, the metric entropy of the parameter space.

Lastly, we present further results in the appendices, including the algorithms for mean estimation under the low-privacy regime, where $\alpha \gtrsim 1$. As mean estimation is an important sub-routine in other algorithms, such as our non-parametric density estimation and sparse mean estimation procedures, this serves as a useful stepping stone for further work for user-level LDP in the low-privacy regime. We also obtain results for a generalisation of the user-level setting under data heterogeneity which can be found in the appendix. To the best of our knowledge, this is the first consideration of user-level LDP with data heterogeneity. Both of these are interesting settings, opening up new avenues for further study.

We end by noting additional problems, previously considered under item-level LDP, which could be extended to user-level LDP. The methods we analyse in this paper may serve as a foundation to tackle further interesting problems such as private quantile estimation \citep[e.g.][]{Liu:2023} and estimation of functionals of distributions \citep[e.g.][]{Butucea:2021, Butucea:2023}.

\section*{Acknowledgements}

The second author was supported by European Research Council Starting Grant 101163546. The third author was partially supported by the Philip Leverhulme Prize.

\bibliographystyle{plainnat}
\bibliography{bibliography}

\appendix

\section*{Appendices}
The appendices are organised as follows: In \Cref{app:sec:moresims}, we present further numerical simulations. The problem of estimating the mean of a mixture distributions under user-level LDP is considered in \Cref{secmix}. The proofs for results of \Cref{sec2} are contained in \Cref{Appendix_sec2}. We state and prove an extension of the results of \Cref{sec2} to the setting of approximate LDP in \Cref{app:sec:approxLDP}. We extend the results of \Cref{sec3} to the low-privacy setting in \Cref{app:sec:meanlowpriv}. We discuss the computational complexity of our estimation procedures in \Cref{app:sec:CompComplex}. For the user-level rates of Sections \ref{sec3}, \ref{sec5} and \ref{sec4}, and Appendices \ref{secmix} and \ref{app:sec:meanlowpriv}, we separate the proofs of the upper bounds and lower bounds into \Cref{app:sec:Upper} and \Cref{app:sec:Lower} respectively. 
Lastly, in \Cref{app:sec:misc} we collect auxiliary technical details.

\section{Additional Numerical Results} \label{app:sec:moresims}
    In this appendix we collect additional results of numerical simulations on synthetic data, including simulations pertaining to sensitivity analyses, and those of the mixture mean estimation setting considered in \Cref{secmix}.
    \subsection{Phase Transition Phenomena} \label{sec:phasetransition}
        We empirically validate the phase transition phenomena as $T \rightarrow \infty$. We focus on the case of univariate mean estimation, employing our estimation procedure developed in \Cref{sec3:univariate}, as the phase transition phenomena for the other estimation procedures developed ultimately arise from the univariate procedure being applied as a subroutine.
        
        We recall that the localisation stage of our developed procedures involve dividing the interval into sub-intervals. In particular, if the mean of a distribution lies close to the boundary of two sub-intervals as opposed to being close to the centre, the performance will decrease slightly. Without adjustments, this may cause the simulation results to appear noisy in a distribution dependent manner. Hence, where appropriate, we specify random ``shifts'' within our experiments to average out this behaviour.
        
        We generate data as follows. We sample a shift $U \sim \mathrm{Unif}[-0.3, 0.3]$ and generate $X_{1:T}^{(1)}, \hdots, X_{1:T}^{(n)} \mid U \overset{\mathrm{i.i.d}}{\sim} \mathrm{Rad}(1/2)^{\otimes T} + U\mathbf{1}$ for $\mathbf{1}$ the $T$-dimensional vector of ones. We vary $n \in \{100, 200\}$, $T \in \{\lfloor 10^m \rfloor: m = 2, 2.2, 2.4, \hdots, 7\}$ and $a \in \{0.5, 1, 2, 4\}$. Further, we manually fix the sub-interval width given in \eqref{sec3:eq:deltaval}, varying $\Delta \in \{0.0005, 0.0010, 0.0015, \hdots, 0.15\}$. We estimate the mean-squared-error for each fixed $\Delta$ across 20,000 repetitions, and for a given $(n, T, \alpha)$ triplet pick the minimum mean-squared-error across all $\Delta$.

        The results of the simulations are presented in \Cref{fig:phasetransition}. On the logarithmic scale, for $\alpha \in \{0.5, 1, 2\}$ we see varying rates at which the curves decay, before starting to flatten, showing for an optimally tuned $\Delta$, the mean-squared-error will not vanish as $T \rightarrow \infty$ for fixed $n$. For $\alpha = 4$, we note a consistent linear decay, even for $T = 10^7$, many orders of magnitude greater than the modest values of $n \in \{100, 200\}$, and for $n = 200$ we observe linear decay for $\alpha = 2$ until roughly $T = 10^5$. As $\alpha \geq 2$ is often taken in practice (for example $\alpha \in \{4, 8\}$ is considered in \citealt{Apple:2023}), this demonstrates that the phase transition phenomena are unlikely to cause significant issues with implementing local user-level methods in practice. In particular, in all our simulations we disregard the issues which would be caused by truncating $T$, and substitute $T$ in place of all occurrences of a corresponding $T^\ast$.

        \begin{figure}[h]
            \centering
            \includegraphics[height=0.2\textheight]{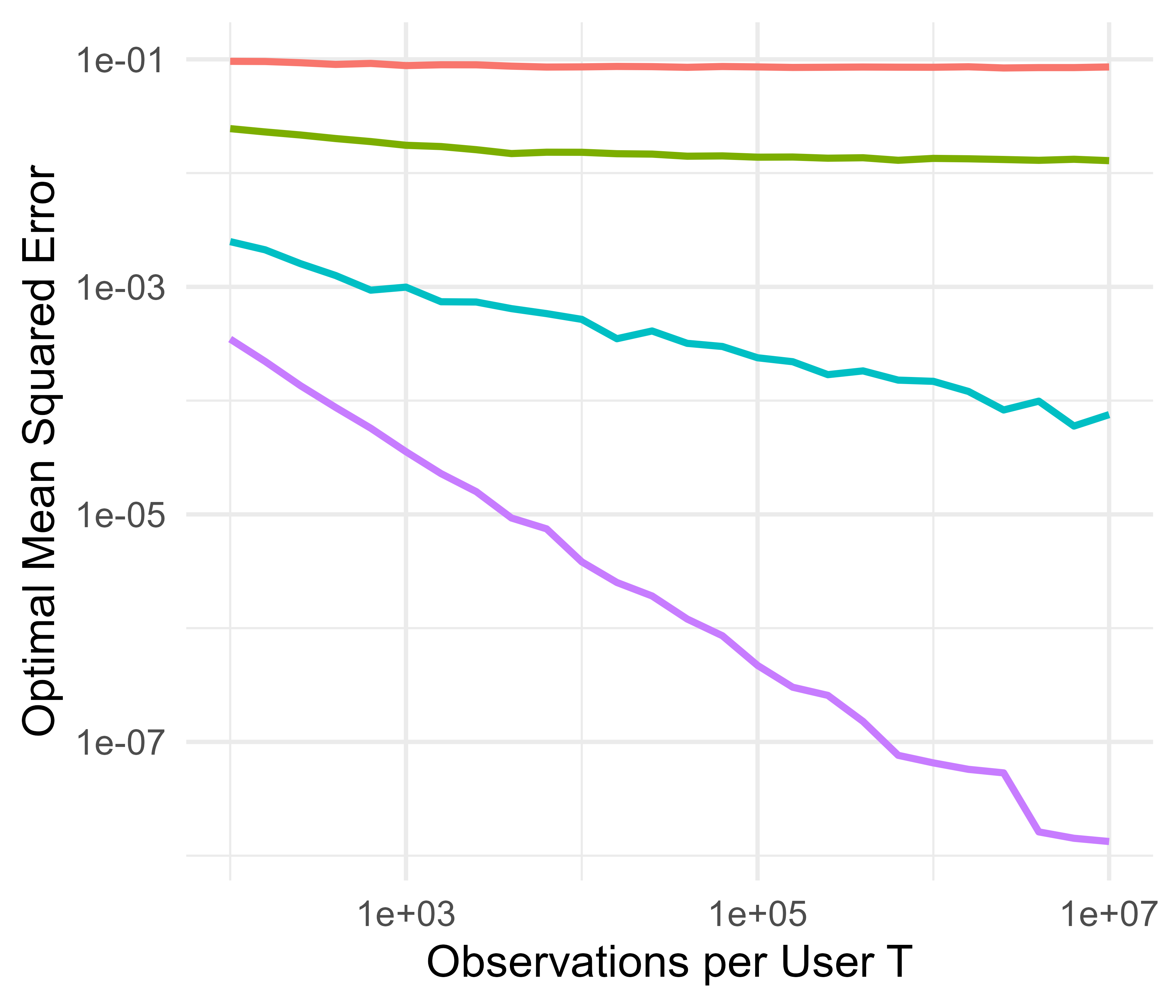}
            \includegraphics[height=0.2\textheight]{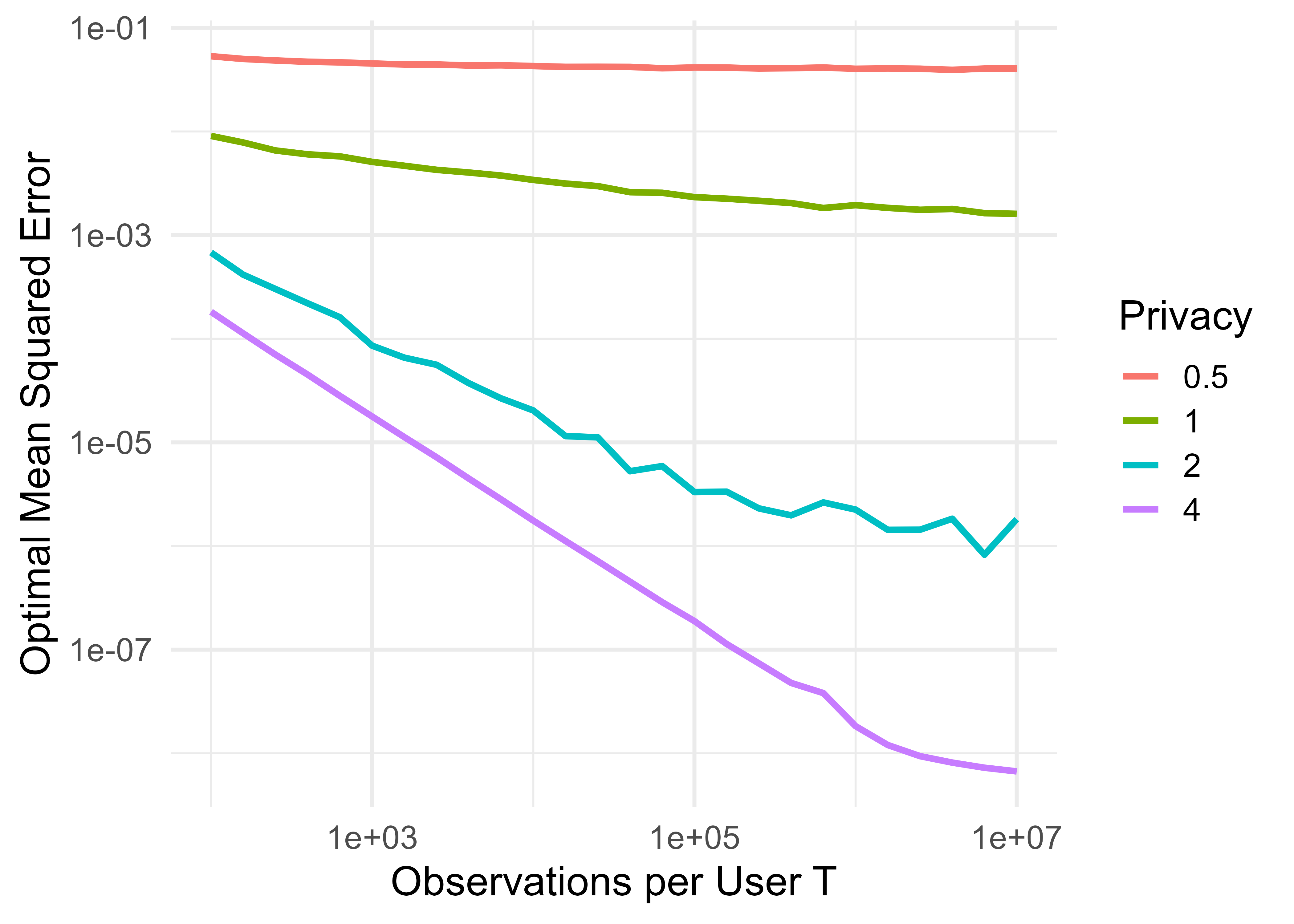}
            \caption{Minimum mean-squared-error across specified sub-interval widths $\Delta$ against $T$, the number of observations per user, across varying privacy levels, with $n = 100$ (left) and $n = 200$ (right).}
            \label{fig:phasetransition}
        \end{figure}

    \subsection{Sensitivity Analyses}\label{sec-sens}
        In this section, we consider the sensitivity of our methods to important tuning parameters. We focus on the role of three different parameters of interest which we highlight.

        \begin{itemize}
            \item The first parameter of interest arises when localising. The parameter of interest is the sub-interval width $\Delta$ defined in \eqref{sec3:eq:deltaval}. In particular, rewriting in terms of an arbitrary constant we have that
            \begin{equation*}
                \Delta_C = C\sqrt{\frac{\log(nT\alpha^2)}{T}}.
            \end{equation*}
            Based on the results of \Cref{sec:phasetransition}, we use $T$ in place of $T^\ast$ in the definition of \eqref{sec3:eq:deltaval} for simplicity. We measure the performance as we vary the value of $C$.

            \item The second parameter of interest regards the localisation step during sparse mean estimation. In \eqref{sec5:eq:truncwidtheq}, the user truncates the sum of the hashes to the interval $[-\eta, \eta]$ where $\eta = \sqrt{30s}$. The value $30$ is taken to aid the theoretical analysis but may not be the best in practice. Denoting $\eta_C = \sqrt{Cs}$, we inspect the sensitivity of the sparse mean estimation procedure as $C$ is varied.

            \item The last parameter we investigate is the assumed sparsity, again in sparse mean estimation. The results of \Cref{sec5} assume that the sparsity level $s$ is known in advance, but this may not always be the case in practice. We investigate the effect of misspecification of the true sparsity.
        \end{itemize}
        
        Starting with the first problem, we focus on the univariate mean estimation procedure as in the other procedures this tuning parameter only arises through the univariate sub-routine. Intuition suggests that for $C$ too small, the sub-intervals are harder to identify during the localisation stage. For $C$ too large, more noise needs to be added to privatise the estimates in the refinement stage. We consider two settings, the first to demonstrate what type of distribution is most sensitive to this parameter, and the second to find suitable values of the tuning parameter.

        \noindent
        \textbf{Setting 1}: Fix $n = 200, T = 100$ and consider, for $B_{p,p}$ the Beta distribution with shape parameters $(p,p)$, $X_{1:T}^{(1)}, \hdots, X_{1:T}^{(n)} \overset{\mathrm{i.i.d}}{\sim} (2B_{p,p} - 1)^{\otimes T}$ for $p \in \{0.1, 0.5, 1, 3, 5\}$, and estimate the mean-squared-error. We vary both $\alpha \in \{0.5, 1, 2\}$ and $C \in \{0.01, 0.02, \hdots, 1.00\}$. We estimate the mean-squared-error across 10,000 repetitions.
        
        \noindent
        \textbf{Setting 2}: Sample a shift $U \sim \mathrm{Unif}[-0.3, 0.3]$ and generate $X_{1:T}^{(1)}, \hdots, X_{1:T}^{(n)} \mid U \overset{\mathrm{i.i.d}}{\sim} \mathrm{Rad}(1/2)^{\otimes T} + U\mathbf{1}$ for $\mathbf{1}$ the $T$-dimensional vector of ones. We vary $n,T \in \{100, 150, \hdots, 500\}$, $\alpha \in \{0.5, 1, 2, 4\}$ and $C \in \{0.01, 0.02, \hdots, 1.00\}$. We estimate the mean-squared-error across 500 repetitions.

        The results of Setting~1 are in \Cref{fig:sensitivitysetting1}. In all cases of the value of $\alpha$, we see that picking $C$ too small results in a large mean-squared-error. As $C$ increases, it decays to some minimum, before gracefully increasing. In particular, the effects of small $C$ are more adverse for the distributions with larger variance, that is, for small shape parameter $p$, whereas the mean-squared-error of the distributions coincide for large $C$. This matches with the intuition that the reliability of a user's data landing in the correct sub-interval is distribution dependent, but the noise incurred from privatisation for large $C$ is distribution agnostic. In summary, it appears preferable to pick $C$ larger than optimal rather than smaller.
        \begin{figure}[htbp]
            \centering
            \includegraphics[height=0.15\textheight]{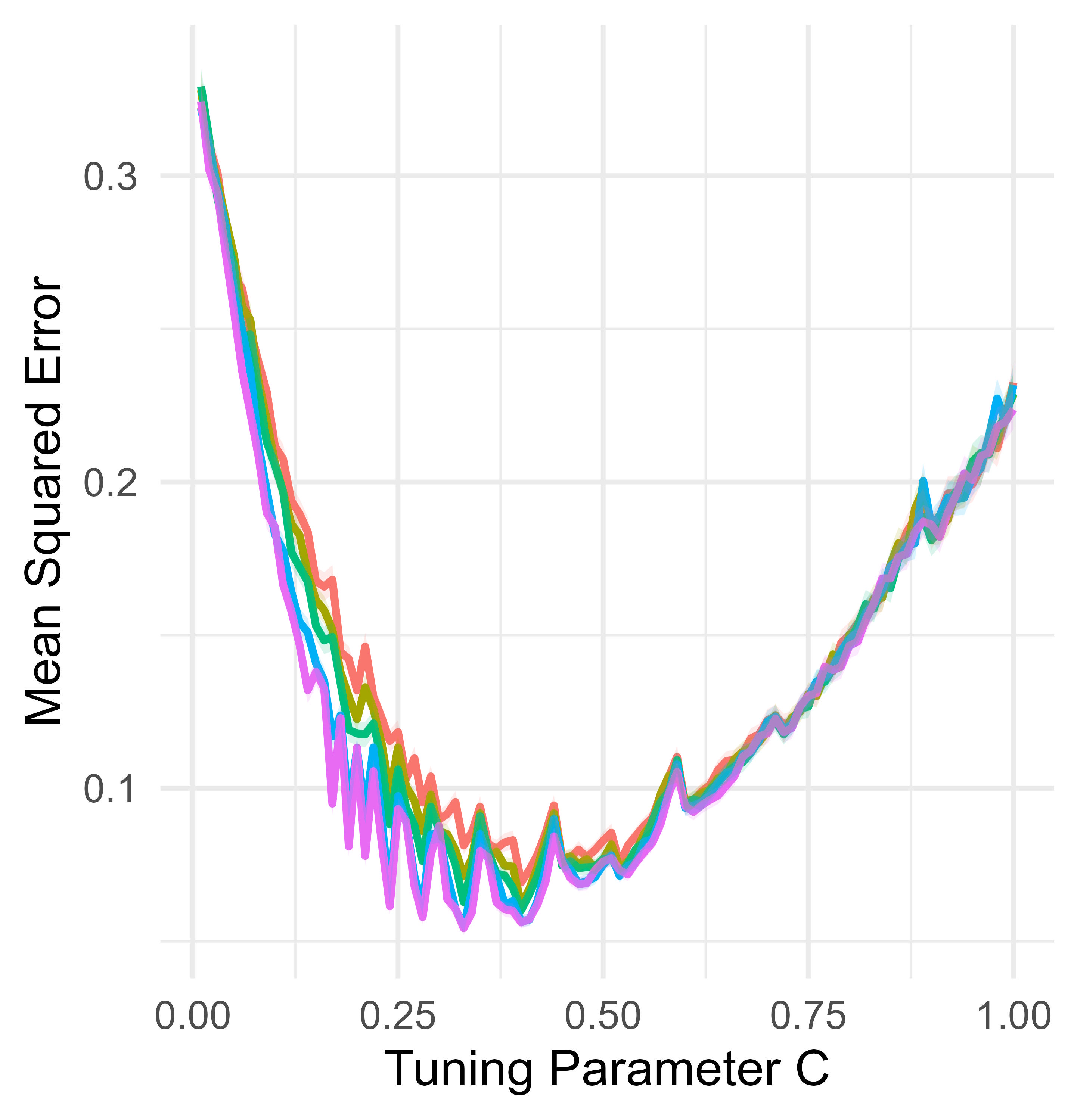}
            \includegraphics[height=0.15\textheight]{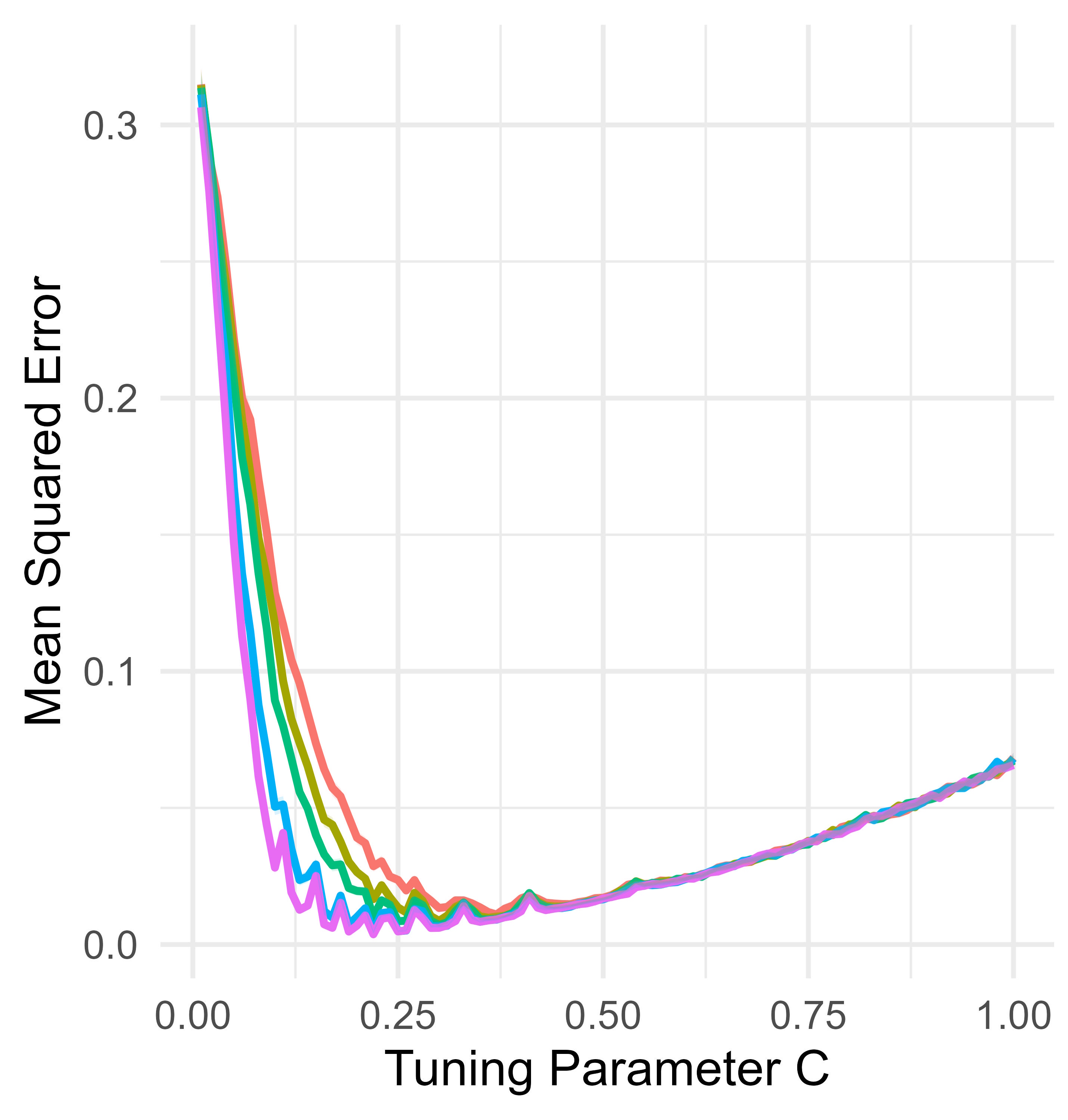}
            \includegraphics[height=0.15\textheight]{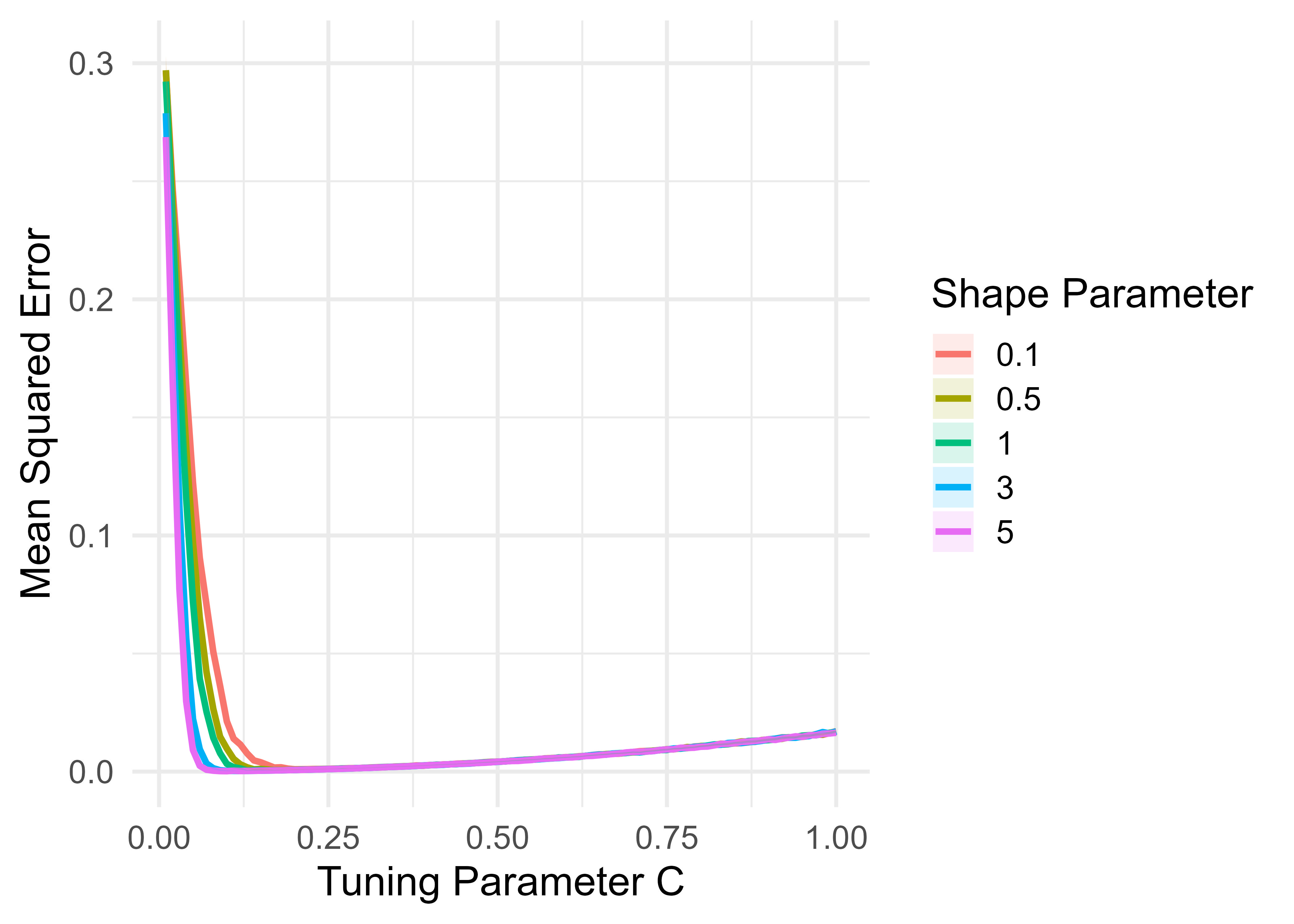}
            \caption{Estimated mean-squared-error against tuning constant for privacy parameter $\alpha = 0.5$ (left), $\alpha = 1$ (middle) and $\alpha = 2$ (right). Shape parameter of Beta distribution indicated by line colour.}
            \label{fig:sensitivitysetting1}
        \end{figure}

        Motivated by the results of Setting~1, where the distributions with greatest variance were the most adversely affected by small $C$, we consider the limiting case of Rademacher distributions as in Setting~2, which has maximal variance for a distribution supported on $[-1,1]$.

        The results of Setting~2 are contained in \Cref{fig:sensitivitysetting2a} and \Cref{fig:sensitivitysetting2b}, wherein we fix $T$, $n$ respectively and grow the other. In both figures, we see that the error decays as $C$ grows for all configurations of $n$ and $T$, before rising, with the case of $n = T = 100$ starting to rise the earliest in $C$.

        \begin{figure}[htbp]
            \centering
            \includegraphics[height=0.15\textheight]{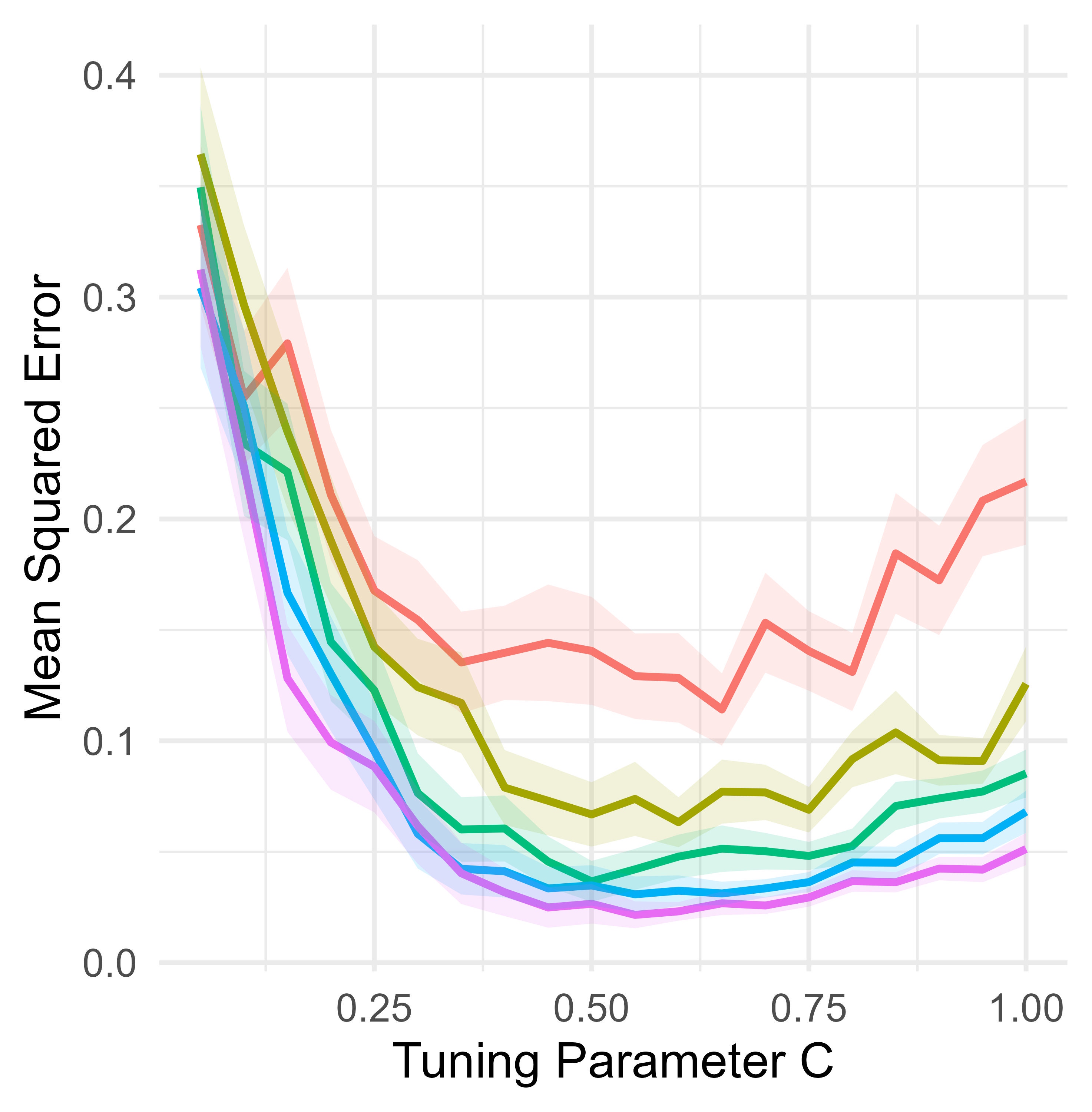}
            \includegraphics[height=0.15\textheight]{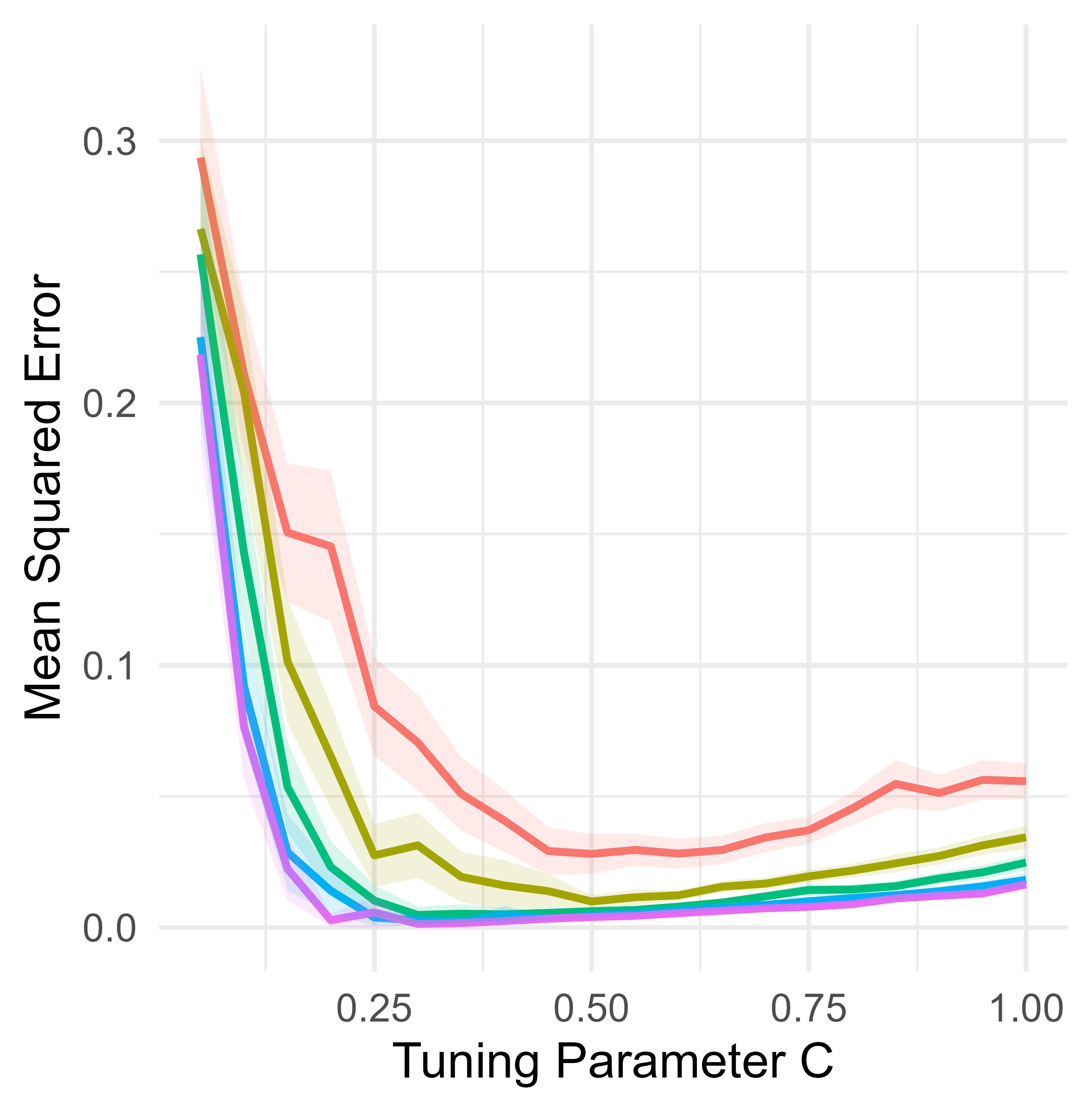}
            \includegraphics[height=0.15\textheight]{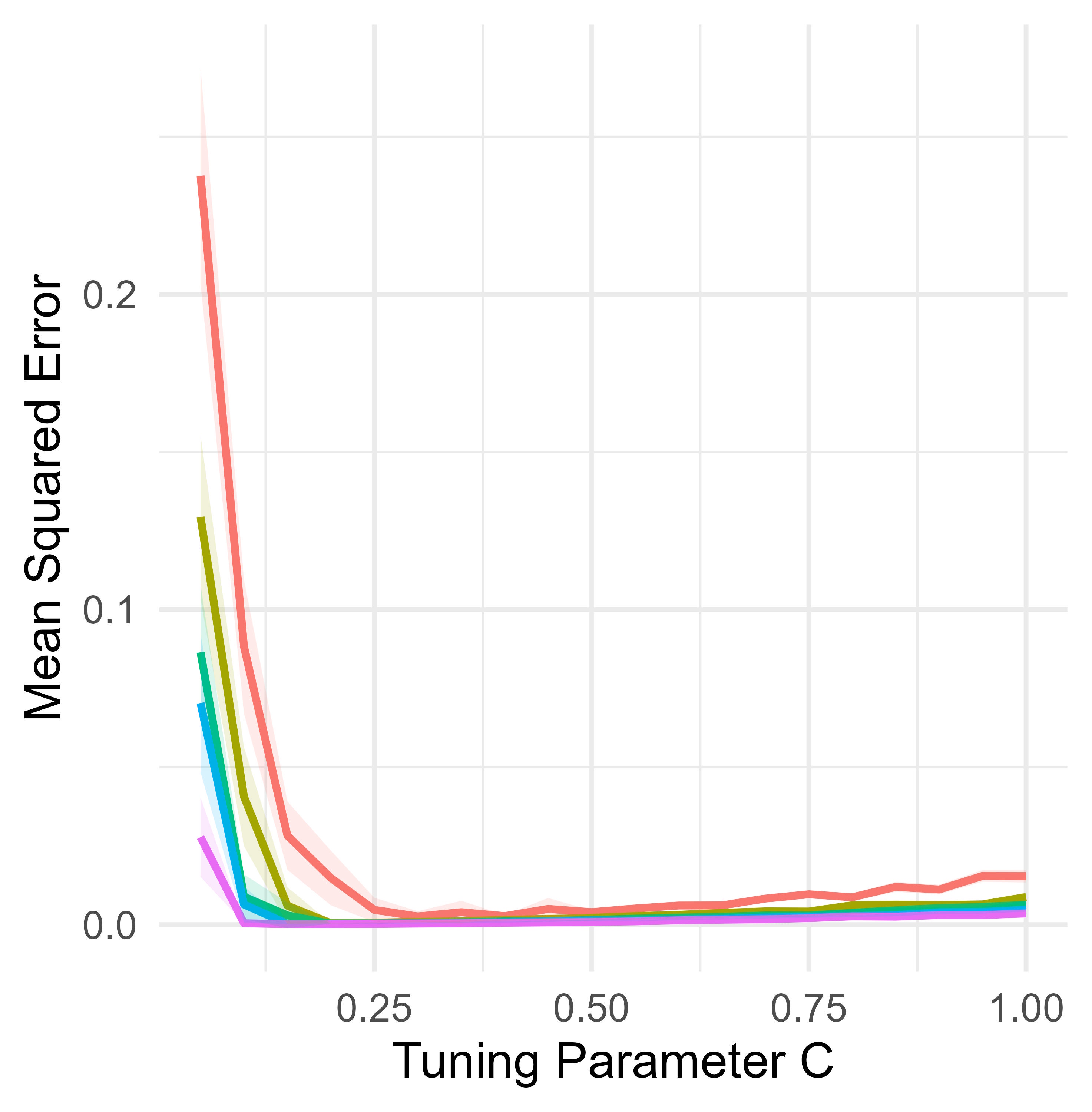}
            \includegraphics[height=0.15\textheight]{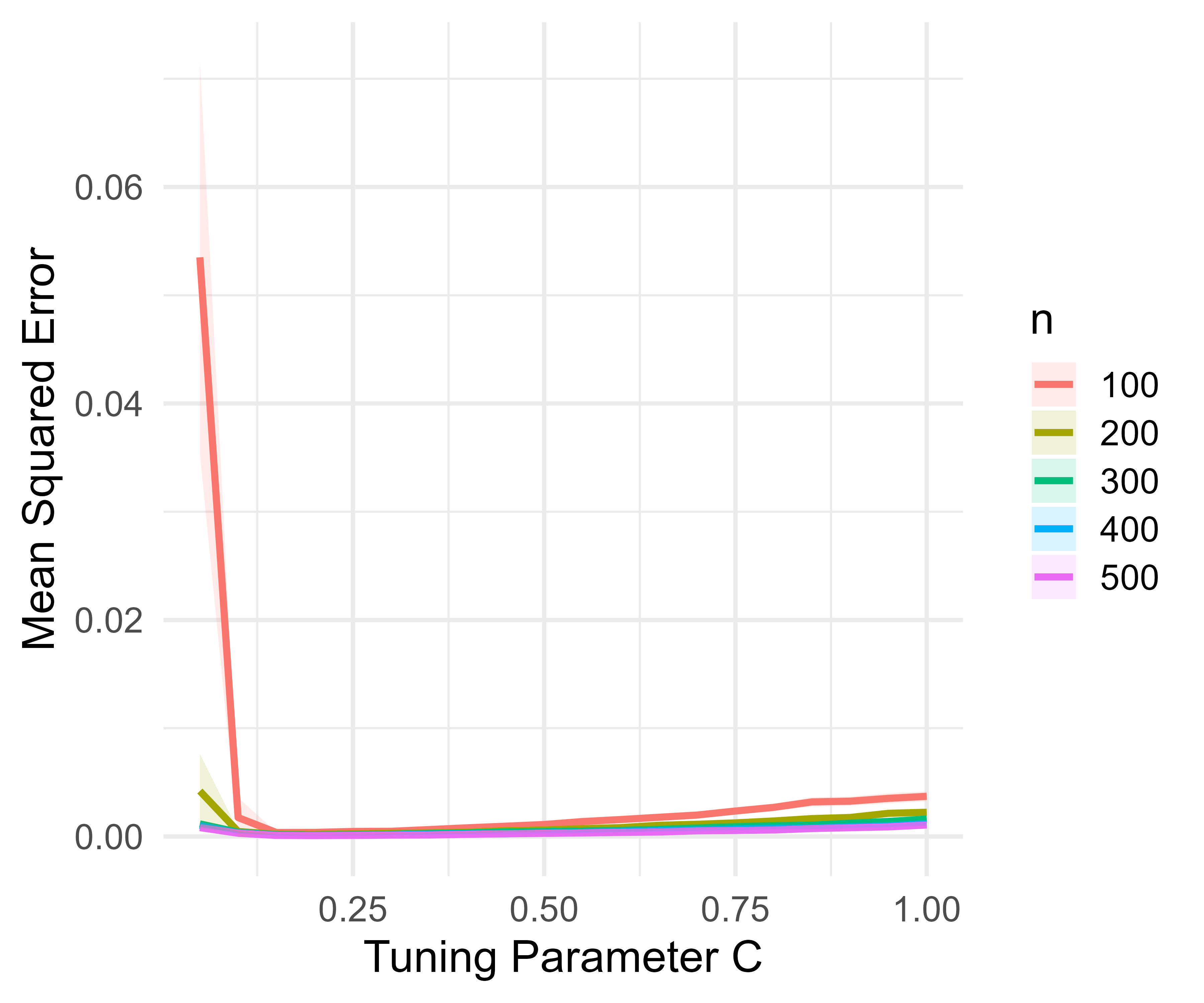}

            \caption{Estimate of mean-squared-error for $T = 100$ fixed and varying values of $n$ denoted by colour, and $\alpha \in \{0.5, 1, 2, 4\}$ (left-to-right). Ribbons denote two standard errors across 500 repetitions.}
            \label{fig:sensitivitysetting2a}
        \end{figure}

        \begin{figure}[htbp]
            \centering
            \includegraphics[height=0.15\textheight]{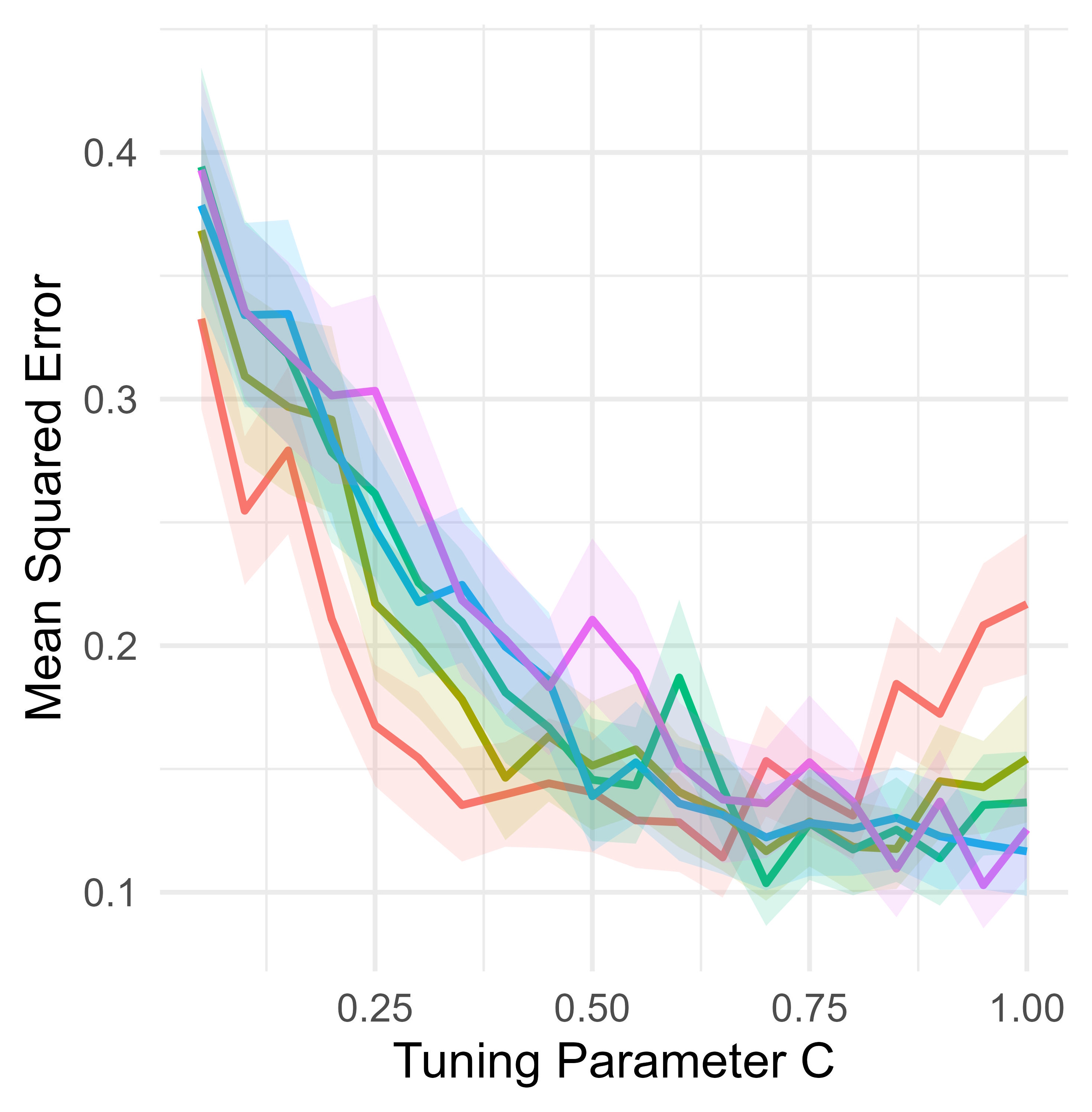}
            \includegraphics[height=0.15\textheight]{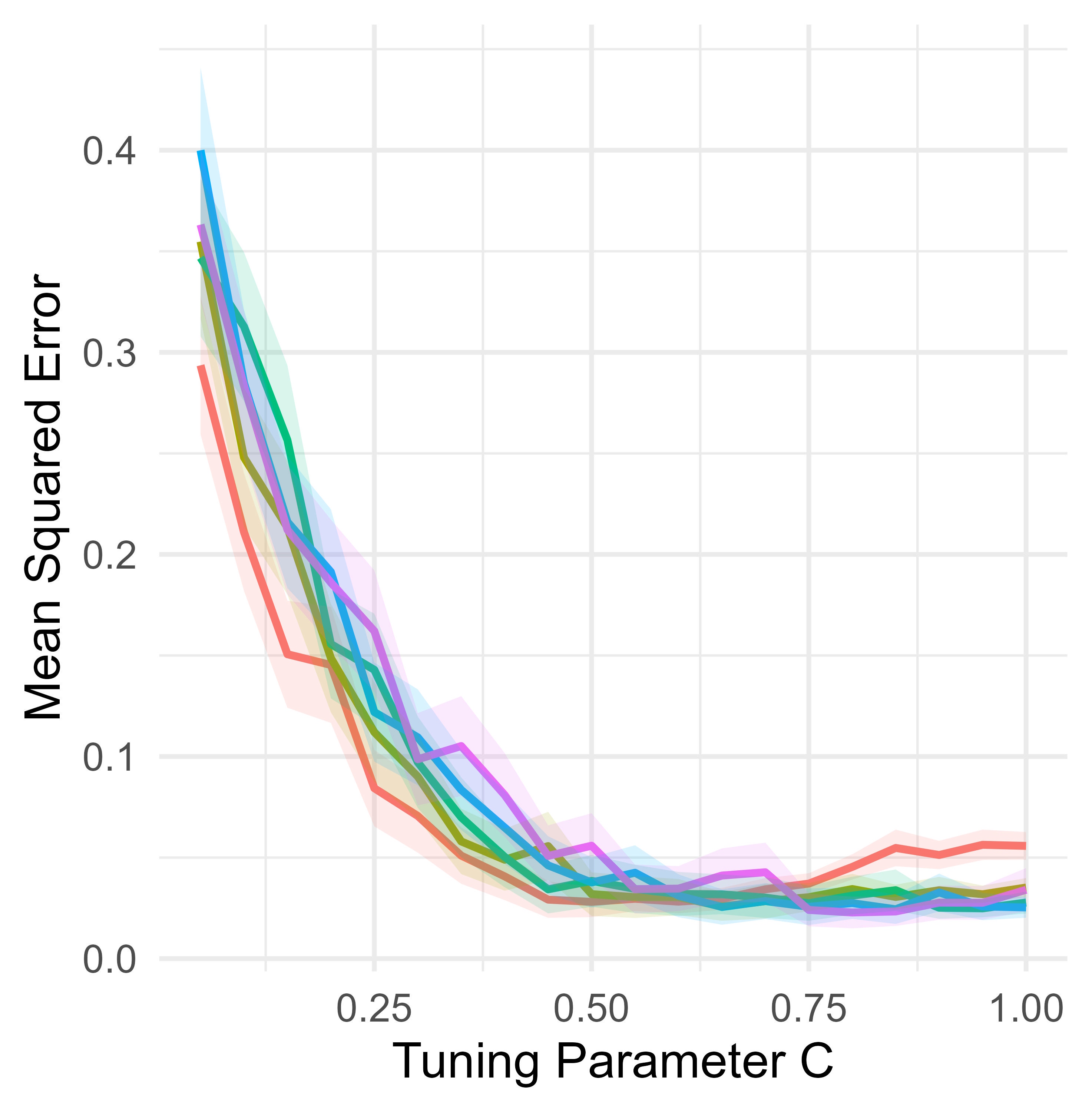}
            \includegraphics[height=0.15\textheight]{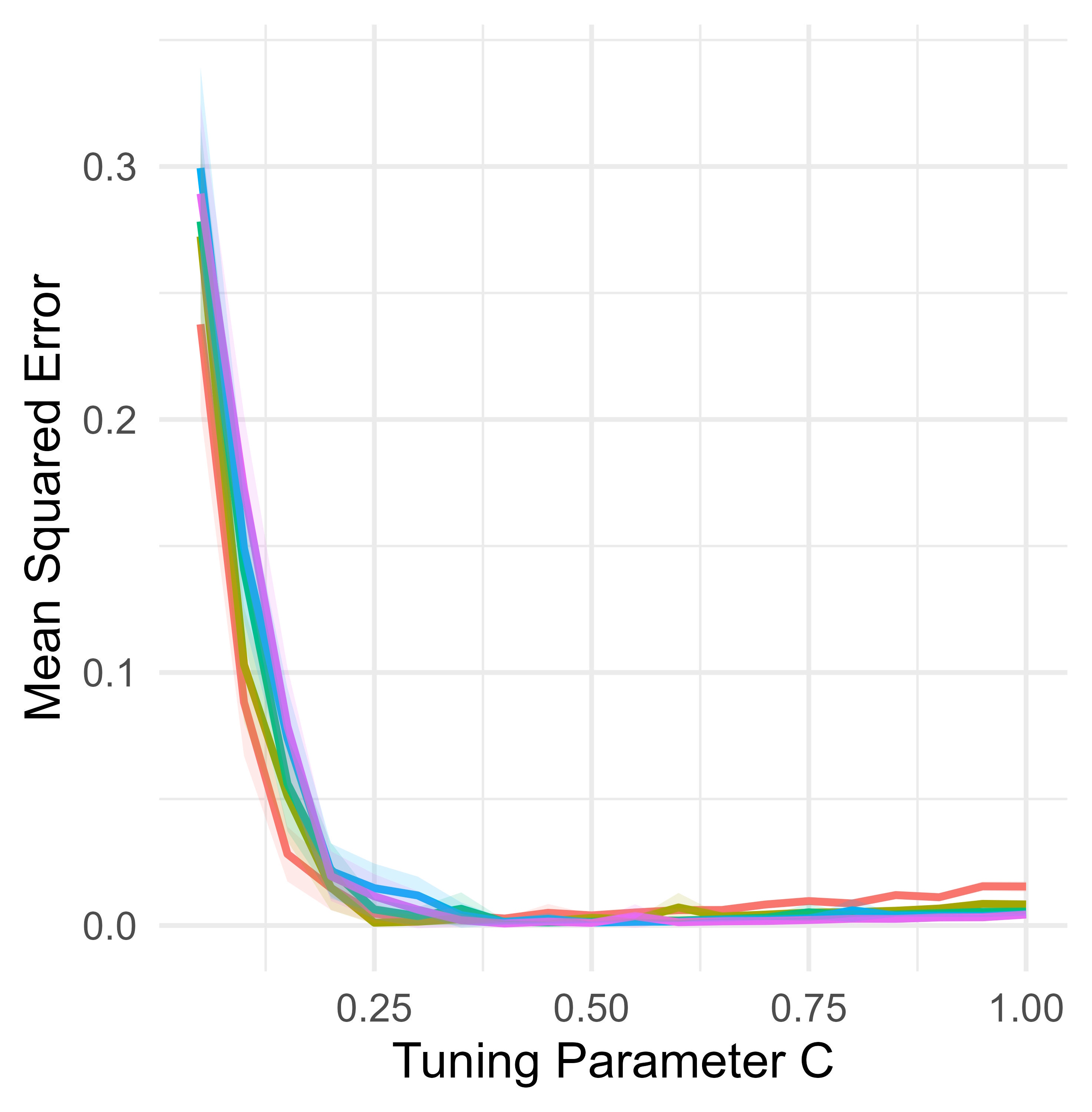}
            \includegraphics[height=0.15\textheight]{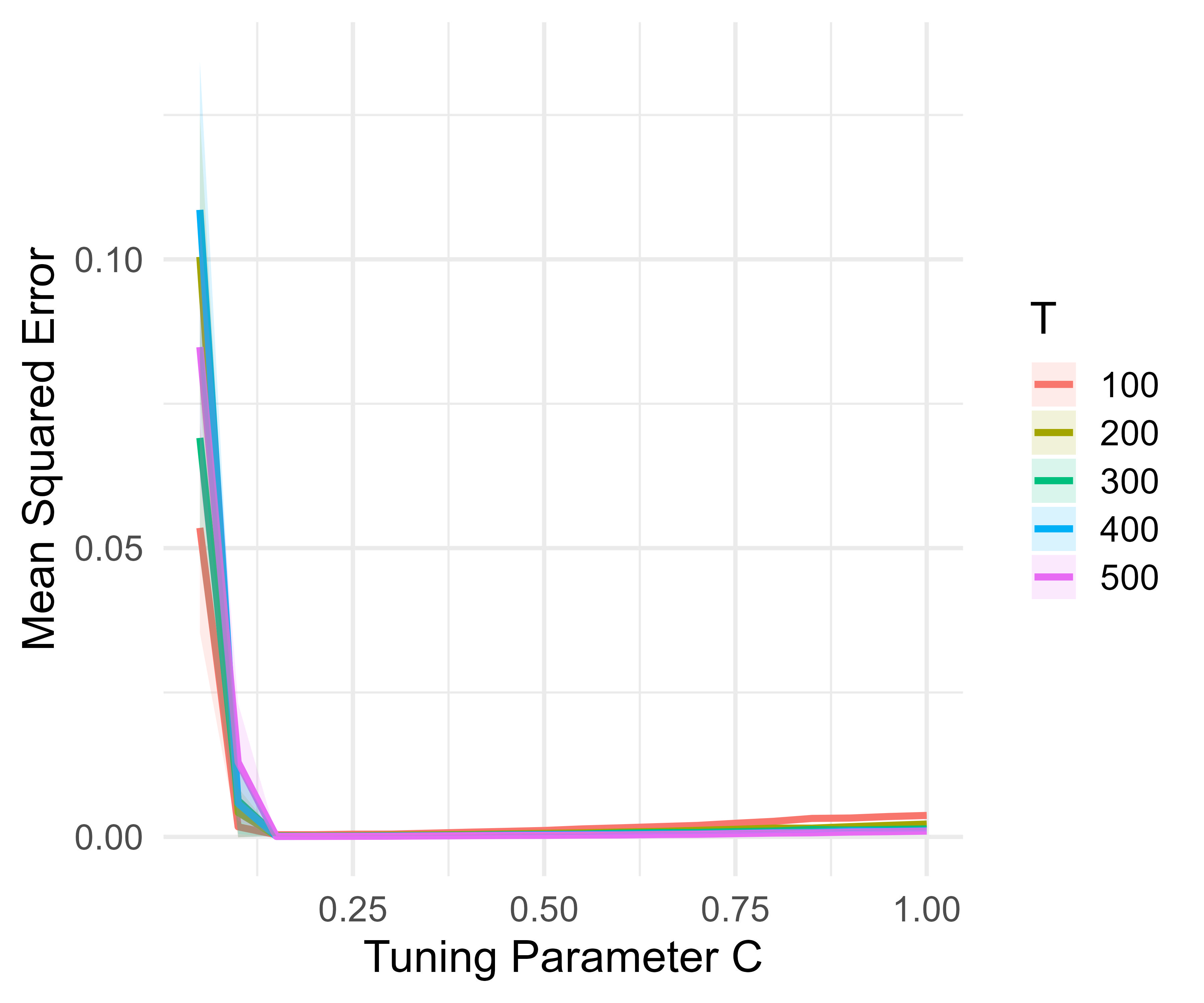}
            \caption{Estimate of mean-squared-error for $n = 100$ fixed and varying values of $T$ denoted by colour, and $\alpha \in \{0.5, 1, 2, 4\}$ (left-to-right). Ribbons denote two standard errors across 500 repetitions.}
            \label{fig:sensitivitysetting2b}
        \end{figure}
        
        In summary, our results suggest that for $C$ chosen too large, distributions are equally affected, whereas for $C$ chosen too small the Rademacher distribution is the most negatively affected. Further, as it appears better to over-estimate than under-estimate the optimal $C$, we calibrate our choice of $C$ based on the results of \Cref{fig:sensitivitysetting2a} and \Cref{fig:sensitivitysetting2b}. Hence, in all simulations that follow, in place of $\sqrt{2}$, we set the tuning constant $C_{n, T, \alpha} = C_{\alpha}$ as follows:
        \begin{equation*}
            C_{n, T, \alpha} =
            \begin{cases}
                0.5, &\text{when } \alpha \in \{0.5, 1\} \\
                0.25, &\text{when } \alpha \in \{2, 4\} 
            \end{cases}.
        \end{equation*}
        We caution for values of $n, T$ smaller that $100$ and of $\alpha$ not considered in this analysis, that we have no guarantee on the suitability of the value of $C_{n, T, \alpha}$. Nevertheless, one could carry out further simulations on synthetic data for values of $\alpha$ of interest to obtain a grid of performant values of $C_{n, T, \alpha}$.

        Inspecting the role of the second parameter of interest in sparse mean estimation, we recall the distributions $P_{d, s, p}^{\mathrm{sparse}}$ defined in \eqref{secsim:eq:sparsefamily}. The choice of $\eta_C$ only affects the variable selection sub-routine corresponding to Step~2 in \Cref{sec:sparse:variableselection}, and so we focus on the proportion of co-ordinates correctly identified. In \Cref{fig:sparseSensitvity}, we see that across all privacy levels, there is a clear increase in the ability of the procedure to identify the correct co-ordinates as $C$ is decreased, up to a limit as $C$ takes a value of approximately one. Hence, when we implement the sparse mean estimation procedure, we take $C = 1$ for simplicity and thus $\eta_C = \sqrt{s}$.
        \begin{figure}[htbp]
            \centering
            \includegraphics[height=0.25\textheight]{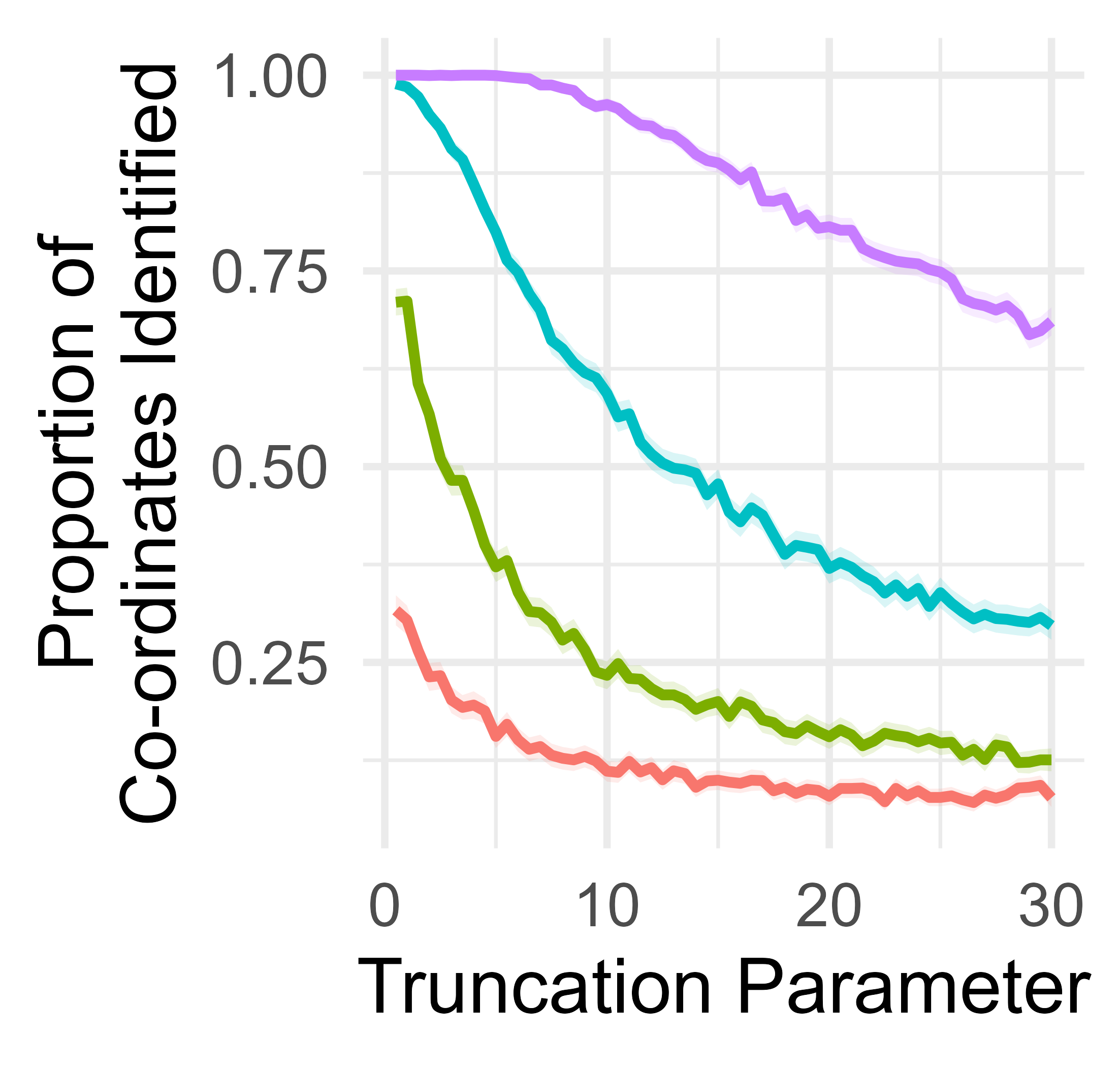}
            \includegraphics[height=0.25\textheight]{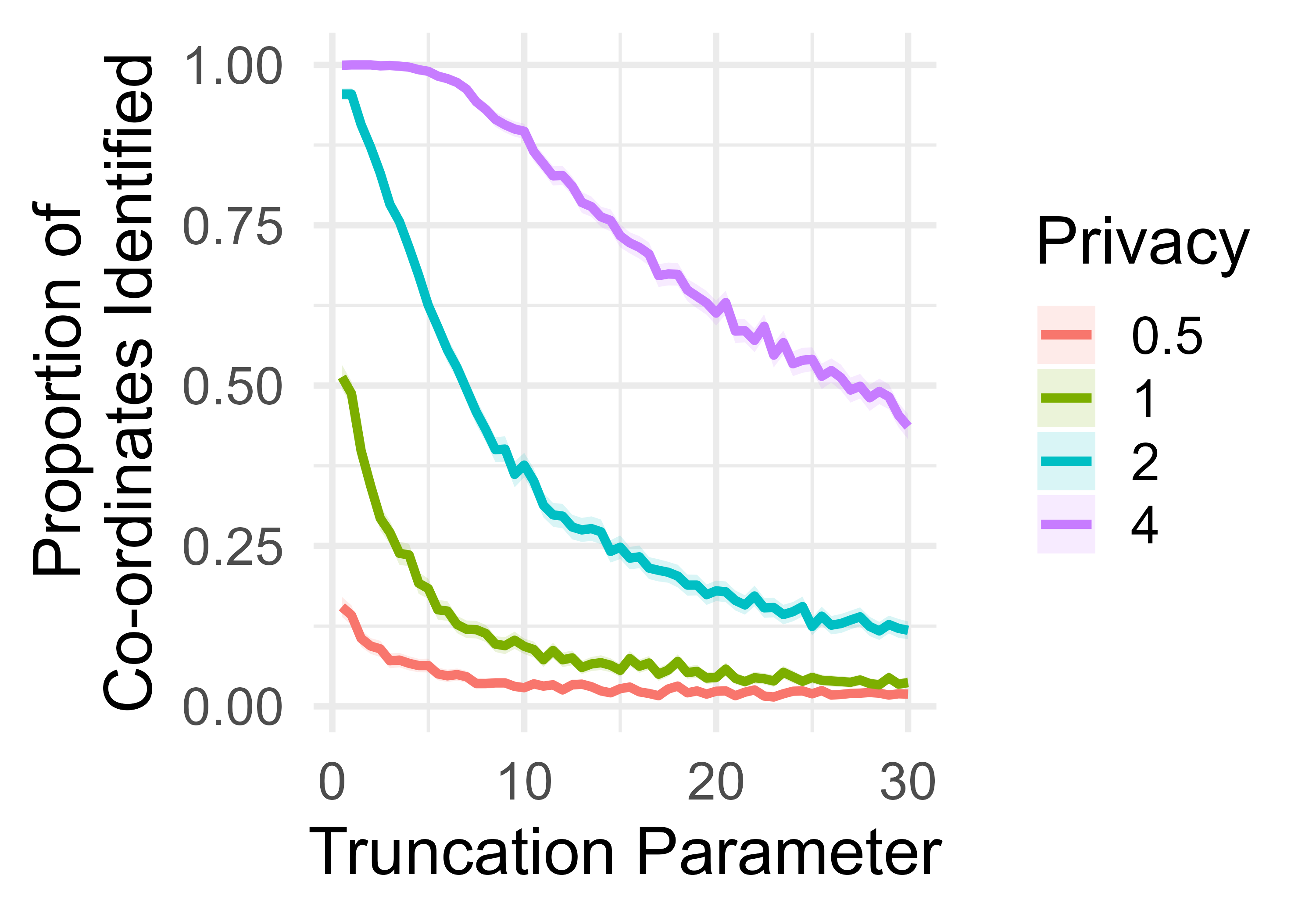}
            \caption{Proportion of correct co-ordinates identified for $P_{100, 4, 0.75}^{\mathrm{sparse}}$ (left) and $P_{500, 4, 0.75}^{\mathrm{sparse}}$ (right) for varying privacy levels with $n = 500$, $T = 100$. Ribbons denote two standard errors across 500 repetitions.}
            \label{fig:sparseSensitvity}
        \end{figure}

        Finally, we investigate the sensitivity of the sparse mean estimation procedure to misspecification of the true sparsity level. We estimate the mean-squared-error of our estimation procedure for the distributions $P_{100, 4, 0.75}^{\mathrm{sparse}}$ and $P_{100, 4, 1}^{\mathrm{sparse}}$. The latter distribution is chosen to exhibit greater bias than the former when one of the non-zero co-ordinates of the mean vector is missed. From the results in \Cref{fig:sparseMisspecify}, we see that the effect of misspecification is delicate, with the relative impact of under-estimating or over-estimating depending on the privacy level and the magnitude of the non-zero entries of the mean. Generally, for the setting we consider we see that for higher values of $\alpha$, over-estimating incurs a smaller penalty than under-estimating. This is because, for large $\alpha$, the extra variance incurred by trying to estimate another co-ordinate is smaller than the bias of missing a non-zero co-ordinate. However, as $\alpha$ decreases, or the magnitude of the non-zero co-ordinates is made smaller, this extra variance from over-estimating becomes greater relative to the bias from under-estimating. We also note that when the assumed sparsity is smaller than four, for the larger privacy levels, the mean-squared-error incurred is also exactly equal to the bias from missing the relevant co-ordinates.
        \begin{figure}[htbp]
            \centering
            \includegraphics[height=0.25\textheight]{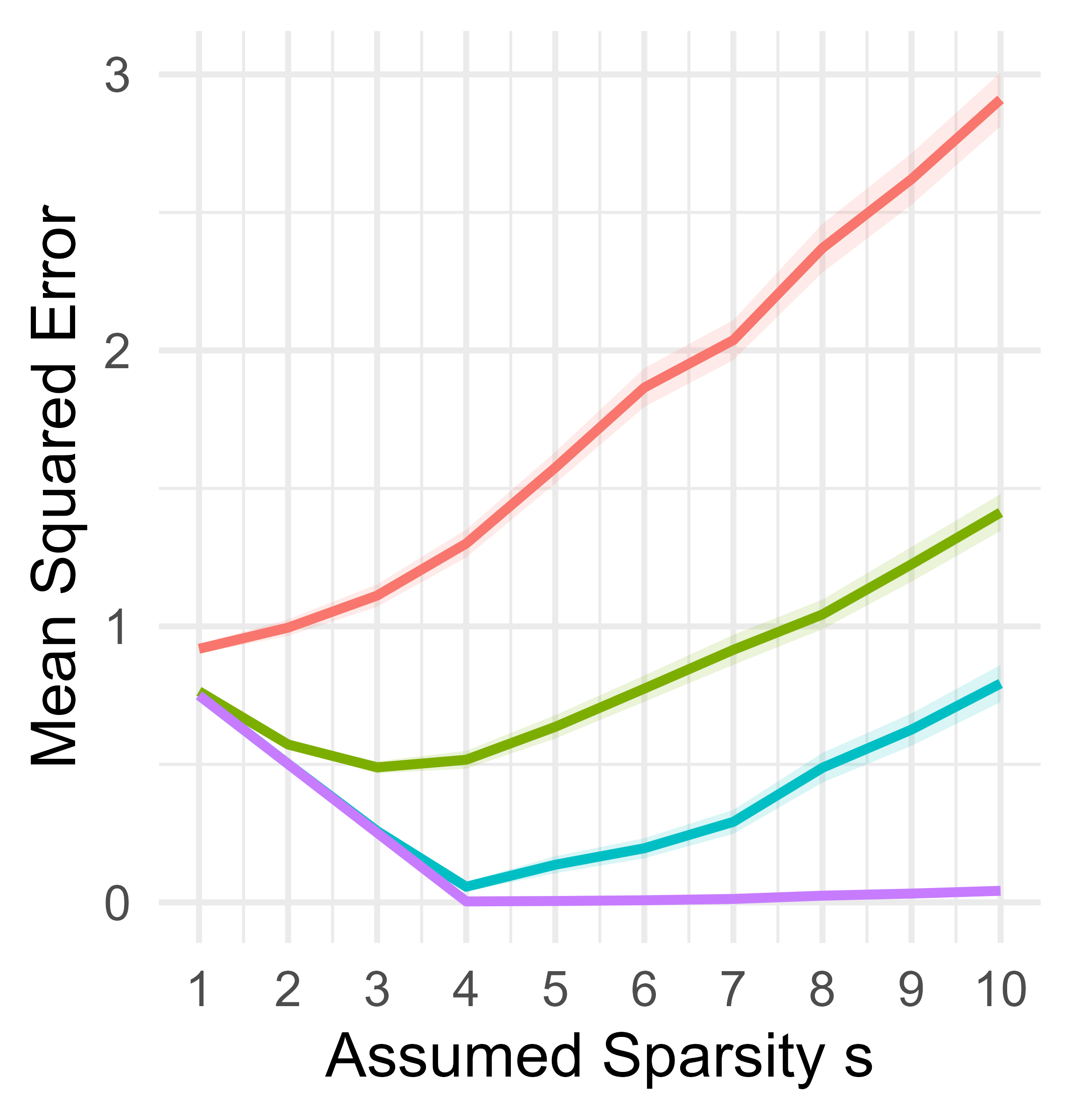}
            \includegraphics[height=0.25\textheight]{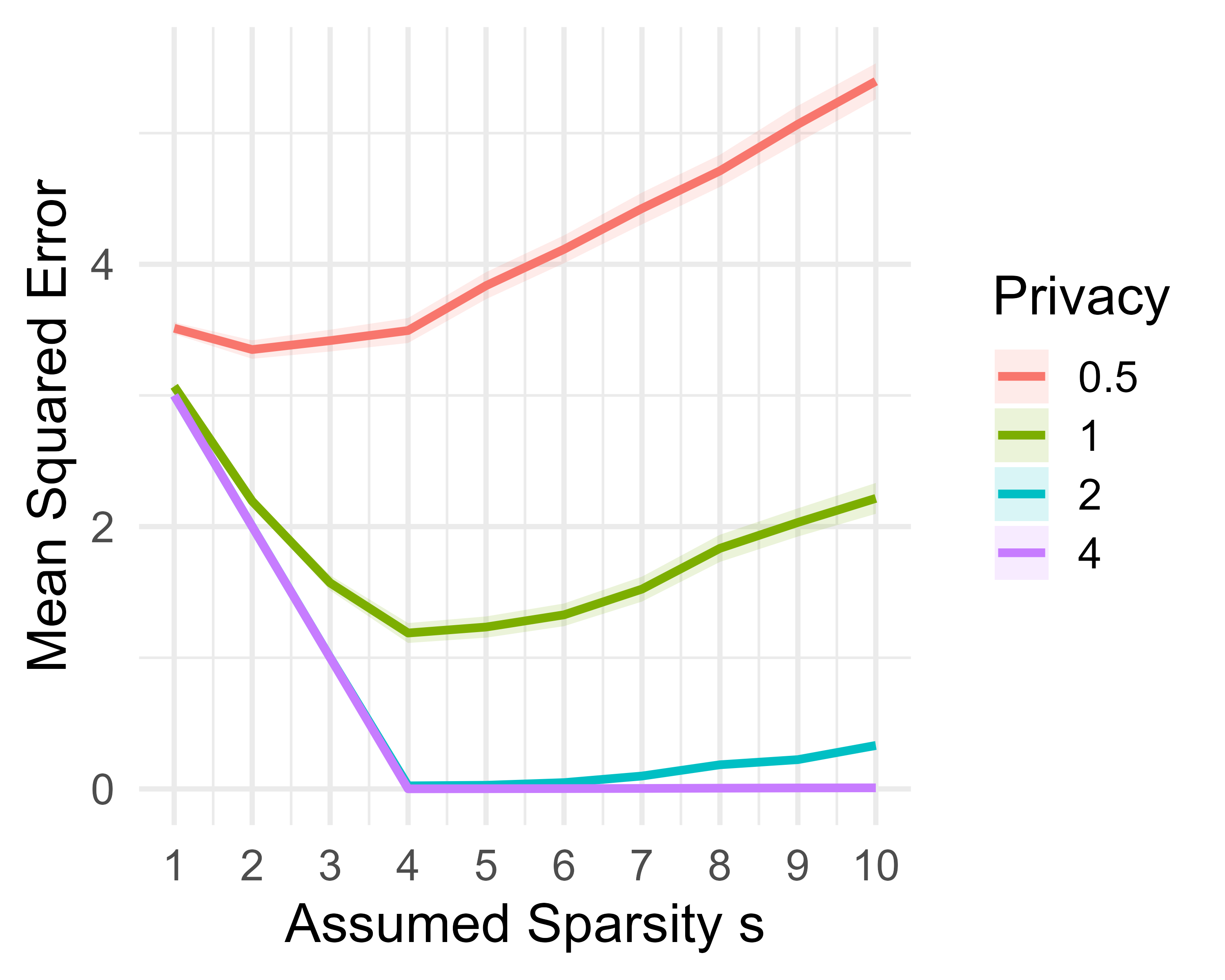}
            \caption{Mean-squared-error of the sparse mean estimation procedure as assumed sparsity is varied for $P_{100, 4, 0.75}^{\mathrm{sparse}}$ (left) and $P_{100, 4, 1}^{\mathrm{sparse}}$ (right) for varying privacy levels with $n = 1000$, $T = 200$. Ribbons denote two standard errors across 500 repetitions.}
            \label{fig:sparseMisspecify}
        \end{figure}

        \subsection{Mixture Distributions} \label{sec:sim:mixture}
            We now empirically validate our procedure in \Cref{secmix} for estimating the means of the components in a mixture model in the following two settings, one corresponding to a balanced collection of mixtures, and another where the mixture component probabilities are greatly imbalanced.

            \noindent
            \textbf{Setting 1}: We consider
            \begin{equation*}
                X_{1:T}^{(1)}, \hdots X_{1:T}^{(n)} \overset{\mathrm{i.i.d.}}{\sim}
                \frac{1}{4}\mathrm{Rad}\bigg(\frac{1}{8}\bigg)^{\otimes T}
                + \frac{1}{4}\mathrm{Rad}\bigg(\frac{3}{8}\bigg)^{\otimes T}
                + \frac{1}{4}\mathrm{Rad}\bigg(\frac{5}{8}\bigg)^{\otimes T}
                + \frac{1}{4}\mathrm{Rad}\bigg(\frac{7}{8}\bigg)^{\otimes T}.
            \end{equation*}
            We consider both (a) fixing $n = 1000$, and varying $T \in \{100, 200, \hdots, 2000\}$, $\alpha \in \{2, 4\}$; and (b) fixing $T = 1000$ and varying $n \in \{100, 200, \hdots, 2000\}$, $\alpha \in \{2, 4\}$. We calculate the mean across 1,000 repetitions.
            
            \noindent
            \textbf{Setting 2}: We consider
            \begin{equation*}
                X_{1:T}^{(1)}, \hdots X_{1:T}^{(n)} \overset{\mathrm{i.i.d.}}{\sim}
                \frac{1}{10}\mathrm{Rad}\bigg(\frac{1}{8}\bigg)^{\otimes T}
                + \frac{8}{10}\mathrm{Rad}\bigg(\frac{1}{2}\bigg)^{\otimes T}
                + \frac{1}{10}\mathrm{Rad}\bigg(\frac{7}{8}\bigg)^{\otimes T}.
            \end{equation*}
            We fix $T = 500$, varying $n \in \{100, 150, \hdots, 2000\}$ and $\alpha \in \{2, 4\}$. We calculate the mean across 1,000 repetitions.
            
            The results for Setting 1 and Setting 2 are contained in \Cref{fig:mixturebalanced} and \Cref{fig:mixtureimbalanced} respectively. In both settings, we see that with $\alpha = 4$ and either one of $n$ or $T$ increasing with the other held fixed, our estimates of the mean of each component are consistent with decreasing variance. In particular the highly imbalanced case of Setting~2, the procedure still successful identifies the components and the estimator is consistent. The case with $\alpha = 2$ is more delicate. For example, in Setting~1 in \Cref{fig:mixturebalanced}, the case of fixed $T$ with increasing $n$ shows decreasing bias and variance as $n$ grows. For fixed $n$, the estimators are consistent but reductions in variance are limited. We hypothesise that this is due to the phase transition phenomena. Indeed, as seen in the lower bound in \Cref{secmix:thm:main}, the $T$-independent lower-bound is stronger than those of previously considered problems, and so we expect larger values of $n$ and/or $\alpha$ to be required to overcome it compared to the non-mixture settings.

            Regarding the cut-off for very small estimates of the mixture proportions to avoid numerical instability as in \eqref{secmix:eq:FinalEstimatorComponent}, we set the indicator $\mathbbm{1}\{\hat{\pi}_k \geq \pi_0/2\}$ therein to be $\mathbbm{1}\{\hat{\pi}_k \geq 1/(100m)\}$ in our simulations.
            \begin{figure}[htbp]
                \centering
                \includegraphics[height=0.15\textheight]{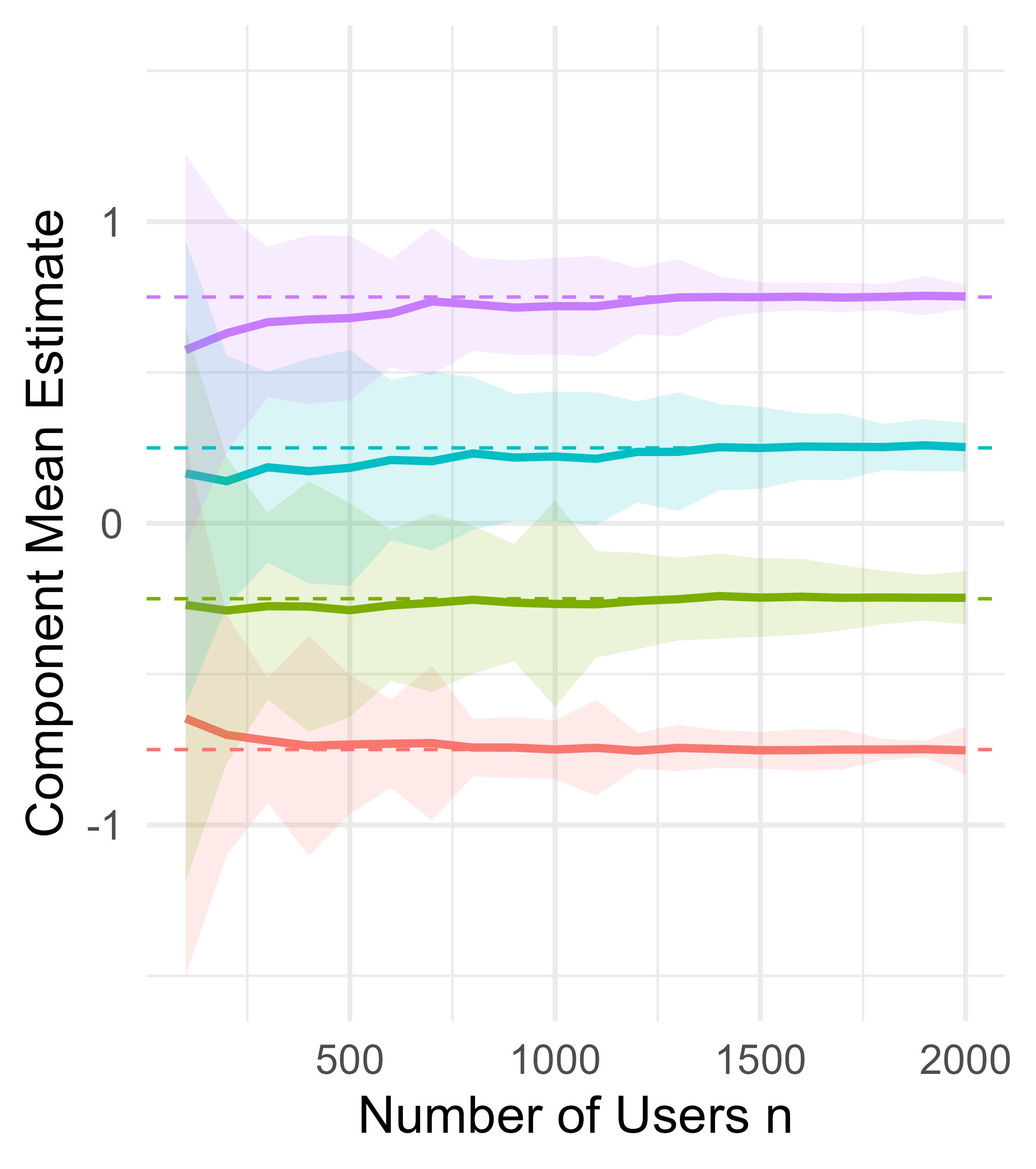}
                \includegraphics[height=0.15\textheight]{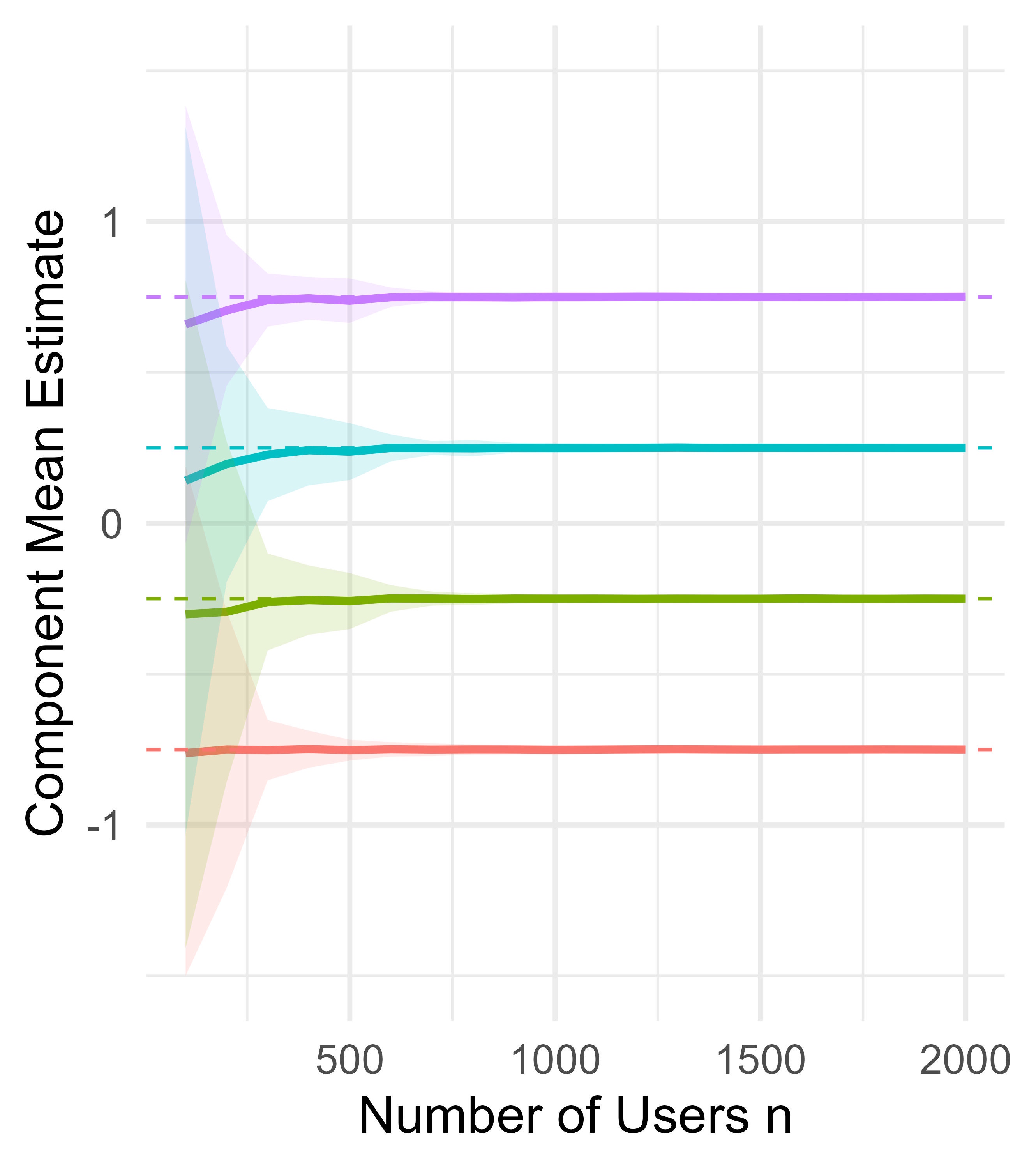}
                \includegraphics[height=0.15\textheight]{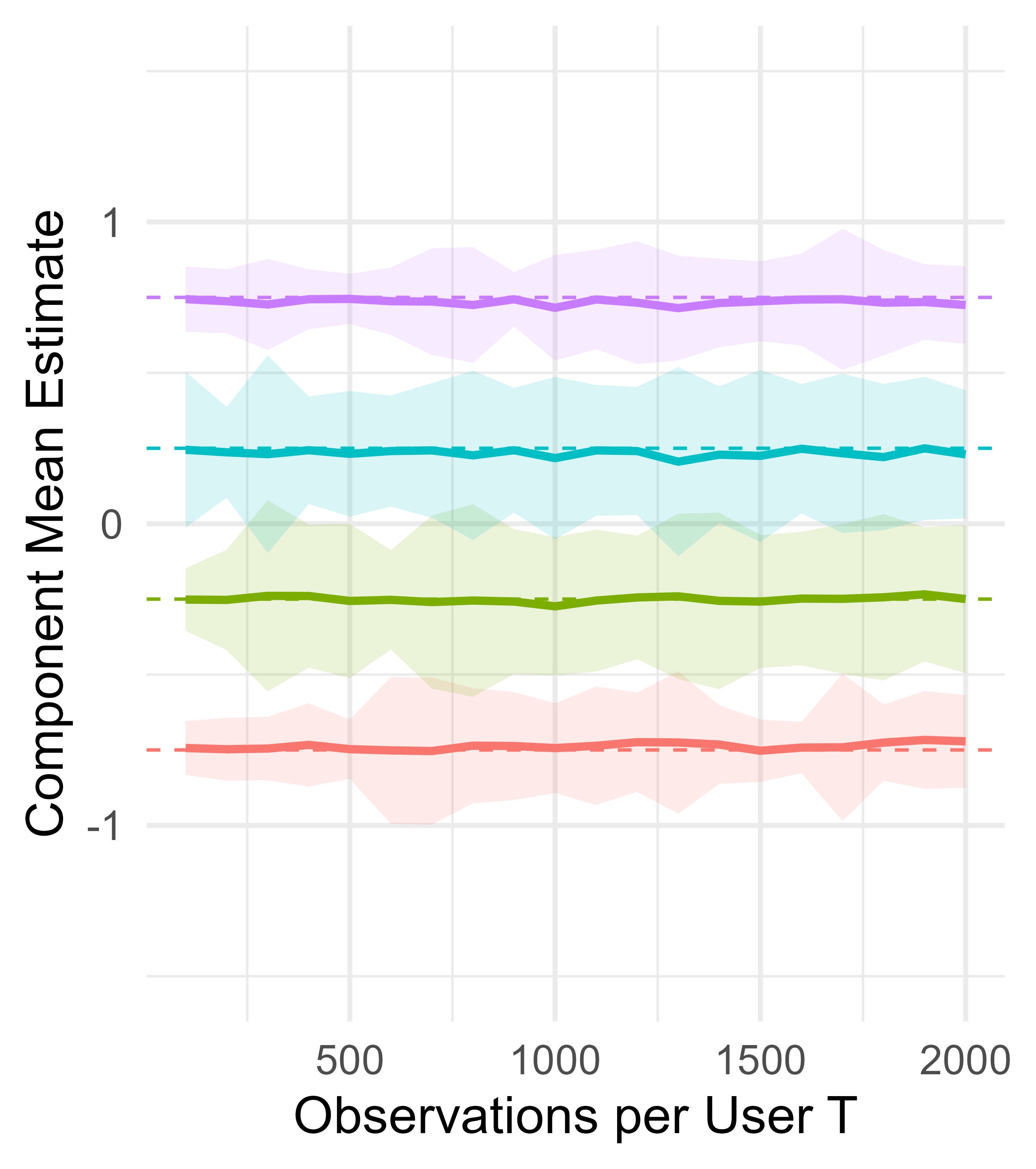}
                \includegraphics[height=0.15\textheight]{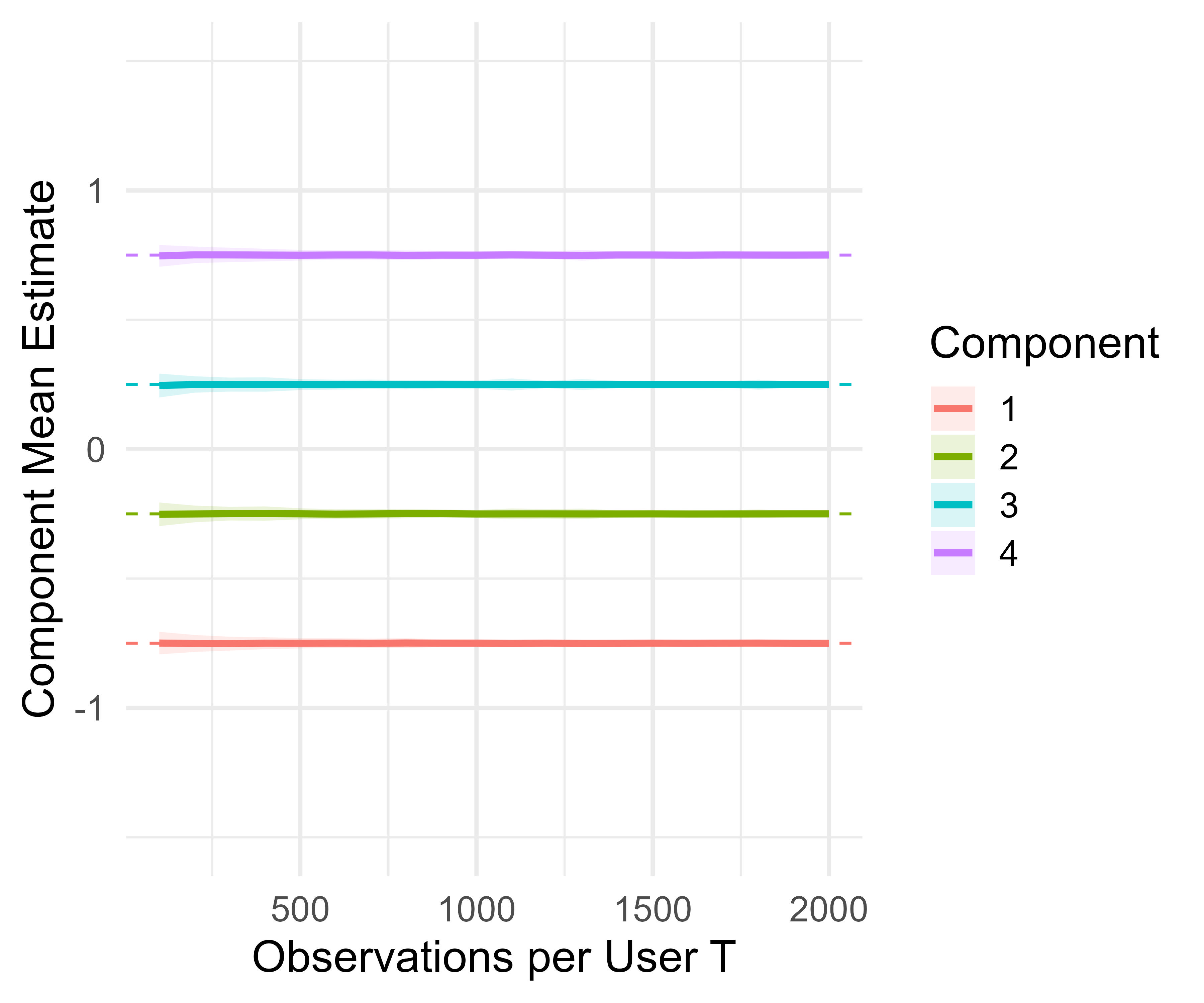}
                \caption{Estimate of component means in Setting~1 for fixing $T$ and varying $n$ (left, centre-left), and fixing $n$ and varying $T$ (centre-right, right), with privacy parameters $\alpha \in \{2, 4\}$ (left, centre-right and centre-left, right respectively). Estimate of each component is denoted by colour. Ribbons denote one standard deviation across 1,000 repetitions.}
                \label{fig:mixturebalanced}
            \end{figure}

            \begin{figure}[htbp]
                \centering

                \includegraphics[height=0.15\textheight]{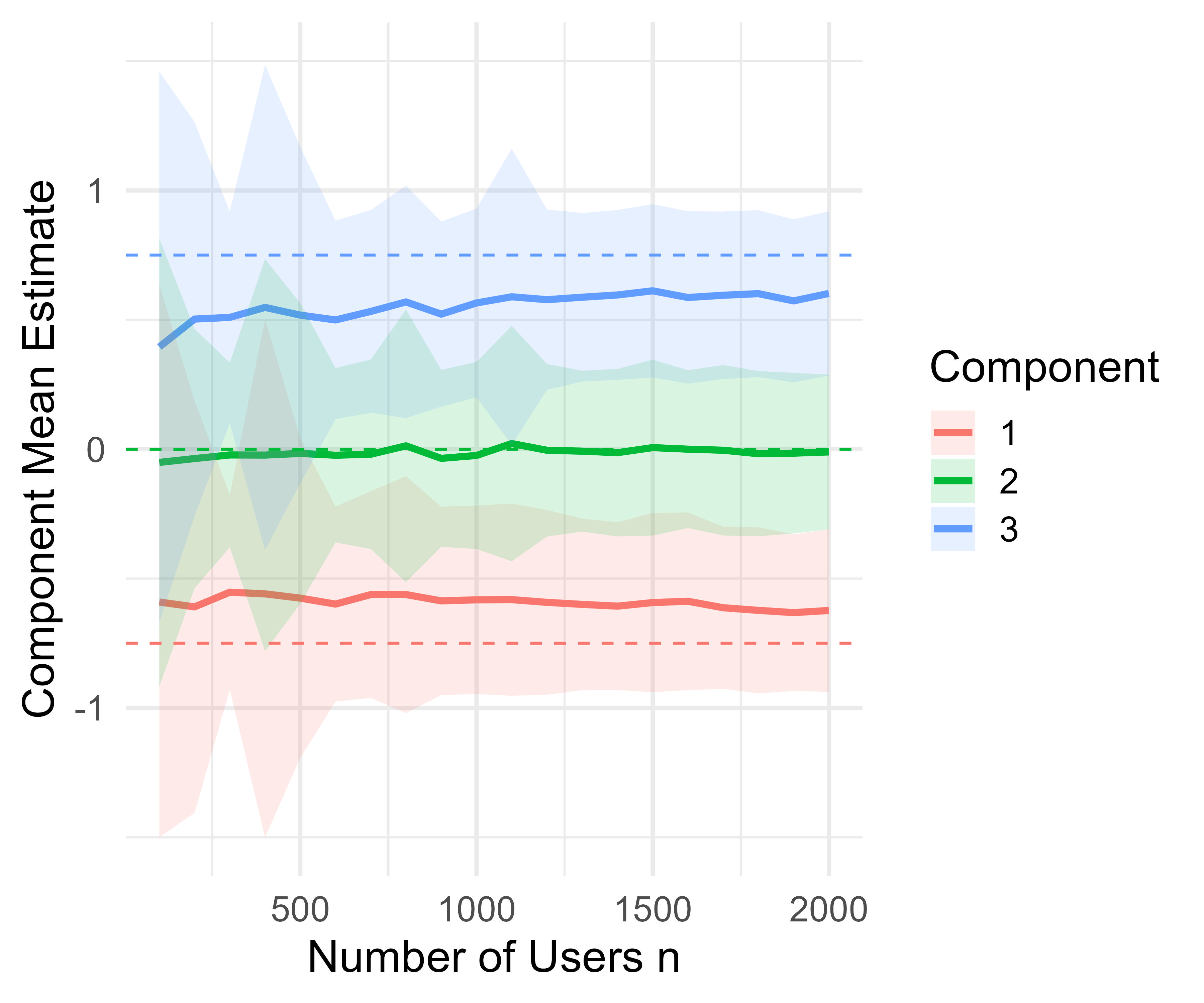}
                \includegraphics[height=0.15\textheight]{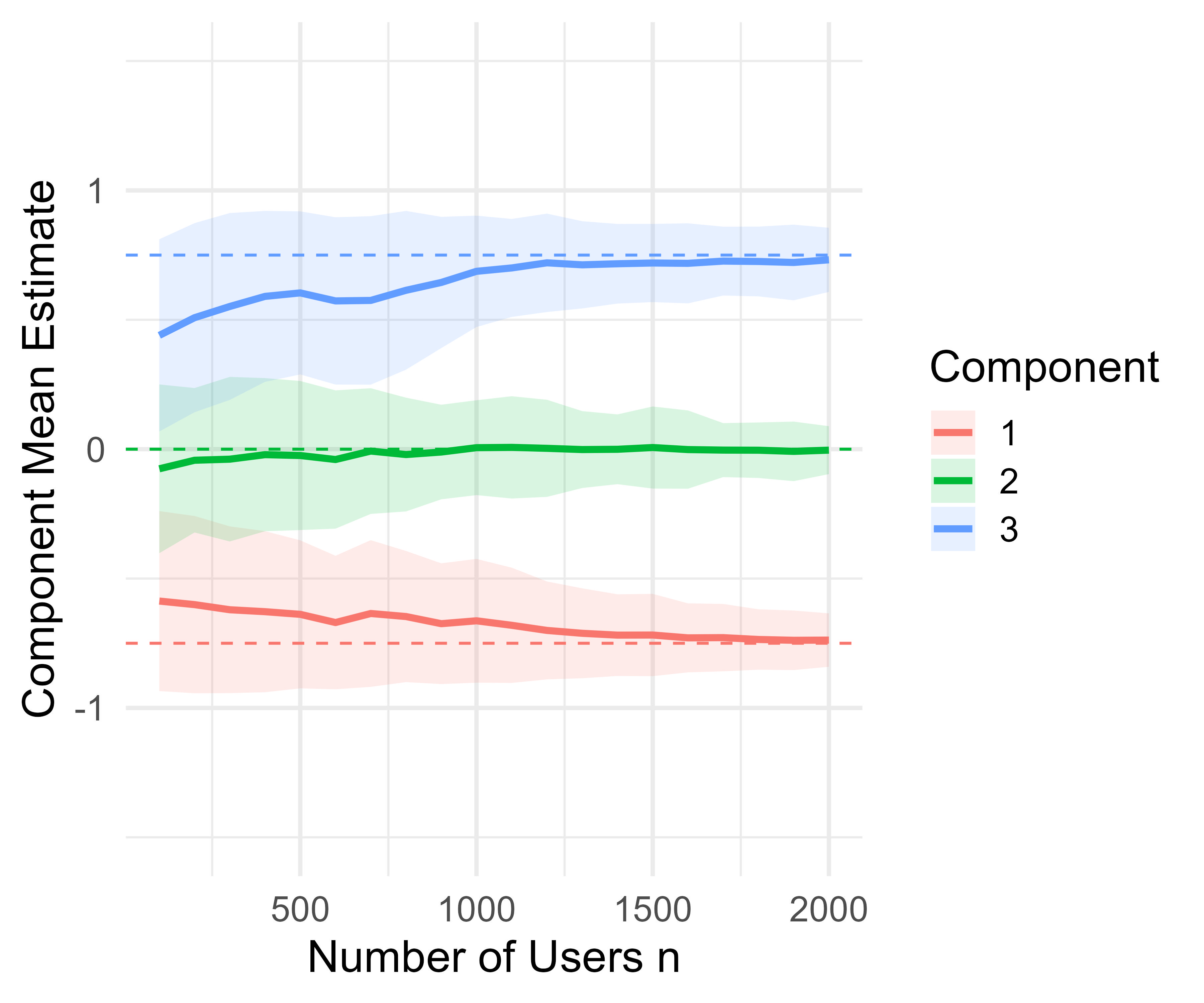}

                \caption{Estimate of component means in Setting~2 for fixing $T$ and varying $n$, with privacy parameters $\alpha \in \{2, 4\}$ (left and right respectively). Estimate of each component is denoted by colour. Ribbons denote one standard deviation across 1,000 repetitions.}
                \label{fig:mixtureimbalanced}
            \end{figure}

\section{Mixture Model Mean Estimation} \label{secmix}
    In this section we consider mean estimation under a generalisation of the user-level framework. Compared to the previous estimation problems where it is assumed that every user draws a sample from the same $T$-fold product distribution, we now consider a setting with data heterogeneity. To be precise, we now assume each user draws data from a component from a mixture distribution, and are interested in estimating the mean of each mixture component under user-level LDP.

    We consider mixture distributions of $m \in \mathbb{N}$ components, where each component is supported on $[-1,1]$ and satisfying suitable separation conditions, defined as
    \begin{equation}
        \begin{aligned}
            \mathcal{P}_{m, \pi_{0}, \theta_{\mathrm{sep}}} = \big\{ (P_1, \hdots, P_m, &\pi_1, \hdots, \pi_m) : \pi_1 + \dots + \pi_m = 1, \mathbb{E}_{P_1}[X] < \dots < \mathbb{E}_{P_m}[X], \\
            &\min_{i \neq j}|\mathbb{E}_{P_i}[X] - \mathbb{E}_{P_j}[X]| \geq \theta_{\mathrm{sep}},
            \forall i \in [m], \mathrm{supp}(P_i) \subseteq [-1,1], \pi_i > \pi_{0} \big\}.
        \end{aligned} \label{secmix:eq:MixClass}
    \end{equation}

    Unlike in the usual user-level setting where we denote by $\mathcal{R}_{n,T,\alpha}$ for the minimax risk where the user obtains $T$-many i.i.d.~copies from a distribution from the family of interest, in this setting we instead have a user obtain $T$-many i.i.d.~copies from a \emph{component} of a distribution from the family of interest. More formally, for $T \geq 1$ we write $X_{1:T} \sim P_{\mathrm{mix}}^{T}$ when $X_{1:T} \sim \pi_1 P_1^{\otimes T} + \dots + \pi_m P_m^{\otimes T}$. We denote the minimax risk in this setting as $\mathcal{R}_{n,T,\alpha}^{\mathrm{mix}}$, equivalent to the standard user-level definition given in~\eqref{sec2:eq:UserMinimax} except that each user draws data as $X_{1:T}^{(1)}, \hdots, X_{1:T}^{(n)} \overset{\mathrm{i.i.d.}}{\sim} P_{\mathrm{mix}}^{T}$ for some $P_{\mathrm{mix}} \in \mathcal{P}_{m, \pi_{0}, \theta_{\mathrm{sep}}}$

    Given $P_{\mathrm{mix}} \in \mathcal{P}_{m, \pi_{0}, \theta_{\mathrm{sep}}}$ with $P_{\mathrm{mix}} = (P_1, \hdots, P_m, \pi_1, \hdots, \pi_m)$, we estimate the means of the components so that $\theta(P_{\mathrm{mix}}) = (\mathbb{E}_{P_1}[X], \hdots, \mathbb{E}_{P_m}[X])^T$. Measuring the error of an estimate by the squared $\ell_2$-loss, we have the following result.

    \begin{theorem} \label{secmix:thm:main}
        Assuming that $n\alpha^2 \geq \widetilde{C}m^2\log(n\alpha^2)$, $T \geq \widetilde{C}' m^2 \log(nT\alpha^2)$, $\pi_{0} > \tilde{c}/m$ and $\theta_{\mathrm{sep}} \geq \max\{\widetilde{C}''\sqrt{\log(nT\alpha^2)/T}, e^{-\tilde{c}'n\alpha^2/m^2}\}$, where $\tilde{c}, \tilde{c}', \widetilde{C}, \widetilde{C}', \widetilde{C}'' > 0$ are absolute constants chosen so that $\theta_{\mathrm{sep}} \leq 2/m$, we have that
        \begin{align*}
            \frac{m^3}{nT\alpha^2} \vee e^{-Cn\alpha^2/m^2}
            \lesssim \mathcal{R}_{n, T, \alpha}^{\mathrm{mix}}(\theta(\mathcal{P}_{m, \pi_{0}, \theta_{\mathrm{sep}}}), \|\cdot\|_2^2)
            \lesssim \frac{m^3 \{\log(nT\alpha^2)\}^2}{nT\alpha^2} + e^{-cn\alpha^2/m^2},
        \end{align*}
        where $\mathcal{P}_{m, \pi_{0}, \theta_{\mathrm{sep}}}$ is the family of distributions as in \eqref{secmix:eq:MixClass} and $C > c > 0$ are absolute constants.
    \end{theorem}

    The lower bound is proved in \Cref{app:sec:Lower}.  An estimation procedure achieving the upper bound is presented below, the performance of which is analysed in \Cref{app:sec:Upper}. The computational complexity of this procedure is discussed in \Cref{app:sec:CompComplex}.

    We note that the $T$-independent lower bound in the above theorem \emph{does not} follow from our general infinite-$T$ lower bounds. This is because of the difference in setting, where due to the mixture structure, as $T \rightarrow \infty$ a user does not attain a perfect view of the entire functional $\theta(P_{\mathrm{mix}})$, but only for a component of the mixture. Applying the general lower bound via \Cref{sec2:thm:TIndepLB} is still possible, but it results in a sub-optimal bound of $e^{-n\alpha^2/m}$, as opposed to the tighter $e^{-n\alpha^2/m^2}$ we obtain through a more tailored approach.

\subsection{Estimation Procedure} \label{secmix:procedure}

    The estimation procedure we develop begins similarly to the univariate mean estimation procedure of \Cref{sec3}, but with three folds of the data, and so we assume without loss of generality that $n$ is a multiple of three. In the first stage, a localisation step is carried using one fold of the sample. Unlike in the case without mixtures, we instead expect there to be multiple candidate localisation regions, one for each of the $m$-many mixture components approximately centred on the mean value for that component.
        
    Once the $m$-many candidate regions are identified, another fold of users obtain estimates of the mixture proportions by releasing a private view of the indicator of which sub-interval their sample means are closest to via unary encoding.

    Finally, the last fold of users calculate the difference of their sample mean and the centre of the nearest bin, before truncating this quantity. However, unlike with the univariate mean estimation procedure of \Cref{sec3:univariate} where there is only a single sub-interval and hence the decision to truncate is data independent, centring around and truncating to the closest sub-interval when multiple are available is a data dependent action which poses a privacy leakage risk. To circumvent this issue, the users embed this information into a univariate quantity using a hashing technique, which is then privatised via Laplace noise. Finally, estimates of the means are obtained inverting the hashing procedure and de-biasing using the estimated mixture proportions. The procedure is formalised in \Cref{alg:Mixture}. The privatised observations $\widetilde{V}_j^{(i)}$ satisfy $\alpha$-LDP by \Cref{app:lem:RR}, as do the values in \eqref{secmix:eq:unicentredest} due to being an instance of the Laplace mechanism.
    
    \begin{algorithm}
        \caption{\textsc{PrivateMixtureMean}\,$(\{X^{(i)}_{t}\}_{i\in[n],t\in[T]},\,n,\,T,\,\alpha,\,m,\,K)$}
        \begin{algorithmic}[1]
            \State Denote $N_k = \{n(k-1)/3 + 1, \hdots, nk/3\}$ for $k \in \{1, 2, 3\}$.
            \State Let $T^\ast = T \wedge \exp\{n\alpha^2/(Km^2)\}$ and $\Delta = \{4\log(n T^\ast \alpha^2)/{T^\ast} \}^{1/2}$.
            \refstepcounter{equation}\label{secmix:eq:deltaval} \hfill(\theequation)
            \State Let $N(\Delta) = \lceil 1/\Delta \rceil$.
            \State Let $I_j = [-1 + 2(j-1)\Delta,\; -1 + 2j\Delta)$ for all $j \in [N(\Delta)]$.
            \State Let $I_{N(\Delta)} = [-1 + 2\{N(\Delta)-1\}\Delta,\; -1 + 2N(\Delta)\Delta]$.
            \For{$i \in N_1$}
              \State Calculate $\hat{\theta}^{(i)} = (T^{\ast})^{-1}\sum_{t=1}^{T^{\ast}} X^{(i)}_{t}$.
              \For{$j=1,\dots,N(\Delta)$}
                \State Let $V^{(i)}_{j} = \mathbbm{1}\{\hat{\theta}^{(i)}\in I_j\}$ and sample $U_{i,j}\sim\mathrm{Unif}[0,1]$.
                \State $\widetilde{V}^{(i)}_j = V^{(i)}_j \mathbbm{1}\{U_{i,j} \leq \omega_{\alpha/2}\} + (1 - V^{(i)}_j)\mathbbm{1}\{U_{i,j} > \omega_{\alpha/2}\}$. \refstepcounter{equation}\label{secmix:eq:uniGRR}\hfill(\theequation)
              \EndFor 
            \EndFor
            \State Let $\mathcal{I}_0 = [N(\Delta)]$.
            \For{$k=1,\dots,m$}
                \State Choose $j_k^\ast = \min\argmax_{a \in \mathcal{I}_{k-1}} \sum_{i \in N_1}\widetilde{V}_a^{(i)}$.
                \refstepcounter{equation}\label{secmix:eq:unichosenindex}\hfill(\theequation)
                \State Denote $L_k,U_k$ the lower/upper endpoints of $I_{j_k^{\ast}}$.
                \State Let $\tilde{L}_k = L_k-2\Delta,  \tilde{U}_k = U_k+2\Delta$ and $\tilde{I}_{j_k^{\ast}} = [\tilde{L}_k,\tilde{U}_k]$. Denote midpoint $M_k = (\tilde{U}_k + \tilde{L}_k)/2$ \refstepcounter{equation}\label{secmix:eq:uniMeanIntervalExpand}\hfill(\theequation)
                \State Update $\mathcal{I}_{k} = \mathcal{I}_{k-1} \setminus \{j_k^\ast - 1, j_k^\ast, j_k^\ast + 1\}$
            \EndFor
            \State Reorder indices to satisfy $M_1 < \dots < M_m$.
            \For{$i \in N_2$}
              \State Calculate $\hat{\theta}^{(i)} = (T^{\ast})^{-1}\sum_{t=1}^{T^{\ast}} X^{(i)}_{t}$.
              
              \For{$k=1,\dots,m$}
                \State Let $V^{(i)}_{k} = \mathbbm{1}\{\hat{\theta}^{(i)}\in I_{j_k^\ast}\}$ and sample $U_{i,k}\sim\mathrm{Unif}[0,1]$.
                \refstepcounter{equation}\label{secmix:eq:mixturebinestimator}  \hfill(\theequation)
                \State $\widetilde{V}^{(i)}_j = V^{(i)}_j \mathbbm{1}\{U_{i,j} \leq \omega_{\alpha/2}\} + (1 - V^{(i)}_j)\mathbbm{1}\{U_{i,j} > \omega_{\alpha/2}\}$.

              \EndFor
            \EndFor
            \For{$k=1,\dots,m$}
                \State Let $\hat{\pi}_k = \{|N_2|^{-1}\sum_{i \in N_2}\widetilde{V}_{k}^{(i)}-(1- \omega_{\alpha/2})\}/(2\omega_{\alpha/2} - 1)$ 
                \refstepcounter{equation}\label{secmix:eq:mixpropests}\hfill(\theequation)
            \EndFor
            \State Sample $r_{k}^{(i)}$ i.i.d.~Rademacher random variables for $i \in N_3$, $k \in [m]$.
            \For{$i \in N_3$}
                \State Let $k^{(i)} = \argmin_{k \in [m]}\{|(T^\ast)^{-1}\sum_{t = 1}^{T^\ast}X_t^{(i)} - M_k|\}$.
                \State Sample $\ell_i \sim \mathrm{Laplace}(1)$ and calculate $\tilde{\theta}^{(i)} = r_k^{(i)}\Pi_{[-3\Delta, 3\Delta]}\{(T^\ast)^{-1}\sum_{t=1}^{T^\ast} - M_{k^{(i)}}\} + 6\Delta\ell_i/\alpha$. \refstepcounter{equation}\label{secmix:eq:unicentredest}\hfill(\theequation)
            \EndFor
            \For{$k=1,\dots,m$}
                \State Let $\hat{\theta}_k = M_k + \{(N_k\hat{\pi}_k)^{-1}\sum_{i \in N_3}(r_k^{(i)}\tilde{\theta}^{(i)})\}\mathbbm{1}\{\hat{\pi}_k > \pi_0/2\}$
                \refstepcounter{equation}\label{secmix:eq:FinalEstimatorComponent}\hfill(\theequation)
            \EndFor
            \State Denote $\hat{\theta} = \sum_{k = 1}^m e_k \hat{\theta}_k$ for $e_k$ the $k$-th standard basis vector of $\mathbb{R}^m$, and \Return $\hat{\theta}$. 
        \end{algorithmic}
        \label{alg:Mixture}
    \end{algorithm}

    We briefly provide some intuition for the form of \eqref{secmix:eq:FinalEstimatorComponent}. Roughly speaking, by the construction of the privatised values $\tilde{\theta}^{(i)}$ for $i \in N_3$, we have that $\sum_{i \in N_3} r_k{(i)} \tilde{\theta}^{(i)}$ will in expectation be the sum of $|N_3| \pi_k$-many estimates of $\theta_k - M_k$. Hence, we use the estimator $\hat{\pi}_k$ for $\pi_k$ to de-bias in the obvious manner. We take the product with the indicator $\mathbbm{1}\{\hat{\pi}_k \geq \pi_{0}/2 \}$ to avoid the error exploding in the rare case that the estimator $\hat{\pi}_k$ takes a negative or near-zero value.

\subsection{Discussion}
    In this appendix we derived the minimax rates for mean estimation under a relaxation of the user-level setting, where each user draws data from a mixture of $T$-fold product distributions, rather than all users drawing from the same distribution. Interestingly, the $T$-independent lower bound is greater than the corresponding bound in Theorem~6. Indeed, viewing the number of mixture components as the ``dimension'' of the problem, the lower bound in the mixture setting is $e^{-n\alpha^2/m^2}$ compared to $e^{-n\alpha^2/d}$. Heuristically, this is to be expected as both problems have $m$ (or $d$, respectively) estimation targets, but in the mixture setting the effective sample size per target is $O(n/m)$ whereas in $d$-dimensional mean estimation the effective sample size is still $O(n)$.

    Regarding mean estimation with data heterogeneity under \emph{central} user-level privacy, \cite{Cummings:2022} consider the problem where users draw data from a Bernoulli distribution where the mean parameter is itself drawn from a global meta-distribution. Here, the estimation target is the mean (and variance) of the meta-distribution, and so there remains a single estimation target regardless of the extent of heterogeneity. On the other hand, we aim to estimate the mean of $m$ distinct mixture components, with greater levels of heterogeneity resulting in a greater number of distinct estimation targets. The difference between these settings is reflected in the rates, where we observe the dependence on $m$ as the number of estimation targets increases in our setting. We also note, in both our work and \cite{Cummings:2022}, the rates established reduce to the standard user-level rates when homogeneity is enforced \citep[][Supplement]{Cummings:2022}.
    
    More generally, tackling sources of heterogeneity is becoming increasingly popular in the differential privacy literature, for example \cite{Setlur:2024} consider data heterogeneity under the central user-level setting, carrying out distribution estimation using clustering to group similar users, tackling heterogeneity in a similar vein to our mixture setting. Regarding heterogeneity in privacy requirements, \cite{Chaudhuri:2024} consider central item-level mean estimation with heterogeneous privacy requirements; \cite{Canonne:2024} consider LDP two-sample tests where the privacy budgets of the two samples may differ; and \cite{Imola:2024} consider heterogeneous privacy requirements via metric differential privacy \citep{Chatzikokolakis:2013} in a range of settings, including user-level LDP. However, the problem of data heterogeneity under the local user-level setting is less understood, and to the best of our knowledge ours is the first attempt to consider this setting. 

    Without privacy constraints, there is a richer literature on estimation problems under mixture settings (see, for example, \citealt{Wu:2020} and the references therein). Due to identifiably issues, among others, restrictive assumptions are often made when obtaining theoretical guarantees in this setting, such as assuming an equal balance of two Gaussians \citep{Xu:2016, Daskalakis:2017}. The user-level setting in fact ameliorates this identifiability issue, as with a large local sample from a single component, it is possible to distinguish data from differing components.
    
    An alternative approach is to aim to estimate the mixture distribution directly, as opposed to parameters of the components, as this allows one to avoid separation assumptions and better handle misspecification \citep[e.g.][]{Heinrich:2018}. Tackling the problem from this direction under local user-level privacy constraints remains unexplored and an interesting direction for future research.

\section{Proofs of Results in Section \ref{sec2}} \label{Appendix_sec2}

\begin{proof}[Proof of \Cref{sec2:ex:example}]
In this proof, we derive the two results~\eqref{sec2:eq:examplefinite} and~\eqref{sec2:eq:exampleinfinite} separately.

\medskip
\noindent \textbf{Sample Size $T$-Dependent Lower Bound \eqref{sec2:eq:examplefinite}.}  We construct two distributions $P_0$ and $P_1$ where
\begin{align*}
    P_0(X = 0) = P_0(X = 1) = \frac{1}{2} \quad \mbox{and} \quad P_1(X = 0) = 1 - P_0(X = 1) = \frac{1}{2} + 2\Delta,
\end{align*}
where $\Delta = \min\{1/5, (192nT\alpha^2)^{-1/2}\}$. Noting that $|\mathbb{E}_{P_0}(X) - \mathbb{E}_{P_1}(X)| = 2\Delta$, it follows from a private version of Le Cam's lemma \citep[][Proposition 1]{Duchi:2018} that
\begin{align*}
    \mathcal{R}_{n,T,\alpha}(\theta(\mathcal{P}), (\cdot)^2) \geq & \frac{\Delta^2}{2} \Big\{1 - \sqrt{ 3n\alpha^2 D_{\text{KL}}(P_0^{\otimes T} \| P_1^{\otimes T}) } \Big\} = \frac{\Delta^2}{2} \Big\{1 - \sqrt{ 3nT\alpha^2 D_{\text{KL}}(P_0 \| P_1) } \Big\} \\
    \geq & \frac{\Delta^2}{2} \Big(1 - \sqrt{ 48nT\alpha^2 \Delta^2} \Big) \gtrsim \min\left\{ 1, \frac{1}{nT\alpha^2} \right\}, 
\end{align*}
where the first inequality is due to the fact that $(e^\alpha - 1)^2 \leq 3\alpha^2$ for $\alpha \in (0,1]$; the identity follows from the assumption that for each user the $T$ samples are independent and the additivity of the Kullback--Leibler divergence for product distributions; the second inequality holds by noticing that for $\Delta \in (0, 1/5)$,
\begin{equation*}
    D_{\text{KL}}(P_0 \| P_1) = -\frac{1}{2} \log \left( 1 + 4\Delta \right) - \frac{1}{2} \log \left( 1 - 4\Delta \right) \leq 16\Delta^2;    
\end{equation*}
and the final inequality holds by the choice of $\Delta$. This concludes the proof of the inequality \eqref{sec2:eq:examplefinite}.

\medskip
\noindent \textbf{Sample Size $T$-Independent Lower Bound \eqref{sec2:eq:exampleinfinite}.}  We consider two distributions $P_0$ and $P_1$ where
\begin{equation*}
    P_0(X = 0) = P_1(X = 1) = 1.
\end{equation*}
For any user-level $\alpha$-LDP privacy mechanism $Q$, let
\begin{equation*}
    M_j^n(\cdot) = \int Q(\cdot \mid x^{(1:n)}) \diff (P_j^{\otimes T})^{\otimes n}(x^{(1:n)}), \quad j \in \{0, 1\},
\end{equation*}
where each $x^{(i)} \in \{0, 1\}^T$ for each $i \in [n]$.
It then follows that 
\begin{align*}
    \mathcal{R}_{n,T,\alpha}(\theta(\mathcal{P}), (\cdot)^2)
    \geq& \frac{1}{2} \big\{1 - D_{\mathrm{TV}}(M_0^{n}, M_1^{n})\big\}
    \geq \frac{1}{4} e^{-D_{\mathrm{KL}}(M_0^{n}, M_1^{n})}  \\
    \geq& \frac{1}{4} e^{-12n\alpha^2 D_{\mathrm{TV}}(P_0^{\otimes T}, P_1^{\otimes T})^2}
    \geq \frac{1}{4} e^{-12n\alpha^2},
\end{align*}         
where the first inequality follows from Le Cam's lemma \citep[e.g.][Lemma~1]{Yu:1997}; the second is from the Bretagnolle–Huber inequality \citep[e.g.][Equation~(2.25)]{Tsybakov:2009}; the third is from Corollary~3 in \cite{Duchi:2018} where we note that $(e^\alpha-1)^2 \leq 3\alpha^2$ for $\alpha \in (0,1]$; and the last is due to the fact that the total variation distance is trivially upper bounded by $1$. Hence, inequality \eqref{sec2:eq:exampleinfinite} is proved which concludes the proof.
\end{proof}

\begin{proof}[Proof of \Cref{sec2:thm:GeneralBound}]

We prove the lower and upper bounds separately, and the claim follows consequently.

\medskip \noindent \textbf{The Lower Bound.}

This proof is an application of Fano's lemma \citep[e.g.][Lemma~3]{Yu:1997}.  For $\Delta > 0$, let $\{P_v: \, v \in \mathcal{V} \subset \Theta\} \subset \mathcal{P}$ satisfy the following: (i) For any $v, v' \in \mathcal{V}$ with $v \neq v'$, it holds that $\rho(\theta (P_v), \theta (P_{v'})) \geq 2 \Delta$; and (ii) with $M(2\Delta)$ being the $2\Delta$-packing number in the metric space~$(\Theta, \rho)$, let $|\mathcal{V}| = M(2\Delta)$. Denoting for $v \in \mathcal{V}$ the private marginals $M_v^n(\cdot) = \int Q(\cdot \mid x_{1:n}) \diff P_v^{\otimes n}(x_{1:n})$, it then holds that
\begin{align}
    \mathcal{R}_{n, \infty, \alpha} \big(\theta(\mathcal{P}), \Phi \circ \rho\big) \geq & \inf_{Q \in \mathcal{Q}_{\alpha}} \inf_{\hat{\theta}} \max_{v \in \mathcal{V}}\mathbb{E}_{P_v, Q} \left[\Phi \circ \rho \big(\hat{\theta}, \theta(P_v)\big)\right] \nonumber \\
    \geq & \inf_{Q \in \mathcal{Q}_{\alpha}} \frac{\Phi(\Delta)}{2}\left\{1 - \frac{K_n + \log(2)}{\log(|\mathcal{V}|)} \right\} \geq \frac{\Phi(\Delta)}{2} \inf_{Q \in \mathcal{Q}_{\alpha}} \left\{1 - \frac{K_n + \log(2)}{\log(N(2\Delta))}\right\}, \label{eq-theorem1-proof-1}
\end{align}
where the second inequality follows from Fano's inequality \citep[e.g.][]{Yu:1997}, with
\begin{equation*}
    K_n = \max_{v, v' \in \mathcal{V}} D_{\mathrm{KL}} (M^{n}_v \| M^{n}_{v'}),
\end{equation*}
and the final inequality follows from the fact that $M(2\Delta) \geq N(2\Delta)$ \citep[e.g.][Lemma 4.2.8]{Vershynin:2018}. We now provide two bounds on $K_n$ for small and large values of $\alpha$ separately.

For $\alpha \in (0,1]$, we note
\begin{equation} \label{eq-theorem1-proof-2}
    \sup_{Q \in \mathcal{Q}_{\alpha}} \max_{v, v' \in \mathcal{V}} D_{\mathrm{KL}} (M^{n}_v \| M^{n}_{v'}) \leq 12n \alpha^2 \max_{v, v' \in \mathcal{V}} D^2_{\mathrm{TV}}(P_v, P_{v'}) \leq 12n \alpha^2,
\end{equation}
where the second inequality follows from Corollary 3 in \cite{Duchi:2018} and the fact that $(e^\alpha -1 )^2 \leq 3\alpha^2$ for $\alpha \in (0, 1]$, and the final inequality follows from the fact that the total variation distance is upper bounded by $1$.

For $\alpha > 1$, we employ the following arguments which in fact hold for all $\alpha > 0$, but are potentially loose for $\alpha < 1$. Denote, for $i \in [n]$, $M^{(i)}_v(\cdot \mid Z_{1:i-1} = z_{1:i-1})$ the private marginal distribution for $Z_i$ with raw data $X_i \sim P_v$, conditional on $Z_{1:i-1} = z_{1:i-1}$. Then, for any $v, v' \in \mathcal{V}$, we have by the chain rule for the Kullback--Leibler divergence \cite[e.g.][Chapter~3]{Gray:2011} that
\begin{align} \label{eq-theorem1-proof-2-chainrule}
    D_{\mathrm{KL}} (M^{n}_v \| M^{n}_{v'})
    = \sum_{i = 1}^n D_{\mathrm{KL}} \big(M^{(i)}_v(\cdot \mid Z_{1:i-1}) \| M^{(i)}_{v'}(\cdot \mid Z_{1:i-1})\big).
\end{align}
Fixing an $i \in [n]$, we inspect an individual term of the sum. We first note by the LDP condition that the channels $Q(\cdot \mid x_i, z_{1:i-1})$ and $Q(\cdot \mid x_i', z_{1:i-1})$ are absolutely continuous with respect to each other for all $x_i, x_i' \in \mathcal{X}$. Hence, there exists a measure $\mu(\cdot \mid z_{1:i-1})$ such that they have densities $q(\cdot \mid x_i, z_{1:i-1})$. Further, $M^{(i)}_v$ is absolutely continuous with respect to $\mu$ for all $v \in \mathcal{V}$ and so we may write the density $m^{(i)}_v$ where $m^{(i)}_v(z_i \mid z_{1:i-1}) = \int_{\mathcal{X}} q(z_i \mid x_i, z_{1:i-1}) \diff P_v(x_i)$.

Hence, we obtain that 
\begin{align}
    D_{\mathrm{KL}}(M^{(i)}_v(\cdot \mid Z_{1:i-1}) &\| M^{(i)}_{v'}(\cdot \mid Z_{1:i-1})) \nonumber \\
    &= \int_{\mathcal{Z}^{i-1}} D_{\mathrm{KL}} (M^{(i)}_v(\cdot \mid z_{1:i-1}) \| M^{(i)}_{v'}(\cdot \mid z_{1:i-1})) \diff M_v^{i-1}(z_{1:i-1}) \nonumber \\
    &= \int_{\mathcal{Z}^{i-1}} \bigg\{ \int_{\mathcal{Z}} m^{(i)}_v(z_i \mid z_{1:i-1}) \log\bigg( \frac{m_v^{(i)}(z_i \mid z_{1:i-1})}{m_{v'}^{(i)}(z_i \mid z_{1:i-1})} \bigg) \bigg\} \diff M_v^{i-1}(z_{1:i-1}) \nonumber \\
    &\leq \int_{\mathcal{Z}^{i-1}} \bigg\{ \int_{\mathcal{Z}}  \alpha m^{(i)}_v(z_i \mid z_{1:i-1}) \diff \mu(z_i \mid z_{1:i-1}) \bigg\} \bigg\} \diff M_v^{i-1}(z_{1:i-1}) = \alpha, \label{eq-theorem1-proof-2-lowpriv}
\end{align}
where the inequality is by \Cref{app:lem:marginalratio}. Combining \eqref{eq-theorem1-proof-1}, \eqref{eq-theorem1-proof-2}, \eqref{eq-theorem1-proof-2-chainrule} and \eqref{eq-theorem1-proof-2-lowpriv} leads to
\begin{equation*}
    \mathcal{R}_{n, \infty, \alpha} \big(\theta(\mathcal{P}), \Phi \circ \rho\big) \geq \frac{\Phi(\Delta)}{2} \left\{1 - \frac{12n \min\{\alpha, \alpha^2\} + \log(2)}{\log(N(2\Delta))}\right\},
\end{equation*}
as claimed.

\bigskip
\noindent \textbf{The Upper Bound.}

\noindent\textbf{Step 1: Preparation.}
Denote the non-overlapping covering set constructed in \eqref{sec2:eq:nonoverlapping} by $\widetilde{\mathcal{B}} = \cup_{i = 1}^{N(\Delta)} \{B_i'\}$ satisfying $|\widetilde{\mathcal{B}}| = N_{\widetilde{\mathcal{B}}} \leq N(\Delta)$.  For $P \in \mathcal{P}$, denote by $l \in [N_{\widetilde{\mathcal{B}}}]$ the fixed but unknown value such that $\theta(P) \in B'_{l} \in \widetilde{\mathcal{B}}$. It follows from the construction that $V^{(i)}_{j} = \mathbbm{1}\{j = l\}$, for $i \in [n]$ and $j \in [N_{\widetilde{\mathcal{B}}}]$. By the construction \eqref{eq-general-upper-bound-privacy}, it holds, for $Z_j = \sum_{i =1}^n \widetilde{V}_{j}^{(i)}$ the total privatised votes for the $j$-th ball, that
\begin{equation*}
    Z_{j} \sim \mathrm{Bin}(n, \omega_\{\alpha/2\}) \mathbbm{1}\{j = l\} + \mathrm{Bin}(n, 1 - \omega_\{\alpha/2\}) \mathbbm{1}\{j \neq l\}, \quad j \in [N_{\widetilde{\mathcal{B}}}].
\end{equation*}
In addition, $\{Z_{j}\}_{j \in [N_{\widetilde{\mathcal{B}}}]}$ is a collection of mutually independent random variables. Define the event $A = \cap_{j \neq l} \{Z_{l} > Z_{j}\}$.  The rest of the proof is concerned with controlling the risks on the event $A$ and its complement $A^c$. 

\medskip
\noindent \textbf{Step 2: Probability Control.}
It holds that 
\begin{equation}\label{eq-proof-general-upper-2}
    \mathbb{P}(A^c) \leq \sum_{j \in [N_{\widetilde{\mathcal{B}}}] \setminus \{ l\}} \mathbb{P}(Z_{l} \leq Z_{j}) \leq N(\Delta) \mathbb{P}(Z_{l} \leq n - Z') = N(\Delta) \mathbb{P}(Z \leq n),
\end{equation}
where $Z' \sim \mathrm{Bin}(n, \omega_\{\alpha/2\})$ is independent of $\{Z_{j}\}_{j \in [N_{\widetilde{\mathcal{B}}}]}$, and $Z = Z_{l} + Z'$. By construction, $Z_{l}$ and $Z'$ are i.i.d.~and $Z \sim \mathrm{Bin}(2n, \omega_\{\alpha/2\})$. Therefore $Z$ has the same distribution as $\sum_{i = 1}^{2n} W_i$, where $\{W_i\}_{i = 1}^{2n}$ is a collection of independent $\mathrm{Ber}(\omega_\{\alpha/2\})$ random variables. 

For $\alpha > 0$, let $\{W_i'\}_{i = 1}^{2n}$ is a collection of independent $\mathrm{Ber}(1 - \omega_\{\alpha/2\})$ random variables. We then have by the Chernoff--Hoeffding bound \citep[e.g.~Theorem~1 in][]{Hoeffding:1963}
\begin{align}
    \mathbb{P}(Z \leq n)
    & = \mathbb{P}\bigg(\sum_{i = 1}^{2n} W_i' \geq n \bigg)
    = \mathbb{P}\bigg(\frac{1}{2n}\sum_{i = 1}^{2n} (W_i' - \mathbb{E}[W_i']) \geq \omega_\{\alpha/2\} - 1/2 \bigg) \nonumber \\
    & \leq \{4\omega_\{\alpha/2\}(1-\omega_\{\alpha/2\})\}^{n} = \bigg( \frac{4e^{\alpha/2}}{(1+e^{\alpha/2})^2}\bigg)^n \leq e^{-n\min\{\alpha, \alpha^2\}/20}, \label{app:eq:chernoffhoeffdingapp}
\end{align}
where the final inequality is by \Cref{app:lem:sechbound}. Hence, combining \eqref{eq-proof-general-upper-2} and \eqref{app:eq:chernoffhoeffdingapp}, we obtain
\begin{equation*}
    \mathbb{P}(A^c) \leq N(\Delta)e^{-n\min\{\alpha, \alpha^2\}/20}.
\end{equation*}

\medskip
\noindent \textbf{Step 3: Conclusion.}  
On the event $A$, due to the construction, we have that $\hat{\theta}, \theta(P) \in B_{l}$ and hence that $\rho(\hat{\theta}, \theta(P)) \leq \Delta$. On the event $A^c$, we can only guarantee that $\hat{\theta}, \theta(P) \in \Theta$ and hence that $\rho(\hat{\theta}, \theta(P)) \leq \text{diam}(\Theta)$. From this and the bound on $\mathbb{P}(A^c)$ we obtain that
\begin{equation*}
    \mathcal{R}_{n, \infty, \alpha}(\theta(\mathcal{P}), \Phi \circ \rho) \leq \Phi(\Delta) \mathbb{P}(A)  + \Phi\left(\mathrm{diam}(\Theta)\right) \mathbb{P}(A^c) \leq \Phi(\Delta) + \Phi(\mathrm{diam}(\Theta)) N(\Delta) e^{-n\min\{\alpha, \alpha^2\}/20},
\end{equation*}
as required.
\end{proof}

\begin{proof}[Proof of \Cref{sec2:thm:Means}]
We prove the $\ell_\infty$- and $\ell_2$-cases separately.  In each case, we first upper and lower bound the covering numbers of the spaces and then apply \Cref{sec2:thm:GeneralBound}.

\medskip
\noindent \textbf{The $\ell_{\infty}$-Ball Case.}  
For $\Delta \in (0, 1)$, to upper bound the covering number $N(\Delta)$ of $\mathbb{B}_{\infty}(1)$, we have that
\begin{align}
    N(\Delta) & \leq \frac{\mathrm{Vol}((2/\Delta)\mathbb{B}_{\infty}(1) + \mathbb{B}_2(1))}{\mathrm{Vol}(\mathbb{B}_2(1))} \leq\frac{\mathrm{Vol}((2/\Delta + 1)\mathbb{B}_{\infty}(1))}{\mathrm{Vol}(\mathbb{B}_2(1))}  \leq \left(\frac{3}{\Delta} \right)^d \frac{\mathrm{Vol}(\mathbb{B}_\infty(1))}{\mathrm{Vol}(\mathbb{B}_2(1))} \nonumber \\
    & =\left( \frac{3}{\Delta} \right)^d \frac{2^d \Gamma(1 +d/2)}{\pi^{d/2}}
    \leq \left( \frac{18 d}{\pi \Delta^2}\right)^{d/2}, \label{eq-proof-prop-4-1}
\end{align}
where the first inequality follows from Lemma~5.7(b) in \cite{Wainwright:2019}, by setting $\mathbb{B}$, $\mathbb{B}'$, $\|\cdot\|$ and $\|\cdot\|'$ therein as $\mathbb{B}_{\infty}(1)$, $\mathbb{B}_2(1)$, $\|\cdot\|_{\infty}$ and $\|\cdot\|_2$; the second from the fact that $\mathbb{B}_2(1) \subseteq \mathbb{B}_\infty(1)$; the third from the fact that $\Delta \in (0, 1)$; and the last is due to $\Gamma(1+x) \leq x^x$ for all $x > 0$.

For the lower bound on $N(\Delta)$, we have that, for any $\Delta \in (0,1)$,
\begin{equation} \label{eq-proof-prop-4-2}
    N(\Delta)
    \geq \left(\frac{1}{\Delta} \right)^d \frac{\mathrm{Vol}(\mathbb{B}_\infty(1))}{\mathrm{Vol}(\mathbb{B}_2(1))}
    =\left( \frac{1}{\Delta} \right)^d \frac{2^d \Gamma(1 +d/2)}{\pi^{d/2}}
    \geq \left( \frac{2d}{3\pi \Delta^2}\right)^{d/2},
\end{equation}
where the first inequality is from Lemma~5.7(a) in \cite{Wainwright:2019} and the final inequality from $\Gamma(1+x) \geq (x/3)^x$ using Stirling's approximation. Combining \eqref{eq-proof-prop-4-1} and \eqref{eq-proof-prop-4-2}, we have that 
\begin{equation}
    \left( \frac{d}{5\Delta^2}\right)^{d/2}
    \leq N(\Delta)
    \leq \left( \frac{6d}{\Delta^2}\right)^{d/2}. \label{app:eq:linfcover}
\end{equation}

We now apply \Cref{sec2:thm:GeneralBound} to obtain upper and lower bounds on the minimax rate. To obtain the lower bound, we let $\Delta = (d/20)^{1/2}e^{-24n\min\{\alpha, \alpha^2\}/d-\log(4)/d} < 1$, where the inequality holds due to the assumption that $n\min\{\alpha, \alpha^2\} > 60d\log(6d)$. An application of the general lower bound in \eqref{sec2:eq:GeneralLB} gives
\begin{align}
    \mathcal{R}_{n, \infty, \alpha} (\theta(\mathcal{P}_d), \|\cdot\|_2^2)
    &\geq \frac{\Delta^2}{2} \bigg\{1 - \frac{12n \min\{\alpha, \alpha^2\} + \log(2)}{\log(N(2\Delta))}\bigg\}
    \geq \frac{\Delta^2}{2} \bigg\{1 - \frac{24n \min\{\alpha, \alpha^2\} + \log(4)}{d\log(d/(20\Delta^2))}\bigg\} \nonumber \\
    = & \frac{d}{40} e^{-48n\min\{\alpha, \alpha^2\}/d-\log(16)/d} \bigg\{1 - \frac{12n \min\{\alpha, \alpha^2\} + \log(2)}{d \{24n\min\{\alpha, \alpha^2\}/d + \log(4)/d\}}\bigg\} \nonumber \\
    = & \frac{d}{80} e^{-48n\min\{\alpha, \alpha^2\}/d-\log(16)/d} \gtrsim e^{-Cn\min\{\alpha, \alpha^2\}/d}, \label{sec2:eq:linflowerbound}
\end{align}
where $C>0$ is an absolute constant; the first inequality is from the first inequality in \eqref{sec2:eq:GeneralLB}; the second from the first inequality in \eqref{app:eq:linfcover}; the first equality follows from the choice of $\Delta$; and the last inequality holds as $n\min\{\alpha, \alpha^2\} \gtrsim 1$ which follows from the assumption $n\min\{\alpha, \alpha^2\} > 60d\log(6d)$.

To obtain an upper bound, we set $\Delta = (6d)^{1/2}e^{-n\min\{\alpha, \alpha^2\}/(40d)} < 1$ where the inequality holds due to the assumption $n\min\{\alpha, \alpha^2\} > 60d\log(6d)$. An application of the general upper bound in \eqref{sec2:eq:GeneralLB} gives that
\begin{align}
    \mathcal{R}_{n, \infty, \alpha} (\theta(\mathcal{P}_{d, r}), \|\cdot\|_2^2)
    &\leq \Delta^2 + \{\mathrm{diam}(\Theta)\}^2 N(\Delta) e^{-n\min\{\alpha, \alpha^2\}/20}
    \leq \Delta^2 + 4d (6d/\Delta^2)^{d/2} e^{-n\min\{\alpha, \alpha^2\}/20} \nonumber \\
    &= 6de^{-n\min\{\alpha, \alpha^2\}/(20d)} + 4d e^{-n\min\{\alpha, \alpha^2\} [1/20 - 1/40]}
    \lesssim  e^{-cn\min\{\alpha, \alpha^2\}/d}, \label{sec2:eq:linfupperbound}
\end{align}
where $c>0$ is an absolute constant; the first inequality follows from the upper bound in \eqref{sec2:eq:GeneralLB}; the second from the second inequality in \eqref{app:eq:linfcover} and the fact that $\mathrm{diam}(\Theta) = 2d^{1/2}$; the equality holds by the choice of $\Delta$; and the final inequality comes from the assumption that $n\min\{\alpha, \alpha^2\} > 60d\log(6d)$ which allows the factor of $d$ to be absorbed into the exponential term. By combining \eqref{sec2:eq:linflowerbound} and \eqref{sec2:eq:linfupperbound}, we conclude the proof of \eqref{sec2:eq:linfMeanStatement}.

\medskip
\noindent \textbf{The $\ell_2$-Ball Case.}  It follows from Example~5.8 of \cite{Wainwright:2019} that, for any $\Delta \in (0, 1)$, the $\Delta$-covering number of $\mathbb{B}_2(1)$ satisfies
\begin{equation}\label{app:eq:l2cover}
    (1 / \Delta)^d \leq N(\Delta) \leq (3 / \Delta)^d.
\end{equation}
To obtain the lower bound on the minimax risk, we let $\Delta = 2^{-1}e^{-24n\min\{\alpha, \alpha^2\}/d-\log(4)/d} < 1$.  An application of the general lower bound in \eqref{sec2:eq:GeneralLB} gives that
\begin{align}
    \mathcal{R}_{n, \infty, \alpha} (\theta(\mathcal{P}_{d}), \|\cdot\|_2^2) &\geq \frac{\Delta^2}{2} \bigg\{1 - \frac{12n \min\{\alpha, \alpha^2\} + \log(2)}{\log(N(2\Delta))}\bigg\}  \geq \frac{\Delta^2}{2} \bigg\{1 - \frac{12n \min\{\alpha, \alpha^2\} + \log(2)}{d\log(1/(2\Delta))}\bigg\} \nonumber \\
    &= \frac{1}{8} e^{-48n\min\{\alpha, \alpha^2\}/d-\log(16)/d} \bigg\{1 - \frac{12n \min\{\alpha, \alpha^2\} + \log(2)}{d \{24n\min\{\alpha, \alpha^2\}/d + \log(4)/d\}}\bigg\} \nonumber \\
    &= \frac{1}{16} e^{-48n\min\{\alpha, \alpha^2\}/d-\log(16)/d} \gtrsim e^{-C'n\min\{\alpha, \alpha^2\}/d}, \label{sec2:eq:l2lowerbound}
\end{align}
where $C' > 0$ is an absolute constant; the first inequality is from \eqref{sec2:eq:GeneralLB}; the second from the first inequality in \eqref{app:eq:l2cover}; the first equality follows from the choice of $\Delta$; and the last inequality holds as $n\min\{\alpha, \alpha^2\} \gtrsim 1$ which follows from the assumption $n\min\{\alpha, \alpha^2\} > 60d$.

To obtain the upper bound, we let $\Delta = 3 e^{-n\min\{\alpha, \alpha^2\}/(30d)} < 1$, where the inequality holds due to the assumption that $n\min\{\alpha, \alpha^2\} > 60d$. An application of the general upper bound in \eqref{sec2:eq:GeneralLB} gives that
\begin{align}
    \mathcal{R}_{n, \infty, \alpha} (\theta(\mathcal{P}_{d, 1}), \|\cdot\|_2^2)
    &\leq \Delta^2 + \{\mathrm{diam}(\Theta)\}^2 N(\Delta) e^{-n\min\{\alpha, \alpha^2\}/20}
    \leq \Delta^2 + 4 (3/\Delta)^d e^{-n\min\{\alpha, \alpha^2\}/20} \nonumber \\
    &= 9 e^{-n\min\{\alpha, \alpha^2\}/(15d)} + 4 e^{-n\min\{\alpha, \alpha^2\} (1/20 - 1/30)}
    \lesssim  e^{-c'n\min\{\alpha, \alpha^2\}/d}, \label{sec2:eq:l2upperbound}
\end{align}
where $c'>0$ is an absolute constant; the first inequality follows the upper bound in \eqref{sec2:eq:GeneralLB}; the second from \eqref{app:eq:l2cover} and the fact that $\mathrm{diam}(\Theta) = 2$; and the equality holds by the choice of $\Delta$. By combining \eqref{sec2:eq:l2lowerbound} and \eqref{sec2:eq:l2upperbound}, we conclude the proof of \eqref{sec2:eq:l2MeanStatement}.
\end{proof}
    
\begin{proof}[Proof of \Cref{sec2:thm:Sparse}]
    For $\Delta > 0$, we first obtain the required bounds on the $\Delta$-covering number $N(\Delta)$ of the metric space $(\Theta, \|\cdot\|_2)$, where $\Theta = \theta(\mathcal{P}_{d, s})$ and $\|\cdot\|_2$ is the vector $\ell_2$-norm.  
    
    For an upper bound on $N(\Delta)$, we note that there are $\binom{d}{s}$-many choices for the non-zero elements of the parameter vector, and for each choice, the covering number of the resulting space can be bounded by \eqref{app:eq:linfcover} with $s$ in place of $d$ therein. Hence, for any $\Delta \in (0,1)$, we have that
        \begin{equation}
            N(\Delta)
            \leq \binom{d}{s} \left( \frac{6s}{\Delta^2}\right)^{\frac{s}{2}}
            \leq \left( \frac{ed}{s} \right)^{s} \left( \frac{6s}{\Delta^2}\right)^{\frac{s}{2}}, \label{app:eq:SparseCoveringBoundUpper}
        \end{equation}
        where the second inequality comes from the bound $\binom{d}{s} \leq (ed/s)^s$ for $s \in [d]$ due to Stirling's approximation.     
        
    For a lower bound on $N(\Delta)$, by \Cref{app:lem:sparsecoverlower} and for any $\Delta \in (0, 1/4)$, it holds that
        \begin{equation}
            N(\Delta)
            \geq \left( \frac{d}{s} \right)^{s} \left\{\frac{s(1/\Delta-2)^2}{20}\right\}^{\frac{s}{2}}. \label{app:eq:SparseCoveringBoundLower}
        \end{equation}
        
        \noindent
        To obtain the lower bound on the risk, we let
        \begin{equation*}
            \Delta
            = \left( \frac{2(20s)^{1/2}}{d}e^{24n\min\{\alpha, \alpha^2\}/s + \log(4)/s} + 4 \right)^{-1} < \frac{1}{8}
        \end{equation*}
        where the inequality holds for $d \geq 2$ due to the assumption $n\min\{\alpha, \alpha^2\} > 120s\log(6d)$. An application of the general lower bound in \eqref{sec2:eq:GeneralLB} gives that
        \begin{align}
            \mathcal{R}_{n, \infty, \alpha} (\theta(\mathcal{P}_{d, s}), \|\cdot\|_2^2)
            &\geq \frac{\Delta^2}{2} \bigg\{1 - \frac{12n \min\{\alpha, \alpha^2\} + \log(2)}{\log(N(2\Delta))}\bigg\} \nonumber \\
            &\geq \frac{\Delta^2}{2} \bigg(1 - \frac{24n \min\{\alpha, \alpha^2\} + \log(4)}{s\log[\{d^2/(20s)\}\{1/(2\Delta) - 2\}^2]}\bigg) \nonumber \\
            &= \left( \frac{2\sqrt{20s}}{d}e^{24n\min\{\alpha, \alpha^2\}/s + 2\log(2)/s} + 4 \right)^{-2} \bigg(1 - \frac{1}{2}\bigg) \gtrsim \frac{d^2}{s} e^{-Cn\min\{\alpha, \alpha^2\}/s}, \label{sec2:eq:sparselowerbound}
        \end{align}
        where $C>0$ is an absolute constant; the first inequality is from \eqref{sec2:eq:GeneralLB}; the second from the bound \eqref{app:eq:SparseCoveringBoundLower}; the equality follows from the choice of $\Delta$; and the last inequality holds as $(s^{1/2}/{d})e^{24n\min\{\alpha, \alpha^2\}/s + \log(4)/s} \gtrsim 1$ due to the assumption $n\min\{\alpha, \alpha^2\} \geq 120s\log(6d)$.
        
        To obtain the upper bound, we let $\Delta = 6^{1/2}eds^{-1/2}e^{-n\min\{\alpha, \alpha^2\}/(40s)} < 1$, where the inequality holds due to the assumption $n\min\{\alpha, \alpha^2\} \geq 120s\log(6d)$. An application of the general upper bound in \eqref{sec2:eq:GeneralLB} shows that
        \begin{align}
            \mathcal{R}_{n, \infty, \alpha} (\theta(\mathcal{P}_{d, s}), \|\cdot\|_2^2)
            &\leq \Delta^2 + \{\mathrm{diam}(\Theta)\}^2 N(\Delta) e^{-n\min\{\alpha, \alpha^2\}/20} \nonumber \\
            &\leq \Delta^2 + 4s \left( \frac{ed}{s} \right)^{s} \left( \frac{6s}{\Delta^2}\right)^{\frac{s}{2}} e^{-n \min\{\alpha, \alpha^2\} / 20} \nonumber \\
            &= 6e^2 \frac{d^2}{s} e^{-n\min\{\alpha, \alpha^2\}/(20s)} + 4s e^{-n\min\{\alpha, \alpha^2\} [1/20 - 1/40]}
            \lesssim  e^{-cn\min\{\alpha, \alpha^2\}/s}, \label{sec2:eq:sparseupperbound}
        \end{align}
        where $c>0$ is an absolute constant; the first inequality follows from the upper bound in \eqref{sec2:eq:GeneralLB}; the second from \eqref{app:eq:SparseCoveringBoundUpper} and the fact that $\mathrm{diam}(\Theta) = 2s^{1/2}$; the equality holds by the choice of $\Delta$; and the final inequality comes from the assumption that $n\min\{\alpha, \alpha^2\} > 120s\log(6d)$. By combining \eqref{sec2:eq:sparselowerbound} and \eqref{sec2:eq:sparseupperbound}, we conclude the proof of \eqref{sec2:eq:sparseMeanStatement}.
    \end{proof}

\begin{proof}[Proof of \Cref{sec2:thm:Density}]
    For $\Delta > 0$, we first obtain the required bounds on the $\Delta$-covering number $N(\Delta)$ of the metric space $(\Theta, \|\cdot\|_2)$, where $\Theta = \theta(\mathcal{P}_{\beta, 1})$ and $\|\cdot\|_2$ is the $L^2([0, 1])$ norm.  We have that
        \begin{equation}
            \exp\left(c_\beta(1/\Delta)^{1/\beta}\right) \leq N(\Delta) \leq \exp\left(C_\beta(1/\Delta)^{1/\beta}\right), \label{app:eq:sobolevcovering}
        \end{equation}
    where the upper bound holds by noting $\mathcal{F}_{\beta, 1} \subset \mathcal{S}_{\beta, 1}$ and by an upper bound on the covering number of $\mathcal{S}_{\beta, 1}$ \citep[e.g.][Example~5.12]{Wainwright:2019}, and the lower bound follows from \Cref{app:lem:densitycoverlower}, for $\Delta$ sufficiently small.  
            
    To obtain the lower bound in~\eqref{sec2:eq:densityStatement}, we set
    \begin{equation*}
        \Delta = \frac{1}{2} \left( \frac{c_\beta}{c_\beta'\{24n\min\{\alpha, \alpha^2\}+2\log(2)\}} \right)^{\beta},
    \end{equation*}
    where $c_\beta' > 0$ is a sufficiently large absolute constant, so that $\Delta$ is sufficiently small therefore the lower bound in \eqref{app:eq:sobolevcovering} holds.  Applying the lower bound in \eqref{sec2:eq:GeneralLB} leads to
    \begin{align*}
        \mathcal{R}_{n, \infty, \alpha} (\theta(\mathcal{F}_{\beta, 1}), \|\cdot\|_2^2)
            &\geq \frac{\Delta^2}{2} \bigg\{1 - \frac{12n \min\{\alpha, \alpha^2\} + \log(2)}{\log(N(2\Delta))}\bigg\} \geq \frac{\Delta^2}{2} \bigg\{1 - \frac{12n \min\{\alpha, \alpha^2\} + \log(2)}{c_\beta\{1/(2\Delta)\}^{1/\beta}}\bigg\} \\
            &= \frac{1}{8} \left( \frac{c_\beta}{c_\beta'\{24n\min\{\alpha, \alpha^2\}+2\log(2)\}} \right)^{2\beta} \bigg(1 - \frac{1}{2c_\beta'}\bigg)
            \gtrsim \left(\frac{c_\beta}{c_\beta'n\min\{\alpha, \alpha^2\}}\right)^{2\beta}, 
    \end{align*}
    where the first inequality is from the lower bound in \eqref{sec2:eq:GeneralLB}, the second from the first inequality in \eqref{app:eq:sobolevcovering}, the first equality follows from the choice of $\Delta$, and the last holds under the assumption that $n\min\{\alpha, \alpha^2\} \gtrsim 1$. Setting $c_\beta'' = (c_\beta/c_\beta')^{2\beta}$ shows that the lower bound can be written as
        \begin{equation}
            \mathcal{R}_{n, \infty, \alpha}(\theta(\mathcal{F}_{\beta, 1}), \|\cdot\|_2^2)
            \gtrsim \frac{c_\beta''}{(n \min\{\alpha, \alpha^2\})^{2\beta}}, \label{app:eq:densitylowerINF}
        \end{equation}
        as required.
            
    To obtain the upper bound in \eqref{sec2:eq:densityStatement}, we let $\Delta = \{40C_\beta / (n \min\{\alpha, \alpha^2\})\}^{\beta}$. When $n\min\{\alpha, \alpha^2\} \geq C_\beta'$, where $C_\beta' > 0$ is a sufficiently large constant, we have $\Delta$ sufficiently small so that the upper bound in \eqref{app:eq:sobolevcovering} holds. Applying the upper bound in \eqref{sec2:eq:GeneralLB} leads to
    \begin{align}
        \mathcal{R}_{n, \infty, \alpha} (\theta(\mathcal{F}_{\beta, 1}), \|\cdot\|_2^2)
        &\leq \Delta^2 + \{\mathrm{diam}(\Theta)\}^2 N(\Delta) e^{-n\min\{\alpha, \alpha^2\}/20} \nonumber \\
        &\leq \Delta^2 + 2\exp\left(C_\beta(1/\Delta)^{1/\beta} - n\min\{\alpha, \alpha^2\}/20\right) \nonumber \\
        &= \{40C_\beta / (n \min\{\alpha, \alpha^2\})\}^{2\beta} + 2e^{-n\min\{\alpha, \alpha^2\}/40}
        \lesssim \{C_\beta''/ (n \min\{\alpha, \alpha^2\})\}^{2\beta}, \label{app:eq:densityupperINF}
    \end{align}
    where $C_\beta'' > 0$ is some constant depending on $\beta$, the first inequality follows from the upper bound in \eqref{sec2:eq:GeneralLB}, the second from the second inequality in  \eqref{app:eq:sobolevcovering} and the fact that $\mathrm{diam}(\Theta) \leq 2^{1/2}$, the equality holds by the choice of $\Delta$ and simplifying the exponential term; and the final inequality comes from noting that for any value of $\beta$ the exponential term is dominated by the polynomial. By combining \eqref{app:eq:densitylowerINF} and \eqref{app:eq:densityupperINF}, we conclude the proof of \eqref{sec2:eq:densityStatement}.
\end{proof}

\section{Extension of Results of Section \ref{sec2} to Approximate Privacy} \label{app:sec:approxLDP}

The results of \Cref{sec2} consider the case of ``pure'' privacy. However, there is also the relaxed notion of ``approximate'' privacy. Loosely speaking, approximate differential privacy allows the condition of pure differential privacy to fail with some, usually very small, probability $\delta \in (0,1]$. Some motivations for this include tighter composition results on how privacy degrades when repeatedly querying a user \cite[Theorem~III.3]{Dwork:2010}, and the ability to privatise continuous valued data using Gaussian noise.

For completeness, we define approximate LDP in the user-level setting: For the same notation and setting as in \eqref{sec1:eq:alphaLDP}, we say that a collection of conditional distributions $\{Q_i\}_{i=1}^n$ satisfies $(\alpha, \delta)$-LDP if, for all $i \in \{1, 2, \ldots, n\}$,
\begin{equation}\label{sec1:eq:alphadeltaLDP}
    \begin{aligned}
        &Q_i(Z^{(i)} \in S \mid X^{(i)} = x^{(i)}, Z^{(1:i-1)} = z^{(1:i-1)}) \leq e^{\alpha}Q_i(Z^{(i)} \in S \mid X^{(i)} = x'^{(i)}, Z^{(1:i-1)} = z^{(1:i-1)}) + \delta, \\
        & \hspace{8cm}\forall x^{(i)}, \, x'^{(i)} \in \mathcal{Y}, \forall z^{(1:i-1)} \in \mathcal{Z}^{i-1} \mbox{ and } S \in \sigma(\mathcal{Z}).
    \end{aligned}
\end{equation}
The concepts of interactive and non-interactive privacy follow similarly, and we denote by $\mathcal{R}_{n, T, \alpha, \delta}$ the extension of the user-level LDP minimax risk $\mathcal{R}_{n, T, \alpha}$ in \eqref{sec2:eq:UserMinimax} to the approximate setting, where we instead take the the infimum over all $(\alpha, \delta)$-LDP mechanisms.

In item-level LDP, for the minimax rates, it is usually the case that approximate and pure LDP constraints lead to the same rate up to logarithmic factors, for $\delta$ sufficiently small \citep[e.g.][]{Bassily:2015, Bun:2019}. Indeed, $\delta$ is almost invariably taken to be $\delta \ll 1/n$ as $\delta$ any larger ceases to provide a meaningful privacy guarantee.\footnote{Consider an otherwise pure $\alpha$-LDP mechanism which publishes the true value with probability $\delta$. This would satisfy approximate LDP, but for $\delta \asymp 1/n$, a constant proportion of users will leak their true value.} Hence, there are usually no improvements in a minimax sense for using the looser concept of approximate LDP within this ``regime of equivalence''.

However, in the user-level setting with $T$ users, it is not a-priori clear whether the sample size in this context should be taken as $n$ or $nT$. If the latter were the case, it would suggest in the limit $T \rightarrow \infty$ that this ``regime of equivalence'' vanishes, suggesting approximate LDP may exhibit interesting behaviour in the user-level setting for even small $\delta$. The following result answers this in the negative, demonstrating that for $\delta \lesssim 1/n$, the $T$-independent lower bounds remain valid and bound the risk away from zero for arbitrary $T$ as in the pure LDP setting. 

\begin{theorem}[General infinite-$T$ rates under approximate LDP]
    \label{sec2:thm:GeneralBoundApproximateLDP}
    Start with the same conditions and notation as \Cref{sec2:thm:GeneralBound}. Provided $\delta$ is such that $\delta \leq c\min\{1, \alpha\}/\{n\log(n)\}$ for $c > 0$ a sufficiently small absolute constant and any $\alpha > 0$, and that $n\min\{1, \alpha\} > C$ for $C > 0$ some sufficiently large constant, we have that the infinite-$T$ private minimax risk and the user-level private minimax risk for all $T \geq 1$ under $(\alpha, \delta)$-LDP are lower bounded as
    \begin{equation*}
        \mathcal{R}_{n, \infty, \alpha, \delta}(\theta(\mathcal{P}), \Phi \circ \rho),\;
        \mathcal{R}_{n, T, \alpha, \delta}(\theta(\mathcal{P}), \Phi \circ \rho)
        \gtrsim \sup_{\Delta : N(2 \Delta)>1}\bigg\{ \Phi(\Delta)\bigg(1 - \frac{n \min\{\alpha, \alpha^2\} + \log(2)}{\log(N(2\Delta))} \bigg) \bigg\}.
    \end{equation*}
\end{theorem}

\begin{proof}
    \noindent \textbf{Step 1: Fano's Inequality and Necessary Notation.}  As in the proof of \Cref{sec2:thm:GeneralBound}, we start by appealing to Fano's inequality. To be precise, let $V$ be uniformly distributed over the set $\mathcal{V}$, and denote the mutual information between $V$ and the private observations $Z_{1:n}$ by $I(V; Z_{1:n})$. We denote, for $v \in \mathcal{V}$, the private marginals $M_v^n(\cdot) = \int Q(\cdot \mid x_{1:n}) \diff P_v^{\otimes n}(x_{1:n})$, and, for $i \in [n]$, $M^{(i)}_v(\cdot \mid Z_{1:i-1})$ the private marginal distribution for $Z_i$ with raw data $X_i \sim P_v$, conditional on $Z_{1:i-1}$. Then, following the proof of Theorem~3 in \citet{Yu:1997}, we obtain that
    \begin{align}
        \mathcal{R}_{n, \infty, \alpha, \delta} \big(\theta(\mathcal{P}), \Phi \circ \rho\big) \geq & \inf_{Q \in \mathcal{Q}_{\alpha}} \inf_{\hat{\theta}} \max_{v \in \mathcal{V}}\mathbb{E}_{P_v, Q} \left[\Phi \circ \rho \big(\hat{\theta}, \theta(P_v)\big)\right] \nonumber \\
        \geq & \inf_{Q \in \mathcal{Q}_{\alpha}} \frac{\Phi(\Delta)}{2}\left\{1 - \frac{I(V; Z_{1:n}) + \log(2)}{\log(|\mathcal{V}|)} \right\}, \label{app:eq:approxTinfFano}
    \end{align}
    where we have that $Z_{1:n} \mid V \sim M_V^n$.

    Our goal is to upper bound the quantity $I(V; Z_{1:n})$. Before proceeding, we \textbf{(a)} show that without loss of generality, we may reduce to the case that $Z_{1:n}$ is a vector of discrete values, enabling us to appeal to useful results from information theory; and \textbf{(b)} prepare necessary notation. 
    
\medskip 
\noindent \textbf{Step 1.1: (a).} Indeed, in the case of continuous random variables, one can take an increasingly fine sequence of partitions of the state space and discretise the observations, provided that the sigma algebra generated by the partitions increase to the sigma algebra on the state space \citep[e.g.][Appendix]{Tulcea:1961}. Further, in our setting the sigma algebra on $\mathcal{Z}$ is assumed Borel and $\mathcal{Z}$ is a separable metric space, and hence the Borel sigma algebra on the product space $\mathcal{Z}^n$ is generated by the rectangles $\{A_1 \times \hdots \times A_n: A_i \in \sigma(\mathcal{Z})\}$ \cite[e.g.][Lemma~1.2]{Kallenberg:2021}. Thus, we may discretise the vector $Z_{1:n}$ along each co-ordinate and assume $Z_i \in [k]$ for all $i \in [n]$ in what follows. In particular, we may assume the private marginals $M_v^n(\cdot)$ and $M_v^{(i)}(\cdot \mid Z_{1:i-1})$ have corresponding probability mass functions denoted $m_v^n(\cdot)$ and $m_v^{(i)}(\cdot \mid Z_{1:i-1})$ respectively.

\medskip 
\noindent \textbf{Step 1.2: (b).}  Given discrete random variables $A \in [a]$, $B \in [b]$ and $C \in [c]$, we define the entropy
    \begin{equation} \label{app:eq:entropy}
        H(A) = \sum_{i = 1}^a \mathbb{P}(A = i) \log_2\{1/\mathbb{P}(A = i)\},
    \end{equation}
    conditional entropy
    \begin{equation} \label{app:eq:condentropy}
        H(A|B) = H(A,B) - H(B) = \sum_{i = 1}^a \sum_{j = 1}^b \mathbb{P}(A = i, B = j) \log_2\{1/\mathbb{P}(A = i \mid B = j)\},
    \end{equation}
    and conditional mutual information $I(A;B \mid C) = I(A; B, C) - I(A; C)$.

\medskip
\noindent \textbf{Step 2: Upper Bounding the Mutual Information.} First, by the chain rule of mutual information \cite[e.g.][Chapter~3]{Gray:2011}, we have that
    \begin{equation} \label{app:eq:approxTinfMIChain}
        I(V; Z_{1:n}) = \sum_{i = 1}^n I(V;Z_i \mid Z_{1:{i-1}}).
    \end{equation}
    We focus on $I(V;Z_i \mid Z_{1:{i-1}})$, an individual term of the sum. Let
    \begin{equation} \label{app:eq:Bvardefs}
        \begin{aligned}
            &B_{v, i, 1}(V, Z_{1:i}) = B_{v, i, 1} = \mathbbm{1}\bigg\{\frac{m_v^{(i)}(Z_i \mid Z_{1:i-1})}{\overline{m}^{(i)}(Z_i \mid Z_{1:i-1})} > e^{2\alpha}, V = v\bigg\}, \mbox{ and } \\
            &B_{v, i, 2}(V, Z_{1:i}) = B_{v, i,2} = \mathbbm{1}\bigg\{\frac{ m_v^{(i)}(Z_i \mid Z_{1:i-1})}{\overline{m}^{(i)}(Z_i \mid Z_{1:i-1})} < e^{-2\alpha}, V = v\bigg\}.
        \end{aligned}
    \end{equation}
    where we denote $\overline{m}^{(i)}(\cdot \mid Z_{1:i-1}) = \sum_{v \in \mathcal{V}} \mathbb{P}(V = v \mid Z_{1:i-1}) m_v^{(i)}(\cdot \mid Z_{1:i-1})$, and likewise for $\overline{M}^{(i)}(\cdot \mid Z_{1:i-1})$. Then, writing $B_{v, i} = B_{v, i, 1} + B_{v, i, 2}$ and $B_i = \sum_{v \in \mathcal{V}} B_{v, i}$, we consider
    \begin{align}
        I(V; Z_i \mid Z_{1:i-1})
        &\leq I(V, B_i; Z_i \mid Z_{1:i-1}) \nonumber \\
        &= I(V; Z_i \mid Z_{1:i-1}, B_i) + I(B_i; Z_i \mid Z_{1:i-1}) \nonumber \\
        &= I(V; Z_i \mid Z_{1:i-1}, B_i) + H(B_i \mid Z_{1:i-1}) - H(B_i \mid Z_{1:i})\nonumber \\
        &\leq I(V; Z_i \mid Z_{1:i-1}, B_i) + H(B_i \mid Z_{1:i-1}) \nonumber \\
        &\leq I(V; Z_i \mid Z_{1:i-1}, B_i = 0) + I(V; Z_i \mid Z_{1:i-1}, B_i = 1)\mathbb{P}(B_i = 1) + H(B_i \mid Z_{1:i-1}) \nonumber \\
        &\leq I(V; Z_i \mid Z_{1:i-1}, B_i = 0) + I(V; Z_i \mid Z_{1:i-1}, B_i = 1)\mathbb{P}(B_i = 1) + H(B_i) \nonumber \\
        &\leq (I) + (II) + (III), \label{app:eq:overallMIbound}
    \end{align}
    where the first and third lines are by the definition of mutual information; the second by the chain rule of mutual information; the fourth by the non-negativity of entropy; and the remaining inequalities by the definition of entropy \eqref{app:eq:entropy}, conditional entropy \eqref{app:eq:condentropy} and conditional mutual information.

    \medskip
    \noindent \textbf{Step 2.1: Term (III).}
    Denoting the set $\mathcal{B}_{v, i, 1}(V, Z_{1:i-1}) = \{z \in [k] : B_{v, i,1}(V, Z_{1:i-1}, z) = 1\}$, and likewise of $\mathcal{B}_{v, i, 2}(V, Z_{1:i-1})$, we start by considering
    \begin{align}
        &\mathbb{P}(B_{v, i} = 1 \mid Z_{1:i-1}) \nonumber \\
        &= \mathbb{P}(B_{v, i, 1} = 1 \mid Z_{1:i-1}) + \mathbb{P}(B_{v, i, 2} = 1 \mid Z_{1:i-1}) \nonumber \\
        &= \mathbb{P}(V = v \mid Z_{1:i-1})\big[M_v^{(i)}(Z_i \in \mathcal{B}_{v, i, 1}(v, Z_{1:i-1}) \mid Z_{1:i-1}) + M_v^{(i)}(Z_i \in \mathcal{B}_{v, i, 2}(v, Z_{1:i-1}) \mid Z_{1:i-1}) \big]. \label{app:eq:BiEntropProbFinal1}
    \end{align}

    Focusing on the first term, we have
    \begin{equation}
        M_v^{(i)}(Z_i \in \mathcal{B}_{v, i, 1}(v, Z_{1:i-1}) \bigm| Z_{1:i-1})
        \geq e^{2\alpha} \overline{M}^{(i)}(Z_i \in \mathcal{B}_{v, i, 1}(v, Z_{1:i-1}) \mid Z_{1:i-1}) \text{ a.s.}, \label{app:eq:BiEntropProb1}
    \end{equation}
    which holds due to the conditions defining $B_{v, i,1}$ in \eqref{app:eq:Bvardefs}. We also have for any $v, v' \in \mathcal{V}$ that
    \begin{align}
        M_{v'}^{(i)}(Z_i \in &\mathcal{B}_{v, i, 1}(v, Z_{1:i-1}) \mid Z_{1:i-1}) \nonumber \\
        &= \int_\mathcal{X} Q(Z_i \in \mathcal{B}_{v, i, 1}(v, Z_{1:i-1}) \mid x_i, Z_{1:i-1}) \diff P_{v'}(x_i) \nonumber \\
        &= \int_\mathcal{X} \int_\mathcal{X} Q(Z_i \in \mathcal{B}_{v, i, 1}(v, Z_{1:i-1}) \mid x_i, Z_{1:i-1}) \diff P_{v}(y_i) \diff P_{v'}(x_i) \nonumber \\
        &\geq \int_\mathcal{X} \int_\mathcal{X} e^{-\alpha} \{Q(Z_i \in \mathcal{B}_{v, i, 1}(v, Z_{1:i-1}) \mid y_i, Z_{1:i-1}) - \delta\}\diff P_{v}(y_i) \diff P_{v'}(x_i) \nonumber \\
        &= e^{-\alpha}M_{v}^{(i)}(Z_i \in \mathcal{B}_{v, i, 1}(v, Z_{1:i-1}) \mid Z_{1:i-1}) - \delta e^{-\alpha} \text{ a.s.}, \label{app:eq:BiEntropProb2}
    \end{align}
    where $Q$ denotes the LDP channel; the second equality is by Tonelli's theorem; and the inequality is by the definition of $(\alpha, \delta)$-LDP. Hence, combining \eqref{app:eq:BiEntropProb1} and~\eqref{app:eq:BiEntropProb2} and rearranging, we obtain
    \begin{equation}
        M_v^{(i)}(Z_i \in \mathcal{B}_{v, i, 1}(v, Z_{1:i-1}) \mid Z_{1:i-1})
        \leq \frac{\delta e^\alpha}{e^\alpha - 1} \text{ a.s.}  \label{app:eq:BiEntropProbFinal2}
    \end{equation}
    A similar calculation yields
    \begin{equation}
        M_v^{(i)}(Z_i \in \mathcal{B}_{v, i, 2}(v, Z_{1:i-1}) \mid Z_{1:i-1})
        \leq \frac{\delta e^{-\alpha}}{e^\alpha - 1} \text{ a.s.}, \label{app:eq:BiEntropProbFinal3}
    \end{equation}
    and hence, combining \eqref{app:eq:BiEntropProbFinal1}, \eqref{app:eq:BiEntropProbFinal2} and \eqref{app:eq:BiEntropProbFinal3}, we obtain
    \begin{equation} \label{app:eq:BiEntropProbFinalControl}
        \mathbb{P}(B_{i} = 1 \mid Z_{1:i-1})
        \leq \sum_{v \in \mathcal{V}} \mathbb{P}(V = v \mid Z_{1:i-1}) \mathbb{P}(B_{v,i} = 1 \mid Z_{1:i-1})
        \leq \frac{\delta(e^\alpha + e^{-\alpha})}{e^\alpha - 1} \leq \frac{2\delta e^\alpha}{e^\alpha - 1} \text{ a.s.}
    \end{equation}

    Denoting $p_i = \mathbb{P}(B_{i} = 1) = \mathbb{E}[\mathbb{P}(B_{i} = 1 \mid Z_{1:i-1})]$, we have $H(B_{i}) = p_i\log_2(1/p_i) - (1-p_i)\log_2(1-p_i) \lesssim p_i\log_2(1/p_i)$ for $0 < p_i \leq 1 - c'$ where $c'$ is some small constant. In particular, as by assumption $\delta \leq c/n$ for $c > 0$ a sufficiently small absolute constant and $\delta \leq \alpha/2$, we have that $p_i$ is bounded away from one, and hence
    \begin{equation}
        H(B_{i}) \lesssim \frac{2\delta e^\alpha}{e^\alpha - 1} \log_2\bigg( \frac{e^\alpha - 1}{2\delta e^\alpha} \bigg) \lesssim \frac{\delta}{\min\{1, \alpha\}}\log_2\bigg(\frac{\min\{1, \alpha\}}{\delta}\bigg), \label{app:eq:overallMIbound1}
    \end{equation}
    where the last inequality comes from noting that for $\alpha \in (0,1]$, $1/\alpha \leq e^\alpha/(e^\alpha - 1) \leq 2/\alpha$, and for $\alpha > 1$ that $1 \leq e^\alpha/(e^\alpha - 1) \leq 2$.
    
    \medskip
    \noindent \textbf{Controlling} (I).
    Write $m_{v, B_i = 0}^{(i)}(z_i \mid Z_{1:i-1}) = m_{v}^{(i)}(z_i \mid Z_{1:i-1}, B_i(V, Z_{1:i}) = 0)$ and $\overline{m}_{B_i = 0}^{(i)}(z_i \mid Z_{1:i-1}) = \sum_{v \in \mathcal{V}}\mathbb{P}(V = v \mid Z_{1:i-1}, B_i(V, Z_{1:i}) = 0)m_{v}^{(i)}(z_i \mid Z_{1:i-1}, B_i(V, Z_{1:i}) = 0)$.
    Denoting $P$ the law of a random variable where appropriate, we have that

    \begin{align}
        &I(V; Z_i \mid Z_{1:i-1}, B_{i}(V, Z_{1:i}) = 0) \nonumber \\
        &= \mathbb{E}_{Z_{1:i-1}}\bigg[ D_{\mathrm{KL}}\big(P_{V, Z_i \mid Z_{1:i-1}, B_{i}(V, Z_{1:i}) = 0} \| P_{V \mid Z_{1:i-1}, B_{i}(V, Z_{1:i}) = 0} \otimes P_{Z_i \mid Z_{1:i-1}, B_{i}(V, Z_{1:i}) = 0}\big) \biggm| B_i(V, Z_{1:i}) = 0\bigg] \nonumber \\
        &\leq \mathbb{E}_{Z_{1:i-1}}\bigg[ D_{\mathrm{KL}}\big(P_{V, Z_i \mid Z_{1:i-1}, B_{i}(V, Z_{1:i}) = 0} \| P_{V \mid Z_{1:i-1}, B_{i}(V, Z_{1:i}) = 0} \otimes P_{Z_i \mid Z_{1:i-1}, B_{i}(V, Z_{1:i}) = 0}\big) \nonumber\\ 
        &\hspace{1cm}+ D_{\mathrm{KL}}\big(P_{V \mid Z_{1:i-1}, B_{i}(V, Z_{1:i}) = 0} \otimes P_{Z_i \mid Z_{1:i-1}, B_{i}(V, Z_{1:i}) = 0} \| P_{V, Z_i \mid Z_{1:i-1}, B_{i}(V, Z_{1:i}) = 0}\big) \biggm| B_i(V, Z_{1:i}) = 0\bigg] \nonumber \\
        &= \mathbb{E}_{Z_{1:i-1}}\bigg[ \sum_{v \in \mathcal{V}, z_i \in [k]} \mathbb{P}(V = v \mid Z_{1:i-1}, B_i(V, Z_{1:i}) = 0)\{ m_{v, B_i = 0}^{(i)}(z_i \mid Z_{1:i-1})  - \overline{m}_{B_i = 0}^{(i)}(z_i \mid Z_{1:i-1})\} \nonumber \\
        &\hspace{8cm}\times \log\bigg(\frac{m_{v, B_i = 0}^{(i)}(z_i \mid Z_{1:i-1})}{\overline{m}_{B_i = 0}^{(i)}(z_i \mid Z_{1:i-1})}\bigg) \biggm| B_i(V, Z_{1:i}) = 0\bigg] \nonumber \\
        &= \mathbb{E}_{Z_{1:i-1}}\bigg[ \sum_{v \in \mathcal{V}} \mathbb{P}(V = v \mid Z_{1:i-1}, B_i(V, Z_{1:i}) = 0) \bigg\{ \sum_{z_i \in [k]}\{ m_{v, B_i = 0}^{(i)}(z_i \mid Z_{1:i-1})  - \overline{m}_{B_i = 0}^{(i)}(z_i \mid Z_{1:i-1})\} \nonumber \\
        &\hspace{8cm}\times \log\bigg(\frac{m_{v, B_i = 0}^{(i)}(z_i \mid Z_{1:i-1})}{\overline{m}_{B_i = 0}^{(i)}(z_i \mid Z_{1:i-1})}\bigg) \bigg\} \biggm| B_i(V, Z_{1:i}) = 0\bigg] \nonumber \\
        &\lesssim \mathbb{E}_{Z_{1:i-1}}\bigg[ \min\{\alpha, \alpha^2\}\sum_{v \in \mathcal{V}} \mathbb{P}(V = v \mid Z_{1:i-1}, B_i(V, Z_{1:i}) = 0) \biggm| B_i(V, Z_{1:i}) = 0\bigg]
        = \min\{\alpha, \alpha^2\}, \label{app:eq:LDPMIBound1}
    \end{align}
    where the first inequality is by the non-negativity of Kullback--Leibler divergence, noting the distributions are absolutely continuous with respect to each other on $\{B_i = 0\}$, and the second by \Cref{app:lem:KLboundonevent}. Hence, we obtain
    \begin{equation} \label{app:eq:overallMIbound2}
        I(V; Z_i \mid Z_{1:i-1}, B_i = 0) \lesssim \min\{\alpha, \alpha^2\}
    \end{equation}
    
    \medskip
    \noindent \textbf{Controlling} (II).  
    We first control the mutual information term
    \begin{align*}
        I(V; Z_i \mid Z_{1:i-1}, B_i(V, Z_{1:i}) = 1)
        &= H(V \mid Z_{1:i-1}, B_i = 1) - H(V \mid Z_i, Z_{1:i-1}, B_i = 1) \\
        &\leq H(V \mid Z_{1:i-1}, B_i = 1) \\
        &\leq H(V)
        \leq \log_2(|\mathcal{V}|),
    \end{align*}
    where the equality is by the definition of conditional mutual information; the first inequality by the non-negativity of entropy; the second inequality from the fact that conditioning cannot increase entropy which can be seen from the definition \eqref{app:eq:condentropy} and the non-negativity of entropy; and the final by the maximal bound on entropy which can be readily verified using Jensen's inequality.

    Hence, by \eqref{app:eq:BiEntropProbFinalControl}, we have
    \begin{equation} \label{app:eq:overallMIbound3}
        I(V; Z_i \mid Z_{1:i-1}, B_i = 1) \mathbb{P}(B_i = 1)
        \leq \frac{2\delta e^{\alpha}\log_2(|\mathcal{V}|)}{e^\alpha - 1}
        \lesssim \frac{\delta\log_2(|\mathcal{V}|)}{\min\{1,\alpha\}},
    \end{equation}
    where the last inequality comes from noting that for $\alpha \in (0,1]$, $e^\alpha/(e^\alpha - 1) \leq 2/\alpha$, and for $\alpha > 1$ that $e^\alpha/(e^\alpha - 1) \leq 2$.

    With all three terms controlled, we now conclude the proof. Combining \eqref{app:eq:approxTinfFano}, \eqref{app:eq:approxTinfMIChain}, \eqref{app:eq:overallMIbound}, \eqref{app:eq:overallMIbound1}, \eqref{app:eq:overallMIbound2} and \eqref{app:eq:overallMIbound3}, we obtain
    \begin{align*}
        \mathcal{R}_{n, \infty, \alpha, \delta} &\big(\theta(\mathcal{P}), \Phi \circ \rho\big) \\
        &\geq \frac{\Phi(\Delta)}{2}\bigg\{ 1 - \frac{C'n \min\{\alpha, \alpha^2\} + \log(2)}{\log(|\mathcal{V}|)} - C'\bigg(\frac{n\delta}{\min\{1, \alpha\}} - \frac{n\delta \log_2(\min\{1, \alpha\}/\delta)}{\min\{1, \alpha\} \log_2(|\mathcal{V}|)}\bigg)\bigg\} \\
        &\geq \frac{\Phi(\Delta)}{2}\bigg( \frac{1}{2} - \frac{C'n \min\{\alpha, \alpha^2\} + \log(2)}{\log(|\mathcal{V}|)} \bigg).
    \end{align*}
    where the second inequality comes from the fact $\delta \leq c\min\{1, \alpha\}/\{n\log(n)\}$ for $c > 0$ some sufficiently small constant and the assumption that $n\min\{1, \alpha\} > C$ for $C > 0$ some sufficiently large constant.
\end{proof}

\section{Mean Estimation in Low-Privacy Regime} \label{app:sec:meanlowpriv}

    In this appendix, we provide the general algorithms for user-level mean estimation Section~3 to all $\alpha \in (0, \infty)$. Recalling the families~$\mathcal{P}_{d}$ and $\mathcal{P}_{d}'$ defined in \eqref{sec2:eq:MeanClasses}, we estimate the mean of a distribution $P$ so that $\theta(P) = \mathbb{E}_P(X)$ and measure the error by the squared $\ell_2$-loss. 
    
    We now present the estimation methods for the univariate case, the multivariate case with the $\ell_{\infty}$-ball and the multivariate case with the $\ell_2$-ball. In the case $\alpha \in (0, 1]$, we recover the methods of Section~3. The performance of these algorithms is analyzed in \Cref{app:sec:Upper}.

    \subsection{Univariate Procedure} \label{app:sec:LowPrivUniMean}
        We first introduce the main sub-routine for general privacy parameter $\alpha > 0$ in \Cref{alg:GeneralUserLevelMean}, followed by the univariate mean estimation procedure in \Cref{alg:unigeneralpriv}.
        \begin{algorithm}
            \caption{\textsc{GeneralUserLevelMean}\,$(\{X^{(i)}_{t}\}_{i\in[n],t\in[T]},\,n,\,T,\,\alpha,\,\Delta)$}
            \begin{algorithmic}[1]
                \State Let $N(\Delta) = \lceil 1/\Delta \rceil$;\; $I_j = [-1 + 2(j-1)\Delta,\; -1 + 2j\Delta)$ for all $j \in [N(\Delta)]$;\; and $I_{N(\Delta)} = [-1 + 2\{N(\Delta)-1\}\Delta,\; -1 + 2N(\Delta)\Delta]$.
                \For{$i=1,\dots,n/2$}
                    \State Calculate $\hat{\theta}^{(i)} = (T^{\ast})^{-1}\sum_{t=1}^{T^{\ast}} X^{(i)}_{t}$. \refstepcounter{equation}\label{app:eq:unilowprivunilocalisationestimator} \hfill(\theequation)
                    
                    \For{$j=1,\dots,N(\Delta)$}
                        \If{$\alpha \in (0, 1]$}
                            \State Let $V^{(i)}_{j} = \mathbbm{1}\{\hat{\theta}^{(i)}\in I_j\}$ and sample $U_{i,j}\sim\mathrm{Unif}[0,1]$.
                            \State $\widetilde{V}^{(i)}_j = V^{(i)}_j \mathbbm{1}\{U_{i,j} \leq \omega_{\alpha/2}\} + (1 - V^{(i)}_j)\mathbbm{1}\{U_{i,j} > \omega_{\alpha/2}\}$. \refstepcounter{equation}\label{sec3:eq:uniGRRGeneral} \hfill(\theequation)
                        \Else
                            \State Let $V^{(i)}_{j} = \mathbbm{1}\{\hat{\theta}^{(i)}\in I_{j-1} \cup I_j \cup I_{j + 1}\}$ and sample $U_{i,j}\sim\mathrm{Unif}[0,1]$.
                            \State $\widetilde{V}^{(i)}_j = V^{(i)}_j \mathbbm{1}\{U_{i,j} \leq \omega_{\alpha/6}\} + (1 - V^{(i)}_j)\mathbbm{1}\{U_{i,j} > \omega_{\alpha/6}\}$. \refstepcounter{equation}\label{sec3:eq:uniGRRGeneral2} \hfill(\theequation)
                        \EndIf
                        
                    \EndFor
                \EndFor
                \State Let $j^{\ast} = \min\arg\max_{j\in[N(\Delta)]}\sum_{i=1}^{n/2}\widetilde{V}^{(i)}_{j}$ and denote $L,U$ the lower/upper endpoints of $I_{j^{\ast}}$. \refstepcounter{equation}\label{app:eq:unilowprivunichosenindex} \hfill(\theequation)
                \If{$\alpha \in (0, 1]$}
                    \State Let $\tilde{L} = L-2\Delta,  \tilde{U} = U+2\Delta$ and
                \Else
                    \State Let $\tilde{L} = L-5\Delta,  \tilde{U} = U+5\Delta$ and
                \EndIf
                \State Let $\tilde{I}_{j^{\ast}} = [\tilde{L},\tilde{U}]$. \refstepcounter{equation}\label{app:eq:unilowprivuniMeanIntervalExpand} \hfill(\theequation) 
                \For{$i=n/2+1,\dots,n$}
                    \State Sample $\ell_i \sim \mathrm{Laplace}(1)$ and let $\hat{\theta}^{(i)} = \Pi_{\tilde{I}_{j^{\ast}}}\{(T^\ast)^{-1}\sum_{t=1}^{T^{\ast}} X^{(i)}_{t}\} + \{(\tilde{U} - \tilde{L})/\alpha\}\,\ell_i$. \refstepcounter{equation}\label{app:eq:unilowprivUniRefinedEstimator} \hfill(\theequation)
                \EndFor
                \State Let $\hat{\theta} = (n/2)^{-1}\sum_{i=n/2+1}^{n}\hat{\theta}^{(i)}$ and
                \Return $\hat{\theta}$.
            \end{algorithmic}
            \label{alg:GeneralUserLevelMean}
        \end{algorithm}

        \begin{algorithm}
            \caption{\textsc{GeneralPrivateUniMean}\,$(\{X^{(i)}_{t}\}_{i\in[n],t\in[T]},\,n,\,T,\,\alpha,\,K)$}
            \begin{algorithmic}[1]
                \State Let $T^{\ast} = T \wedge \exp(n\min\{\alpha, \alpha^2\}/K)$ 
                \If{$\alpha \in (0, 1]$}
                    \State $\Delta = \{2\log(nT^{\ast}\alpha^{2})/T^{\ast}\}^{1/2}$.
                \Else
                    \State $\Delta = \{4\log(nT^{\ast})/T^{\ast}\}^{1/2}$. \refstepcounter{equation}\label{app:eq:unilowprivdeltaval2} \hfill(\theequation) 
                \EndIf
                \State Let $\hat{\theta} = \textsc{GeneralUserLevelMean}(\{X_t^{(i)}\}_{i \in [n], t \in [T^\ast]}, n, T^\ast, \alpha, \Delta)$ and
                \Return $\hat{\theta}$.
            \end{algorithmic}
            \label{alg:unigeneralpriv}
        \end{algorithm}
    
        The privatization in \eqref{sec3:eq:uniGRRGeneral} and \eqref{sec3:eq:uniGRRGeneral2} satisfies $\alpha$-LDP by \Cref{app:lem:RR}, as does the estimator $\hat{\theta}^{(i)}$ due to being an instance of the Laplace mechanism.
    
        We briefly discuss the differing voting procedure dependent on the value of $\alpha$. The core difficulty is that the true $\theta$ may be close to the endpoint of a bin. In this case, the votes may be evenly split between each bin. For $\alpha \in (0,1]$, it is enough to consider that at least one sub-interval will contain half the users with high probability via Hoeffding's inequality. For $\alpha > 1$ however, to be able to obtain tight minimax rates we must instead employ a more direct Chernoff bound argument, and require that at least one sub-interval close to the true mean will get voted for by every user before privatization with high probability, requiring the vote for the three closest bins.

    \subsection{Multivariate Procedure for \texorpdfstring{$\ell_\infty$}{Sup-Norm}-ball} \label{app:sec:LowPrivlinfball}

        We generalise the multivariate procedure for the $\ell_\infty$-ball to general $\alpha \in (0, \infty)$. We recall for $\alpha \leq 1$, we split the $n$ users into $d$ many equally-sized folds, each of which is assigned to estimate a co-ordinate using the univariate procedure. For the case $\alpha > 1$, we instead assign each user to $(\lfloor \alpha \rfloor \wedge d)$-many differing folds, with privacy budget $1 \vee (\alpha/d)$. In particular, when $1 \vee (\alpha/d) > 1$, it is necessary to use the univariate procedure tailored for the low-privacy setting as in \Cref{app:sec:LowPrivUniMean}. We assume without loss of generality that $n$ is a multiple of $2d$. The procedure is formalised in \Cref{alg:linflowpriv}. The procedure is private for $\alpha \leq 1$ as the univariate mechanism is $\alpha$-LDP. For $\alpha > 1$, by the compositionality of LDP \citep{Dwork:2006}, the procedure is LDP with parameter $(\lfloor \alpha \rfloor \wedge d)\{1 \vee (\alpha/d)\} \leq \alpha$.

        \begin{algorithm}
            \caption{\textsc{GeneralPrivateMultiDimMean}\,$(\{X^{(i)}_{t}\}_{i\in[n],t\in[T]},\,n,\,T,\,\alpha,\,K)$}
            \begin{algorithmic}[1]
                \State Let $N_j = N_{j, 1} \cup N_{j, 2}$ where $N_{j,1} = \left\{(j-1)\frac{n}{2d} + 1, \hdots, j\frac{n}{2d} \right\}$,\newline $N_{j,2} = \left\{(d+j-1)\frac{n}{2d} + 1, \hdots, (d+j)\frac{n}{2d} \right\}$, for $j \in [d]$. For convenience, denote $N_k = N_{k - d}$ for $k \in \{d + 1, \hdots, 2d\}$.
                \For{$j = 1, \hdots, d$}
                    \If{$\alpha \in (0, 1]$} 
                        \State Let $T^\ast = T \wedge \exp(n\alpha^2/Kd)$, $\Delta = \{2\log(nT^{\ast}\alpha^{2}/d)/T^{\ast}\}^{1/2}$.
                        \State Calculate $\hat{\theta}_j = \textsc{GeneralUserLevelMean}(\{X^{(i)}_{t, j}\}_{i \in N_j, t \in [T^\ast]},\, n,\, T^\ast,\,\alpha,\,\Delta)$.
                    \ElsIf{$\alpha \in (1, d]$}
                        \State Let $T^\ast = T \wedge \exp(n/Kd)$, $\Delta = \{2\log(nT^{\ast}/d)/T^{\ast}\}^{1/2}$.
                        \State Calculate $\hat{\theta}_j = \textsc{GeneralUserLevelMean}(\{X^{(i)}_{t, j}\}_{i \in N_j \cup \hdots \cup N_{j + \lfloor \alpha \rfloor}, t \in [T^\ast]},\, n,\, T^\ast,\,1,\,\Delta)$.
                    \Else
                        \State Let $T^\ast = T \wedge \exp(n\alpha/Kd)$, $\Delta = \{4\log(nT^{\ast}/d)/T^{\ast}\}^{1/2}$.
                        \State Calculate $\hat{\theta}_j = \textsc{GeneralUserLevelMean}(\{X^{(i)}_{t, j}\}_{i \in [n], t \in [T^\ast]},\, n,\, T^\ast,\, \alpha/d ,\,\Delta)$.
                    \EndIf
                \EndFor
            
                \State Let $\hat{\theta} = \sum_{j=1}^d e_j \hat{\theta}_{j}$ where $e_j$ is the $j$-th standard basis vector of $\mathbb{R}^d$, and
                \Return $\hat{\theta}$. \refstepcounter{equation}\label{app:eq:LowPrivlinfFinalEstimator} \hfill(\theequation) 
            \end{algorithmic}
            \label{alg:linflowpriv}
        \end{algorithm}

    \subsection{Multivariate Procedure for \texorpdfstring{$\ell_2$}{Euclidean}-ball} \label{app:sec:LowPrivl2ball}  
        
        For the $\ell_2$-ball setting, the procedure for $\alpha > 1$ operates similarly to the $\ell_\infty$-ball setting, suitably distributing each user across multiple co-ordinates simultaneously by splitting the privacy budget. This is formalised in \Cref{alg:l2lowpriv}.

        \begin{algorithm}
            \caption{\textsc{GeneralPrivateMultiDimMean2}\,$(\{X^{(i)}_{t}\}_{i\in[n],t\in[T]},\,n,\,T,\,\alpha,\,K,\,C)$}
            \begin{algorithmic}[1]
                \State Let $N_j = N_{j, 1} \cup N_{j, 2}$ where $N_{j,1} = \left\{(j-1)\frac{n}{2d} + 1, \hdots, j\frac{n}{2d} \right\}$,\newline $N_{j,2} = \left\{(d+j-1)\frac{n}{2d} + 1, \hdots, (d+j)\frac{n}{2d} \right\}$, for $j \in [d]$. For convenience, denote $N_k = N_{k - d}$ for $k \in \{d + 1, \hdots, 2d\}$. Sample random rotation matrix $R_d$ as in~(13).
                \For{$j = 1, \hdots, d$}
                    \If{$\alpha \in (0, 1]$} 
                        \State Denote $T^\ast = T \wedge \exp(n\alpha^2/Kd)$, $\Delta = \{C\log(nT^{\ast}\alpha^{2}/d)/T^{\ast}\}^{1/2}$.
                        \State Let $\hat{\theta}_j = \textsc{GeneralUserLevelMean}(\{(R_d X^{(i)}_{t})_j\}_{i \in N_j, t \in [T^\ast]},\, n,\, T^\ast,\,\alpha,\,\Delta)$.
                    \ElsIf{$\alpha \in (1, d]$}
                        \State Denote $T^\ast = T \wedge \exp(n/Kd)$, $\Delta = \{C\log(nT^{\ast}/d)/T^{\ast}\}^{1/2}$.
                        \State Let $\hat{\theta}_j = \textsc{GeneralUserLevelMean}(\{(R_d X^{(i)}_{t})_j\}_{i \in N_j \cup \hdots \cup N_{j + \lfloor \alpha \rfloor}, t \in [T^\ast]},\, n,\, T^\ast,\,1,\,\Delta)$.
                    \Else
                        \State Denote $T^\ast = T \wedge \exp(n\alpha/Kd)$, $\Delta = \{C\log(nT^{\ast}/d)/T^{\ast}\}^{1/2}$.
                        \State Let $\hat{\theta}_j = \textsc{GeneralUserLevelMean}(\{(R_d X^{(i)}_{t})_j\}_{i \in [n], t \in [T^\ast]},\, n,\, T^\ast,\, \alpha/d ,\,\Delta)$.
                    \EndIf
                \EndFor
            
                \State Let $\hat{\theta} = R_d^{-1}\sum_{j=1}^d e_j \hat{\theta}_{j}$ where $e_j$ is the $j$-th standard basis vector of $\mathbb{R}^d$, and
                \Return $\hat{\theta}$. \refstepcounter{equation}\label{app:eq:LowPrivFinall2Estimator} \hfill(\theequation) 
            \end{algorithmic}
            \label{alg:l2lowpriv}
        \end{algorithm}
    
        As with the $\ell_\infty$-ball case, this procedure is private for $\alpha \leq 1$ as the univariate mechanism is $\alpha$-LDP. For $\alpha > 1$, by the compositionality of LDP \citep{Dwork:2006}, the procedure is LDP with parameter $(\lfloor \alpha \rfloor \wedge d)\{1 \vee (\alpha/d)\} \leq \alpha$.
    \subsection{Discussion}\label{app:sec:LowPrivDiscussion}
        
        In this appendix we extended the results of \Cref{sec3} to the low privacy setting $\alpha > 1$. In particular we observe, just as in the $\alpha \in (0,1]$ setting, that our procedure attains the $T$-independent lower bounds for $T$ sufficiently large. Importantly, here we consider the $T$-independent lower bound for general $\alpha$ as given in \Cref{sec2:thm:GeneralBoundApproximateLDP}, where the modifications to the univariate procedure in \Cref{app:sec:LowPrivUniMean} are critical for attaining these matching rates.
        
        To our knowledge, the case of $\alpha > 1$ under user-level LDP has been explored in depth only in \cite{Acharya:2023} for the problem of discrete density estimation. There, a multitude of regimes exist, depending delicately on $n, T, \alpha$ and the dimension of the problem. We instead observe fewer regimes in the problem of mean estimation, but also share the requirement $d^2/nT \lesssim 1$ for the low-privacy lower bounds to be valid. The greater number of regimes observed in \cite{Acharya:2023} is due the subtleties introduced by the dependence between observations drawn from the multinomial distributions considered therein when obtaining lower bounds, whereas no such dependence arises in the cases of mean estimation we consider. For a more detailed treatment, see the discussion and use of Poissonisation in \cite{Acharya:2023}.

\section{Computational Complexity of Estimation Procedures} \label{app:sec:CompComplex}
    In this appendix we summarise the computational complexity of the developed estimation procedures.

    \medskip
    \noindent
    \textbf{Univariate Mean Procedure:}
    We first note that each user can calculate their quantities in parallel, barring the one step of interactivity. In the localisation stage, each user calculates the mean of $O(T^\ast)$ observations, requiring this many operations, and then queries the membership of this mean out of $O((T^\ast)^{1/2})$ sub-intervals to obtain an $O((T^\ast)^{1/2})$-dimensional vector. To privatise the vector, each user checks whether to flip an entry of the vector via \eqref{sec3:eq:uniGRR}, and so the privatisation step requires $O((T^\ast)^{1/2})$ operations. The data aggregator then sums the votes of $O(n)$ users for each $O((T^\ast)^{1/2})$-dimensional vote vector. Hence, the total complexity of this stage is $\max\{O(T^\ast), O(n(T^\ast)^{1/2})\}$.
    
    In the refinement stage, each user again calculates a mean of $O(T^\ast)$ observations. The truncation and subsequent privatisation can be carried out in $O(1)$ time. Finally, the data aggregator averages over $O(n)$ users. Hence, the total complexity of this stage is $\max\{O(T^\ast), O(n)\}$, and the overall computational complexity is $\max\{O(T^\ast), O(n(T^\ast)^{1/2})\}$.

    \medskip
    \noindent
    \textbf{$\ell_\infty$-Ball Procedure:}
    We first consider the procedure for $\alpha \leq 1$ as in \Cref{sec3:linfball}. We recall this multivariate procedure splits the users into $d$ groups, each with a sample size of $O(n/d)$. The localisation stage occurs in parallel for all users across all groups, requiring $O(T^\ast)$ operations. The central server then aggregates the results for each group, requiring $O(n(T^\ast)^{1/2}/d)$ operations per group, with this being carried out in series by the central serval for a total of $O(n(T^\ast)^{1/2})$ operations in total.

    For the refinement stage, each user across all groups again carries out $O(T^\ast)$ operations in parallel, with the central server aggregating the results of $O(n/d)$ users per group across $d$ groups. Hence, the overall computational complexity is $\max\{O(T^\ast), O(n(T^\ast)^{1/2})\}$.

    For the procedure when $\alpha > 1$ as in \Cref{app:sec:LowPrivlinfball}, we note that a user applies the univariate procedure to at most $(d \wedge \lfloor \alpha \rfloor)$-many of the co-ordinates individually, and the sample size of each group is $O((d \wedge \lfloor \alpha \rfloor)n/d)$, giving a complexity of $\max\{O((d \wedge \lfloor \alpha \rfloor)T^\ast), O((d \wedge \lfloor \alpha \rfloor)n(T^\ast)^{1/2})\}$.

    \medskip
    \noindent
    \textbf{$\ell_2$-Ball Procedure:}
    We first consider the procedure for $\alpha \leq 1$ as in \Cref{sec-multi-ell2}. The $\ell_2$-ball procedure is carried out as the $\ell_\infty$-ball procedure except with a) $O(d^2)$ operations carried out per user to transform their sample mean; b) $O(d^2)$ operations carried out by the central server when inverting the transformation at the final step, and; c) a total of $O((dT^\ast)^{1/2})$ sub-intervals are used as per the sub-interval width \eqref{sec3:eq:l2deltaval}. Hence, we focus on the time complexity of the localisation step in this case. Each user calculates their sample mean of $O(T^\ast)$ observations, and queries the membership of this mean out of $O((dT^\ast)^{1/2})$ sub-intervals. The data aggregator then sums the votes of $O(n/d)$ users for each $O((dT^\ast)^{1/2})$-dimensional vote vector. Hence, the total complexity of this stage after the central server aggregates the results for each group is $\max\{O(d^2), O(T^\ast), O(n(dT^\ast)^{1/2})\}$. The remaining stages follow as in the $\ell_\infty$-ball procedure, and so the total complexity is $\max\{O(d^2), O(T^\ast), O(n(dT^\ast)^{1/2})\}$

    For the procedure when $\alpha > 1$ as in \Cref{app:sec:LowPrivl2ball}, as with the $\ell_\infty$-ball procedure, each user considers at most $(d \wedge \lfloor \alpha \rfloor)$-many co-ordinates, and the sample size of each group is $O((d \wedge \lfloor \alpha \rfloor)n/d)$. We note that each user only transforms their sample mean once, regardless of how many co-ordinates they consider. Hence, the total complexity is $\max\{O(d^2), O((d \wedge \lfloor \alpha \rfloor)T^\ast), O((d \wedge \lfloor \alpha \rfloor)n(dT^\ast)^{1/2})\}$.

    \medskip
    \noindent
    \textbf{Variable Selection-Based Sparse Mean Procedure:}
    The variable selection-based procedure corresponds to that in \Cref{sec:sparse:variableselection}. Observe that the localisation step requires $O(d)$ computations per user where they verify the magnitude of each co-ordinate of the sample mean, on top of the $O(T^\ast)$ to compute the sample mean, and the server then aggregates the results of $O(n)$ users for each co-ordinate. Hence, the localisation step has complexity $\max\{O(T^\ast), O(nd)\}$. Combining with the complexity of the $s$-dimensional $\ell_\infty$-ball estimation procedure, the overall complexity is $\max\{O(T^\ast), O(nd), O(n(T^\ast)^{1/2})\}$.

    \medskip
    \noindent
    \textbf{Thresholding-Based Sparse Mean Procedure:}
    The thresholding-based procedure corresponds to that in \Cref{sec:sparse:thresholding}. We observe that it is simply the $\ell_\infty$-ball procedure (which has a complexity of $\max\{O(T^\ast), O(n(T^\ast)^{1/2})\}$), with a truncation operation applied to each of the $d$ co-ordinates at the final step. Hence, the complexity is $\max\{O(d), O(T^\ast), O(n(T^\ast)^{1/2}\}$. We remark that this procedure only outperforms the variable selection-based procedure when $n\alpha^2 \gtrsim d\log(nT\alpha^2)$, and the $O(d)$ term is dominated in this setting.

    \medskip
    \noindent
    \textbf{Non-Parametric Density Estimation Procedure:}
    Regarding the computational complexity of the procedure in \Cref{sec4:sec:procedure}, we observe that it is simply the $\ell_\infty$-ball procedure (which has a complexity of $\max\{O(T^\ast), O(n(T^\ast)^{1/2})\}$), applied after each user evaluates each of their $T^\ast$ observations on the functions $\{\phi_j\}_{j \in [M]}$ for a total of $O(MT^\ast)$ operations. Hence, the overall computation complexity is $\max\{O(MT^\ast), O(n(T^\ast)^{1/2})\}$ for $M$ as in \eqref{sec4:eq:Mval}
    
    \medskip
    \noindent
    \textbf{Mixture Model Mean Procedure:}
    For the procedure developed in \Cref{secmix:procedure}, each user requires $O(T^\ast)$ computations locally as in the other algorithms. When attaining an estimate for the $k$-th proportion $\pi_k$ or mean $\theta_k$ as in \eqref{secmix:eq:mixpropests} and \eqref{secmix:eq:FinalEstimatorComponent}, the server aggregates results from all $n$ users. Hence, the overall complexity is $\max\{O(T^\ast), O(mn)\}$ for $m$ the number of mixture components.

    \begin{table}[h]
        \centering
        \begin{tabular}{ll}
        \hline
        \hline
        Procedure                             & Complexity                                       \\ \hline
        $\ell_\infty$-ball Procedure     & $\max\{O(T^\ast), O(n(T^\ast)^{1/2})\}$          \\
        $\ell_2$-ball Procedure  & $\max\{O(d^2), O(T^\ast), O(n(dT^\ast)^{1/2})\}$ \\
        Sparse Mean Estimation (large $d$) & $\max\{O(T^\ast), O(nd), O(n(T^\ast)^{1/2})\}$   \\
        Sparse Mean Estimation (small $d$) & $\max\{O(d), O(T^\ast), O(n(T^\ast)^{1/2}\}$     \\
        Non-Parametric Density Estimation  & $\max\{O(MT^\ast), O(n(T^\ast)^{1/2})\}$         \\
        Mixture-Model Mean Estimation      & $\max\{O(T^\ast), O(mn)\}$                       \\ \hline
        \end{tabular}
        \caption{Computation complexities of developed procedures for $\alpha \leq 1$.}
    \end{table}

\section{Proof of User-Level Upper Bounds} \label{app:sec:Upper}
\subsection{Proof of Theorem \ref{sec3:thm:main} Upper Bound (Univariate)} \label{app:sec:Univariate}
\begin{proof}[Proof of \Cref{sec3:thm:main} $(d = 1)$]
    \
    \\
    \noindent \textbf{Step 1: Initial Estimator Interval $[\tilde{L}, \tilde{U}]$.}
    For the first step, we consider the cases $\alpha \in (0,1]$ and $\alpha > 1$ separately.

    \noindent \textbf{Step 1.1: Case $\alpha \in (0,1]$.}
    Recall that $n$ is assumed to be even.  We denote by $l \in [N(\Delta)]$ the fixed but unknown index such that the truth $\theta \in I_l$. Further, for each $i \in [n]$ and $t \in [T]$, the random variable $X^{(i)}_t \in [-1, 1]$. Hence, recalling the definition of $\hat{\theta}^{(i)}$ for $i \in [n/2]$ in \eqref{app:eq:unilocalisationestimator}, we see that 
    \begin{equation} \label{app:eq:UnivariateConcentration}
        \mathbb{P}\left(|\hat{\theta}^{(i)} - \theta|\geq \Delta \right) = \mathbb{P}\left\{\left| \frac{1}{T^\ast} \sum_{t = 1}^{T^\ast} X^{(i)}_t\right|\geq \Delta \right\} \leq 2e^{-T^\ast\Delta^2/2} = \frac{2}{nT^\ast\alpha^2} \leq \frac{1}{4}
    \end{equation}
    where the first inequality is due to Hoeffding's inequality for bounded random variables \citep[e.g.][Proposition 2.5]{Wainwright:2019}; the second equality by the definition of $\Delta$ from~\eqref{sec3:eq:deltaval}; and the last inequality by the assumption that $n\alpha^2 \geq \widetilde{C}$ for some sufficiently large $\widetilde{C}$ and $T^* \geq 1$.
    
    Denote $l'$ for the index of the closest sub-interval not containing the true mean, that is
    \begin{equation*}
        l' =
        \begin{cases}
            l - 1, &\text{when } \theta \in [-1 + 2(l-1)\Delta, -1 + 2(l-1)\Delta + \Delta) \\
            l + 1, &\text{when } \theta \in [-1 + 2(l-1)\Delta + \Delta, -1 + 2l\Delta) \\
        \end{cases}.
    \end{equation*}
    On the event $\{|\hat{\theta}^{(i)} - \theta| < \Delta\}$, it holds that $\hat{\theta}^{(i)} \in I_l \cup I_{l'}$ where we denote $I_0 = I_{N(\Delta) + 1} = \emptyset$. Consequently, we can bound the following probability
    \begin{equation*}
        \mathbb{P}(\{V_l^{(i)} = 0\} \cap \{V_{l'}^{(i)} = 0\}) \leq \mathbb{P}\left(|\hat{\theta}^{(i)} - \theta| \geq \Delta \right) \leq \frac{1}{4}.
    \end{equation*}
    Noting that $\mathbb{P}(\{V_l^{(i)} = 1\} \cap \{V_{l'}^{(i)} = 1\}) = 0$, we obtain
    \begin{align*}
        \mathbb{P}(V_l^{(i)} = 1) + \mathbb{P}(V_{l'}^{(i)} = 1)
        &\geq \mathbb{P}(\{V_l^{(i)} = 1\} \cap \{V_{l'}^{(i)} = 0\}) + \mathbb{P}(\{V_l^{(i)} = 0\} \cap \{V_{l'}^{(i)} = 1\}) \\
        &= \mathbb{P}(\{V_l^{(i)} = 1\} \cup \{V_{l'}^{(i)} = 1\}) \\
        &= 1 - \mathbb{P}(\{V_l^{(i)} = 0\} \cap \{V_{l'}^{(i)} = 0\}) > 3/4
    \end{align*}
    We then denote $l^\ast = \argmax_{a \in \{l, l'\}}\{\mathbb{P}(V_a^{(1)} = 1)\}$, such that $\mathbb{P}(V_{l^\ast}^{(i)} = 1) > 3/8$ for all $i \in [n/2]$.
    
    We also have the bound
    \begin{equation*}
        \mathbb{P}(V_j^{(i)} = 1) \leq \mathbb{P}\left(|\hat{\theta}^{(i)} - \theta| \geq \Delta \right) \leq \frac{1}{4}, \quad j \in[N(\Delta)] \setminus \{l, l'\},
    \end{equation*}
    following from \eqref{app:eq:UnivariateConcentration}. Then, for $k \in [N(\Delta)]$, we have that  
    \begin{align*}
        p_k := \mathbb{P}(\widetilde{V}_k^{(i)} = 1) = (2\omega_{\alpha/2} - 1)\mathbb{P}(V_k^{(i)} = 1) + (1 - \omega_{\alpha/2})         
    \end{align*}
    Hence, $p_{l^\ast} > (5 - 2\omega_{\alpha/2})/8$ and $p_j \leq (3-2\omega_{\alpha/2})/4$ for $j \in[N(\Delta)] \setminus \{l, l'\}$. 
    
    Letting $Z_j = \sum_{i=1}^{n/2} \widetilde{V}^{(i)}_j$ denote the total privatised votes for the $j$-th sub-interval, we consider the event
    \begin{equation} \label{app:eq:eventB}
        A_1 = \bigcap_{j \in [N(\Delta)] \setminus\{l, l'\}} \{Z_{l^\ast} > Z_j\}.
    \end{equation}   
    We note that
    \begin{equation*}
        Z_{l^\ast} - Z_j = \sum_{i = 1}^{n/2} \big(\widetilde{V}^{(i)}_{l^\ast} - \widetilde{V}^{(i)}_j\big),
    \end{equation*}
    where $\mathbb{E}[\widetilde{V}^{(i)}_{l^\ast} - \widetilde{V}^{(i)}_j] = p_{l^\ast} - p_j$. Hence, 
    \begin{align}
        \mathbb{P}( Z_{l^\ast} \leq Z_j) &= \mathbb{P}\left(\sum_{i = 1}^{n/2} \big(\widetilde{V}^{(i)}_{l^\ast} - \widetilde{V}^{(i)}_j\big) \leq 0\right) \nonumber \\
        &= \mathbb{P}\left(\sum_{i = 1}^{n/2} \big(\widetilde{V}^{(i)}_{l^\ast} - \widetilde{V}^{(i)}_j - p_{l^\ast} + p_j\big) \leq -n/2(p_{l^\ast} - p_j)\right) \nonumber \\
        &\leq  e^{-n(p_{l^\ast} - p_j)^2/4} \leq e^{-n(1/2 - \omega_{\alpha/2})^2/64} \leq e^{-2n\alpha^2/K}. \label{app:eq:voteprobbound}
    \end{align}
    for $K>0$ an absolute constant, where the first inequality is by inequality comes from Hoeffding's inequality, and the third inequality by the fact that $\alpha^2/K \leq (1/2 - \omega_{\alpha/2})^2$ holds for $\alpha \in (0,1]$ for $K$ sufficiently large.

    Hence, we have
    \begin{equation}
        \begin{split}
            \mathbb{P}(A_1^c) &\leq \sum_{j \in [N(\Delta)] \setminus\{l, l'\}} \mathbb{P}( Z_{l^\ast} \leq Z_j)
            \leq N(\Delta) e^{-2n\alpha^2/K}
            \leq \frac{3}{\Delta \wedge 1} e^{-2n\alpha^2/K}, \label{app:eq:probA} 
        \end{split}    
    \end{equation}
    where the first inequality is by the union bound; the second by \eqref{app:eq:voteprobbound}; and the last by bounding the covering number using \eqref{app:eq:l2cover}. We note that for the covering number bound we require $\Delta < 1$, so to account for when this fails to hold we take the minimum $\Delta \wedge 1$ which suffices for an upper bound.

    On the event $A_1$, we have that $j^\ast$, as defined in \eqref{sec3:eq:unichosenindex}, satisfies $j^\ast \in \{l , l'\}$ almost surely. As the distance between the parameter $\theta \in I_l$ and the closest endpoint of $I_{l'}$ is at most $\Delta$, expanding the endpoints of the chosen interval $I_{j^\ast}$ by $2\Delta$ as in \eqref{sec3:eq:uniMeanIntervalExpand} ensures that $\min\{ \theta - \tilde{L}, \tilde{U} - \theta\} \geq \Delta$, equivalently that $\theta \in [\tilde{L}+\Delta,\tilde{U}-\Delta]$, regardless of the value of $j^\ast$.

    \noindent \textbf{Step 1.2: Case $\alpha > 1$.}
    As before, denote by $l \in [N(\Delta)]$ the fixed but unknown index such that the truth $\theta \in I_l$, and for each $i \in [n]$ and $t \in [T]$, the random variable $X^{(i)}_t \in [-1, 1]$. Hence, recalling the definition of $\hat{\theta}^{(i)}$ for $i \in [n/2]$ in \eqref{app:eq:unilowprivunilocalisationestimator}, we see that 
    \begin{equation} \label{app:eq:UnivariateConcentrationlowpriv}
        \mathbb{P}\left(|\hat{\theta}^{(i)} - \theta|\geq \Delta \right) = \mathbb{P}\left\{\left|\frac{1}{T^*} \sum_{t = 1}^T X^{(i)}_t\right|\geq \Delta \right\} \leq 2e^{-T^\ast\Delta^2/2} = \frac{2}{(nT^\ast)^2}
    \end{equation}
    where the first inequality is due to Hoeffding's inequality for bounded random variables \citep[e.g.][Proposition 2.5]{Wainwright:2019}; the second equality by the definition of $\Delta$ from~\eqref{app:eq:unilowprivdeltaval2}.

    On the event $\{|\hat{\theta}^{(i)} - \theta| < \Delta\}$, it holds that $\hat{\theta}^{(i)} \in I_{l - 1} \cup I_l \cup I_{l + 1}$ where we denote $I_0 = I_{N(\Delta) + 1} = \emptyset$. Consequently, we can bound the following probability
    \begin{equation*}
        1 - p_l := \mathbb{P}(V_l^{(i)} = 0) \leq \mathbb{P}\left(|\hat{\theta}^{(i)} - \theta| \geq \Delta \right) \leq \frac{2}{(nT^\ast)^2}.
    \end{equation*}
    
    We also have the bound
    \begin{equation*}
        p_j := \mathbb{P}(V_j^{(i)} = 1) \leq \mathbb{P}\left(|\hat{\theta}^{(i)} - \theta| \geq \Delta \right) \leq \frac{2}{(nT^\ast)^2}, \quad j \in \{k \in [N(\Delta)] : |l - k| > 2\},
    \end{equation*}
    following from \eqref{app:eq:UnivariateConcentration}. Then, for $k \in [N(\Delta)]$, we have that  
    \begin{align*}
        \tilde{p}_k := \mathbb{P}(\widetilde{V}_k^{(i)} = 1) = (2\omega_{\alpha/6} - 1)p_k + (1 - \omega_{\alpha/6})         
    \end{align*}
    Hence, $\tilde{p}_l > \omega_{\alpha/6} - \varepsilon_{n, T}$ and $p_j \leq 1 - \omega_{\alpha/6} + \varepsilon_{n, T}$ for $j \in \{k \in [N(\Delta)] : |l - k| > 2\}$, where we denote $\varepsilon_{n, T} = 2(2\omega_{\alpha/6} - 1)/(nT^\ast)^2$
    
    Letting $Z_j = \sum_{i=1}^{n/2} \widetilde{V}^{(i)}_j$ denote the total privatised votes for the $j$-th sub-interval, where we recall $\widetilde{V}^{(i)}_j$ as in \eqref{sec3:eq:uniGRRGeneral2}, we consider the event
    \begin{equation} \label{app:eq:eventBlowpriv}
        A_2 = \bigcap_{j \in [N(\Delta)] : |l - j| > 2} \{Z_l > Z_j\}.
    \end{equation}
    We first simplify the problem by rewriting $\mathbb{P}( Z_l \leq Z_j)$ as the tail probability of a sum of i.i.d.~Bernoulli random variables. Indeed, letting $B_1, \hdots, B_n \overset{\mathrm{i.i.d.}}{\sim} \mathrm{Ber}(1 - \omega_{\alpha/6} + \varepsilon_{n, T})$, we have the bound
    \begin{equation*}
        \mathbb{P}( Z_l \leq Z_j)
        = \mathbb{P}( Z_l' + Z_j \geq n/2)
        = \mathbb{P}\left(\sum_{i = 1}^{n/2} \big(\Breve{V}^{(i)}_l + \widetilde{V}^{(i)}_j\big) \geq n/2 \right) 
        \leq \mathbb{P}\left(\sum_{i = 1}^{n} B_i \geq n/2 \right),
    \end{equation*}
    where the equalities follow from letting $Z_l' = \sum_{i = 1}^{n/2} \Breve{V}^{(i)}_l$ for $\Breve{V}^{(i)}_l = 1 - \widetilde{V}^{(i)}_l$; and the inequality from noting that $\Breve{V}^{(i)}_l, \widetilde{V}^{(i)}_j$ are independent $\mathrm{Ber}(1 - \tilde{p}_l)$ and $\mathrm{Ber}(\tilde{p}_j)$ and that the probability is increasing in $1 - \tilde{p}_l$ and $\tilde{p}_j$.

    Hence, we have
    \begin{align}
        \mathbb{P}( Z_l \leq Z_j)
        &= \mathbb{P}\left(\sum_{i = 1}^{n} B_i \geq n/2 \right)
        = \mathbb{P}\left(\frac{1}{n}\sum_{i = 1}^{n} (B_i - \mathbb{E}[B_i]) \geq \omega_{\alpha/6} - 1/2 - \varepsilon_{n, T})\right) \nonumber \\
        &\leq \{4(\omega_{\alpha/6} - \varepsilon_{n, T})(1 - \omega_{\alpha/6} + \varepsilon_{n, T})\}^{n} 
        = \{4\omega_{\alpha/6}(1 - \omega_{\alpha/6}) + 4\varepsilon_{n, T}(2\omega_{\alpha/6} - 1) - 4\varepsilon_{n, T}^2\}^{n} \nonumber \\
        &\leq \{4\omega_{\alpha/6}(1 - \omega_{\alpha/6}) + 4\varepsilon_{n, T}\}^{n} \nonumber \\
        &= \{4\omega_{\alpha/6}(1 - \omega_{\alpha/6})\}^{n}
        + 4^n \sum_{i = 1}^n \binom{n}{i} \{\omega_{\alpha/6}(1 - \omega_{\alpha/6})\}^{n-i} \varepsilon_{n, T}^i \nonumber \\
        &\leq \{4\omega_{\alpha/6}(1 - \omega_{\alpha/6})\}^{n}
        +  \sum_{i = 1}^n \binom{n}{i} 4^i \varepsilon_{n, T}^i
        \leq \{4\omega_{\alpha/6}(1 - \omega_{\alpha/6})\}^{n}
        +  \sum_{i = 1}^n \bigg(\frac{8n}{(nT^\ast)^2}\bigg)^i \nonumber \\
        &\leq e^{-n\alpha/20} + \frac{16}{n(T^\ast)^2}, \label{app:eq:voteprobboundlowpriv}
    \end{align}
    where the first inequality is by the Chernoff--Hoeffding bound \cite[Theorem~1][]{Hoeffding:1963}; the third inequality by the fact that $\omega_{\alpha/6}(1 - \omega_{\alpha/6}) \leq 1/4$; the fourth by the value of $\varepsilon_{n, T}$; and the final by \Cref{app:lem:sechbound} and the fact that the sum forms a geometric series with $8/\{n (T^\ast)^2\} \leq 1/2$, which holds as we may assume $nT^2 \geq 16$ without changing the final rate except up to constants.

    Hence, we have
    \begin{align}
        \mathbb{P}(A_2^c) &\leq \sum_{j \in [N(\Delta)] : |l - j| > 2} \mathbb{P}( Z_l \leq Z_j)
        \leq N(\Delta) \bigg(e^{-n\alpha/20} + \frac{16}{n(T^\ast\alpha^2)^2} \bigg) \nonumber \\
        &\leq \frac{3}{\Delta \wedge 1} e^{-2n\alpha/K} + \frac{48}{(\Delta \wedge 1)n(T^\ast)^2}
        \leq \frac{3}{\Delta \wedge 1} e^{-2n\alpha/K} + \frac{48}{n T^\ast}, \label{app:eq:probAlowpriv}   
    \end{align}
    where the constant $K > 0$ is as in \eqref{app:eq:probA} and taken sufficiently large; where the first inequality is by the union bound; the second by \eqref{app:eq:voteprobboundlowpriv}; the third by bounding the covering number using \eqref{app:eq:l2cover}; and the final by noting that $1/(\Delta \wedge 1) \leq (T^\ast)^{1/2}$. We note that for the covering number bound we require $\Delta < 1$, so to account for when this fails to hold we take the minimum $\Delta \wedge 1$ which suffices for an upper bound.

    On the event $A_2$, we have that $j^\ast$, as defined in \eqref{app:eq:unilowprivunichosenindex}, satisfies $j^\ast \in \{l - 2, l - 1, l, l + 1, l + 2\}$ almost surely. As the distance between the parameter $\theta \in I_l$ and the closest endpoint of either $I_{l-2}$ or $I_{l+2}$ is at most $4\Delta$, expanding the endpoints of the chosen interval $I_{j^\ast}$ by $5\Delta$ as in \eqref{app:eq:unilowprivuniMeanIntervalExpand} ensures that $\min\{ \theta - \tilde{L}, \tilde{U} - \theta\} \geq \Delta$, equivalently that $\theta \in [\tilde{L}+\Delta,\tilde{U}-\Delta]$, regardless of the value of $j^\ast$.

    \medskip
    \noindent\textbf{Step 2: Final Estimator.}
    It remains to consider the refined estimator using the remaining users $i \in [n] \setminus [n/2]$. In what follows, we will consider the cases $\alpha \in (0,1]$ and $\alpha > 1$ simultaneously. In particular, we note that $\Delta^2 \lesssim \log\{n T^\ast (\alpha^2 \wedge 1)\}/T^\ast$ for all $\alpha > 0$. Further, denoting by $A$ the event $A_1$ or $A_2$ in the cases $\alpha \in (0,1]$ and $\alpha > 1$ respectively, we see by combining \eqref{app:eq:probA} and \eqref{app:eq:probAlowpriv} that
    \begin{equation} \label{app:eq:probAallpriv}
        \mathbb{P}(A^c) \leq \frac{3}{\Delta \wedge 1} e^{-2n\min\{\alpha, \alpha^2\}/K} + \frac{48}{n T^\ast}.
    \end{equation}
    
    Recalling the private estimator $\hat{\theta}^{(i)}$ as defined in \eqref{sec3:eq:UniRefinedEstimator} and \eqref{app:eq:unilowprivUniRefinedEstimator} and writing $\overline{X^{(i)}} = (T^\ast)^{-1} \sum_{t = 1}^{T^\ast} X_t^{(i)}$ for $i \in [n]$, we have for any $i \in [n] \setminus [n/2]$ that
    \begin{align}
        &\mathbb{E}\left[ | \hat{\theta} -\theta |^2 \right]
        = \mathbb{E}\left[ \mathbb{E} \left\{ | \hat{\theta} -\theta |^2 \bigm| j^\ast \right\} \right]  = \mathbb{E} \left[ \left\{ \mathbb{E}(\hat{\theta} | j^\ast) - \theta \right\}^2 + \mathrm{Var}(\hat{\theta}|j^\ast) \right] \nonumber \\
        & = \mathbb{E}\left[ \left\{ \mathbb{E}(\hat{\theta}^{(i)} - \overline{X^{(i)}}  | j^\ast) \right\}^2 + \frac{2}{n}\mathrm{Var}(\hat{\theta}^{(i)} | j^\ast) \right] \nonumber \\
        &= \mathbb{E} \left[ \left\{ \mathbb{E}\left( (\hat{\theta}^{(i)} - \overline{X^{(i)}})\mathbbm{1}\{A\} \bigm| j^\ast \right) + \mathbb{E}\left((\hat{\theta}^{(i)} - \overline{X^{(i)}})\mathbbm{1}\{A^c\} \bigm| j^\ast \right) \right\}^2 + \frac{2}{n}\mathrm{Var}(\hat{\theta}^{(i)} | j^\ast ) \right] \nonumber \\
        &\leq 2 \mathbb{E} \left[\left\{ \mathbb{E}\left( (\hat{\theta}^{(i)} - \overline{X^{(i)}})\mathbbm{1}\{A\} \bigm| j^\ast \right) \right\}^2 + \left\{ \mathbb{E}\left( (\hat{\theta}^{(i)} - \overline{X^{(i)}})\mathbbm{1}\{A^c\} \bigm| j^\ast \right) \right\}^2 + \frac{1}{n}\mathrm{Var}(\hat{\theta}^{(i)} | j^\ast ) \right] \nonumber \\
        & = 2\mathbb{E}\{(I) + (II) + (III)\}, \label{app:eq:biasvariance}
    \end{align}
    where the first equality is by the tower law of expectation; the second is by the bias-variance decomposition; the third by the definition of $\hat{\theta}$, the fact that $\overline{X^{(i)}}$ is an unbiased estimator of $\theta$ and the fact that $\overline{X^{(i)}}$ is independent of $j^\ast$; and the inequality follows from the Cauchy--Schwarz inequality.
    
\textbf{Term $(III)$.} Letting the value $W_\alpha = 2\mathbbm{1}\{\alpha \in (0,1]\} + 5\mathbbm{1}\{\alpha > 1\}$, the variance term can be upper bounded as 
    \begin{equation} \label{app:eq:VarianceControl}
        \mathrm{Var}(\hat{\theta}^{(i)} | j^\ast) = \mathrm{Var}\left( \Pi_{\tilde{I}_{j^\ast}}(\overline{X^{(i)}}) \bigm| j^\ast \right) + \mathrm{Var}\left( \frac{2(1 + W_\alpha)\Delta}{\alpha} \ell_i \right) \leq 49\Delta^2 + \frac{8(1 + W_\alpha)^2\Delta^2}{\alpha^2} \leq \frac{337\Delta^2}{\alpha^2 \wedge 1},
    \end{equation}
    where the variance of the truncated sample mean is bounded as it is itself a bounded random variable taking values on an interval of width $2(1 + W_\alpha)\Delta$, and the last inequality uses the fact that $W_\alpha \leq 5$. 

\textbf{Term $(I)$.}   We have, following a similar argument as in \citet[][Appendix B.1]{Duchi:2018}, that
    \begin{align}
        \mathbb{E}&\left[\left| \overline{X^{(i)}} - \Pi_{\tilde{I}_{j^\ast}}(\overline{X^{(i)}}) \right| \mathbbm{1}\{A\} \Bigm| j^\ast \right] \nonumber \\
        &= \mathbb{E}\left[\left|\overline{X^{(i)}} - \tilde{U}\right| \mathbbm{1}\left\{\overline{X^{(i)}} > \tilde{U}\right\} \mathbbm{1}\{A\} \Bigm| j^\ast \right] + \mathbb{E}\left[ \left|\overline{X^{(i)}} - \tilde{L}\right| \mathbbm{1}\left\{\overline{X^{(i)}} < \tilde{L}\right\} \mathbbm{1}\{A\} \Bigm| j^\ast \right] \nonumber \\
        &\leq 2\mathbb{P}\left(\left\{ \overline{X^{(i)}} > \tilde{U} \right\} \cap A \Bigm| j^\ast  \right) + 2\mathbb{P}\left(\left\{ \overline{X^{(i)}} < \tilde{L} \right\} \cap A \Bigm| j^\ast \right), \label{app:eq:biascontrolprelim}
    \end{align}
    where the inequality uses the result that $\overline{X^{(i)}} - \tilde{U} \leq 2$ which follows from the facts $\overline{X^{(i)}} \in [-1, 1]$ and $\tilde{U} \geq -1+2\Delta+2\Delta = 4\Delta-1 \geq -1$ and similarly the result that $\overline{X^{(i)}} - \tilde{L} \geq -2$ which uses the fact that $\tilde{L} \leq 1$.
    
    The first term in \eqref{app:eq:biascontrolprelim} can be bounded as  
    \begin{align}
        \mathbb{P}\left(\left\{ \overline{X^{(i)}} > \tilde{U} \right\} \cap A \Bigm| j^\ast \right)
        &= \mathbb{P}\left( \left\{ \overline{X^{(i)}} - \theta > \tilde{U} - \theta \right\} \cap A \Bigm| j^\ast \right) \nonumber \\
        &\leq \mathbb{P}\left( \left\{ \overline{X^{(i)}} - \theta > \Delta \right\} \cap A \right) \nonumber \\
        &\leq \mathbb{P}\left(\overline{X^{(i)}} - \theta > \Delta \right)
        \leq e^{-T^\ast \Delta^2/2}, \label{app:eq:BiasControl}
    \end{align}
    where the first inequality is due to the facts that $\overline{X^{(i)}}$ is independent of $j^\ast$ and that, on the event $A$, $\tilde{U} - \theta > \Delta$; and the last inequality follows from Hoeffding's inequality. A similar argument gives the same inequality for the other term where we note that $\theta - \tilde{L} > \Delta$ on $A$. From this, we have that 
    \begin{equation}
        \left\{\mathbb{E}\left[(\hat{\theta}^{(i)} - \overline{X^{(i)}}) \mathbbm{1}\{A\}\Bigm| j^\ast\right]\right\}^2
        = \left\{\mathbb{E}\left[ \left(\Pi_{\tilde{I}_{j^\ast}}(\overline{X^{(i)}}) - \overline{X^{(i)}} \right) \mathbbm{1}\{A\} \Bigm| j^\ast \right]\right\}^2
        \leq 16e^{-T^\ast \Delta^2}, \label{app:eq:squaredbiasbound}
    \end{equation}
    where in the equality we use the fact that the Laplace noise variables $\ell_i$ are independent of all other randomness in the estimator with mean zero, and the inequality comes from combining \eqref{app:eq:biascontrolprelim} and \eqref{app:eq:BiasControl}.

\textbf{Term $(II)$.}  We have that
    \begin{align}
        \mathbb{E} \left[ \left\{ \mathbb{E}\left( (\hat{\theta}^{(i)} - \overline{X^{(i)}})\mathbbm{1}\{A^c\} \bigm| j^\ast \right) \right\}^2 \right] &= \mathbb{E} \left[ \left\{ \mathbb{E} \left( \left\{ \Pi_{\tilde{I}_{j^\ast}}(\overline{X^{(i)}}) - \overline{X^{(i)}} \right\}\mathbbm{1}\{A^c\} \bigm| j^\ast \right) \right\}^2 \right] \nonumber \\
        & \leq 4\mathbb{E} \left[ \mathbb{P}(A^c | j^\ast)^2 \right] \leq 4\mathbb{P}(A^c)
        \leq \frac{12}{\Delta \wedge 1} e^{-2n\min\{\alpha, \alpha^2\}/K} + \frac{192}{n T^\ast}  \label{app:eq:unicompbound}
    \end{align}
    where the equality follows from the fact that the Laplace noise variables $\ell_i$ are independent of all other randomness in the estimator and are mean zero; the first inequality comes from the fact that both  $\overline{X^{(i)}}$ and $\Pi_{\tilde{I}_{j^\ast}}(\overline{X^{(i)}})$ take values in $[-1,1]$; and the last by \eqref{app:eq:probAallpriv}.

    We combine \eqref{app:eq:biasvariance},~\eqref{app:eq:VarianceControl},~\eqref{app:eq:squaredbiasbound} and~\eqref{app:eq:unicompbound} to see that
    \begin{align*}
        \mathbb{E}\left[ | \hat{\theta} -\theta |^2 \right]
        &\leq 32 e^{-T^\ast\Delta^2}
        + \frac{24}{\Delta \wedge 1} e^{-2n\min\{\alpha, \alpha^2\}/K} + \frac{384}{n T^\ast}
        + \frac{674\Delta^2}{n(\alpha^2 \wedge 1)} \\
        &\lesssim \frac{1}{nT^\ast(\alpha^2 \wedge 1)}
        + \left( \frac{(T^\ast)^{1/2}}{[\log\{n T^\ast (\alpha^2 \wedge 1)\}]^{1/2}} \vee 1 \right) e^{-2n\min\{\alpha, \alpha^2\}/K}
        + \frac{\log\{nT^\ast(\alpha^2 \wedge 1)\}}{nT^\ast(\alpha^2 \wedge 1)}, \\
        &\lesssim \frac{ \log\{nT^\ast(\alpha^2 \wedge 1)\}}{nT^\ast(\alpha^2 \wedge 1)} + \frac{(T^\ast)^{1/2}}{[\log\{n T^\ast (\alpha^2 \wedge 1)\}]^{1/2}}e^{-2n\min\{\alpha, \alpha^2\}/K} + e^{-2n\min\{\alpha, \alpha^2\}/K}, \\
        &\lesssim \frac{ \log\{nT^\ast(\alpha^2 \wedge 1)\}}{nT^\ast(\alpha^2 \wedge 1)},
    \end{align*}
    where the second inequality comes from substituting in the value of $\Delta$ as in \eqref{sec3:eq:deltaval} and \eqref{app:eq:unilowprivdeltaval2} for the cases $\alpha \in (0,1]$ and $\alpha > 1$ respectively; the third by bounding the maximum term by the sum of the two terms; and the fourth using the fact that
    \begin{equation*}
        e^{-2n\min\{\alpha, \alpha^2\}/K}
        \leq (T^\ast)^{-3/2} e^{-(n/2)\min\{\alpha, \alpha^2\}/K}
        \lesssim 1/\{(T^\ast)^{3/2} n \min\{\alpha, \alpha^2\}\}
        \leq 1/\{(T^\ast)^{3/2} n \min\{\alpha^2, 1\}\},
    \end{equation*}
    and controlling the log term in the denominator using the assumption that $n\min\{\alpha, \alpha^2\} > \widetilde{C}$ for some sufficiently large $\widetilde{C}$. 
    
    In the case $T \leq e^{n\min\{\alpha, \alpha^2\}/K}$, we have $T^\ast = T$ which immediately gives the upper bound in \eqref{sec3:eq:linfMeanStatement}. In the case $T > e^{n\min\{\alpha, \alpha^2\}/K}$, we have $T^\ast =  e^{n\min\{\alpha, \alpha^2\}/K}$, which gives
    \begin{equation*}
        \frac{\log\{n T^\ast (\alpha^2 \wedge 1)\}}{n T^\ast (\alpha^2 \wedge 1)}
        =  \left( \frac{\log\{n (\alpha^2 \wedge 1)\}}{n(\alpha^2 \wedge 1)} + \frac{\log(T^\ast)}{n(\alpha^2 \wedge 1)} \right) \frac{1}{T^\ast}
        \lesssim (\alpha \vee 1) e^{-n\min\{\alpha, \alpha^2\}/K}
        \lesssim e^{-n\min\{\alpha, \alpha^2\}/(2K)},
    \end{equation*}
    where the inequality comes from the assumption $n\min\{\alpha, \alpha^2\} \geq \widetilde{C}$, and from substituting in the value $e^{n\min\{\alpha, \alpha^2\}/K}$ for $T^\ast$. Putting together the pieces, we have
    \begin{equation}
        \mathbb{E}\left[ | \hat{\theta} -\theta |^2 \right] \lesssim
        \begin{cases}
            \frac{ \log\{nT(\alpha^2 \wedge 1)\}}{nT(\alpha^2 \wedge 1)} &\text{ when } T \leq e^{n\min\{\alpha, \alpha^2\}/K}, \\
            e^{-n\min\{\alpha, \alpha^2\}/(2K)} &\text{otherwise.}
        \end{cases}
        \label{app:eq:unimeanMSE}
    \end{equation}
    Lastly, bounding the error by the sum of the errors in the two cases for $T$ completes the proof.
\end{proof}

\subsection{Proof of Theorem \ref{sec3:thm:main} Upper Bound (multivariate)}
    \begin{proof}[Proof of \Cref{sec3:thm:main} $(\ell_\infty \text{-ball case})$]
        We consider three cases separately depending on the value of $\alpha$.

        \medskip
        \noindent \textbf{Case 1: $\alpha \in (0, 1]$} 

        When $\alpha \in (0,1]$, the estimation procedure in \Cref{sec3:linfball} is equivalent to splitting the sample and applying the univariate procedure, defined in Section \ref{sec3:univariate}, to each co-ordinate in turn. For $j \in [d]$, the collection of users indexed by the set $N_j$ in \eqref{sec3:eq:linfGroups}, uses the $j$-th co-ordinate of their data to estimate the $j$-th co-ordinate of the mean. We note that the assumption that $n\alpha^2 > \widetilde{C}d\log(ed)$ ensures that the required assumption for the univariate sub-problems is satisfied.
        
        The estimator for the $j$-th co-ordinate $\hat{\theta}_j$ has error bounded as 
        \begin{equation}
            \mathbb{E}\left[|\hat{\theta}_j - \theta_j|^2 \right]
            \lesssim
                \begin{cases}
                    \frac{d \log(nT\alpha^2/d)}{nT\alpha^2}, & T \leq e^{n\alpha^2/(K d)}, \\
                    e^{-n\alpha^2/(2K d)}, &\text{otherwise,}
                \end{cases}
                \label{app:eq:linfComponentMSE}
        \end{equation}
        which follows from the univariate error in \eqref{app:eq:unimeanMSE} for $\alpha \in (0,1]$, noting the factors of $d$ arising from the number of users in each collection being such that $|N_j| = n/d$ for $j \in [d]$. For the final estimator $\hat{\theta}$ as defined in \eqref{sec3:eq:linfFinalEstimator}, the error can be bounded as
        \begin{equation}
            \mathbb{E}\left[\|\hat{\theta} - \theta\|_2^2 \right]
            = \sum_{j=1}^d \mathbb{E}\left[|\hat{\theta}_j - \theta_j|_2^2 \right]
            \lesssim
            \begin{cases}
                \frac{d^2 \log(nT\alpha^2/d)}{nT\alpha^2}, & T \leq e^{n\alpha^2/(K d)}, \\
                de^{-n\alpha^2/(2K d)}, &\text{otherwise,}
            \end{cases} \label{app:eq:linfMSE}
        \end{equation}
        where the equality comes from the decomposition of multivariate mean-squared-error and the inequality by summing up $d$-many copies of \eqref{app:eq:linfComponentMSE}. Lastly, we have that
        \begin{equation}
            \mathbb{E}\left[\|\hat{\theta} - \theta\|_2^2 \right]
            \lesssim 
            \frac{d^2 \log(nT\alpha^2/d)}{nT\alpha^2} + e^{-cn\alpha^2/(K d)}, \label{app:eq:linffinalMSE}
        \end{equation}
        which comes from bounding the error by the sum of the errors in the two cases for $T$ and using the assumption that $n\alpha^2 > \widetilde{C}d\log(ed)$ for some sufficiently large $\widetilde{C}$ to absorb the factor of $d$ into the exponential term.

        \medskip
        \noindent \textbf{Case 2: $1 < \alpha \leq d$} 

        For simplicity we assume $\alpha$ is an integer, noting that this will not change the minimax rates obtained except up to constants. When $1 < \alpha \leq d$, the estimation procedure in \Cref{app:sec:LowPrivlinfball} is equivalent to applying the univariate procedure with $\alpha$ set to $1$, with each user contributing to $\alpha$-many co-ordinates. Hence, each fold has a sample size of $n\alpha/d$, and the estimator for the $j$-th co-ordinate $\hat{\theta}_j$ has error bounded as 
        \begin{equation}
            \mathbb{E}\left[|\hat{\theta}_j - \theta_j|^2 \right]
            \lesssim
                \begin{cases}
                    \frac{d \log(nT\alpha/d)}{nT\alpha}, & T \leq e^{n\alpha/(K d)}, \\
                    e^{-n\alpha/(2K d)}, &\text{otherwise,}
                \end{cases}
                \label{app:eq:linfComponentMSE2}
        \end{equation}
        which follows from the univariate error in \eqref{app:eq:unimeanMSE} with $\alpha = 1$ and sample size $n\alpha/d$ therein. For the final estimator $\hat{\theta}$ as defined in \eqref{app:eq:LowPrivlinfFinalEstimator}, the error can be bounded as
        \begin{equation}
            \mathbb{E}\left[\|\hat{\theta} - \theta\|_2^2 \right]
            = \sum_{j=1}^d \mathbb{E}\left[|\hat{\theta}_j - \theta_j|_2^2 \right]
            \lesssim
            \begin{cases}
                \frac{d^2 \log(nT\alpha/d)}{nT\alpha}, & T \leq e^{n\alpha/(K d)}, \\
                de^{-n\alpha/(2K d)}, &\text{otherwise,}
            \end{cases} \label{app:eq:linfMSE2}
        \end{equation}
        where the equality comes from the decomposition of multivariate mean-squared-error and the inequality by summing up $d$-many copies of \eqref{app:eq:linfComponentMSE2}. Lastly, we have that
        \begin{equation}
            \mathbb{E}\left[\|\hat{\theta} - \theta\|_2^2 \right]
            \lesssim 
            \frac{d^2 \log(nT\alpha/d)}{nT\alpha} + e^{-cn\alpha/(K d)}, \label{app:eq:linffinalMSE2}
        \end{equation}
        which comes from bounding the error by the sum of the errors in the two cases for $T$ and using the assumption that $n\alpha > \widetilde{C}d\log(ed)$ for some sufficiently large $\widetilde{C}$ to absorb the factor of $d$ into the exponential term.

        \medskip
        \noindent \textbf{Case 3: $\alpha > d$} 

        For $\alpha > d$, the estimation procedure in \Cref{app:sec:LowPrivlinfball} is equivalent to applying the univariate procedure to each co-ordinate with a privacy parameter of $\alpha/d > 1$. Hence, each fold has a sample size of $n$, and the estimator for the $j$-th co-ordinate $\hat{\theta}_j$ has error bounded as 
        \begin{equation}
            \mathbb{E}\left[|\hat{\theta}_j - \theta_j|^2 \right]
            \lesssim
                \begin{cases}
                    \frac{\log(nT)}{nT}, & T \leq e^{n\alpha/(K d)}, \\
                    e^{-n\alpha/(2K d)}, &\text{otherwise,}
                \end{cases}
                \label{app:eq:linfComponentMSE3}
        \end{equation}
        which follows from the univariate error in \eqref{app:eq:unimeanMSE} with privacy parameter $\alpha/d > 1$. For the final estimator $\hat{\theta}$ as defined in \eqref{app:eq:LowPrivlinfFinalEstimator}, the error can be bounded as
        \begin{equation}
            \mathbb{E}\left[\|\hat{\theta} - \theta\|_2^2 \right]
            = \sum_{j=1}^d \mathbb{E}\left[|\hat{\theta}_j - \theta_j|_2^2 \right]
            \lesssim
            \begin{cases}
                \frac{d \log(nT)}{nT}, & T \leq e^{n\alpha/(K d)}, \\
                de^{-n\alpha/(2K d)}, &\text{otherwise,}
            \end{cases} \label{app:eq:linfMSE3}
        \end{equation}
        where the equality comes from the decomposition of multivariate mean-squared-error and the inequality by summing up $d$-many copies of \eqref{app:eq:linfComponentMSE3}. Lastly, we have that
        \begin{equation}
            \mathbb{E}\left[\|\hat{\theta} - \theta\|_2^2 \right]
            \lesssim 
            \frac{d \log(nT)}{nT} + e^{-cn\alpha/(K d)}, \label{app:eq:linffinalMSE3}
        \end{equation}
        which comes from bounding the error by the sum of the errors in the two cases for $T$ and using the assumption that $n\alpha > \widetilde{C}d\log(ed)$ for some sufficiently large $\widetilde{C}$ to absorb the factor of $d$ into the exponential term.
    \end{proof}

    \begin{proof}[Proof of \Cref{sec3:thm:main} $(\ell_2 \text{-ball case})$]
        Recall that we assume without loss of generality that the dimension $d$ is a power of $2$, noting that we can always append zero values to the data until this holds, with this procedure not affecting the final minimax rate except up to constants. We then note that we can always take the origin of the $\ell_2$-ball as a trivial estimator. We upper bound the error of this trivial estimator by the radius of the ball. As a consequence, the upper bound on the risk consists of the minimum between the constant $1$, and the error of the estimators we propose in \Cref{sec-multi-ell2} and \Cref{app:sec:LowPrivl2ball}. In the rest of the proof, we focus on our proposed estimators. 
        
        The proof for the $\ell_2$-ball case also involves considering each co-ordinate separately and applying a univariate estimation procedure to each co-ordinate. However, due to the application of the random rotation of \Cref{sec3:lem:RandomRotation}, the steps are not identical to the $\ell_{\infty}$ case. We first focus on the univariate sub-problem for general $\alpha \in (0, \infty)$, with a sample size per co-ordinate of $n/d$. 

        \medskip
        \noindent \textbf{Step 1: Random Rotations.}
        We are to show a user's estimator deviates from the truth by $\Delta$ with small probability.
        
        As $X^{(i)}_t \in \mathbb{B}_2(1)$ for all $i \in [n]$ and $t \in [T]$, we have that $\| X^{(i)}_t - \theta \|_2 \in [0,2]$. Hence, we have by \citet[Lemma~1]{Jin:2019} that for some absolute constant $c_1 > 0$ and for all $\varepsilon > 0$, 
        \begin{align*}
            \mathbb{P}(\| X^{(i)}_t - \theta \|_2 \geq \varepsilon) \leq 2e^{-\varepsilon^2/(2c_1)}.
        \end{align*}
        Denoting the transformed sample mean via $\hat{\theta}^{(i)} = R_d \{(T^\ast)^{-1}\sum_{t = 1}^{T^\ast} X_t^{(i)}\}$, for each $i \in [n/2]$ we have by \Cref{app:lem:SGnorm} that there exists a universal constant $c_2>0$ such that, for any $\gamma > 0$,
        \begin{equation}
            \|\hat{\theta}^{(i)} - \theta\|_2 \leq c_2 \left
            \{\frac{ \log(2d/\gamma)}{T^\ast} \right\}^{1/2} \label{app:eq:NormConcentrationSingle}
        \end{equation}
        holds with probability at least $1 - \gamma$. We define the event
        \begin{equation*}
            A_i = \left\{\|\hat{\theta}^{(i)} -\theta\|_2 \leq c_2\left\{\frac{2\log\{nT^\ast(\alpha^2 \wedge 1)\}}{T^\ast}\right\}^{1/2} \right\}
        \end{equation*}
        from which we have that
        \begin{equation}
            \mathbb{P}(A_i^c) \leq \frac{2d}{\{nT^\ast(\alpha^2 \wedge 1)\}^2} \label{app:eq:l2auxconcevent}
        \end{equation}
        due to \eqref{app:eq:NormConcentrationSingle} with $\gamma = 2d/\{nT^\ast(\alpha^2 \wedge 1)\}^2$ therein.

        We thus have
        \begin{align}
            &\mathbb{P}\left( |(R_d\hat{\theta}^{(i)})_j - (R_d\theta)_j| \geq \Delta \right) \\
            &\leq \mathbb{P}\left( \left\{|(R_d\hat{\theta}^{(i)})_j - (R_d\theta)_j| \geq \Delta \right\} \cap A_i \right) + \mathbb{P}(A_i^c) \nonumber \\
            & \leq \mathbb{P}\left( \left\{|(R_d\hat{\theta}^{(i)})_j - (R_d\theta)_j| \geq \left(\frac{C}{dT^\ast}\right)^{1/2}\log\{nT^\ast(\alpha^2 \wedge 1)\} \right\} \cap A_i \right) + \frac{2d}{\{nT^\ast(\alpha^2 \wedge 1)\}^2)} \nonumber \\
            & \leq \mathbb{P}\left( \left\{|(R_d\hat{\theta}^{(i)})_j - (R_d\theta)_j| \geq \|\hat{\theta}^{(i)}-\theta\|_2 \left(\frac{C\log\{nT^\ast(\alpha^2 \wedge 1)\}}{c_2^2d}\right)^{1/2} \right\} \cap A_i \right) + \frac{2d}{\{nT^\ast(\alpha^2 \wedge 1)\}^2} \nonumber \\
            & \leq \frac{1}{\{nT^\ast(\alpha^2 \wedge 1)\}^2} + \frac{2d}{\{nT^\ast(\alpha^2 \wedge 1)\}^2}
            \leq
            \begin{cases}
                1/4 &\text{ when } \alpha \in (0,1] \\
                3d/(nT^\ast)^2 &\text{ when } \alpha > 1
            \end{cases}, \label{app:eq:rotatedconc}
        \end{align}
        where the second inequality uses \eqref{app:eq:l2auxconcevent} and substitutes in the value of $\Delta$ as in \eqref{sec3:eq:l2deltaval} and \Cref{alg:l2lowpriv} in the cases $\alpha \in (0,1]$ and $\alpha > 1$ respectively; the third uses the definition of the event $A_i$; the penultimate inequality follows from applying \Cref{sec3:lem:RandomRotation} with $\gamma = 1/\{nT^\ast(\alpha^2 \wedge 1)\}$ therein, provided $C \geq 200c_2^2$; and the final, for the case $\alpha \in (0,1]$, using the assumption that $n\alpha^2 \geq \widetilde{C}d\log(ed)$ for $\widetilde{C}$ sufficiently large and that $T \geq 1$.

    \medskip 
    \noindent\textbf{Step 2: Co-ordinate-Wise Treatment.}
        We now focus on a single fixed co-ordinate $j \in [d]$ and follow the same steps as in the proof of the univariate case with some modifications to account for the random rotation. We first note that as we are working with the transformed data, we have, for a fixed $j \in [d]$, $i \in N_{j,1}, k \in [N(\delta)]$, that $V^{(i)}_k = \mathbbm{1}\{ (R_d \hat{\theta}^{(i)})_j \in I_{k-1} \cup I_k \cup I_{k+1} \}$ with $I_0 = I_{N(\delta) + 1} = \emptyset$ and $R_d$ as in \eqref{eq-rotation-matrix}. The privatized value $\widetilde{V}^{(i)}_k$ is then generated via the same randomized response mechanism as in \eqref{sec3:eq:uniGRR}. We then recall that for $j \in [d]$ and $k \in [N(\Delta)]$, we let $Z_{j,k} = \sum_{i \in N_{j,1}} \widetilde{V}^{(i)}_k$ denote the total number of privatised votes from the $j$-th group of users for the $k$-th sub-interval and denote by $l_j \in [N(\Delta)]$ the fixed but unknown index such that $(R_d\theta)_j \in I_{l_j}$. We consider the cases $\alpha \in (0,1]$ and $\alpha > 1$ separately. 

        \noindent \textbf{Step 2.1: Case $\alpha \in (0,1]$.}
    
        Define $l_j$ for the closest sub-interval not containing the true mean of the $j$-th transformed co-ordinate, that is,
        \begin{equation*}
            l_j' =
            \begin{cases}
                l_j - 1, &\text{when } (R_d\theta)_j \in [-1 + 2(l_j-1)\Delta, -1 + 2(l_j-1)\Delta + \Delta) \\
                l_j + 1, &\text{when } (R_d\theta)_j \in [-1 + 2(l_j-1)\Delta + \Delta, -1 + 2l_j\Delta) \\
            \end{cases}.
        \end{equation*}
        We consider the event 
        \begin{equation}
            B_{1, j} = \bigcap_{k \in [N(\Delta)]\setminus\{l_j, l_j'\}} \{Z_{j,l_j} > Z_{j,k}\}, \label{app:eq:l2probBjlowpriv}
        \end{equation}
        which is analogous to the event \eqref{app:eq:eventB} from the univariate case. On the event $B_j$, denoting by $k_j^\ast$ the chosen sub-interval analogously to \eqref{sec3:eq:unichosenindex}, whichever interval is selected is sufficiently close to $(R_d\theta)_j$ so that $(R_d\theta)_j \in \tilde{I}_{k_j^\ast}$ and further, due to inflating the end points of the intervals, $\min\{(R_d\theta)_j - \widetilde{L}_j, \widetilde{U}_j - (R_d\theta)_j\} > \Delta$. By \eqref{app:eq:rotatedconc}, we have that the analogous concentration inequality to \eqref{app:eq:UnivariateConcentration} holds for the data rotated by $R_d$. Hence, bounding the probability of the complement $B_j^c$ by the same analysis as that in \eqref{app:eq:voteprobbound} and \eqref{app:eq:probA}, we have
        \begin{equation}
            \mathbb{P}(B_{1, j}^c)
            \leq \frac{3}{\Delta \wedge 1} e^{-2n\alpha^2/(K d)} \label{app:eq:l2probBj}
        \end{equation}
        where $K$ is as appears in \eqref{app:eq:voteprobbound} and, as in \eqref{app:eq:probA}, we require $\Delta < 1$, so to account for when this fails to hold we take the minimum $\Delta \wedge 1$ which suffices for an upper bound. 
        
        \noindent \textbf{Step 2.2: Case $\alpha > 1$.}
        We consider the event 
        \begin{equation}
            B_{2, j} = \bigcap_{k \in [N(\Delta)]: |k - l| > 2} \{Z_{j,l_j} > Z_{j,k}\}, \label{app:eq:l2probCj}
        \end{equation}
        We similarly proceed to bound the probability of the event $B_j^c$ in the case $\alpha > 1$. Using the bound \eqref{app:eq:rotatedconc} in place of the analogous \eqref{app:eq:UnivariateConcentrationlowpriv} from the univariate setting, and following the same argument as \eqref{app:eq:eventBlowpriv}, \eqref{app:eq:voteprobboundlowpriv} and \eqref{app:eq:probAlowpriv}, we obtain for $k \in [N(\Delta)] \setminus \{l_j - 2, l_j - 1, l_j, l_j + 1, l_j + 2\}$,
        \begin{align}
            \mathbb{P}( Z_{j, l_j} \leq Z_{j, k})
            \leq \{4\omega_{\alpha/6}(1 - \omega_{\alpha/6})\}^{n/d}
            +  \sum_{i \in N_{j, 1}} \bigg(\frac{12nd}{(nT^\ast)^2}\bigg)^i
            \leq e^{-n\alpha/(20d)} + \frac{24d}{n(T^\ast)^2},
        \end{align}
        where the second inequality comes from noting the geometric series with the fact $12d/\{n (T^\ast)^2\} \leq 1/2$, which holds as we have $n \gtrsim d$.
        
        Hence, similarly to \eqref{app:eq:probAlowpriv}, we obtain
        \begin{align}
            \mathbb{P}(B_{2,j}^c) &\leq \sum_{j \in [N(\Delta)] : |l - j| > 2} \mathbb{P}( Z_l \leq Z_j)
            \leq N(\Delta) \bigg(e^{-n\alpha/(20d)} + \frac{24d}{n(T^\ast)^2} \bigg) \nonumber \\
            &\leq \frac{3}{\Delta \wedge 1} e^{-2n\alpha/(Kd)} + \frac{72}{(\Delta \wedge 1)n(T^\ast)^2}
            \leq \frac{3}{\Delta \wedge 1} e^{-2n\alpha/(Kd)} + \frac{72}{n T^\ast}, \label{app:eq:probBjlowpriv}   
        \end{align}

        Finally, in preparation for the next step, writing $\overline{X^{(i)}} = (T^\ast)^{-1} \sum_{t = 1}^{T^\ast} X_t^{(i)}$ for $i \in [n]$, we define the final events
        \begin{align*}
            D_{i,j} = \left\{ |(R_d\overline{X^{(i)}})_j - (R_d\theta)_j| \leq \frac{10 \|\overline{X^{(i)}} - \theta\|_2 [\log\{nT^\ast(\alpha^2 \wedge 1)\}]^{1/2}}{d^{1/2}} \right\}.
        \end{align*}
        By applying \Cref{sec3:lem:RandomRotation} with $\gamma = 1/\{nT^\ast(\alpha^2 \wedge 1)\}$ therein we see that we have
        \begin{equation}
            \mathbb{P}(D_{i,j}^c) < 1/\{nT^\ast(\alpha^2 \wedge 1)\}. \label{app:eq:l2probDj}
        \end{equation}

        \medskip
        \noindent \textbf{Step 3: Controlling Estimator Error.}
        For $i \in N_{j,2}$, we write $\tilde{\theta}^{(i)} =  \Pi_{\tilde{I}_{k_j^\ast}}\{  (T^\ast)^{-1}\sum_{t = 1}^{T^\ast} ( R_dX_{t}^{(i)} )_j \} + 2(1+W_\alpha)\Delta\ell_i/\alpha$, where the $\ell_i$ are i.i.d.~standard Laplace random variables and we recall $W_\alpha = 2\mathbbm{1}\{\alpha \in (0,1]\} + 5\mathbbm{1}\{\alpha > 1\}$. For the refined estimator for the $j$-th co-ordinate, we write $\tilde{\theta}_{j} = (2d/n)\sum_{ i \in N_{j,2}} \tilde{\theta}^{(i)} \in \mathbb{R}$.
        
        Recalling the final estimator as in \eqref{sec3:eq:Finall2Estimator} and \eqref{app:eq:LowPrivFinall2Estimator}, we have that
        \begin{equation}
            \mathbb{E}\left[\|\hat{\theta} - \theta \|_2^2\right] = \mathbb{E}\left[\| R_d\hat{\theta} - R_d\theta \|_2^2\right] = \sum_{j = 1}^d \mathbb{E}\left[|\tilde{\theta}_j - (R_d\theta)_j|^2\right]\label{app:eq:OrthogonalFinalEstimator}
        \end{equation}
        where the first equality holds as the rotation $R_d$ is an orthogonal transformation. Then, focusing on a single co-ordinate, we have for all $j \in [d]$ and any $i \in N_{j,2}$, and denoting by $B_j$ the event $B_{1, j}$ or $B_{2, j}$ in the cases $\alpha \in (0,1]$ and $\alpha > 1$ respectively, we see that
        \begin{align}
            \mathbb{E}&\left[|\tilde{\theta}_j - (R_d\theta)_j|^2 \right] \nonumber \\
            &\leq 2\mathbb{E}\biggl[ \mathbb{E}\left\{ \{\tilde{\theta}^{(i)} - (R_d\overline{X^{(i)}})_j\}\mathbbm{1}\{B_j\} \Bigm| k_j^\ast, R_d \right\}^2 + \mathbb{E}\left\{ \{\tilde{\theta}^{(i)} - (R_d\overline{X^{(i)}})_j\}\mathbbm{1}\{(B_j)^c\} \Bigm| k_j^\ast, R_d \right\}^2 \nonumber  \\
            & \hspace{250pt} + \frac{d}{n}\mathrm{Var}(\tilde{\theta}^{(i)} | k_j^\ast, R_d) \biggr] \nonumber \\
            & = 2\mathbb{E}\{(I) + (II) + (III)\}, \label{app:eq:l2co-ordinatebiasvariance}
        \end{align}
        where the inequality is by a similar decomposition to that in \eqref{app:eq:biasvariance}.

    \medskip
    \noindent \textbf{Step 3.1: Term $(III)$.} The variance term can easily be bounded as 
        \begin{align}
            \mathrm{Var}(\tilde{\theta}^{(i)}| k_j^\ast, R_d)
            &= \mathrm{Var}\left( \Pi_{\tilde{I}_{k_j^\ast}}\{(R_d\overline{X^{(i)}})_j \} \biggm| k_j^\ast, R_d \right) + \mathrm{Var}\left( \frac{2(1 + W_\alpha)\Delta}{\alpha} \ell_i \right) \nonumber \\
            &\leq 49\Delta^2 + \frac{8(1 + W_\alpha)^2\Delta^2}{\alpha^2} \leq \frac{337\Delta^2}{\alpha^2 \wedge 1}, \label{app:eq:l2VarianceControl}
        \end{align}
        where the variance of the truncated sample mean is bounded as it is itself a bounded random variable taking values on an interval of width $2(1 + W_\alpha)$ recalling the value $W_\alpha = 2\mathbbm{1}\{\alpha \in (0,1]\} + 5\mathbbm{1}\{\alpha > 1\}$, and the last inequality uses the fact $W_\alpha \leq 5$
        
       \medskip
    \noindent \textbf{Step 3.2: Term $(I)$.}  Following the same arguments as in \eqref{app:eq:biascontrolprelim} and \eqref{app:eq:BiasControl}, we have that
        \begin{align}
        \label{Eq:NewBias1}
            \mathbb{E}&\left[\left| (R_d\overline{X^{(i)}})_j - \Pi_{\tilde{I}_{k_j^\ast}}\{(R_d\overline{X^{(i)}})_j\} \right| \mathbbm{1}\{B_j \} \Bigm| k_j^\ast, R_d \right] \nonumber \\
            &\leq 2\mathbb{P}\left(\left\{ (R_d\overline{X^{(i)}})_j > \tilde{U}_j \right\} \cap B_j \Bigm| k_j^\ast, R_d \right) + 2\mathbb{P}\left(\left\{ (R_d\overline{X^{(i)}})_j < \tilde{L}_j \right\} \cap B_j \Bigm| k_j^\ast, R_d \right) \nonumber \\ 
            &\leq 2\mathbb{P}\left( \left\{ \bigl| (R_d\overline{X^{(i)}})_j - (R_d\theta)_j \bigr| > \Delta \right\} \cap B_j \Bigm| k_j^\ast, R_d \right),
        \end{align}
        where the second inequality is due to the fact that on the event $B_j$, $\tilde{U}_j - (R_d\theta)_j > \Delta$ and $(R_d\theta)_j - \tilde{L}_j > \Delta$. We upper bound the contribution of \eqref{Eq:NewBias1} in the event $D_{i,j}$, that 
        \begin{align}
        \label{Eq:NewBias2}
            \mathbb{E} &\left[ \mathbb{P}\left( \left\{ \bigl| (R_d\overline{X^{(i)}})_j - (R_d\theta)_j \bigr| > \Delta \right\} \cap B_j \Bigm| k_j^\ast, R_d \right)^2 \right]
            \leq \mathbb{P}\left( \left\{ \bigl| (R_d\overline{X^{(i)}})_j - (R_d\theta)_j \bigr| > \Delta \right\} \cap B_j \right) \nonumber  \\
            & \leq \mathbb{P}\left( \left\{ \bigl| (R_d\overline{X^{(i)}})_j - (R_d\theta)_j \bigr| > \Delta \right\} \cap B_j \cap D_{i,j} \right) + \mathbb{P}(D_{i,j}^c) \nonumber  \\
            & \leq \mathbb{P} \left( \left\{\|\overline{X^{(i)}} - \theta\|_2 \geq \frac{d^{1/2} \Delta}{10[\log\{nT^\ast(\alpha^2 \wedge 1)\}]^{1/2}} \right\} \cap B_j \cap D_{i,j} \right) + \mathbb{P}(D_{i,j}^c) \nonumber \\
            &\leq 2d\exp\left(-\frac{T^\ast d\Delta^2}{100c_3 \log\{nT^\ast(\alpha^2 \wedge 1)\}}\right) + \frac{1}{nT^\ast(\alpha^2 \wedge 1)}, 
        \end{align}
        where the third inequality holds on the event $D_{i,j}$ and the fourth inequality follows from applying \Cref{app:lem:SGnorm}, with some absolute constant $c_3 > 0$, for the first term and by~\eqref{app:eq:l2probDj} for the second term.

        We simplify the exponential term by noting that
        \begin{align}
            \exp\left(-\frac{T^\ast d\Delta^2}{100c_3 \log(nT^\ast\alpha^2)}\right)
            &= \exp\left(-\frac{C [\log\{nT^\ast(\alpha^2 \wedge 1)\}]^2}{100c_3 \log\{nT^\ast(\alpha^2 \wedge 1)\}}\right) \leq \bigg\{\frac{1}{nT^\ast(\alpha^2 \wedge 1)}\bigg\}^2,  \label{app:eq:exponentialsimplification}
        \end{align}
        where in the last line we take $C > 200c_3$.

    \medskip
    \noindent \textbf{Step 3.3: Term $(II)$.}
        We have by the same argument as \eqref{app:eq:unicompbound} that
        \begin{equation}
             \mathbb{E} \left( \left[\mathbb{E}\left\{ \{\tilde{\theta}^{(i)} - (R_d\overline{X^{(i)}})_j\}\mathbbm{1}\{B_j^c\} \Bigm| k_j^\ast, R_d \right\}\right]^2 \right)
            \leq 4\mathbb{P}(B_j^c).\label{app:eq:l2compbound}
        \end{equation}
        We hence have that
        \begin{align}
            \mathbb{E}[| (\tilde{\theta}_j - (R_d\theta)_j |^2] &\leq 2d\exp\left(-\frac{T^\ast d\Delta^2}{100c_3 \log\{nT^\ast(\alpha^2 \wedge 1)\}}\right) + \frac{1}{nT^\ast(\alpha^2 \wedge 1)} + 8\mathbb{P}(B_j^c) + \frac{374d\Delta^2}{n(\alpha^2 \wedge 1)} \nonumber \\
            &\lesssim \frac{d}{\{nT^\ast(\alpha^2 \wedge 1)\}^2} + \frac{1}{nT^\ast(\alpha^2 \wedge 1)} + \left( \frac{3}{\Delta} \vee 1\right) e^{-2n\min\{\alpha, \alpha^2\}/(K d)} + \frac{d\Delta^2}{n(\alpha^2 \wedge 1)} \nonumber \\
            &\lesssim \frac{1}{nT^\ast(\alpha^2 \wedge 1)} + \left( \frac{3}{\Delta} \vee 1\right) e^{-2n\min\{\alpha, \alpha^2\}/(K d)} + \frac{d\Delta^2}{n(\alpha^2 \wedge 1)}, \label{app:eq:l2co-ordinateerror}
        \end{align}
        where the first inequality is due to \eqref{app:eq:l2co-ordinatebiasvariance},~\eqref{app:eq:l2VarianceControl},~\eqref{Eq:NewBias1},~\eqref{Eq:NewBias2}  and~\eqref{app:eq:l2compbound}; the second is from~\eqref{app:eq:exponentialsimplification}, \eqref{app:eq:l2compbound} and~\eqref{app:eq:l2probBj}; and the final using the assumption that $n(\alpha^2 \wedge 1) > \widetilde{C}'d\log(ed)$ for some sufficiently large $\widetilde{C}'$. Simplifying further, we have
        \begin{align*}
            \mathbb{E}[| (\tilde{\theta}_j - (R_d\theta)_j |^2]
            &\lesssim \frac{1}{nT^\ast (\alpha^2 \wedge 1)} + \left( \frac{1}{\Delta} + 1\right) e^{-2n\min\{\alpha, \alpha^2\}/(K d)} + \frac{d\Delta^2}{n(\alpha^2 \wedge 1)} \\
            &\lesssim \frac{d^{1/2}(T^\ast)^{1/2}}{\log\{nT^\ast(\alpha^2 \wedge 1)\}} e^{-2n\min\{\alpha, \alpha^2\}/(K d)} + e^{-2n\min\{\alpha, \alpha^2\}/(K d)} + \frac{[\log\{nT^\ast(\alpha^2 \wedge 1)\}]^2}{nT^\ast(\alpha^2 \wedge 1)} \\
            &\lesssim \frac{d^{1/2}}{T^\ast}e^{-n\min\{\alpha, \alpha^2\}/(2Kd)} + \frac{[\log\{nT^\ast(\alpha^2 \wedge 1)\}]^2}{nT^\ast(\alpha^2 \wedge 1)} \\
            &\lesssim \frac{1}{T^\ast}e^{-C'n\min\{\alpha, \alpha^2\}/d} +\frac{[\log\{nT^\ast(\alpha^2 \wedge 1)\}]^2}{nT^\ast(\alpha^2 \wedge 1)} \\
            &\lesssim \frac{[\log\{nT^\ast(\alpha^2 \wedge 1)\}]^2}{nT^\ast(\alpha^2 \wedge 1)},
        \end{align*}
        for $C'>0$ some absolute constant, where the second inequality comes from substituting in the value of $\Delta$ as in \eqref{sec3:eq:l2deltaval}; the third using the fact that $e^{-2n\min\{\alpha, \alpha^2\}/(K d)} \leq (T^\ast)^{-3/2} e^{n\min\{\alpha, \alpha^2\}/(2K d)}$ and controlling the log term in the denominator using the assumption that $n\min\{\alpha, \alpha^2\} > \widetilde{C}'d\log(ed)$ for some sufficiently large $\widetilde{C}'$; and the penultimate by using the fact that $n\min\{\alpha, \alpha^2\} > \widetilde{C}'d\log(ed)$ for some sufficiently large $\widetilde{C}'$ to absorb the factor of $d^{1/2}$ into the exponential term.

        We conclude bounding the mean-squared-error of a single co-ordinate depending on the value of $T$. When $T \leq e^{n\min\{\alpha, \alpha^2\}/(K d)}$, we have that $T^\ast = T$ immediately giving the result for this case. When $T > e^{n\min\{\alpha, \alpha^2\}/(K d)}$, we have that $T^\ast = e^{n\min\{\alpha, \alpha^2\}/(K d)}$ giving
        \begin{align*}
            \mathbb{E}[| (\tilde{\theta}_j - (R_d\theta)_j |^2]
            &\lesssim \bigg( \frac{[\log\{n(\alpha^2 \wedge 1)\}]^2}{n(\alpha^2 \wedge 1)} + \frac{\{\log(T^\ast)\}^2}{n(\alpha^2 \wedge 1)} \bigg) \frac{1}{T^\ast} \\
            &\lesssim (\alpha \vee 1)e^{n\min\{\alpha, \alpha^2\}/(K d)}
            \lesssim e^{-C''n\min\{\alpha, \alpha^2\}/(K d)},
        \end{align*}
        for $C''>0$ an absolute constant, where the first inequality is by the Cauchy--Schwarz inequality, and the second by substituting in the value of $T^\ast = e^{n\min\{\alpha, \alpha^2\}/(K d)}$. Hence, we have that the error of the co-ordinate wise estimator $\tilde{\theta}_j$ is bounded as
        \begin{equation} \label{app:eq:l2coordwiseerror}
            \mathbb{E}[| (\tilde{\theta}_j - (R_d\theta)_j |^2] \lesssim
            \begin{cases}
                \frac{[\log\{n T (\alpha^2 \wedge 1)\}]^2}{nT(\alpha^2 \wedge 1)}, & T \leq e^{n\min\{\alpha, \alpha^2\}/(K d)}, \\
                e^{-C'' n\min\{\alpha, \alpha^2\}/(K d)}, &\text{otherwise.}
            \end{cases}
        \end{equation}
    \medskip
    \noindent \textbf{Step 4: Completing the Proof.} 
        We consider three cases separately depending on the value of $\alpha$

        \medskip
        \noindent \textbf{Case 1: $\alpha \in (0, 1]$} 

        When $\alpha \in (0,1]$, we simply sum up the error per co-ordinate in \eqref{app:eq:l2coordwiseerror}, yielding
        \begin{align*}
            \mathbb{E}\left[\| \hat{\theta} - \theta \|_2^2\right]
            =\mathbb{E}\left[\| R_d\hat{\theta} - R_d\theta \|_2^2\right]
            &= \sum_{j = 1}^d \mathbb{E}[| \tilde{\theta}_j - (R_d\theta)_j |^2] \\
            &\lesssim
            \begin{cases}
               \frac{d\{\log(n T \alpha^2)\}^2}{nT\alpha^2}, & T \leq e^{n\alpha^2/(K d)}, \\
                e^{-C''' n\alpha^2/(K d)}, &\text{otherwise.}
            \end{cases},
        \end{align*}
        for $C'''$ an absolute constant, where the first equality uses the fact that $R_d$ is an orthogonal transformation, and the inequality uses the fact that $n\alpha^2 > \widetilde{C}'d\log(ed)$ for some sufficiently large $\widetilde{C}'$ to absorb the factor of $d$ into the exponential term. Bounding the error by the sum of the two cases for the different values of $T$ completes the proof for the case $\alpha \leq 1$.

        \medskip
        \noindent \textbf{Case 2: $1 < \alpha \leq d$} 

        For simplicity we assume $\alpha$ is an integer, noting that this will not change the minimax rates obtained except up to constants. When $1 < \alpha \leq d$, the estimation procedure in \Cref{app:sec:LowPrivl2ball} has $\alpha$ set to $1$, with each user contributing to $\alpha$-many co-ordinates. Hence, each fold has a sample size of $n\alpha/d$, and the estimator for the $j$-th transformed co-ordinate $\tilde{\theta}_j$ has error bounded as 
        \begin{equation} \label{app:eq:l2coordwiseerror2}
            \mathbb{E}[| \tilde{\theta}_j - (R_d\theta)_j |^2] \lesssim
            \begin{cases}
                \frac{\{\log(nT\alpha)\}^2}{nT\alpha}, & T \leq e^{n\alpha/(K d)}, \\
                e^{-C'' n\alpha/(K d)}, &\text{otherwise.}
            \end{cases}
        \end{equation}
        which follows from the univariate error in \eqref{app:eq:l2coordwiseerror} with $\alpha = 1$ and sample size $n\alpha/d$ therein. For the final estimator $\hat{\theta}$ as defined in \eqref{sec3:eq:Finall2Estimator}, the error can be bounded as
        \begin{align*}
            \mathbb{E}\left[\| \hat{\theta} - \theta \|_2^2\right]
            =\mathbb{E}\left[\| R_d\hat{\theta} - R_d\theta \|_2^2\right]
            &= \sum_{j = 1}^d \mathbb{E}[| (\tilde{\theta}_j - (R_d\theta)_j |^2] \\
            &\lesssim
            \begin{cases}
               \frac{d\{\log(n T \alpha)\}^2}{nT\alpha}, & T \leq e^{n\alpha/(K d)}, \\
                e^{-C''' n\alpha/(K d)}, &\text{otherwise.}
            \end{cases},
        \end{align*}
        for $C'''$ an absolute constant, where the first equality uses the fact that $R_d$ is an orthogonal transformation, and the inequality uses the fact that $n\alpha > \widetilde{C}'d\log(ed)$ for some sufficiently large $\widetilde{C}'$ to absorb the factor of $d$ into the exponential term. Bounding the error by the sum of the two cases for the different values of $T$ completes the proof for the case $1 < \alpha \leq d$.

        \medskip
        \noindent \textbf{Case 3: $\alpha > d$} 

        For $\alpha > d$, the estimation procedure in \Cref{app:sec:LowPrivl2ball} has $\alpha$ set to $\alpha/d > 1$, with each user contributing to $\alpha$-many co-ordinates. Hence, each fold has a sample size of $n$, and the estimator for the $j$-th transformed co-ordinate $\tilde{\theta}_j$ has error bounded as 
        \begin{equation*}
            \mathbb{E}\left[|\hat{\theta}_j - \theta_j|^2 \right]
            \lesssim
                \begin{cases}
                    \frac{\{\log(nT)\}^2}{dnT}, & T \leq e^{n\alpha/(K d)}, \\
                    e^{-n\alpha/(2K d)}, &\text{otherwise,}
                \end{cases}
        \end{equation*}
        which follows from the univariate error in \eqref{app:eq:l2coordwiseerror} with for privacy parameter $\alpha/d > 1$. For the final estimator $\hat{\theta}$ as defined in \eqref{sec3:eq:Finall2Estimator}, the error can be bounded as
        \begin{align*}
            \mathbb{E}\left[\| \hat{\theta} - \theta \|_2^2\right]
            =\mathbb{E}\left[\| R_d\hat{\theta} - R_d\theta \|_2^2\right]
            &= \sum_{j = 1}^d \mathbb{E}[| (\tilde{\theta}_j - (R_d\theta)_j |^2] \\
            &\lesssim
            \begin{cases}
               \frac{\{\log(d n T)\}^2}{nT}, & T \leq e^{n\alpha/(K d)}, \\
                e^{-C''' n\alpha/(K d)}, &\text{otherwise.}
            \end{cases},
        \end{align*}
        for $C'''$ an absolute constant, where the first equality uses the fact that $R_d$ is an orthogonal transformation, and the inequality uses the fact that $n\alpha > \widetilde{C}'d\log(ed)$ for some sufficiently large $\widetilde{C}'$ to absorb the factor of $d$ into the exponential term. Bounding the error by the sum of the two cases for the different values of $T$ completes the proof for the case $\alpha > d$. Combining all three cases completes the proof.
    \end{proof}

\subsection{Proof of Theorem \ref{sec5:thm:main} Upper Bound}  

    \begin{proof}[Proof of \Cref{sec5:thm:main} (The first procedure)]
        We define the following sets of indices
        \begin{align*}
            S_1 = \{j \in [d] : |\theta_j| > 2\varepsilon \}, \quad S_2 = \{j \in [d] : 0 < |\theta_j| \leq 2\varepsilon \} \quad \mbox{and} \quad S_0 = \{j \in [d] : \theta_j = 0 \}.
        \end{align*}
        For $j \in [d]$, denote
        \begin{equation}
            I_j = (n/2)^{-1}\sum_{i=1}^{n/2} r_{i,j}Z_i \label{app:eq:sparseselectionvoteagg}
        \end{equation}
        which is such that $\mathcal{I} = \{j \in [d] : I_j \geq 1/2\}$. We then consider the following event
        \begin{equation}
            A = \left( \bigcap_{j \in S_1} \left\{ I_j \geq 1/2 \right\}\right) \cap \left(\bigcap_{j \in S_0} \left\{ I_j < 1/2 \right\} \right). \label{app:eq:sparsebinningevent}
        \end{equation}
        On this event, those co-ordinates in $S_1$ are correctly identified and those in $S_0$ are correctly rejected. We then consider the error of the estimator on the event $A$ and its complement.

        The construction of the estimator $\hat{\theta}$ in \eqref{sec5:eq:sparseestimator1} ensures that $\hat{\theta}_j = 0$ for $j \notin \mathcal{I}'$, from which we obtain
        \begin{equation} 
            \| \hat{\theta} - \theta \|_2^2 = \sum_{j \notin \mathcal{I}'} \theta_j^2 + \sum_{j \in \mathcal{I}'} (\tilde{\theta}_j - \theta_j)^2. \label{app:eq:sparseestimatorcoords}
        \end{equation}

        We consider the first summand in \eqref{app:eq:sparseestimatorcoords}. On $A$, we have that $S_1 \subseteq \mathcal{I}$, and we also have $S_0 \cap \mathcal{I} = \emptyset$ and hence $|\mathcal{I}| \leq s$, giving $\mathcal{I}' = \mathcal{I}$. We have no guarantees on the membership of co-ordinates of $S_2$ in $\mathcal{I}$, and so the error is at most $|S_2|(2\varepsilon)^2 \leq 4s\varepsilon^2$. On the complement event $A^c$, the error can be bounded by the worst case error of $4s$, giving
        \begin{equation}
            \sum_{j \notin \mathcal{I}'} \theta_j^2
            = \sum_{j \notin \mathcal{I}'} \theta_j^2 \mathbbm{1}\{A\} + \sum_{j \notin \mathcal{I}'} \theta_j^2 \mathbbm{1}\{A^c\} \leq 4s\varepsilon^2 + 4s\mathbbm{1}\{A^c\}\quad \text{ a.s.} \label{app:eq:sparseestimatorcontrolled}
        \end{equation}
        
        As the second summand in \eqref{app:eq:sparseestimatorcoords} corresponds to the $\ell_\infty$-ball estimator of \Cref{sec3:linfball}, the mean-squared-error incurred in this term is bounded by \eqref{app:eq:linfMSE} where the dimension of the problem is $s' \leq s$. We now show that the error of the $s'$-dimensional problem can be bounded by the error of the $s$-dimensional problem up to constants.
        
        Indeed, we first note that $s^2 \log(nT\alpha^2/s)/(nT\alpha^2)$ is an increasing function of $s$ due to the assumption $n\alpha^2 > \widetilde{C}s\log(ed)$. Further, with $c$ as in \eqref{app:eq:linffinalMSE}, $e^{-cn\alpha^2/s}$ as a function of $s$ is also increasing. 

        It remains to consider the case where $e^{cn\alpha^2/s} < T < e^{cn\alpha^2/s'}$ where the exponential rate of the $s$-dimensional problem must be compared to the polynomial rate of the $s'$-dimensional problem. We see that
        \begin{align*}
            \frac{(s')^2\log(nT\alpha^2/s')}{nT\alpha^2}
            \leq \left(\frac{s^2\log(n\alpha^2) +n\alpha^2/s'}{n\alpha^2}\right)e^{-cn\alpha^2/s}
            \lesssim s^2 e^{-cn\alpha^2/s} \lesssim e^{-c'n\alpha^2/s}
        \end{align*}
        where $c' > 0$ is an absolute constant; the first inequality is by the fact that $e^{-cn\alpha^2/s} < T < e^{-cn\alpha^2/s'}$ in the considered regime; and the last inequality by the fact that $n\alpha^2 > \widetilde{C}s\log(ed)$ for some sufficiently large $\widetilde{C}$.
        
        Hence, the error of the $s$-dimensional $\ell_\infty$-ball estimator is indeed an increasing function of $s$ up to constants, giving
        \begin{equation}
            \mathbb{E}\left[\sum_{j \in \mathcal{I}'} (\tilde{\theta}_j - \theta_j)^2 \right]
            \lesssim 
            \begin{cases}
                \frac{s^2 \log(nT\alpha^2/s)}{nT\alpha^2}, & T \leq e^{n\alpha^2/(K s)}, \\
                e^{-c'n\alpha^2/s}, &\text{otherwise}.
            \end{cases} \label{app:eq:sparseestimatorlinfsub}
        \end{equation}
        We then combine \eqref{app:eq:sparseestimatorcoords}, \eqref{app:eq:sparseestimatorcontrolled}, \Cref{app:lem:sparsebinningevent}, and \eqref{app:eq:sparseestimatorlinfsub} to see that
        \begin{align*}
            \mathbb{E}\left( \| \hat{\theta} - \theta \|_2^2 \right)
            &\lesssim \frac{s\log(dnT^\ast\alpha^2)}{T^\ast} +
                \begin{cases}
                    \frac{s^2 \log(nT\alpha^2/s)}{nT\alpha^2}, &  T \leq e^{n\alpha^2/(K s)}, \\
                    e^{-c'n\alpha^2/s}, &\text{otherwise,}
                \end{cases}
        \end{align*}
        where we substitute in the value of $\varepsilon$ as in \eqref{sec5:eq:thresholdboundary}.
        When $T \leq e^{n\alpha^2/(K s)}$, we have that $T^\ast = T$, and hence,
        \begin{align*}
            \mathbb{E}\left( \| \hat{\theta} -\theta \|_2^2 \right)
            &\lesssim \frac{s\log(dnT\alpha^2)}{T}
            + \frac{s^2 \log(nT\alpha^2/s)}{nT\alpha^2}
            \lesssim \frac{s\log(dnT\alpha^2)}{T},
        \end{align*}    
        where the second inequality follows from the assumption that $n\alpha^2 > \widetilde{C}s\log(ed)$ for some suitably large constant $\widetilde{C}$.
        When instead $T > e^{n\alpha^2/(K s)}$, we have that $T^\ast = e^{n\alpha^2/(K s)}$, giving
        \begin{align*}
            \mathbb{E}\left( \| \hat{\theta} -\theta \|_2^2 \right)
            \lesssim \{s\log(dn\alpha^2)+n\alpha^2\}e^{-n\alpha^2/(K s)}
            + e^{-c'n\alpha^2/s}
            \lesssim e^{-c''n\alpha^2/s},
        \end{align*}
        where $c'' > 0$ is some absolute constant, the first inequality comes from the value of $T^\ast$, and the second by the assumption that $n\alpha^2 > \widetilde{C}s\log(d)$ for some suitably large constant $\widetilde{C} > 0$ which allows the multiplicative prefactors to be absorbed into their respective exponential factor.
        Hence, we have
        \begin{equation*}
            \mathbb{E}\left( \| \hat{\theta} -\theta \|_2^2 \right) 
            \lesssim
            \begin{cases}
                s\log(dnT\alpha^2)/T, &  T \leq e^{n\alpha^2/(K s)}, \\
                e^{-c''n\alpha^2/s}, &\text{otherwise.}
            \end{cases}
        \end{equation*}
        Lastly, bounding the error by the sum of the two cases for the different values of $T$ completes the proof.
    \end{proof}

    \begin{proof}[Proof of \Cref{sec5:thm:main} (Second Procedure)]
    We recall that this second procedure is only used when $n\alpha^2 > \widetilde{C}d\log(dnT\alpha^2)$, which is henceforth assumed throughout this proof. We proceed by showing that the entries of the estimator $\tilde{\theta}$ obtained from the $\ell_\infty$-ball procedure of \Cref{sec3:linfball} are suitably close to the true values in $\theta$ with high probability so that that the thresholding procedure in \eqref{sec5:eq:thresholding} correctly shrinks to zero the co-ordinates of $\hat{\theta}$ which correspond to zeros of $\theta$.

    It suffices to consider a single fixed co-ordinate $j \in [d]$ as the $\ell_\infty$-ball estimation procedure is independent across co-ordinates.

    For $i \in [n]$ we write $\overline{X^{(i)}} = (T^\ast)^{-1} \sum_{t = 1}^{T^\ast} X_{t, j}^{(i)}$ and $Z_{j,k} = \sum_{i \in N_{j,1}} \widetilde{V}^{(i)}_k$ for the total number of privatised votes for the $k$-th sub-interval. We denote by $l_j \in [N(\Delta)]$ the fixed but unknown index such that $\theta_j \in I_{l_j}$ and define the event
    \begin{equation*}
        A_j = \bigcap_{k:|l_j-k|>2} \{Z_{j,l_j} > Z_{j,k}\},
    \end{equation*}
    which is analogous to the event \eqref{app:eq:eventB} considered in the analysis of the univariate mean estimator. Bounding the probability of the complement by the same analysis as that in \eqref{app:eq:voteprobbound} and \eqref{app:eq:probA}, we have
        \begin{equation}
            \mathbb{P}(A_j^c)
            \leq \frac{3}{\Delta \wedge 1} e^{-2n\alpha^2/(K d)} \label{app:eq:sparseprobAj}
        \end{equation}
    where $K$ is as appears in \eqref{app:eq:voteprobbound} and, as in \eqref{app:eq:probA}, we require $\Delta < 1$, so to account for when this fails to hold we take the minimum $\Delta \wedge 1$ which suffices for an upper bound.

    We denote by $\tilde{I}_{j^\ast}$ the chosen interval for refinement and denote the lower and upper endpoints of this interval by $L_j$ and $U_j$ respectively. Recalling the inflated intervals $\tilde{L}_j$ and $\tilde{U}_j$ defined in \eqref{sec3:eq:uniMeanIntervalExpand}, we now proceed to show that for a fixed user $i \in N_{j,2}$, on the event $A_j$ the truncated sample mean is close to the sample mean with high probability. Indeed, we first note that
    \begin{align}
        \mathbb{E}\left[\left| \Pi_{\tilde{I}_{j^\ast}}(\overline{X^{(i)}}) - \overline{X^{(i)}} \right| \mathbbm{1}\{A_j\} \Bigm| j^\ast \right] \leq 4e^{-T^\ast \Delta^2/2} \quad \mathrm{a.s.} \label{app:eq:biascontrolinsparse}
    \end{align}
    which follows from the same argument as \eqref{app:eq:biascontrolprelim} and \eqref{app:eq:BiasControl}, noting the different value of $\Delta$ in this case.

    Hence, we have
    \begin{align}
        &\mathbb{P}\left( \left\{ \left| \frac{1}{n/(2d)} \sum_{i \in N_{j,2}} \left( \Pi_{\tilde{I}_{j^\ast}}(\overline{X^{(i)}}) - \overline{X^{(i)}} \right) \right| > \varepsilon \right\} \cap A_j \Biggm| j^\ast \right) \nonumber \\
        &\leq \frac{1}{\varepsilon}\mathbb{E}\left[ \left|\frac{1}{n/(2d)} \sum_{i \in N_{j,2}}  \left( \Pi_{\tilde{I}_{j^\ast}}(\overline{X^{(i)}}) - \overline{X^{(i)}} \right) \right|\mathbbm{1}\{A_j\} \Biggm| j^\ast \right]
        \leq \frac{4e^{-T^\ast\Delta/2}}{\varepsilon} \nonumber \\
        &\leq \frac{4}{C}\left(\frac{nT^\ast\alpha^2}{2d\{\log(nT^\ast\alpha^2)\}^2}\right)^{1/2} e^{-T^\ast\Delta^2/2} \lesssim \frac{1}{nT^\ast\alpha^2} \quad \mathrm{a.s.}
        \label{app:eq:TruncDeviationBound}
    \end{align}
    where the first inequality is by Markov's inequality; the second by \eqref{app:eq:biascontrolinsparse}; the last two inequalities by the values of $\varepsilon$ and $\Delta$ as in \eqref{sec5:eq:Deltaval2nd}.

    We also note the definition of a sub-Exponential random variable \citep[e.g.][~Definition~2.7]{Wainwright:2019} where it is easy to verify that the Laplace random variables are sub-Exponential with parameters $(2,2)$. Hence, by e.g.~\citet[][Equation~2.18]{Wainwright:2019} we have that
    \begin{equation}
        \mathbb{P} \left( \left\{ 
        \left| \frac{1}{n/(2d)} \sum_{i \in {N_{j,1}}} \left( \frac{14\Delta}{\alpha} \ell_i \right) \right| > \varepsilon \right\} \cap A_j \right)
        \leq 2
        \begin{cases}
            e^{-\frac{n}{16d}\{\varepsilon\alpha/(14\Delta)\}^2}, &\text{ for } 0 \leq \varepsilon\alpha/(14\Delta) \leq 2\\
            e^{-\frac{n}{8d}\{\varepsilon\alpha/(14\Delta)\}}, &\text{ for } \varepsilon\alpha/(14\Delta) > 2.
        \end{cases}\label{app:eq:bernsteinbound}
    \end{equation}
    We note that
    \begin{align*}
        \frac{\varepsilon\alpha}{14\Delta} = \frac{C'}{14}\left(\frac{d\log(nT^\ast\alpha^2)}{n\alpha^2}\right)^{1/2} \leq \frac{C'}{14\widetilde{C}} \leq 2
    \end{align*}
    for $\widetilde{C}$ taken sufficiently large using the fact $n\alpha^2 > \widetilde{C}d\log(dnT\alpha^2)$. Hence we have the sub-Gaussian regime in \eqref{app:eq:bernsteinbound}, giving
    \begin{align}
        &\mathbb{P}\left( \left\{ \left| \tilde{\theta}_j - \theta_j \right| > 3\varepsilon \right\} \cap A_j \right)
        \leq \mathbb{P} \left( \left\{ \left| \frac{1}{n/(2d)} \sum_{i \in {N_{j,1}}} \left( \Pi_{\tilde{I}_{j^\ast}}(\overline{X^{(i)}}) - \overline{X^{(i)}} \right) \right| > \varepsilon \right\} \cap A_j \right) \nonumber \\
        &+ \mathbb{P} \left(\left\{ \left| \frac{1}{n/(2d)} \sum_{i \in {N_{j,1}}} \left( \overline{X^{(i)}} - \theta_j \right) \right| > \varepsilon \right\} \cap A_j \right) + \mathbb{P} \left( \left\{ 
        \left| \frac{1}{n/(2d)} \sum_{i \in {N_{j,1}}} \frac{14\Delta}{\alpha} \ell_i \right| > \varepsilon \right\} \cap A_j \right) \nonumber \\
        &\lesssim \frac{1}{nT^\ast\alpha^2} + 2e^{-\frac{nT^\ast \varepsilon^2}{4d}} + 2e^{-\frac{n}{16d}\left(\frac{\varepsilon\alpha}{14\Delta}\right)^2}
        = \frac{1}{nT^\ast\alpha^2} + 2e^{-C'^2\{\log(nT^\ast\alpha^2)\}^2/\alpha^2} + 2e^{-C'^2\log(nT^\ast\alpha^2)/(3136\alpha^2)} \nonumber \\
        &\lesssim \frac{1}{nT^\ast\alpha^2}
        \label{app:eq:SparseDistanceBound},   
    \end{align}
    where in the second inequality the first probability term is bounded as in \eqref{app:eq:TruncDeviationBound}, the second term by Hoeffding's inequality, and the third term by the sum of the two regimes in \eqref{app:eq:bernsteinbound}. The equality comes from substituting the value of $\varepsilon$ and $\Delta$ as in \eqref{sec5:eq:Deltaval2nd}, and the final inequality is from taking $C'$ sufficiently large.

    For a co-ordinate $j \in [d]$, define the event
    \begin{equation*}
        B_j = \left\{ \left| \tilde{\theta}_j - \theta_j \right| \leq 3\varepsilon \right\},
    \end{equation*}
    and let $B = \cap_{j \in [d]} B_j$, whence we have
    \begin{align*}
        \mathbb{P}(B^c) &\leq \sum_{j=1}^d \left\{ \mathbb{P}(A_j^c) + \mathbb{P}(B_j^c \cap A_j) \right\} \\
        &\lesssim d\left( \frac{1}{\Delta} + 1 \right) e^{-2n\alpha^2/(K d)} + d\mathbb{P}\left( \left\{ \left| \tilde{\theta}_j - \theta_j \right| > 3\varepsilon \right\} \cap A_j \right)  \\
        &\lesssim \left( \frac{d(T^\ast)^{1/2}}{\{\log(nT\alpha^2)\}^{1/2}} + d \right) e^{-2n\alpha^2/(K d)} + \frac{d}{nT^\ast\alpha^2}  \\
        &\lesssim \frac{1}{T^\ast}e^{-cn\alpha^2/d} + \frac{d}{nT^\ast\alpha^2} 
        \lesssim \frac{d}{nT^\ast\alpha^2},
    \end{align*}
    where $c > 0$ is an absolute constant; the second inequality uses \eqref{app:eq:sparseprobAj}; the third by \eqref{app:eq:SparseDistanceBound} and the value of $\Delta$ as in \eqref{sec5:eq:Deltaval2nd}; and the penultimate by controlling the exponential term as $e^{-2n\alpha^2/(K d)} \leq (T^\ast)^{-3/2} e^{-n\alpha^2/(2K d)}$ and absorbing the prefactor of $d$ using the fact that $n\alpha^2 \geq \widetilde{C} d\log(dnT\alpha^2)$ for $\widetilde{C}$ sufficiently large.

    We now analyse the error of the estimator on the event $B$, and control its error off this event. We define the sets
    \begin{equation*}
        S_1 = \{j \in [d] : |\theta_j| > 6\varepsilon \}, \qquad S_2 = \{j \in [d] : |\theta_j| \leq 6\varepsilon \}, \qquad S_0 = \{j \in [d] : \theta_j = 0 \}.
    \end{equation*}
    We note by the construction of the final estimator \eqref{sec5:eq:finalThresholdedEstimator}, the worst case error is $4s$ as at most $s$ many co-ordinates are non-zero, and we see that
    \begin{align*}
        \mathbb{E}\left[ \|\hat{\theta} - \theta\|_2^2 \right]
        &\leq 4s\mathbb{P}(B^c) + \mathbb{E}(\|\hat{\theta} - \theta\|_2^2 \mathbbm{1}\{B\}) \\
        &\lesssim \frac{sd}{nT^\ast\alpha^2} + \sum_{j \in S_1 \cup S_2}\mathbb{E}(|\hat{\theta} - \theta|^2 \mathbbm{1}\{B\}) + \sum_{j \in S_0}\mathbb{E}(|\hat{\theta} - \theta|^2 \mathbbm{1}\{B\}) \\
        &\lesssim \frac{sd}{nT^\ast\alpha^2} + |S_1 \cup S_2|\varepsilon^2
        \lesssim \frac{sd}{nT^\ast\alpha^2} + s\varepsilon^2
        \lesssim \frac{sd\{\log(nT^\ast\alpha^2)\}^2}{nT^\ast\alpha^2},
    \end{align*}
    where in the first inequality we bound the error on the event $C^c$ by the worst case of $4s$, and the third inequality uses the fact that on the event $B$, (i) co-ordinates in $S_0$ are correctly thresholded to $0$ in the estimator $\hat{\theta}$, contributing no error, and (ii) co-ordinates in $S_1 \cup S_2$ contribute error at most $O(\varepsilon^2)$. The final inequality then follows from the value of $\varepsilon$ as in \eqref{sec5:eq:Deltaval2nd}.

    In the case $T \leq e^{n\alpha^2/(K d)}$, we have $T^\ast = T$ which immediately gives the desired rate. In the case $T > e^{n\alpha^2/(K d)}$, we have $T^\ast = e^{n\alpha^2/(K d)}$, which gives
    \begin{align*}
        \mathbb{E}\left[ \| \hat{\theta} -\theta \|^2 \right]
        &\lesssim \frac{sd \{\log(nT^\ast\alpha^2)\}^2}{nT^\ast\alpha^2}
        = \frac{sd}{nT^\ast \alpha^2} \left\{ \log(n\alpha^2) + \log(T^\ast) \right\}^2, \\
        &= \frac{sd}{n\alpha^2} \left( \log(n\alpha^2) + \frac{n\alpha^2}{K d} \right)^2 e^{-n\alpha^2/(K d)}
        \lesssim e^{-c'n\alpha^2/d},
    \end{align*}
    where $c'>0$ is an absolute constant; the second inequality comes from substituting in the value $e^{n\alpha^2/(K d)}$ for $T^\ast$; and the last inequality by the fact that $n\alpha^2 \geq \widetilde{C} d\log(dnT\alpha^2)$ for some sufficiently large $\widetilde{C}$. Trivially upper bounding the polynomial rate by inserting an extra factor of $d$ into the log term to match that of the previously constructed estimator completes the proof.
    \end{proof}

\subsection{Proof of Theorem \ref{sec4:thm:main} Upper Bound}
    \begin{proof}[Proof of \Cref{sec4:thm:main}]  
    \
    \\
    \noindent \textbf{Step 1: Reducing from Non-Parametric Density Estimation to Mean Estimation.}
     Due to the definition of $\hat{f}$ in \eqref{sec4:eq:finalestimator}, it holds that 
        \begin{equation}
            \mathbb{E}\left[ \|\hat{f} - f\|_2^2 \right] = \mathbb{E} \left[ \int \{\hat{f}(x) - f(x)\}^2 \diff x \right]
            = \mathbb{E} \left[ \sum_{j=1}^M (\hat{\theta}_j - \theta_j)^2\right] + \sum_{j=M+1}^\infty \theta_j^2 
            \leq \mathbb{E}\left[\|\hat{\theta} -\theta\|_2^2 \right] + \frac{1}{M^{2\beta}}, \label{app:eq:densityupperrate}
        \end{equation}
        where the second equality is due to the orthonormality of the basis \eqref{sec2:eq:TrigBasis} and the inequality follows from the definition of the Sobolev ellipsoid \eqref{sec2:eq:Sobolev}. As $\sup_j \sup_x |\varphi_j(x)| \leq \sqrt{2}$ for all $j \geq 1$, we have that $\theta \in \mathbb{B}_\infty(\sqrt{2}) \subset \mathbb{R}^M$. We recall the required assumption for the $\ell_\infty$-ball estimator in \Cref{sec3:thm:main}, that is, $n\alpha^2 > \widetilde{C}M\log(eM)$ where $\widetilde{C} > 0$ is the same constant as in \Cref{sec3:thm:main}.  This is verified below.

        \begin{enumerate}[{Case} 1.]
            \item When $T \leq (n\alpha^2)^{2\beta+1}/\{C\log(n\alpha^2)\}^{2\beta + 2}$, we have that $M = (nT\alpha^2)^{1/(2\beta + 2)}$ and therefore $M \leq n\alpha^2/\{C\log(n\alpha^2)\}$.  It then follows that
            \begin{align*}
                \widetilde{C}M\log(eM)
                \leq \frac{\widetilde{C}n\alpha^2}{C\log(n\alpha^2)} \log\left\{ \frac{en\alpha^2}{C\log(n\alpha^2)} \right\}
                \leq \frac{\widetilde{C}n\alpha^2}{C} \leq n\alpha^2,
            \end{align*}
            due to the assumption that $n\alpha^2 \geq \widetilde{C}_\beta$ for some sufficiently large $\widetilde{C}_\beta$ and choosing $C \geq \widetilde{C}$.
            \item When $T > (n\alpha^2)^{2\beta+1}/\{C\log(n\alpha^2)\}^{2\beta + 2}$, we have that $M = n\alpha^2/\{C\log(n\alpha^2)\}$ and therefore
            \begin{equation*}
                \widetilde{C}M\log(eM)
                = \frac{\widetilde{C}n\alpha^2}{C\log(n\alpha^2)} \log\left\{ \frac{en\alpha^2}{C\log(n\alpha^2)} \right\}
                \leq \frac{\widetilde{C}n\alpha^2}{C} \leq n\alpha^2,
            \end{equation*}
           due to the assumption that $n\alpha^2 \geq \widetilde{C}_\beta$ for some sufficiently large $\widetilde{C}_\beta$ and choosing $C \geq \widetilde{C}$.
        \end{enumerate}

        \medskip
        \noindent \textbf{Step 2: Applying \Cref{sec3:thm:main}.}
        We proceed with two cases of $M$ separately.  

        \medskip
        \noindent \textbf{Step 2.1: Case $T \leq (n\alpha^2)^{2\beta+1}/\{C\log(n\alpha^2)\}^{2\beta + 2}$.}  For sufficiently large $C \geq K(2\beta+1)$ where $K$ is as in \eqref{app:eq:linfMSE}, we have that
        \begin{equation*}
            T \leq (n\alpha^2)^{2\beta+1} \leq \exp\biggl\{\frac{C}{K} \log(n \alpha^2) \biggr\} \leq \exp \biggl\{\frac{1}{K} (n\alpha^2)^\frac{2\beta+1}{2 \beta+2} T^{-1/(2\beta+2)} \biggr\} \leq \exp \biggl( \frac{n \alpha^2}{K M} \biggr).
        \end{equation*}
        Hence, combining \eqref{app:eq:linfMSE} and \eqref{app:eq:densityupperrate} we see that
        \begin{align}
            \mathbb{E}\left[ \|\hat{f} - f\|_2^2 \right]
            &\lesssim \frac{M^2\log(nT\alpha^2/M)}{nT\alpha^2} + \frac{1}{M^{2\beta}}, \nonumber \\
            &= \frac{\log\{nT\alpha^2/(nT\alpha^2)^{1/(2\beta + 2)}\}}{(nT\alpha^2)^{2\beta/(2\beta + 2)}} + \frac{1}{(nT\alpha^2)^{2\beta/(2\beta + 2)}}, \nonumber \\
            &\lesssim C_\beta \log(nT\alpha^2) (nT\alpha^2)^{-\frac{2\beta}{2\beta + 2}}, \label{app:eq:densitysmallTrate}
        \end{align}
        where $C_\beta>0$ is some constant depending on $\beta$ and the equality comes from the value of $M$.

    \medskip
        \noindent \textbf{Step 2.2: Case $T > (n\alpha^2)^{2\beta + 1} /\{C\log(n\alpha^2)\}^{2\beta + 2}$.}
        Consider the two possible regimes for the mean estimation rate \eqref{app:eq:linfMSE} separately. 
        
        When $T > \exp(n\alpha^2/(K M))$, the exponential rate occurs.  Combining  \eqref{app:eq:linfMSE} and \eqref{app:eq:densityupperrate}, we see that
        \begin{align}
            \mathbb{E}\left[ \|\hat{f} - f\|_2^2 \right]
            &\lesssim e^{-cn\alpha^2/M} + \frac{1}{M^{2\beta}} \nonumber \\
            &= e^{-cC\log(n\alpha^2)} + \frac{\{C \log(n\alpha^2)\}^{2\beta}}{(n\alpha^2)^{2\beta}} \nonumber \\
            &\leq \frac{1}{(n\alpha^2)^{2\beta}} + \frac{\{C\log(n\alpha^2)\}^{2\beta}}{(n\alpha^2)^{2\beta}}
            \lesssim \frac{\{C\log(n\alpha^2)\}^{2\beta}}{(n\alpha^2)^{2\beta}}, \label{app:eq:densitylargeTexprate}
        \end{align}
        where $c > 0$ is an absolute constant; the equality from the value of $M$; and the second inequality by taking $C$ sufficiently large so the $cC \geq 2 \beta$.
        
        When $T \leq \exp(n\alpha^2/(K M))$, the polynomial rate in \eqref{app:eq:linfMSE} occurs.  Note that
        \begin{equation}
            T \leq e^{n\alpha^2/(K M)} = (n\alpha^2)^{C/K}, \label{app:eq:densitylargeTpolycasebound}
        \end{equation}
        where the equality comes from the value of $M$. Combining  \eqref{app:eq:linfMSE} and \eqref{app:eq:densityupperrate}, we see that
        \begin{align}
            \mathbb{E}\left[ \|\hat{f} - f\|_2^2 \right]
            &\lesssim \frac{M^2 \log(nT\alpha^2/M)}{nT\alpha^2} + \frac{1}{M^{2\beta}} \nonumber \\
            & \asymp \frac{(n\alpha^2)^2}{\{\log(n\alpha^2)\}^2} \frac{\log(T\log(n\alpha^2))}{nT\alpha^2} + \frac{\{C\log(n\alpha^2)\}^{2\beta}}{(n\alpha^2)^{2\beta}} \nonumber \\
            & \lesssim \frac{n\alpha^2}{T} \log(T) + \frac{\{C\log(n\alpha^2)\}^{2\beta}}{(n\alpha^2)^{2\beta}} \nonumber \\
            & \lesssim \frac{n\alpha^2}{(n\alpha^2)^{2\beta + 1}} \{\log(n\alpha^2)\}^{2\beta+3} + \frac{\{C\log(n\alpha^2)\}^{2\beta}}{(n\alpha^2)^{2\beta}} \lesssim \frac{\{\log(n\alpha^2)\}^{2\beta + 3}}{(n\alpha^2)^{2\beta}}, \label{app:eq:densitylargeTpolyrate}
        \end{align}
        where the equality comes from the value of $M$ and the second inequality comes from \eqref{app:eq:densitylargeTpolycasebound} and the fact that $T > (n\alpha^2)^{2\beta + 1} /\{C\log(n\alpha^2)\}^{2\beta + 2}\}$. Combining \eqref{app:eq:densitylargeTexprate} and $\eqref{app:eq:densitylargeTpolyrate}$ shows that the rate is $O\{(n\alpha^2)^{-2\beta}\}$ up to constant and poly-logarithmic factors in both cases. Combining also with \eqref{app:eq:densitysmallTrate} completes the proof.
    \end{proof}

    \subsection{Proof of Theorem \ref{secmix:thm:main} Upper Bound} \label{app:sec:Mix}
            \begin{proof}[Proof of \Cref{secmix:thm:main}]
            Let $\{\iota^{(i)}\}_{i \in [n]}$ be i.i.d. $\mathrm{Multi}(1, \pi)$ random variables with $\pi = (\pi_1, \ldots, \pi_m)^T$. It then follows that 
            \begin{equation*}
                X^{(i)}_{1:T} | \iota^{(i)} = k \sim P^{\otimes T}_k, \quad k \in [m], \, i \in [n].
            \end{equation*}

            \medskip
            \noindent \textbf{Step 1: Localisation.}
            
            For $k \in [m]$, we denote by $l_k \in [N(\Delta)]$ the fixed but unknown index such that the truth $\theta_k \in I_{l_k}$. Note that due to the separation condition that $\min_{k \neq k'}|\theta_k - \theta_{k'}| \geq \theta_{\mathrm{sep}}$, along with the choice of $\Delta$ as in~\eqref{secmix:eq:deltaval} and the condition on $\theta_{\mathrm{sep}}$, we have that $\theta_{k'} \notin I_{l_k - 1} \cup I_{l_k} \cup I_{l_k + 1}$ for $k' \neq k$.  Denote $I_0 = I_{N(\Delta) + 1} = \emptyset$.
            
            For each $i \in N_1$ and $t \in [T]$, the random variable $X^{(i)}_t \in [-1, 1]$. Hence, we see that
            \begin{align} 
                \mathbb{P}\left\{\left| \frac{1}{T^*} \sum_{t = 1}^{T^\ast} X^{(i)}_t - \theta_{\iota^{(i)}} \right|\geq \Delta \right\}
                &= \mathbb{E}\bigg[\mathbb{P}\left\{\left| \frac{1}{T^*} \sum_{t = 1}^{T^\ast} X^{(i)}_t - \theta_{\iota^{(i)}} \right|\geq \Delta \biggm| \iota^{(i)} \right\} \bigg] \nonumber \\
                &\leq 2e^{-T^\ast\Delta^2/2} = \frac{2}{(nT^\ast\alpha^2)^2}
                \leq \frac{\pi_{0}}{4} \label{app:eq:UnivariateConcentrationMix}
            \end{align}
            where the first inequality is due to Hoeffding's inequality for bounded random variables \citep[e.g.][Proposition 2.5]{Wainwright:2019}; the second equality by the definition of $\Delta$ in~\eqref{secmix:eq:deltaval}; and the last inequality by recalling the condition that $\pi_{0} \geq \tilde{c}/m$ and by the assumption that $n\alpha^2 \geq \widetilde{C}m^2\log(n\alpha^2)$ for some sufficiently large $\widetilde{C}$.
            
            For $k \in [m]$, denote $l_k'$ for the index of the closest sub-interval not containing the true mean $\theta_k$, that is
            \begin{equation*}
                l_k' =
                \begin{cases}
                    l_k - 1, &\text{when } \theta_k \in [-1 + 2(l_k-1)\Delta, -1 + 2(l_k-1)\Delta + \Delta) \\
                    l_k + 1, &\text{when } \theta_k \in [-1 + 2(l_k-1)\Delta + \Delta, -1 + 2l_k\Delta) \\
                \end{cases}.
            \end{equation*}
            For readability, we write $l(i) = l_{\iota^{(i)}}$ and $l'(i) = l_{\iota^{(i)}}'$. Then, on the event $\{|(T^\ast)^{-1} \sum_{t = 1}^{T^{\ast}} X^{(i)}_t - \theta_{\iota^{(i)}}| < \Delta\}$, it holds that $(T^\ast)^{-1} \sum_{t = 1}^{T^{\ast}} X^{(i)}_t \in I_{l(i)} \cup I_{l'(i)}$. Consequently, we can bound the following probability
            \begin{equation*}
                \mathbb{P}(\{V_{l(i)}^{(i)} = 0\} \cap \{V_{l'(i)}^{(i)} = 0\})
                \leq \mathbb{P}\left\{\left| \frac{1}{T^*} \sum_{t = 1}^{T^\ast} X^{(i)}_t - \theta_{\iota^{(i)}} \right|\geq \Delta \right\} \leq \frac{\pi_{0}}{4},
            \end{equation*}
            where the final inequality is by \eqref{app:eq:UnivariateConcentrationMix}. As the bins are disjoint, we have $\mathbb{P}(\{V_{l(i)}^{(i)} = 1\} \cap \{V_{l'(i)}^{(i)} = 1\}) = 0$, and hence obtain
            \begin{align}
                \mathbb{P}(V_{l(i)}^{(i)} = 1) + \mathbb{P}(V_{l'(i)}^{(i)} = 1)
                &= \mathbb{P}(\{V_{l(i)}^{(i)} = 1\} \cup \{V_{l'(i)}^{(i)} = 1\}) \nonumber \\
                &= 1 - \mathbb{P}(\{V_{l(i)}^{(i)} = 0\} \cap \{V_{l'(i)}^{(i)} = 0\}) > 3/4, \label{app:eq:mixintruebins}
            \end{align}
            where the inequality is by \eqref{app:eq:UnivariateConcentrationMix}. For $k \in [m]$, we then denote $l_k^\ast = \argmax_{a \in \{l_k, l_k'\}}\{\mathbb{P}(V_a^{(i)} = 1 \mid \iota^{(i)} = k)\}$ which by \eqref{app:eq:mixintruebins} is such that $\mathbb{P}(V_{l_k^\ast}^{(i)} = 1 \mid \iota^{(i)} = k) > 3/8$. Denoting $\mathcal{L}_m = \{l_1, \hdots, l_m, l_1', \hdots, l_m'\}$, we also have the bound
            \begin{equation} \label{app:eq:mixinfalsebins}
                \mathbb{P}(V_j^{(i)} = 1) \leq \mathbb{P}\left\{\left| \frac{1}{T^*} \sum_{t = 1}^{T^\ast} X^{(i)}_t - \theta_{\iota^{(i)}} \right|\geq \Delta \right\}
                \leq \frac{\pi_{0}}{4}, \quad j \in[N(\Delta)] \setminus \mathcal{L}_m,
            \end{equation}
            with the final inequality following from \eqref{app:eq:UnivariateConcentrationMix}. Then, for $l \in [N(\Delta)]$, we have that  
            \begin{align*}
                p_l := \mathbb{P}(\widetilde{V}_l^{(i)} = 1) = (2\omega_{\alpha/2} - 1)\mathbb{P}(V_l^{(i)} = 1) + (1 - \omega_{\alpha/2}).
            \end{align*}
            Hence, for $k \in [m]$, we have
            \begin{align}
                p_{l_k^\ast}
                &= \mathbb{P}(\widetilde{V}_{l_k^\ast}^{(i)} = 1) \nonumber \\
                &= (2\omega_{\alpha/2} - 1)\mathbb{P}(V_{l_k^\ast}^{(i)} = 1) + (1 - \omega_{\alpha/2}) \nonumber \\
                &= (2\omega_{\alpha/2} - 1)\mathbb{P}(V_{l_k^\ast}^{(i)} = 1 \mid \iota^{(i)} = k)\pi_k
                + (2\omega_{\alpha/2} - 1)\mathbb{P}(V_{l_k^\ast}^{(i)} = 1 \mid \iota^{(i)} \neq k)(1 - \pi_k)
                + (1 - \omega_{\alpha/2}) \nonumber \\
                &> \frac{3(2\omega_{\alpha/2} - 1)\pi_k}{8} + (1 - \omega_{\alpha/2}), \label{app:eq:mixgoodbinvoteprob}
            \end{align}
            where the inequality is by \eqref{app:eq:mixintruebins} and the definition of $l_k^\ast$. Further, for $j \in [N(\Delta)] \setminus \mathcal{L}_m$, we have
            \begin{equation}
                p_j
                = \mathbb{P}(\widetilde{V}_{j}^{(i)} = 1)
                = (2\omega_{\alpha/2} - 1)\mathbb{P}(V_{j}^{(i)} = 1) + (1 - \omega_{\alpha/2})
                \leq \frac{(2\omega_{\alpha/2} - 1)\pi_{0}}{4} + (1 - \omega_{\alpha/2}), \label{app:eq:mixbadbinvoteprob}
            \end{equation}
            where the inequality is by \eqref{app:eq:mixinfalsebins}. Combining \eqref{app:eq:mixgoodbinvoteprob} and \eqref{app:eq:mixbadbinvoteprob}, we have for $k \in [m]$ and $j \in [N(\Delta)] \setminus \mathcal{L}_m$ that
            \begin{equation} \label{app:eq:mixvoteprobsdiff}
                p_{l_k^\ast} - p_j
                \geq \frac{3(2\omega_{\alpha/2} - 1)\pi_k}{8} - \frac{(2\omega_{\alpha/2} - 1)\pi_{0}}{4}
                \geq \frac{(2\omega_{\alpha/2} - 1)\pi_k}{8}
            \end{equation}
            where the first inequality is by \eqref{app:eq:mixgoodbinvoteprob} and \eqref{app:eq:mixbadbinvoteprob}, and the second by the fact that $\pi_k \geq \pi_{0}$ for all $k \in [m]$.
            
            Letting $Z_j = \sum_{i \in N_1} \widetilde{V}^{(i)}_j$ for $j \in [N(\Delta)]$ denote the total privatised votes for the $j$-th sub-interval, we consider the event
            \begin{equation} \label{app:eq:eventBmix}
                A = \bigcap_{k \in [m]}\bigcap_{j \in [N(\Delta)] \setminus \mathcal{L}_m} \{Z_{l_k^\ast} > Z_j\}.
            \end{equation}   
            We have the bound
            \begin{align}
                \mathbb{P}( Z_{l_k^\ast} \leq Z_j) &= \mathbb{P}\left(\sum_{i \in N_1} \big(\widetilde{V}^{(i)}_{l_k^\ast} - \widetilde{V}^{(i)}_j\big) \leq 0\right) \nonumber \\
                &= \mathbb{P}\left(\sum_{i \in N_1} \big(\widetilde{V}^{(i)}_{l_k^\ast} - \widetilde{V}^{(i)}_j - p_{l_k^\ast} + p_j\big) \leq -|N_1|(p_{l_k^\ast} - p_j)\right) \nonumber \\
                &\leq  e^{-|N_1|(p_{l_k^\ast} - p_j)^2/2} \leq e^{-|N_1|\pi_k^2(1/2 - \omega_{\alpha/2})^2/32} \leq e^{-2n\alpha^2\pi_k^2/K}. \label{app:eq:voteprobboundmix}
            \end{align}
            for $K>0$ an absolute constant, where the first inequality follows from Hoeffding's inequality; the second by \eqref{app:eq:mixvoteprobsdiff}; and the third inequality by the fact that $\alpha^2/K \leq (1/2 - \omega_{\alpha/2})^2$ holds for $\alpha \in (0,1]$ for $K$ sufficiently large.
        
            Hence, we have
            \begin{align}
                    \mathbb{P}(A^c)
                    &\leq \sum_{k \in [m]} \sum_{j \in [N(\Delta)] \setminus \mathcal{L}_m} \mathbb{P}( Z_{l_k^\ast} \leq Z_j)
                    \leq m N(\Delta) e^{-2n\alpha^2\pi_k^2/K} \nonumber \\
                    &\leq \frac{3m}{\Delta \wedge 1} e^{-2n\alpha^2\pi_k^2/K}
                    \leq 3m(T^\ast)^{1/2}e^{-2n\alpha^2\pi_k^2/K}
                    \leq 3me^{n\alpha^2\pi_{0}^2/(2K)}e^{-2n\alpha^2\pi_k^2/K} \nonumber \\
                    &\leq 3me^{-3n\alpha^2\pi_{0}^2/(2K)}
                    \leq e^{-n\alpha^2\pi_{0}^2/K}, \label{app:eq:probAmix}  
            \end{align}
            where the first inequality is by the union bound; the second by \eqref{app:eq:voteprobboundmix}; and the third by bounding the covering number using \eqref{app:eq:l2cover}; the fourth and fifth by the values of $\Delta$ and $T^\ast$; and the final two by the assumptions $\pi_k \geq \pi_{0}$, $n\alpha^2 \geq \widetilde{C}m^2\log(n\alpha^2)$ for $\widetilde{C}$ sufficiently large, and that $\pi_{0} \geq \tilde{c}/m$. We note that for the covering number bound we require $\Delta < 1$, so to account for when this fails to hold we take the minimum $\Delta \wedge 1$ which suffices for an upper bound.

            On the event $A$, we have for all $l \in \mathcal{L}_m$ and $j \in [N(\Delta)] \setminus \mathcal{L}_m$ that $Z_l \geq Z_j$. Hence, $j_1^\ast$ as defined in \eqref{secmix:eq:unichosenindex} satisfies $j_1^\ast \in \mathcal{L}_m$ almost surely. Then, as we remove the adjacent sub-interval indices $j_1^\ast + 1$ and $j_1^\ast - 1$ from consideration, the next $j_2^\ast \in \mathcal{L}_m \setminus \{j_1^\ast - 1, j_1^\ast, j_1^\ast + 1\}$ almost surely. Repeating, and ordering so that $j_1^\ast < \dots < j_m^\ast$, we have for all $k \in [m]$ that $j_k^\ast \in \{l_k , l_k'\}$ almost surely. As the distance between the parameter $\theta_k \in I_{l_k}$ and the closest endpoint of $I_{l_k'}$ is at most $\Delta$, expanding the endpoints of the chosen interval $I_{j^\ast}$ by $2\Delta$ as in \eqref{secmix:eq:uniMeanIntervalExpand} ensures that $\min\{ \theta_k - \tilde{L}_k, \tilde{U}_k - \theta\} \geq \Delta$, equivalently that $\theta_k \in [\tilde{L}_k +\Delta, \tilde{U}_k-\Delta]$, regardless of the value of $j_k^\ast$.
        
            \medskip
            \noindent\textbf{Step 2: Proportion Estimation.}
            We now consider the estimation of the mixture proportions using the users $i \in N_2$. In preparation for what follows, we write $\overline{X^{(i)}} = (T^\ast)^{-1}\sum_{t = 1}^{T^\ast} X_t^{(i)}$ and for readability denote $j^\ast(i) = j_{\iota^{(i)}}^\ast$. We then define the event $B_i = \{\overline{X^{(i)}} \in \tilde{I}_{j^\ast(i)}\}$, that is, the event the $i$-th user's sample mean lies in the sub-interval which corresponds to their mixture component.
            
            On the event $A$ from the previous step, the chosen bins indexed by $j_{1:m}^\ast$ contain the true parameters $\theta_{1:m}$, and so $\overline{X^{(i)}}$, concentrating around the true mean $\theta_{\iota^{(i)}}$ corresponding to its mixture component, should lie in the interval $\tilde{I}_{j^\ast(i)}$ with high probability. Indeed, denoting $B = \cap_{i \in N_2} B_i$, we have that
            \begin{align}
                \mathbb{P}(B^c)
                &\leq \mathbb{P}(A^c) + \mathbb{P}(B^c \cap A)
                \leq \mathbb{P}(A^c) + \sum_{i \in N_2} \mathbb{P}(B_i^c \cap A) \nonumber \\
                &\leq e^{-n\alpha^2\pi_{0}^2/K} + \sum_{i \in N_2} \mathbb{P}(\{\overline{X^{(i)}} \notin \tilde{I}_{j^\ast(i)}\} \cap A ) \nonumber \\
                &\leq e^{-n\alpha^2\pi_{0}^2/K} + \sum_{i \in N_2} \mathbb{P}(\{|\overline{X^{(i)}} - \theta_{\iota^{(i)}}| \geq \Delta \} \cap A)\nonumber \\
                &\leq e^{-n\alpha^2\pi_{0}^2/K} + \frac{2}{n T^\ast \alpha^2} \label{app:eq:secondfoldconcmix}
            \end{align}
            where the first two equalities are by the union bound; the third by \eqref{app:eq:probAmix}; the fourth by the tower property and the fact that on $A$, the chosen bins given by the indices $j_{1:m}^\ast$ contain the true parameter with at least $\Delta$ buffer from the endpoints; and the final inequality by \eqref{app:eq:UnivariateConcentrationMix}.

            We now provide an important bound on a quantity depending on the mixture proportion estimates $\hat{\pi}_k$ as in \eqref{secmix:eq:mixpropests} which will be vital when analysing the MSE of the final estimator. We consider
            \begin{align*}
                \bigg| \mathbb{E}\bigg[ \frac{\pi_k - \hat{\pi}_k}{\hat{\pi}_k} &\mathbbm{1}\{\hat{\pi}_k \geq \pi_{0}/2\} \biggm| j_{1:m}^\ast \bigg]\bigg| \\
                &\leq \frac{2}{\pi_{0}} \mathbb{E}[|\pi_k - \hat{\pi}_k|\mathbbm{1}\{\hat{\pi}_k \geq \pi_{0}/2\} \mid j_{1:m}^\ast]  \\
                &\leq \frac{2}{\pi_{0}} \mathbb{E}[|\pi_k - \hat{\pi}_k| \mid j_{1:m}^\ast]
                + \frac{4}{\pi_{0}} \mathbb{P}(\hat{\pi}_k < \pi_{0}/2 \mid j_{1:m}^\ast) \\
                &\leq \frac{2C_1\sqrt{\log(n\alpha^2)}}{\pi_{0}\sqrt{n\alpha^2}}
                + \frac{4}{\pi_{0}} \mathbb{P}\big(|\hat{\pi}_k - \pi_k| \geq C_1\sqrt{\log(n\alpha^2)/(n\alpha^2)} \bigm| j_{1:m}^\ast\big)
                + \frac{4}{\pi_{0}} \mathbb{P}(\hat{\pi}_k < \pi_{0}/2 \mid j_{1:m}^\ast)
            \end{align*}
            where the second and third inequalities uses the fact that $|\hat{\pi}_k - \pi_k| \leq 2$. Hence, we can bound the following quantity
            \begin{align}
                &\mathbb{E}\bigg[ \mathbb{E}\bigg[ \frac{\pi_k - \hat{\pi}_k}{\hat{\pi}_k} \mathbbm{1}\{\hat{\pi}_k \geq \pi_{0}/2\} \biggm| j_{1:m}^\ast \bigg]^2 \bigg] \nonumber \\
                &\leq \mathbb{E}\bigg[ \frac{12C_1^2\log(n\alpha^2)}{\pi_0^2 n\alpha^2} 
                + \frac{48}{\pi_0^2} \mathbb{P}\big(|\hat{\pi}_k - \pi_k| \geq C_1\sqrt{\log(n\alpha^2)/(n\alpha^2)} \bigm| j_{1:m}^\ast \big)^2
                + \frac{48}{\pi_0^2} \mathbb{P}(\hat{\pi}_k < \pi_{0}/2 \mid j_{1:m}^\ast)^2 \bigg] \nonumber \\
                &\leq \frac{12C_1^2\log(n\alpha^2)}{\pi_0^2 n\alpha^2}
                + \frac{48}{\pi_0^2} \mathbb{P}\big(|\hat{\pi}_k - \pi_k| \geq C_1\sqrt{\log(n\alpha^2)/(n\alpha^2)} \big)
                + \frac{48}{\pi_0^2} \mathbb{P}(\hat{\pi}_k < \pi_{0}/2 ) \nonumber \\
                &\leq \frac{12C_1^2\log(n\alpha^2)}{\pi_0^2 n\alpha^2}+ (I) + (II) \label{app:eq:probestratiobound}
            \end{align}
            where the first inequality is by the Cauchy--Schwarz inequality, and the second using the fact that the probabilities are bounded by one and then applying the law of total probability. We now control the terms $(I)$ and $(II)$.

            Starting with $(I)$, first observe that on $A \cap B$ the raw votes $V^{(i)}_k$ defined in \eqref{secmix:eq:mixturebinestimator} are equal to $\mathbbm{1}\{\iota^{(i)} = k\}$ for all $i \in N_2, k \in [m]$. Hence, denoting $\Breve{V}_k^{(i)}$ the output of applying the unary encoding mechanism as in \eqref{secmix:eq:uniGRR} to $\mathbbm{1}\{\iota^{(i)} = k\}$, we write
            \begin{equation*}
                \breve{\pi}_k = \bigg( \frac{\sum_{i \in N_2} \breve{V}_k^{(i)}}{|N_2|} - \frac{1}{e^{\alpha/2} + 1} \bigg) \frac{e^{\alpha/2} + 1}{e^{\alpha/2} - 1}.
            \end{equation*}
            Noting that the $\breve{V}_k^{(i)}$ are independent across $i \in N_2$ and take values in $\{0, 1\}$, $\breve{\pi}_k$ is $25/(4|N_2|\alpha^2)$-sub-Gaussian \citep[e.g.][Exercise~2.4]{Wainwright:2019}, where we use the fact that $(e^{\alpha/2} + 1)/(e^{\alpha/2} - 1) \leq 5/\alpha$ for $\alpha \in (0,1]$. Further, we have $\mathbb{E}[\breve{\pi}_k] = \pi_k$. Hence, we have for $x > 0$ the concentration inequality
            \begin{align}
                \mathbb{P}(|\hat{\pi}_k - \pi_k| \geq x)
                &\leq \mathbb{P}(|\breve{\pi}_k - \pi_k| \geq x)
                + \mathbb{P}(A^c) + \mathbb{P}(B^c) \nonumber \\
                &\leq e^{-|N_2|\alpha^2 x^2/100} + 2e^{-n\alpha^2\pi_{0}^2/K} + \frac{2}{nT^\ast\alpha^2}, \label{app:eq:propestconcmix}
            \end{align}
            where the first inequality is by the union bound; the second as $x > 0$ requires the indicator to be non-zero for the second probability term to not vanish; the third by the union bound and the fact an indicator function is upper bounded by one; and the final by \eqref{app:eq:UnivariateConcentrationMix}, \eqref{app:eq:secondfoldconcmix} and the fact that $\breve{\pi}_k$ is $25/(4|N_2|\alpha^2)$-sub-Gaussian.

            To bound the term $(II)$, we note
            \begin{align}
                \mathbb{P}(\hat{\pi}_k < \pi_{0}/2)
                \leq \mathbb{P}(|\hat{\pi}_k - \pi_k| \geq \pi_{0}/2)
                &\leq e^{-|N_2|\alpha^2 \pi_0^2/400} + 2e^{-n\alpha^2\pi_{0}^2/K} + \frac{2}{nT^\ast\alpha^2} \nonumber \\
                &\leq 3e^{-n\alpha^2\pi_{0}^2/K} + \frac{2}{nT^\ast\alpha^2} \label{app:eq:probestsmallbound}
            \end{align}
            where the first inequality is by the fact that $\pi_k \geq \pi_{0}$; the second by \eqref{app:eq:propestconcmix}; and the final by taking $K > 0$ sufficiently large.

            Finally, combining \eqref{app:eq:probestratiobound}, \eqref{app:eq:propestconcmix} and \eqref{app:eq:probestsmallbound} yields
            \begin{align}
                \mathbb{E}\bigg[ \mathbb{E}\bigg[ \frac{\pi_k - \hat{\pi}_k}{\hat{\pi}_k} \mathbbm{1}\{\hat{\pi}_k \geq \pi_{0}/2\} \biggm| j_{1:m}^\ast \bigg]^2 \bigg]
                &\leq \frac{12C_1^2\log(n\alpha^2)}{\pi_0^2 n\alpha^2}
                + \frac{48}{\pi_{0}^2}\bigg(\frac{1}{n\alpha^2} + 5e^{-n\alpha^2\pi_{0}^2/K} + \frac{4}{nT^\ast\alpha^2}\bigg) \nonumber \\
                &\lesssim \frac{m^2\log(n\alpha^2)}{n\alpha^2}, \label{app:eq:probestrationboundfinal}
            \end{align}
            where in the final inequality we used the fact that $\pi_{0} \gtrsim 1/m$ and $n\alpha^2 \gtrsim m^2\log(n\alpha^2)$.

            \medskip
            \noindent\textbf{Step 3: Final Estimator.} 
            As the mean-squared-error decomposes co-ordinate wise, it suffices to consider an estimator of a single component $\hat{\theta}_k$ as in~\eqref{secmix:eq:FinalEstimatorComponent}, for which we analyse the MSE as follows
            \begin{align}
                &\mathbb{E}[|\hat{\theta}_k - \theta_k|^2] \nonumber \\
                &= \mathbb{E}\big[\mathbb{E}[|\hat{\theta}_k - \theta_k|^2 \mid j_{1:m}^\ast](\mathbbm{1}\{A\} + \mathbbm{1}\{A^c\}) \big] \nonumber \\
                &\leq \mathbb{E}\big[ \{ \mathbb{E}[\hat{\theta}_k \mid j_{1:m}^\ast] - \theta_k\}^2\mathbbm{1}\{A\}
                + \mathrm{Var}(\hat{\theta}_k \mid j_{1:m}^\ast)\mathbbm{1}\{A\} \big]
                + 4\mathbb{P}(A^c) \nonumber \\
                &= \mathbb{E}\bigg[ \{ \mathbb{E}[\hat{\theta}_k \mid j_{1:m}^\ast] - \theta_k\}^2\mathbbm{1}\{A\}
                + \mathrm{Var}\bigg(\bigg(\frac{\sum_{i \in N_3} r_k{(i)} \tilde{\theta}^{(i)} }{|N_3|\hat{\pi}_k}\bigg)\mathbbm{1}\{\hat{\pi}_k \geq \pi_{0}/2 \} \biggm| j_{1:m}^\ast \bigg)\mathbbm{1}\{A\} \bigg]
                + 4\mathbb{P}(A^c), \label{app:eq:ComponentMSEMix}
            \end{align}
            where in the final line we use the fact that $M_k$ is constant conditional on $j_{1:m}^\ast$.

            In preparation for what follows, for $i \in N_3$ we define the event $D_i = \{\overline{X^{(i)}} \in \tilde{I}_{j^\ast(i)}\}$ and $D = \cap_{i \in N_3} D_i$. These events are analogous to $B_i$ for $i \in |N_2|$ and $B$ previously defined, and so via the same argument as \eqref{app:eq:secondfoldconcmix} we obtain
            \begin{equation} \label{app:eq:thridfoldconcmix}
                \mathbb{P}(D^c) \leq e^{-n\alpha^2\pi_{0}^2/K} + \frac{2}{n T^\ast \alpha^2}.
            \end{equation}
            
            For the inner expectation term in \eqref{app:eq:ComponentMSEMix}, we see
            \begin{align*}
                \{ &\mathbb{E}[\hat{\theta}_k \mid j_{1:m}^\ast] - \theta_k\}^2 \mathbbm{1}\{A\} \\
                &= \bigg\{(M_k - \theta_k)
                + \mathbb{E}\bigg[\frac{1}{|N_3|\hat{\pi}_k} \sum_{i \in N_3} \mathbbm{1}\{k^{(i)} = k\}\Pi_{[-3\Delta, 3\Delta]}(\overline{X^{(i)}} - M_k) \mathbbm{1}\{\hat{\pi}_k \geq \pi_{0}/2 \} \biggm| j_{1:m}^\ast\bigg] \bigg\}^2 \mathbbm{1}\{A\} \\
                &\leq \bigg\{(M_k - \theta_k)
                + \mathbb{E}\bigg[\frac{1}{|N_3|\hat{\pi}_k} \sum_{i \in N_3} \mathbbm{1}\{\iota^{(i)} = k\}\Pi_{[-3\Delta, 3\Delta]}(\overline{X^{(i)}} - M_k) \mathbbm{1}\{D\}\mathbbm{1}\{\hat{\pi}_k \geq \pi_{0}/2 \} \biggm| j_{1:m}^\ast\bigg] \\
                &\hspace{1cm}+ \mathbb{E}\bigg[\frac{1}{|N_3|\hat{\pi}_k} \sum_{i \in N_3} \mathbbm{1}\{k^{(i)} = k\}\Pi_{[-3\Delta, 3\Delta]}(\overline{X^{(i)}} - M_k) \mathbbm{1}\{D^c\}\mathbbm{1}\{\hat{\pi}_k \geq \pi_{0}/2 \} \biggm| j_{1:m}^\ast\bigg] \bigg\}^2 \mathbbm{1}\{A\} \\
                &\leq 2\bigg\{(M_k - \theta_k)
                + \mathbb{E}\bigg[\frac{\pi_k}{\hat{\pi}_k} (\overline{X^{(n)}} - M_k) \mathbbm{1}\{D\}\mathbbm{1}\{\hat{\pi}_k \geq \pi_{0}/2 \} \biggm| \iota^{(n)} = k, j_{1:m}^\ast\bigg] \bigg\}^2 \mathbbm{1}\{A\} \\
                &\hspace{1cm}+ \frac{72\Delta^2}{\pi_{0}^2}\mathbb{P}(D^c \mid j_{1:m}^\ast)^2 \mathbbm{1}\{A\},
            \end{align*}
            where the equality is by the definition of $\hat{\theta}_k$ in \eqref{secmix:eq:FinalEstimatorComponent} noting that $M_k$ is constant conditional on $j_{1:m}^\ast$, that the Laplace noise is mean zero, and that the product of Rademacher hashes $r_{k^{(i)}}^{(i)}r_k^{(i)}$ is mean zero except when $k^{(i)}$ and $k$ coincide; the first inequality by decomposing the expectation on the event $D$ via indicators and noting that on $D$ and for $j_{1:m}^\ast$ chosen under the event $A$, $k^{(i)} = \iota^{(i)}$ almost surely; and the second inequality by the Cauchy--Schwarz inequality, noting for the first expectation term, on $D$ and $\iota^{(i)} = k$, that $\{\overline{X^{(i)}} \in \tilde{I}_{j^\ast(i)}\}$ implies $|\overline{X^{(i)}} - M_k| \leq 3\Delta$ almost surely, and bounding the second expectation term by the fact that $\hat{\pi}_k \geq \pi_{0}/2$ on the indicator, and that the truncated summands take values in $[-3\Delta, 3\Delta]$.

            Continuing, we have
            \begin{align*}
                \{ &\mathbb{E}[\hat{\theta}_k \mid j_{1:m}^\ast] - \theta_k\}^2 \mathbbm{1}\{A\}   \\
                &\leq 2\bigg\{(M_k - \theta_k)
                + \mathbb{E}\bigg[\frac{\pi_k}{\hat{\pi}_k} (\overline{X^{(n)}} - M_k)\mathbbm{1}\{\hat{\pi}_k \geq \pi_{0}/2 \} \biggm| \iota^{(n)} = k, j_{1:m}^\ast\bigg]   \\
                &\hspace{1cm}- \mathbb{E}\bigg[\frac{\pi_k}{\hat{\pi}_k} (\overline{X^{(n)}} - M_k) \mathbbm{1}\{D^c\} \mathbbm{1}\{\hat{\pi}_k \geq \pi_{0}/2 \} \} \biggm| \iota^{(n)} = k, j_{1:m}^\ast\bigg]\bigg\}^2 \mathbbm{1}\{A\}
                + \frac{72\Delta^2}{\pi_{0}^2}\mathbb{P}(D^c \mid j_{1:m}^\ast)^2 \mathbbm{1}\{A\}   \\
                &= 2\bigg\{(M_k - \theta_k)
                + \mathbb{E}[(\overline{X^{(n)}} - M_k) \mid \iota^{(n)} = k, j_{1:m}^\ast]
                - \mathbb{E}[(\overline{X^{(n)}} - M_k)\mathbbm{1}\{\hat{\pi}_k < \pi_{0}/2 \} \mid \iota^{(n)} = k, j_{1:m}^\ast]  \\
                &\hspace{1cm}+ \mathbb{E}\bigg[\frac{\pi_k - \hat{\pi}_k}{\hat{\pi}_k} (\overline{X^{(n)}} - M_k)\mathbbm{1}\{\hat{\pi}_k \geq \pi_{0}/2 \} \biggm| \iota^{(n)} = k, j_{1:m}^\ast\bigg]  \\
                &\hspace{1cm}- \mathbb{E}\bigg[\frac{\pi_k}{\hat{\pi}_k} (\overline{X^{(n)}} - M_k) \mathbbm{1}\{D^c\} \mathbbm{1}\{\hat{\pi}_k \geq \pi_{0}/2 \} \} \biggm| \iota^{(n)} = k, j_{1:m}^\ast\bigg]\bigg\}^2 \mathbbm{1}\{A\}
                + \frac{72\Delta^2}{\pi_{0}^2}\mathbb{P}(D^c \mid j_{1:m}^\ast)^2 \mathbbm{1}\{A\}  \\
                &\leq 6(\theta_k - M_k)^2\mathbb{E}[\mathbbm{1}\{\hat{\pi}_k < \pi_{0}/2 \} \mid j_{1:m}^\ast]^2 \mathbbm{1}\{A\}
                + 6(\theta_k - M_k)^2\mathbb{E}\bigg[\frac{\pi_k - \hat{\pi}_k}{\hat{\pi}_k} \mathbbm{1}\{\hat{\pi}_k \geq \pi_{0}/2 \} \biggm| j_{1:m}^\ast\bigg]^2 \mathbbm{1}\{A\}  \\
                &\hspace{1cm}+ 6\mathbb{E}\bigg[\frac{\pi_k}{\hat{\pi}_k} (\overline{X^{(n)}} - M_k) \mathbbm{1}\{D^c\} \mathbbm{1}\{\hat{\pi}_k \geq \pi_{0}/2 \} \} \biggm| \iota^{(n)} = k, j_{1:m}^\ast\bigg]^2 \mathbbm{1}\{A\}
                + \frac{72\Delta^2}{\pi_{0}^2}\mathbb{P}(D^c \mid j_{1:m}^\ast)^2 \mathbbm{1}\{A\}  \\
                &\leq 6(\theta_k - M_k)^2\mathbb{E}[\mathbbm{1}\{\hat{\pi}_k < \pi_{0}/2 \} \mid j_{1:m}^\ast]^2 \mathbbm{1}\{A\}
                + 6(\theta_k - M_k)^2\mathbb{E}\bigg[\frac{\pi_k - \hat{\pi}_k}{\hat{\pi}_k} \mathbbm{1}\{\hat{\pi}_k \geq \pi_{0}/2 \} \biggm| j_{1:m}^\ast\bigg]^2 \mathbbm{1}\{A\}  \\
                &\hspace{1cm}+ \frac{96}{\pi_{0}^2}\mathbb{P}(D^c | j_{1:m}^\ast)^2 \mathbbm{1}\{A\}
                + \frac{72\Delta^2}{\pi_{0}^2}\mathbb{P}(D^c \mid j_{1:m}^\ast)^2 \mathbbm{1}\{A\}  \\
                &\lesssim \Delta^2 m^2 \bigg(e^{-c_1n\alpha^2/m^2} + \frac{\log(n\alpha^2)}{n\alpha^2}\bigg)\mathbbm{1}\{A\} 
            \end{align*}
            for $c_1 > 0$ some absolute constant, where the equality and first inequality come from rewriting the indicator functions using the complement events; the second inequality by noting that $\mathbb{E}[(\overline{X^{(n)}} - M_k) \mid \iota^{(n)} = k, j_{1:m}^\ast] = \theta_k - M_k$, that $\hat{\pi}_k$ and $(\overline{X^{(n)}} - M_k)$ are independent given $j_{1:m}^\ast$, and applying the Cauchy--Schwarz inequality; and the final by noting $\pi_k \leq 1$ and $(\overline{X^{(n)}} - M_k) \leq 2$.

            Finally, we bound the expectation of this conditional bias, yielding
            \begin{align}
                \mathbb{E}\big[\{ &\mathbb{E}[\hat{\theta}_k \mid j_{1:m}^\ast] - \theta_k\}^2 \mathbbm{1}\{A\} \big] \nonumber \\
                &\lesssim \Delta^2 \mathbb{E}[\mathbb{P}(\hat{\pi}_k < \pi_{0}/2 \mid j_{1:m}^\ast)^2 \mathbbm{1}\{A\}]
                + \Delta^2\mathbb{E}\bigg[ \mathbb{E}\bigg[\frac{\pi_k - \hat{\pi}_k}{\hat{\pi}_k} \mathbbm{1}\{\hat{\pi}_k \geq \pi_{0}/2 \} \biggm| j_{1:m}^\ast\bigg]^2 \mathbbm{1}\{A\} \bigg] \nonumber \\
                &\hspace{2cm}+ m^2\mathbb{E}[\mathbb{P}(D^c | j_{1:m}^\ast)^2 \mathbbm{1}\{A\}]
                + \Delta^2m^2\mathbb{E}[\mathbb{P}(D^c \mid j_{1:m}^\ast)^2 \mathbbm{1}\{A\}] \nonumber \\
                &\leq \Delta^2 \mathbb{P}(\hat{\pi}_k < \pi_{0}/2)
                + \Delta^2\mathbb{E}\bigg[ \mathbb{E}\bigg[\frac{\pi_k - \hat{\pi}_k}{\hat{\pi}_k} \mathbbm{1}\{\hat{\pi}_k \geq \pi_{0}/2 \} \biggm| j_{1:m}^\ast\bigg]^2 \mathbbm{1}\{A\} \bigg]
                + m^2\mathbb{P}(D^c)
                + \Delta^2m^2\mathbb{P}(D^c) \nonumber \\
                &\lesssim m^2(1 \vee \Delta^2) e^{-c_1n\alpha^2/m^2} + \frac{\Delta^2m^2\log(n\alpha^2)}{n\alpha^2} \label{app:eq:mixMSEexpectation}
            \end{align}
            where the first inequality uses the facts that $|\theta_k - M_k| \leq 3\Delta$ on $A$, and $\pi_{0} \geq \tilde{c}/m$; the second by bounding the square of the probability terms by the non-squared probability and applying the tower property; and the final by \eqref{app:eq:probestsmallbound}, \eqref{app:eq:probestrationboundfinal} and \eqref{app:eq:thridfoldconcmix}, and using the assumptions $\pi_k \geq \pi_{0} \geq \tilde{c}/m$, $n\alpha^2 \geq \widetilde{C}m^2\log(n\alpha^2)$ to simplify and collect together similar terms.
        
            We now focus on the variance term in \eqref{app:eq:ComponentMSEMix}, obtaining
            \begin{align}
                &\mathrm{Var}\bigg(\bigg(\frac{\sum_{i \in N_3} r_k^{(i)} \tilde{\theta}^{(i)} }{|N_3|\hat{\pi}_k}\bigg)\mathbbm{1}\{\hat{\pi}_k \geq \pi_{0}/2 \} \biggm| j_{1:m}^\ast \bigg)\mathbbm{1}\{A\} \nonumber \\
                &= \mathrm{Var}\bigg(\mathbb{E}\bigg[\frac{\sum_{i \in N_3} r_k^{(i)} \tilde{\theta}^{(i)} }{|N_3|\hat{\pi}_k} \biggm| \hat{\pi}_k, j_{1:m}^\ast \bigg]\mathbbm{1}\{\hat{\pi}_k \geq \pi_{0}/2 \} \biggm| j_{1:m}^\ast \bigg)\mathbbm{1}\{A\} \nonumber \\
                &\hspace{1cm}+ \mathbb{E}\bigg[ \mathrm{Var}\bigg(\frac{\sum_{i \in N_3} r_k^{(i)} \tilde{\theta}^{(i)} }{|N_3|\hat{\pi}_k} \biggm| \hat{\pi}_k, j_{1:m}^\ast \bigg) \mathbbm{1}\{\hat{\pi}_k \geq \pi_{0}/2 \} \biggm| \hat{\pi}_k, j_{1:m}^\ast \bigg]\mathbbm{1}\{A\}  \nonumber\\
                &= \mathrm{Var}\bigg(\frac{1}{|N_3| \hat{\pi}_k}  \sum_{i \in N_3} \mathbb{E}\bigg[ \mathbbm{1}\{k^{(i)} = k\}\Pi_{[-3\Delta, 3\Delta]}(\overline{X^{(i)}} - M_k) \nonumber \\
                &\hspace{6cm}+ \frac{6\Delta}{\alpha}r_k^{(i)}\ell_i \biggm| \hat{\pi}_k, j_{1:m}^\ast \bigg] \mathbbm{1}\{\hat{\pi}_k \geq \pi_{0}/2 \} \biggm| j_{1:m}^\ast \bigg)\mathbbm{1}\{A\} \nonumber \\
                &\hspace{1cm}+ \mathbb{E}\bigg[ \frac{1}{|N_3|^2\hat{\pi}_k^2}\sum_{i \in N_3}\mathrm{Var}\bigg( r_k^{(i)}r_{k^{(i)}}^{(i)}\Pi_{[-3\Delta, 3\Delta]}(\overline{X^{(i)}} - M_k) \nonumber \\
                &\hspace{6cm}+ \frac{6\Delta}{\alpha}r_k^{(i)}\ell_i \biggm| \hat{\pi}_k, j_{1:m}^\ast \bigg) \mathbbm{1}\{\hat{\pi}_k \geq \pi_{0}/2 \} \biggm| \hat{\pi}_k, j_{1:m}^\ast \bigg]\mathbbm{1}\{A\} \nonumber \\
                &\leq \frac{36\Delta^2}{|N_3|\pi_{0}^2}\mathbbm{1}\{A\}
                + \mathbb{E}\bigg[\frac{1}{|N_3|\hat{\pi}_k^2} \bigg(9\Delta^2 + \frac{72\Delta^2}{\alpha^2}\bigg) \mathbbm{1}\{\hat{\pi}_k \geq \pi_{0}/2 \} \biggm| j_{1:m}^\ast \bigg]\mathbbm{1}\{A\}
                \lesssim \frac{\Delta^2m^2}{n\alpha^2}\mathbbm{1}\{A\}, \label{app:eq:mixMSEvariance}
            \end{align}
            where the first equality is by the law of total variance so that we may condition on $\hat{\pi}_k$; the second inequality by that fact that the summands are independent conditional on $\hat{\pi}_k$ and $j_{1:m}^\ast$; the first inequality almost surely by using the bound $\hat{\pi}_k \geq \pi_{0}/2$ on the indicator, the fact that the truncated values $\Pi_{[-3\Delta, 3\Delta]}(\overline{X^{(i)}} - M_k)$ have variance bounded by $9\Delta^2$, and the independence of the Laplace noise from the rest of the randomness; and the final inequality holds almost surely by collecting the terms and the fact that $\pi_{0} \geq \tilde{c}/m$.

            Hence, combining \eqref{app:eq:ComponentMSEMix}, \eqref{app:eq:mixMSEexpectation} and \eqref{app:eq:mixMSEvariance}, we obtain
            \begin{align*}
                \mathbb{E}[\|\hat{\theta} - \theta\|_2^2]
                = \sum_{k = 1}^m \mathbb{E}[|\hat{\theta}_k - \theta_k|^2]
                &\lesssim \frac{\Delta^2m^3\log(n\alpha^2)}{n\alpha^2} + (1 \vee \Delta^2)m^3 e^{-c_1n\alpha^2/m^2} + m\mathbb{P}(A^c) \\
                &\lesssim \frac{m^3\log(nT^\ast \alpha^2)\log(n\alpha^2)}{n T^\ast \alpha^2} + \frac{m^3\log(nT^\ast\alpha^2)}{T^\ast} e^{-c_2n\alpha^2/m^2} + m^3e^{-c_2n\alpha^2/m^2} \\
                &\lesssim \frac{m^3\log(nT^\ast \alpha^2)\log(n\alpha^2)}{n T^\ast \alpha^2},
            \end{align*}
            for $c_2 > 0$ some absolute constant, where the second inequality is by \eqref{app:eq:probAmix} and the value of $\Delta$ given in \eqref{secmix:eq:deltaval}, and the third inequality uses the assumption $n\alpha^2 \geq \widetilde{C}m^2\log(n\alpha^2)$ for some sufficiently large constant $\widetilde{C}$ to absorb the pre-factor of $m^3$ and $\log(n\alpha^2)$ into the exponent, and then noting the exponential term is always of smaller order than the polynomial one.

            It remains to consider the form of this error in the two regimes of $T$. In the case $T \leq e^{n\alpha^2/(Km^2)}$, we have $T^\ast = T$ which immediately gives the desired rate. In the case $T > e^{n\alpha^2/(Km^2)}$, we have $T^\ast = e^{n\alpha^2/(Km^2)}$, which gives
            \begin{align*}
                \mathbb{E}[ \| \hat{\theta} -\theta \|^2 ]
                &\lesssim \frac{m^3\log(n\alpha^2)}{n\alpha^2}\{\log(n\alpha^2) + \log(T^\ast)\} e^{-n\alpha^2/(Km^2)} \\
                &\lesssim \frac{m^3\log(n\alpha^2)}{n\alpha^2}\bigg\{\log(n\alpha^2) + \frac{n\alpha^2}{K\pi_{0}}\bigg\} e^{-n\alpha^2/(Km^2)} \\
                &\lesssim e^{-c_3n\alpha^2/(Km^2)}
            \end{align*}
            where $c_3>0$ is an absolute constant; the second inequality comes from substituting in the value $e^{-n\alpha^2/(Km^2)}$ for $T^\ast$; and the last inequality by the fact that $n\alpha^2 \geq \widetilde{C} m^2\log(n\alpha^2)$ for some sufficiently large $\widetilde{C}$. This completes the proof
        \end{proof}
\section{Proof of User-Level Lower Bounds} \label{app:sec:Lower}

\subsection{Preliminary results} 
    To prove the lower bounds throughout this section we will use Assouad's method \citep[see, e.g.][]{Yu:1997}, in tandem with the results of \cite{Acharya:2022}. While Assouad's method reduces the problem to bounding average pairwise TV distances between privatised versions of carefully constructed distributions, these latter results allow us to bound these TV distances uniformly over all suitable privacy mechanisms. For completeness, we will state a version of Assouad's method and the required results of \cite{Acharya:2022} which will be used. We first introduce them in the item-level case before generalising to the user-level case. For convenience, throughout this section we assume that $\mathcal{X}$ and $\mathcal{Z}$ are Euclidean spaces and that an arbitrary privacy mechanism $Q(\cdot \mid x, z)$, for $x \in \mathcal{X}$ and $z \in \mathcal{Z}^i$ for some $i \in \mathbb{N}$, has a density with respect to Lebesgue measure which we denote via $q(\cdot \mid x, z)$. More generally, one can argue that by the definition of $\alpha$-LDP that the privacy mechanisms for differing inputs are all absolutely continuous with respect to each other, and so have a density with respect to some common dominating measure.
    
    For some $k \in \mathbb{N}$, denote by $\mathcal{V} = \{-1, 1\}^k$ the hypercube and let $\{P_v : v \in \mathcal{V}\}$ be a family of distributions indexed by this hypercube. For each $v \in \mathcal{V}$, we write $\theta_v = \theta(P_v)$.  The collection $\{ P_v : v \in \mathcal{V} \}$ is $2\varrho$-Hamming separated with respect to the loss $\Phi \circ \rho$, if for all $v, v' \in \mathcal{V}$ we have that
    \begin{align}
        \Phi\left( \frac{1}{2}\rho(\theta_v, \theta_{v'}) \right) \geq 2\varrho \sum_{j=1}^k \mathbbm{1} \{v_j \neq v'_j \}. \label{app:eq:HammingSeparation}
    \end{align}

    Note the factor of $1/2$ inside $\Phi(\cdot)$; as is common in some formulations of Assouad's lemma, when we have $\rho$ as the $\ell_p$ norm and $\Phi(\cdot) = (\cdot)^p$, we have a factor of $2^{-p}$.

    For each $j \in [k]$, we define the mixture distributions
    \begin{equation}
        P_{+j}^n = \frac{1}{2^{k-1}} \sum_{v \in \mathcal{V} : v_j = 1} P_v^{\otimes n} \quad \mbox{and} \quad P_{-j}^n = \frac{1}{2^{k-1}} \sum_{v \in \mathcal{V} : v_j = -1} P_v^{\otimes n}.  \label{app:eq:mixtures}
    \end{equation}

    \begin{lemma}[Assouad's Method] \label{app:lem:Assouad}
        Assuming that \eqref{app:eq:HammingSeparation} holds, we have that
        \begin{align*}
            \mathcal{R}_{n,1,\infty} (\theta(\mathcal{P}), \Phi \circ \rho) \geq \varrho \sum_{j=1}^k \{1 - D_{\mathrm{TV}}(P_{+j}^n, P_{-j}^n)\},
        \end{align*}
        where $\mathcal{R}_{n,1,\infty}$ is the non-private item-level minimax risk defined in \Cref{sec-general-setup}.
    \end{lemma}
    \noindent
    Assouad's method is often presented in slightly differing forms and the above is a similar statement to those in \cite{Yu:1997} and \citet[Lemma~2.12]{Tsybakov:2009}.
    
    For the distribution $P_v$ and any $\alpha$-LDP mechanism $Q$, we denote the distribution of the private outcomes for $n$ users by $M_v^n(S) = \int Q(S \mid x^{(1:n)}) \diff P_v^{\otimes n}(x^{(1:n)})$ and the analogous private mixtures, for $j \in [k]$,
    \begin{align} \label{app:eq:AssouadPrivateMarginal-1}
        M_{+j}^n(S) = \frac{1}{2^{k-1}} \sum_{v \in \mathcal{V} : v_j = 1} M_v^n(S) = \int Q(S \mid x^{(1:n)}) \diff P_{+j}^n(x^{(1:n)})
    \end{align}
    and
    \begin{align} \label{app:eq:AssouadPrivateMarginal-2}
        M_{-j}^n(S) = \frac{1}{2^{k-1}} \sum_{v \in \mathcal{V} : v_j = -1} M_v^n(S) = \int Q(S \mid x^{(1:n)}) \diff P_{-j}^n(x^{(1:n)})
    \end{align}
    
    Applying Lemma~\ref{app:lem:Assouad} we then obtain that
    \begin{align*}
        \mathcal{R}_{n, 1, \alpha}(\theta(\mathcal{P}), \Phi \circ \rho) \geq \varrho \inf_{Q \in \mathcal{Q}_\alpha} \sum_{j=1}^k \{1 - D_{\mathrm{TV}}(M_{+j}^n, M_{-j}^n)\},
    \end{align*}
    where $\mathcal{R}_{n, 1, \alpha}$ is the item-level private risk defined in \Cref{sec-general-setup}, and it remains to bound these TV distances uniformly over $Q \in \mathcal{Q}_\alpha$.

    To this end, we now introduce the framework of \cite{Acharya:2022}. For the considered family of distributions $\{P_v : v \in \mathcal{V}\}$ as in the earlier paragraphs, we first introduce an assumption. For a vector $u \in \{-1, 1\}^k$ and $i \in [k]$, we define $u^{\oplus i}$ to be the vector with co-ordinates,
    \begin{equation*}
        (u^{\oplus i})_j =
        \begin{cases}
            u_j, & j \neq i, \\
            -u_j, &\text{otherwise}.
        \end{cases}
    \end{equation*}

    The first step of the method of \cite{Acharya:2022} is to bound the average pairwise TV~distances by the sum of $n$-many local discrepancies over each perturbation dimension $i \in [k]$ in terms of a Hellinger distance. 

    \begin{lemma}[\citealp{Acharya:2022}] \label{app:lem:Acharya}
        For $M_{+j}^n$ and $M_{-j}^n$ as defined in \eqref{app:eq:AssouadPrivateMarginal-1} and \eqref{app:eq:AssouadPrivateMarginal-2}, it holds that
        \begin{equation} \label{app:eq:AcharyaLemEq}
            \bigg(\frac{1}{k} \sum_{j=1}^k D_{\mathrm{TV}}(M_{+j}^n, M_{-j}^n) \bigg)^2
            \leq \frac{14n}{k} \max_{v \in \mathcal{V}} \max_{Q \in \mathcal{Q}} \sum_{j = 1}^k D_\mathrm{H} (M_{v}, M_{v^{\oplus j}})^2.
        \end{equation}
    \end{lemma}

    If the Radon--Nikodym derivative of $P_v$ and $P_{v^{\oplus i}}$ is assumed to take the form of a fluctuation around unity, the right-hand-side of \eqref{app:eq:AcharyaLemEq} can be simplified significantly as in the following theorem.
    
    \begin{assumption}\label{app:ass:1}
        For every $v \in \mathcal{V}$ and $i \in [k]$, assume that $P_{v^{\oplus i}} \ll P_v$ and there exist measurable functions $\phi_{v,i} : \mathcal{Z} \rightarrow \mathbb{R}$ such that
        \begin{align*}
            \frac{\diff P_{v^{\oplus i}}}{\diff P_{v}} = 1 + \phi_{v,i}.
        \end{align*}        
    \end{assumption}
    \begin{theorem}[\citealp{Acharya:2022}] \label{app:thm:Acharya}
        Under \Cref{app:ass:1}, we have that
        \begin{equation} \label{app:eq:AcharyaThmEq}
            \bigg(\frac{1}{k} \sum_{j=1}^k D_{\mathrm{TV}}(M_{+j}^n, M_{-j}^n) \bigg)^2
            \leq \frac{7n}{k} \max_{v \in \mathcal{V}} \sup_{Q \in \mathcal{Q}_\alpha} \sum_{j = 1}^k \int_{\mathcal{Z}} \frac{\mathbb{E}_{P_v}[\phi_{v,j}(X) q(z \mid X)]^2}{\mathbb{E}_{P_v}[q(z \mid X)]} \diff{z}.
        \end{equation}
    \end{theorem}
    \noindent
    Note it particular that it follows from \Cref{app:ass:1} that $\mathbb{E}_{P_v}(\phi_{v, i}) = 0$ for all $i \in [k]$.
        
    Under a further assumption that the fluctuations are orthogonal and applying the definition of LDP, \eqref{app:eq:AcharyaThmEq} can be further simplified.
    \begin{assumption}\label{app:ass:2}
        Assume that there exists $\eta \geq 0$ such that for all $v \in \mathcal{V}$ and $i,j \in [k]$ with $i \neq j$, it holds that
        \begin{align*}
            \mathbb{E}_{P_v}(\phi_{v, i}\phi_{v, j}) = 0 \quad \mbox{and} \quad \mathbb{E}_{P_v}(\phi_{v, i}^2) \leq \eta^2.
        \end{align*}
    \end{assumption}
    
    For $i, j \in [k]$ and $v \in \mathcal{V}$, it is sometimes convenient to let $\varphi_{v, i} = 1 + \phi_{v, i}$ so that the requirements of \Cref{app:ass:2} are equivalent to
    \begin{align*}
        \mathbb{E}_{P_v}(\varphi_{v, i} \varphi_{v, j}) = 1 \quad \mbox{and} \quad \mathbb{E}_{P_v}(\varphi_{v, i}^2) \leq 1 + \eta^2.
    \end{align*}
    
    The following result then follows immediately from \citet[][Corollary~1]{Acharya:2022} and the fact that $(e^\alpha - 1)^2 \leq 3\alpha^2$ for $\alpha \in (0,1]$.
    \begin{lemma}[\citealp{Acharya:2022}] \label{app:lem:AcharyaSimple}
        When Assumptions \ref{app:ass:1} and \ref{app:ass:2} hold, we have for any $\alpha \in (0,1]$ and any $Q \in \mathcal{Q}_\alpha$ that
        \begin{align*}
            \frac{1}{k} \sum_{j=1}^k D_{\mathrm{TV}}(M_{+j}^n, M_{-j}^n) \leq \sqrt{\frac{21}{k} n\alpha^2\eta^2}, 
        \end{align*}
        where $\{M_{+j}, M_{-j}\}_{j \in [k]}$ are the private marginals defined in \eqref{app:eq:AssouadPrivateMarginal-1} and \eqref{app:eq:AssouadPrivateMarginal-2}.
    \end{lemma}
    Combining the above result with \Cref{app:lem:Assouad}, we obtain
    \begin{align} \label{app:eq:AcharyaAssouad}
        \mathcal{R}_{n, 1, \alpha}(\theta(\mathcal{P}), \Phi \circ \rho) \geq \varrho k \left(1 - \sqrt{\frac{21}{k}n\alpha^2\eta^2}\right).
    \end{align}
    
    We now consider the user-level analogue. With $\mathcal{V}$ and the family $\{P_v : v \in \mathcal{V}\}$ as before, suppose that each user draws $T$-many i.i.d.~observations from the distribution $P_v$ such that each user's collection of data is distributed as $P_v^{\otimes T}$.
    Then, for each $j \in [k]$, we define the mixture distributions
    \begin{align*}
        P_{+j}^{n,T} = \frac{1}{2^{d-1}} \sum_{v \in \mathcal{V} : v_j = 1} (P_v^{\otimes T})^{\otimes n} \quad \mbox{and} \quad P_{-j}^{n,T} = \frac{1}{2^{d-1}} \sum_{v \in \mathcal{V} : v_j = -1} (P_v^{\otimes T})^{\otimes n}.
    \end{align*}
    Given any user-level $\alpha$-LDP mechanism $Q$, we denote the private marginal for $n$ users each with a sample of size $T$ as
    \begin{align*}
        M_v^{n,T}(S) = \int Q(S \mid x_{1:T}^{(1)}, \ldots, x_{1:T}^{(n)}) \diff (P_v^{\otimes T})^{\otimes n}(x_{1:T}^{(1)}, \ldots, x_{1:T}^{(n)}).
    \end{align*}
    Consider the mixtures
    \[
        M_{+j}^{n,T}(S) = \frac{1}{2^{d-1}} \sum_{v \in \mathcal{V} : v_j = 1} M_v^{n,T}(S) = \int Q(S \mid x_{1:T}^{(1)}, \ldots, x_{1:T}^{(n)}) \diff P_{+ j}^{n, T}(x_{1:T}^{(1)}, \ldots, x_{1:T}^{(n)})
    \]
    and
    \[
        M_{-j}^{n,T}(S) = \frac{1}{2^{d-1}} \sum_{v \in \mathcal{V} : v_j = -1} M_v^{n,T}(S) = \int Q(S \mid x_{1:T}^{(1)}, \ldots, x_{1:T}^{(n)}) \diff P_{- j}^{n, T}(x_{1:T}^{(1)}, \ldots, x_{1:T}^{(n)}).
    \]
    Following the exact arguments above, we have that
    \begin{align} \label{app:eq:userlevelriskassouad}
        \mathcal{R}_{n, T, \alpha}(\theta(\mathcal{P}), \Phi \circ \rho) \geq \varrho \sum_{j=1}^k \{1 - D_{\mathrm{TV}}(M_{+j}^{n, T}, M_{-j}^{n, T})\}.
    \end{align}
    Applying the item-level results above to the distributions $\widetilde{P}_v = P_v^{\otimes T}$, we see that, after verifying Assumptions~\ref{app:ass:1} and \ref{app:ass:2} for the distributions $\widetilde{P}_v$ in place of $P_v$, we will see that \Cref{app:lem:AcharyaSimple} holds in the user-level case where the dependence on $T$ is absorbed into the quantity $\eta^2$. In what follows, for each estimation problem of interest we will construct a family of distributions which obey the separation condition \eqref{app:eq:HammingSeparation} and verify Assumptions~\ref{app:ass:1} and \ref{app:ass:2}.

\subsection{Proof of Theorem \ref{sec3:thm:main} Lower Bound (High Privacy Regime)} \label{app:sec:thm6highprivacy}
    \begin{proof}[Proof of \Cref{sec3:thm:main} (High privacy regime $\alpha \lesssim 1$)]
    \
    \\
    \textbf{Step 1: Constructing a Separated Family.} We consider the hypercube $\mathcal{V} = \{-1, 1\}^d$.  For $v \in \mathcal{V}$, let $P_v$ and $P_v'$ be the distributions where, for $X \sim P_v$ and $X' \sim P_v'$, we have that the $j$-th, $j \in [d]$, co-ordinates are, independently of the other co-ordinates, distributed as
        \begin{align}\label{app:eq:MeanFamily2}
            X_j =
            \begin{cases}
               1, &\text{with probability } (1+\Delta v_j)/2, \\
               -1, &\text{with probability } (1-\Delta v_j)/2,
            \end{cases}
        \end{align}
        and
        \begin{align}\label{app:eq:MeanFamily}     
            X_j' =
            \begin{cases}
               d^{-1/2}, &\text{with probability } (1+\Delta v_j)/2, \\
               -d^{-1/2}, &\text{with probability } (1-\Delta v_j)/2.
            \end{cases}
        \end{align}
        By construction we have that $P_v \in \mathcal{P}_d$ and $P_v' \in \mathcal{P}_d'$ as defined in \eqref{sec2:eq:MeanClasses}, for all $v \in \mathcal{V}$. Note that $\theta(P_v) = \mathbb{E}_{P_v}(X) = \Delta v$ and $\theta(P_v') = \mathbb{E}_{P_v'}(X') = \Delta v / d^{1/2}$.  For the separation condition \eqref{app:eq:HammingSeparation}, it then holds that
        \begin{align*}
            \frac{1}{4}\|\theta(P_v) - \theta(P_{v'})\|_2^2 \geq \Delta^2 \sum_{j=1}^d \mathbbm{1} \{ v_j \neq v_j'\} \quad \mbox{and} \quad \frac{1}{4}\|\theta(P_v') - \theta(P_{v'}')\|_2^2 \geq \frac{\Delta^2}{d} \sum_{j=1}^d \mathbbm{1} \{ v_j \neq v_j'\},
        \end{align*}
        which shows that the distributions are $\Delta^2$-separated and $\Delta^2/d$-separated for the $\ell_\infty$- and $\ell_2$-ball constructions respectively.

        \medskip 
        \noindent
        \textbf{Step 2: Verifying Assumptions~\ref{app:ass:1} and \ref{app:ass:2}.}
        We claim that, for both of the families of distributions obtained by taking the $T$-fold product of the distributions defined in \eqref{app:eq:MeanFamily2} and \eqref{app:eq:MeanFamily}, we have that Assumptions \ref{app:ass:1} and \ref{app:ass:2} hold with
            \begin{align}\label{eq-thm6-lb-eta}
                \eta^2 = \left( \frac{1 + 3\Delta^2}{1 - \Delta^2} \right)^T - 1.
            \end{align} 
        In this step, we show this holds for the construction \eqref{app:eq:MeanFamily2}. The same result for the construction \eqref{app:eq:MeanFamily} follows by replacing all instances of $\mathbbm{1}\{ x_{t, j} = \pm1\}$ in what follows with $\mathbbm{1}\{ x_{t, j} = \pm d^{-1/2}\}$ and $P_v$ by $P_v'$, with the calculations being otherwise identical.
   
        We see that, for $i \in [d]$,
        \begin{align*}
            \frac{\diff P_{v^{\oplus i}}^{\otimes T}}{\diff P_{v}^{\otimes T}}(x_{1:T})
            &= \frac{\prod_{t=1}^T \prod_{j=1}^d \left(\frac{1 + \Delta v_j^{\oplus i}}{2} \right)^{\mathbbm{1}\{ x_{t, j} = 1\}}\left(\frac{1 - \Delta v_j^{\oplus i}}{2} \right)^{\mathbbm{1}\{ x_{t, j} = -1\}}}
            {\prod_{t=1}^T \prod_{j=1}^d \left(\frac{1 + \Delta v_j}{2} \right)^{\mathbbm{1}\{ x_{t, j} = 1\}}\left(\frac{1 - \Delta v_j}{2} \right)^{\mathbbm{1}\{ x_{t, j} = -1\}}} \\
            &= \prod_{t=1}^T \left( \frac{1 - \Delta v_i}{1 + \Delta v_i} \right)^{\mathbbm{1}\{ x_{t, i} = 1\}} \left( \frac{1 + \Delta v_i}{1 - \Delta v_i} \right)^{\mathbbm{1}\{ x_{t, i} = -1\}} =: \varphi_{v, i}(x_{1:T}).
        \end{align*}
        This verifies \Cref{app:ass:1}.  As for \Cref{app:ass:2}, we see further that have for $i \neq j$ that
        \begin{align}
            &\mathbb{E}_{P_v}(\varphi_{v, i} \varphi_{v, j}) \nonumber \\
            =& \prod_{t=1}^{T} \mathbb{E}_{P_v}\left[ \left( \frac{1 - \Delta v_i}{1 + \Delta v_i} \right)^{\mathbbm{1}\{ X_{t,i} = 1\}} \left( \frac{1 + \Delta v_i}{1 - \Delta v_i} \right)^{\mathbbm{1}\{ X_{t,i} = -1\}} \left( \frac{1 - \Delta v_j}{1 + \Delta v_j} \right)^{\mathbbm{1}\{ X_{t,j} = 1\}} \left( \frac{1 + \Delta v_j}{1 - \Delta v_j} \right)^{\mathbbm{1}\{ X_{t,j} = -1\}} \right] \nonumber \\
            =& \left\{\mathbb{E}_{P_v}\left[ \left( \frac{1 - \Delta v_i}{1 + \Delta v_i} \right)^{\mathbbm{1}\{ X_{1,i} = 1\}} \left( \frac{1 + \Delta v_i}{1 - \Delta v_i} \right)^{\mathbbm{1}\{ X_{1,i} = -1\}} \left( \frac{1 - \Delta v_j}{1 + \Delta v_j} \right)^{\mathbbm{1}\{ X_{1,j} = 1\}} \left( \frac{1 + \Delta v_j}{1 - \Delta v_j} \right)^{\mathbbm{1}\{ X_{1,j} = -1\}} \right]\right\}^T \label{app:eq:meanlbcalcij} \\
            =& \left\{\mathbb{E}_{P_v}\Bigg[ \left( \frac{1 + \Delta v_i}{1 - \Delta v_i} \right) \left( \frac{1 - \Delta v_i}{1 + \Delta v_i} \right)^{2\mathbbm{1}\{ X_{1,i} = 1\}} \left( \frac{1 + \Delta v_j}{1 - \Delta v_j} \right) \left( \frac{1 - \Delta v_j}{1 + \Delta v_j} \right)^{2\mathbbm{1}\{ X_{1,j} = 1\}}  \Bigg]\right\}^T \nonumber \\
            =& \left\{\mathbb{E}_{P_v}\Bigg[ \left( \frac{1 + \Delta v_i}{1 - \Delta v_i} \right) \left( \frac{1 - \Delta v_i}{1 + \Delta v_i} \right)^{2\mathbbm{1}\{ X_{1,i} = 1\}} \Bigg]\right\}^T \left\{\mathbb{E}_{P_v}\Bigg[ \left( \frac{1 + \Delta v_j}{1 - \Delta v_j} \right) \left( \frac{1 - \Delta v_j}{1 + \Delta v_j} \right)^{2\mathbbm{1}\{ X_{1,j} = 1\}}  \Bigg]\right\}^T, \nonumber    
        \end{align}
        where in the second equality we use the fact that the $X_{t,i}$ are i.i.d.~across $t \in [T]$, the third by the fact that $\mathbbm{1}\{ X_{t, i} = 1\} = 1 - \mathbbm{1}\{ X_{t, i} = -1\}$, and the last by the independence of the co-ordinates of $X_1$. Focusing on one of the expectation factors, we have that
        \begin{align*}
            \mathbb{E}_{P_v}\left[ \left( \frac{1 + \Delta v_i}{1 - \Delta v_i} \right) \left( \frac{1 - \Delta v_i}{1 + \Delta v_i} \right)^{2\mathbbm{1}\{ X_{1,i} = 1\}} \right] 
            &= \left( \frac{1 + \Delta v_i}{1 - \Delta v_i} \right) \mathbb{E}_{P_v}\left[ \left( \frac{1 - \Delta v_i}{1 + \Delta v_i} \right)^{2\mathbbm{1}\{ X_{1,i} = 1\}} \right] \\
            &= \left( \frac{1 + \Delta v_i}{1 - \Delta v_i} \right) \left[ \left( \frac{1 - \Delta v_i}{1 + \Delta v_i} \right)^2 \frac{(1 + \Delta v_i)}{2} + \frac{(1 - \Delta v_i)}{2}\right] = 1.
        \end{align*}
        Hence, we have that $\mathbb{E}_{P_v}(\varphi_{v, i} \varphi_{v, j}) = 1$ for all $i \neq j$, as required.
    
        To specify $\eta^2$ in \Cref{app:ass:2}, we consider that
        \begin{align*}
            \mathbb{E}_{P_v}(\varphi_{v, i}^2)
            &= \left\{\mathbb{E}_{P_v}\left[ \left( \frac{1 - \Delta v_i}{1 + \Delta v_i} \right)^{2\mathbbm{1}\{ X_{1,i} = 1\}} \left( \frac{1 + \Delta v_i}{1 - \Delta v_i} \right)^{2\mathbbm{1}\{ X_{1,i} = -1\}} \right]\right\}^T \\
            &= \left\{\mathbb{E}_{P_v}\left[ \left( \frac{1 + \Delta v_i}{1 - \Delta v_i} \right)^2 \left( \frac{1 - \Delta v_i}{1 + \Delta v_i} \right)^{4\mathbbm{1}\{ X_{1,i} = 1\}} \right]\right\}^T \\
            &= \left[ \left( \frac{1 + \Delta v_i}{1 - \Delta v_i} \right)^2 \left\{ \left( \frac{1 - \Delta v_i}{1 + \Delta v_i} \right)^4 \frac{(1 + \Delta v_i)}{2} + \frac{(1 - \Delta v_i)}{2}\right\}\right]^T
            = \left( \frac{1 + 3\Delta^2}{1 - \Delta^2}\right)^T,
        \end{align*}
        where the first equality comes from setting $j = i$ in \eqref{app:eq:meanlbcalcij}. Hence, we have that \eqref{eq-thm6-lb-eta} holds.
        
        \medskip        
        \noindent
        \textbf{Step 3: Obtaining Lower Bounds.}
        We set $\Delta = \{cd/(nT\alpha^2)\}^{1/2}$ for some absolute constant $c>0$. Hence, using our assumption that $n \alpha^2 \gtrsim d \log(ed)$, when $c$ is sufficiently small we have that
        \begin{align}
            \eta^2
            = \left( 1 + \frac{4\Delta^2}{1-\Delta^2} \right)^T - 1
            \leq \left( 1 + \frac{8cd}{nT\alpha^2} \right)^T - 1
            \leq \frac{16cd}{n\alpha^2}, \label{app:eq:linfetabound}
        \end{align}
        where the first and second inequalities follow once we choose $c > 0$ small enough that $\Delta^2 \leq 1/2$ and $8cd/(n\alpha^2) < 1$, respectively. As Assumptions \ref{app:ass:1} and \ref{app:ass:2} hold, we have that for the $\ell_\infty$-ball case that
        \begin{align*}
            \mathcal{R}_{n, T, \alpha}(\theta(\mathcal{P}_{d}), \|\cdot\|_2^2) \geq \frac{\Delta^2d}{2} \left(1 - \sqrt{\frac{21}{d}n\alpha^2\eta^2}\right) \geq \frac{c d^2}{2nT\alpha^2}\left(1 - \sqrt{336c} \right) \gtrsim \frac{d^2}{nT\alpha^2}.
        \end{align*}
        where the first inequality is from \eqref{app:eq:AcharyaAssouad} and the fact that the distributions are $\Delta^2$-separated; the second from \eqref{app:eq:linfetabound} and the value of $\Delta$; and the last as $c > 0$ is sufficiently small. Combining with the infinite-$T$ lower bound in \eqref{sec2:eq:linfMeanStatement}, for which the required assumptions are satisfied, completes the proof for this case.

        For the $\ell_2$-ball case, we similarly have that
        \begin{align*}
            \mathcal{R}_{n, T, \alpha}(\theta(\mathcal{P}_{d}'), \|\cdot\|_2^2)
            \geq \frac{\Delta^2}{2} \left(1 - \sqrt{\frac{21}{d}n\alpha^2\eta^2}\right)
            \geq \frac{cd}{2nT\alpha^2}\left(1 - \sqrt{336c} \right)
            \gtrsim \frac{d}{nT\alpha^2},
        \end{align*}
        and combining with the infinite-$T$ lower bound in \eqref{sec2:eq:l2MeanStatement}, for which the required assumptions are satisfied, completes the proof.
    \end{proof}

\subsection{Proof of Theorem \ref{sec3:thm:main} Lower Bound (Low Privacy Regime)}
    
    The proof starts similarly to that of the high privacy regime with the same constructed family. To obtain tight lower bounds in the low-privacy regime, we decompose the quantities $\phi_{v,i}(X_{1:T})$ as the sum of a carefully calibrated sub-Gaussian term and residual term, such that the residual term, whilst not sub-Gaussian itself, has significantly reduced variance, and so its contribution can be directly controlled. The contribution of the sub-Gaussian term can then be controlled similarly to the argument in \citet[Assumption~3]{Acharya:2022}. We now proceed with the proof.
    \begin{proof}[Proof of \Cref{sec3:thm:main} (Low privacy regime $\alpha \gtrsim 1$)]
    \
    \\
    \textbf{Step 1: Constructing a Separated Family.} We recall the constructions \eqref{app:eq:MeanFamily2} and \eqref{app:eq:MeanFamily} inducing the distributions $P_v$ and $P_v'$ respectively, and that by construction we have that $P_v \in \mathcal{P}_d$ and $P_v' \in \mathcal{P}_d'$ as defined in \eqref{sec2:eq:MeanClasses}, for all $v \in \mathcal{V}$. For the separation condition \eqref{app:eq:HammingSeparation}, it then holds that
        \begin{align*}
            \frac{1}{4}\|\theta(P_v) - \theta(P_{v'})\|_2^2 \geq \Delta^2 \sum_{j=1}^d \mathbbm{1} \{ v_j \neq v_j'\} \quad \mbox{and} \quad \frac{1}{4}\|\theta(P_v') - \theta(P_{v'}')\|_2^2 \geq \frac{\Delta^2}{d} \sum_{j=1}^d \mathbbm{1} \{ v_j \neq v_j'\},
        \end{align*}
        which shows that the distributions are $\Delta^2$-separated and $\Delta^2/d$-separated for the $\ell_\infty$- and $\ell_2$-ball constructions respectively.

        In what follows, we focus on the construction \eqref{app:eq:MeanFamily2}. The same result for the construction \eqref{app:eq:MeanFamily} follows by replacing all instances of $\mathbbm{1}\{ x_{t, j} = \pm1\}$ in what follows with $\mathbbm{1}\{ x_{t, j} = \pm d^{-1/2}\}$ and $P_v$ by $P_v'$, with the calculations being otherwise identical.
        
        \medskip 
        \noindent
        \textbf{Step 2: Applying \Cref{app:thm:Acharya}.}
        Following the proof of \Cref{sec3:thm:main} in the high-privacy regime, we have that \Cref{app:ass:1} is satisfied, for $i \in [d]$ and $v \in \mathcal{V}$, by
        \begin{align*}
            \frac{\diff P_{v^{\oplus i}}^{\otimes T}}{\diff P_{v}^{\otimes T}}(x_{1:T}) - 1
            &= \frac{\prod_{t=1}^T \prod_{j=1}^d \left(\frac{1 + \Delta v_j^{\oplus i}}{2} \right)^{\mathbbm{1}\{ x_{t, j} = 1\}}\left(\frac{1 - \Delta v_j^{\oplus i}}{2} \right)^{\mathbbm{1}\{ x_{t, j} = -1\}}}
            {\prod_{t=1}^T \prod_{j=1}^d \left(\frac{1 + \Delta v_j}{2} \right)^{\mathbbm{1}\{ x_{t, j} = 1\}}\left(\frac{1 - \Delta v_j}{2} \right)^{\mathbbm{1}\{ x_{t, j} = -1\}}} -1 \\
            &= \prod_{t=1}^T \left( \frac{1 - \Delta v_i}{1 + \Delta v_i} \right)^{\mathbbm{1}\{ x_{t, i} = 1\}} \left( \frac{1 + \Delta v_i}{1 - \Delta v_i} \right)^{\mathbbm{1}\{ x_{t, i} = -1\}} - 1\\
            &= \bigg( \frac{p_i}{1 - p_i}\bigg)^T \bigg( \frac{1 - p_i}{p_i}\bigg)^{2 S_{i, +}} - 1 \\
            &= \bigg( \frac{p_i}{1 - p_i}\bigg)^{T(1-2p_i)} \bigg( \frac{1 - p_i}{p_i}\bigg)^{2 S_i} - 1
            =: \phi_{v, i}(x_{1:T}),
        \end{align*}
        where the penultimate line follows from setting, $S_{i, +} = \sum_{t=1}^T \mathbbm{1}\{x_{t,i} = 1\}$ and $p_i = (1 + \Delta v_i)/2$, and the final line by letting $S_i = S_{i, +} - \mathbb{E}[S_{i, +}] = S_{i, +} - Tp_i$.

        We decompose $\phi_{v, i}(x_{1:T})$ into a sub-Gaussian term and residual term as
        \begin{align*}
            \phi_{v,i}(x_{1:T})
            &= \bigg( \frac{p_i}{1 - p_i}\bigg)^{T(1-2p_i)}2 \Omega_i S_i + \bigg[ \bigg( \frac{p_i}{1 - p_i}\bigg)^{T(1-2p_i)}\bigg\{ \bigg( \frac{1 - p_i}{p_i}\bigg)^{2 S_i} - 2 \Omega_i S_i \bigg\} - 1 \bigg] \\
            &= G_{v, i}(x_{1:T}) + \psi_{v, i}(x_{1:T}),
        \end{align*}
        where we define
        \begin{equation} \label{app:eq:omegaval}
            \Omega_i = \frac{1-2p_i}{2p_i(1-p_i)} \bigg(\frac{p_i}{1-p_i}\bigg)^{-T(1-2p_i)}
        \end{equation}
        We then decompose the sum and integral term on the right-hand-side of \eqref{app:eq:AcharyaThmEq} as
        \begin{align}
            \sum_{j = 1}^d \int_{\mathcal{Z}} &\frac{\mathbb{E}_{P_v^{\otimes T}}[\phi_{v,j}(X_{1:T}) q(z \mid X_{1:T})]^2}{\mathbb{E}_{P_v^{\otimes T}}[q_{v,j}(X_{1:T})]} \diff{z} \nonumber \\
            &\leq 2\sum_{j = 1}^d \int_{\mathcal{Z}} \frac{\mathbb{E}_{P_v^{\otimes T}}[G_{v,j}(X_{1:T}) q(z \mid X_{1:T})]^2}{\mathbb{E}_{P_v^{\otimes T}}[q(z \mid X_{1:T})]} \diff{z}
            + 2\sum_{j = 1}^d \int_{\mathcal{Z}} \frac{\mathbb{E}_{P_v^{\otimes T}}[\psi_{v,j}(X_{1:T}) q(z \mid X_{1:T})]^2}{\mathbb{E}_{P_v^{\otimes T}}[q(z \mid X_{1:T})]} \diff{z} \nonumber \\
            &\leq 2\sum_{j = 1}^d \int_{\mathcal{Z}} \frac{\mathbb{E}_{P_v^{\otimes T}}[G_{v,j}(X_{1:T}) q(z \mid X_{1:T})]^2}{\mathbb{E}_{P_v^{\otimes T}}[q(z \mid X_{1:T})]} \diff{z}
            + 2\sum_{j = 1}^d \int_{\mathcal{Z}} \mathbb{E}_{P_v^{\otimes T}}[\psi_{v,j}(X_{1:T})^2 q(z \mid X_{1:T})] \diff{z} \nonumber \\
            &= 2\sum_{j = 1}^d \int_{\mathcal{Z}} \frac{\mathbb{E}_{P_v^{\otimes T}}[G_{v,j}(X_{1:T}) q(z \mid X_{1:T})]^2}{\mathbb{E}_{P_v^{\otimes T}}[q(z \mid X_{1:T})]} \diff{z} 
            + 2\sum_{j = 1}^d \mathbb{E}_{P_v^{\otimes T}}[\psi_{v,j}(X_{1:T})^2], \label{app:eq:acharyabounddecomp}
        \end{align}
        where both inequalities follow from applications of the Cauchy--Schwarz inequality. We now proceed to bound both of the terms in \eqref{app:eq:acharyabounddecomp} individually.
        
        \medskip        
        \noindent
        \textbf{Step 3: Verifying Sub-Gaussianity of $G_{v, i}$.}
        For $X_{1:T} \sim P_v^{\otimes T}$, we note that $S_i$ is a centred $\mathrm{Binomial}(T, p_i)$ random variable, and so we have $\mathbb{E}[G_{v, i}] = 0$.

        Knowing that $G_{v, i}$ is centred, it remains to consider the sub-Gaussianity of $G_{v, i}(X_{1:T})$. Indeed, by noting that a Bernoulli random variable is $(1/4)$-sub-Gaussian \cite[e.g.~Exercise~2.4 and Proposition~2.5][]{Wainwright:2019}, and that $S_i$ is equal in distribution to the centred sum of $T$-many independent Bernoulli random variables, we have that $S_i$ is $(T/4)$-sub-Gaussian \cite[e.g.~Proposition~2.5][]{Wainwright:2019}.

        Hence, after rescaling $S_i$ as per the definition of $G_{v, i}(X_{1:T})$, we have that $G_{v, i}(X_{1:T})$ is $\sigma_i^2$-sub-Gaussian for
        \begin{equation} \label{app:eq:SGtermvarproxy}
            \sigma_i^2
            = \frac{T\Omega_i^2}{4}\bigg( \frac{p_i}{1-p_i} \bigg)^{2T(1-2p_i)}
            = \frac{T}{4}\bigg( \frac{1-2p_i}{2p_i(1-p_i)}\bigg)^2.
        \end{equation}
        Finally, following the calculations of the proofs of Theorem~2 and Corollary~1 of \cite{Acharya:2022}, we obtain the bound that
        \begin{align} \label{app:eq:acharyaboundterm1}
            \sum_{j = 1}^d \int_{\mathcal{Z}} \frac{\mathbb{E}_{P_v^{\otimes T}}[G_{v,j}(X_{1:T}) q(z \mid X_{1:T})]^2}{\mathbb{E}_{P_v^{\otimes T}}[q(z \mid X_{1:T})]} \diff{z}
            \leq 2\sigma_i^2\alpha.
        \end{align}

        \medskip        
        \noindent
        \textbf{Step 4: Controlling Variance of $\psi_{v, i}$.}
        For $X_{1:T} \sim P_v^{\otimes T}$ with $P_v$ as defined by \eqref{app:eq:MeanFamily2}, we begin by noting that $\psi_{v, i}(X_{1:T}) = \phi_{v, i}(X_{1:T}) - G_{v,i}(X_{1:T})$, and hence $\mathbb{E}[\psi_{v, i}(X_{1:T})] = 0$. Hence, to control the variance of $\psi_{v, i}$ it suffices to consider the second moment which we decompose as
        \begin{align}
            \mathbb{E}[\psi_{v, i}(X_{1:T})^2]
            = \mathbb{E}[\phi_{v, i}(X_{1:T})^2] + \mathbb{E}[G_{v, i}(X_{1:T})^2] - 2\mathbb{E}[\phi_{v, i}(X_{1:T})G_{v, i}(X_{1:T})]. \label{app:eq:lowprivdecomp}
        \end{align}

        For the first term in the right-hand-side of \eqref{app:eq:acharyaboundterm1}, from the same calculations in \Cref{app:sec:thm6highprivacy} leading to \eqref{eq-thm6-lb-eta}, we have that
        \begin{equation}
            \mathbb{E}[\phi_{v, i}(X_{1:T})^2] = \bigg( 1 + \frac{4\Delta^2}{1-\Delta^2} \bigg)^T - 1. \label{app:eq:lowprivterm1}
        \end{equation}

        For the second term in the right-hand-side of \eqref{app:eq:acharyaboundterm1}, we have
        \begin{align}
            \mathbb{E}[G_{v, i}(X_{1:T})^2]
            &= 4\Omega_i^2\bigg( \frac{p_i}{1-p_i} \bigg)^{2T(1-2p_i)} \mathbb{E}[S_i^2] \nonumber \\
            &= 4\Omega_i^2\bigg( \frac{p_i}{1-p_i} \bigg)^{2T(1-2p_i)} \mathrm{Var}(S_{i, +}) \nonumber \\
            &= 4\Omega_i^2 \bigg( \frac{p_i}{1-p_i} \bigg)^{2T(1-2p_i)} Tp_i(1-p_i), \label{app:eq:lowprivterm2}
        \end{align}
        where the second equality uses the fact that $S_i = S_{i, +} - \mathbb{E}[S_{i, +}]$, and the final equality by that fact that $S_{i, +}$ is equal in distribution to a $\mathrm{Binomial}(T, p_i)$ random variable.

        For the final term in the right-hand-side of \eqref{app:eq:acharyaboundterm1}, we obtain
        \begin{align}
            \mathbb{E}[\phi_{v, i}(X_{1:T}) &G_{v, i}(X_{1:T})] \nonumber \\
            &= \mathbb{E}\bigg[ 2\Omega_i S_i \bigg( \frac{p_i}{1-p_i}\bigg)^{2T(1-2p_i)} \bigg( \frac{1-p_i}{p_i}\bigg)^{2S_i} \bigg] \nonumber \\
            &= 2\Omega_i\bigg(\frac{p_i}{1-p_i}\bigg)^{2T(1-2p_i)} \bigg(\frac{1-p_i}{p_i}\bigg)^{-2Tp_i} \mathbb{E}\bigg[ (S_{i, +} - Tp_i)   \bigg( \frac{1-p_i}{p_i}\bigg)^{2S_{i, +}} \bigg] \nonumber \\
            &= 2\Omega_i \bigg(\frac{p_i}{1-p_i}\bigg)^{2T(1-2p_i)} \bigg(\frac{1-p_i}{p_i}\bigg)^{-2Tp_i} \bigg[ Tp_i \bigg( \frac{1 - p_i}{p_i}\bigg)^2 \bigg\{1 - p_i + \bigg( \frac{1-p_i}{p_i}\bigg)^2 p_i\bigg\}^{T-1} \nonumber \\
            &\hspace{8cm}- Tp_i\bigg\{ 1 - p_i + \bigg(\frac{1 - p_i}{p_i} \bigg)^2p_i\bigg\}^T\bigg] \nonumber \\
            &= 2\Omega_i \bigg(\frac{p_i}{1-p_i}\bigg)^{2T(1-2p_i)} \bigg(\frac{1-p_i}{p_i}\bigg)^{-2Tp_i} \bigg[ Tp_i \bigg( \frac{1 - p_i}{p_i}\bigg)^{T+1} - Tp_i\bigg( \frac{1 - p_i}{p_i} \bigg)^T\bigg] \nonumber \\
            &= 2\Omega_i \bigg(\frac{p_i}{1-p_i}\bigg)^{2T(1-2p_i)} \bigg(\frac{1-p_i}{p_i}\bigg)^{-2Tp_i} \bigg\{ T(1-2p_i)\bigg( \frac{1 - p_i}{p_i}\bigg)^T \bigg\} \nonumber \\
            &= 2\Omega_i T (1-2p_i)\bigg( \frac{p_i}{1 - p_i} \bigg)^{T(1-2p_i)}, \label{app:eq:lowprivterm3}
        \end{align}
        where the third equality is due to \Cref{app:lem:PGFArgument}, which provides a bound on the expectation term using the probability generating function of $S_i$.

        We set $\Delta = \{cd/(nT\alpha)\}^{1/2}$ for some absolute constant $c>0$, and hence, combining \eqref{app:eq:lowprivdecomp}, \eqref{app:eq:lowprivterm1}, \eqref{app:eq:lowprivterm2} and \eqref{app:eq:lowprivterm3}, we obtain the bound
        \begin{align}
            \mathbb{E}[&\psi_{v, i}(X_{1:T})^2] \nonumber \\
            &= 4\Omega_i^2 \bigg( \frac{p_i}{1-p_i} \bigg)^{2T(1-2p_i)} Tp_i(1-p_i)
            - 4\Omega_i T (1-2p_i)\bigg( \frac{p_i}{1 - p_i} \bigg)^{T(1-2p_i)}
            + \bigg( 1 + \frac{4\Delta^2}{1-\Delta^2} \bigg)^T - 1 \nonumber \\
            &= 4\widetilde{\Omega}_i^2 Tp_i(1-p_i)
            - 4\widetilde{\Omega}_i T(1-2p_i)
            + \bigg( 1 + \frac{4\Delta^2}{1-\Delta^2} \bigg)^T - 1 \nonumber \\
            &= \bigg( 1 + \frac{4\Delta^2}{1-\Delta^2} \bigg)^T
            - \frac{4T\Delta^2}{1-\Delta^2}
            - 1
            \leq T^2\Delta^4 = \bigg(\frac{cd}{n\alpha}\bigg)^2
            \label{app:eq:residualtermvar}
        \end{align}
        where the second equality is by letting $\widetilde{\Omega}_i = \Omega_i \{p_i/(1 - p_i)\}^{T(1-2p_i)}$; the third by noting that $(2p_i - 1)^2 = \Delta^2$ and by the value of $\Omega_i$, we have that $\widetilde{\Omega}_i = (1-2p_i)/\{2p_i(1-p_i)\}$ solves the following quadratic equation in $y$
        \begin{align*}
            4y^2Tp_i(1-p_i) - 4yT(1-2p_i) + \frac{4T(2p_i - 1)^2}{1-(2p_i - 1)^2} = 0;
        \end{align*}
        and the inequality by the value $\Delta = \{cd/(nT\alpha)\}^{1/2}$, the assumption that $n\alpha \gtrsim d\log(ed)$ and the fact we may take $c$ sufficiently small.

        \medskip        
        \noindent
        \textbf{Step 5: Completing.}
        Combining \eqref{app:eq:acharyabounddecomp}, \eqref{app:eq:SGtermvarproxy}, \eqref {app:eq:acharyaboundterm1} and \eqref{app:eq:residualtermvar}, we have
        \begin{align} 
            \sum_{j = 1}^d \int_{\mathcal{Z}} \frac{\mathbb{E}_{P_v^{\otimes T}}[\phi_{v,j}(X_{1:T}) Q(z \mid X_{1:T})]^2}{\mathbb{E}_{P_v^{\otimes T}}[\phi_{v,j}(X_{1:T})]} \diff{z}
            &\leq \alpha\frac{T}{2}\bigg( \frac{1-2p_i}{2p_i(1-p_i)}\bigg)^2 + 2T^2\Delta^4 \nonumber \\
            &\leq C'(\alpha T\Delta^2 + dT^2\Delta^4), \label{app:eq:finallowprivdecomp}
        \end{align}
        for $C' > 0$ some absolute constant, where the second inequality is by the value of $p_i = (1 + \Delta v_i)/2$.

        Hence, by \Cref{app:thm:Acharya}, \eqref{app:eq:userlevelriskassouad} and \eqref{app:eq:finallowprivdecomp}, we have
        \begin{align*}
            \mathcal{R}_{n, T, \alpha}(\theta(\mathcal{P}_{d}), \|\cdot\|_2^2)
            \geq \frac{d\Delta^2}{2} \left(1 - C''\sqrt{\frac{nT\alpha\Delta^2}{d} + nT^2\Delta^4}\right)
            \geq \frac{c d^2}{2nT\alpha}\left(1 - c^{1/2}C''\sqrt{1+\frac{d^2}{n\alpha^2}} \right)
            \gtrsim \frac{d^2}{nT\alpha}.
        \end{align*}
        for $C'' > 0$ some absolute constant, where the first inequality is from the fact that the distributions are $\Delta^2$-separated; the second from the value of $\Delta$; and the last as $c > 0$ is sufficiently small and the assumption $n\alpha^2 \gtrsim d^2$. Combining with the infinite-$T$ lower bound in \eqref{sec2:eq:linfMeanStatement}, for which the required assumptions are satisfied, and the non-private lower bound in \Cref{app:lem:nonprivateLB} completes the proof for this case.

        For the $\ell_2$-ball case, we similarly have that
        \begin{align*}
            \mathcal{R}_{n, T, \alpha}(\theta(\mathcal{P}_{d}'), \|\cdot\|_2^2)
            \geq \frac{\Delta^2}{2} \left(1 - C''\sqrt{\frac{nT\alpha\Delta^2}{d} + nT^2\Delta^4}\right)
            \geq \frac{c d}{2nT\alpha}\left(1 - c^{1/2}C''\sqrt{1+\frac{d^2}{n\alpha^2}} \right) \gtrsim \frac{d}{nT\alpha}.
        \end{align*}
        and combining with the infinite-$T$ lower bound in \eqref{sec2:eq:l2MeanStatement}, for which the required assumptions are satisfied, and the non-private lower bound in \Cref{app:lem:nonprivateLB} completes the proof.
    \end{proof}

\subsection{Proof of Theorem \ref{sec5:thm:main} Lower Bound}
    As the distributions we are considering have $s$-sparse means, a standard Assouad's method construction where, for example, $-1$ corresponds to a zero entry for the mean and $+1$ corresponds to a non-zero entry, will fail as any $v \in \mathcal{V}$ with $\sum_{j=1}^k \mathbbm{1}\{ v_j = 1\} > s$ will not lie in the family we are considering. We instead consider a variant of Assouad's lemma, similar to that considered in \cite{Acharya:2022}, that allows us to consider a family of distributions indexed by the hypercube even when not all such distributions are a member of the family our estimation problem considers.

    Firstly, we require a strengthening of the separation condition \eqref{app:eq:HammingSeparation} which holds with equality. We say the collection $\{ P_v : v \in \mathcal{V} \}$ is $2\varrho$-Hamming separated in equality with respect to the loss $\Phi \circ \rho$ if for all $v, v' \in \mathcal{V}$,
    \begin{align}
        \Phi\left( \frac{1}{2}\rho(\theta_v, \theta_{v'}) \right) = 2\varrho \sum_{j=1}^k \mathbbm{1} \{v_j \neq v'_j \}. \label{app:eq:HammingSeparationEquality}
    \end{align}

    We then have the following lemma and corollary, the proofs of which are contained in \Cref{app:sec:misc}.
    \begin{lemma} \label{app:lem:GenAssouadPrelim}
        For $k \in \mathbb{N}$, let $\mathcal{P}$ be a family of distributions and consider the hypercube $\mathcal{V} = \{-1, 1\}^k$, the elements of which index a family $\{P_v : v \in \mathcal{V} \} \subseteq \mathcal{P}$ satisfying the separation condition \eqref{app:eq:HammingSeparationEquality}.  Let $\mathcal{P}^\ast \subseteq \mathcal{P}$ be some subset of the family of distributions and $V$ be a random variable on the set $\mathcal{V} = \{-1, 1\}^k$. Writing $\mathcal{V}^\ast = \{v \in \mathcal{V} : P_v \in \mathcal{P}^\ast \}$, we have that
        \begin{align*}
            \sup_{P \in \mathcal{P}^\ast} \mathbb{E}_{P} \left[ \Phi \circ \rho (\hat{\theta}, \theta(P)) \right]
            \geq \mathbb{E}_{V} \left[ \mathbb{E}_{P_{V}} \left[ \Phi \circ \rho (\hat{\theta}, \theta(P)) \right] \right] - 2k \varrho \mathbb{P}(V \notin \mathcal{V}^\ast).
        \end{align*}
    \end{lemma}

    \begin{corollary} \label{app:cor:GenAssouad}
        Under the same conditions in \Cref{app:lem:GenAssouadPrelim}, suppose that the random vector $V$ consists of i.i.d.~random variables as co-ordinates such that, for each $j \in [k]$, we have that $\mathbb{P}(V_j = 1) = \tau = 1 - \mathbb{P}(V_j = -1)$ with $\tau \leq 1/2$. Suppose also that for some constant $c > 0$, we have that $\mathbb{P}(V \notin \mathcal{V}^\ast) \leq c\tau$, then
        \begin{align*}
            \sup_{P \in \mathcal{P}^\ast} \mathbb{E}_{P} \left[ \Phi \circ \rho (\hat{\theta}, \theta(P)) \right]
            \geq 2\varrho \tau \left[ \sum_{j=1}^k \{ 1 - D_{\mathrm{TV}}(P_{+j}, P_{-j}) \} - ck \right],
        \end{align*}
        where $P_{+j}$ and $P_{-j}$ are the mixture distributions as defined in \eqref{app:eq:mixtures}.
    \end{corollary}

    With these results, we now prove the lower bound for sparse mean estimation.

    \begin{proof}[Proof of \Cref{sec5:thm:main}]
    \
    \\
    \noindent  \textbf{Step 1: Constructing a Separated Family.}  We consider the hypercube $\mathcal{V} = \{-1, 1\}^d$.  For $v \in \mathcal{V}$, let $P_v$ be the distribution where, for $X \sim P_v$, we have that the $j$-th co-ordinate is, independently of the other co-ordinates, distributed as
        \begin{align}\label{app:eq:SparseFamilyLower}
            X_j =
            \begin{cases}
               1, &\text{with probability } \frac{1}{2} + \frac{\Delta}{4}(v_j + 1), \\
               -1, &\text{with probability } \frac{1}{2} - \frac{\Delta}{4}(v_j + 1),
            \end{cases}
            \quad j \in [d].
        \end{align}
        Note that $\theta(P_v) = \mathbb{E}_{P_v}(X) = (\Delta/2)(v + 1)$, and hence the separation condition  in \eqref{app:eq:HammingSeparationEquality} holds as 
        \begin{align*}
            \frac{1}{4}\|\theta(P_v) - \theta(P_{v'})\|_2^2 = \frac{\Delta^2}{4} \sum_{j=1}^d \mathbbm{1} \{v_j \neq v'_j \}.
        \end{align*}

        \medskip
        \noindent
        \textbf{Step 2: Verifying Assumptions~\ref{app:ass:1} and \ref{app:ass:2}.}  We claim that for the family of distributions obtained by taking the $T$-fold product of the distributions defined in \eqref{app:eq:SparseFamilyLower}, when $\Delta \leq 2^{-1/2}$, Assumptions \ref{app:ass:1} and \ref{app:ass:2} hold with $\eta^2 = (1 + 2\Delta^2)^T - 1$.
        
        \medskip
        \noindent \textbf{Step 2.1: \Cref{app:ass:1}.}  For $i \in [d]$, denote $\#_{+i}(x_{1:T}) = \sum_{t=1}^T \mathbbm{1}\{ x_{t, i} =  1\}$.  We have that
        \begin{align*}
            \frac{\diff P_{v^{\oplus i}}^{\otimes T}}{\diff P_v^{\otimes T}}(x_{1:T})
            &= \frac{\prod_{t=1}^T \prod_{j=1}^d \left( \frac{1}{2} + \frac{\Delta}{4}(1 + v_j^{\oplus i}) \right)^{\mathbbm{1}\{ x_{t, j} = 1\}}\left( \frac{1}{2} - \frac{\Delta}{4}(1 + v_j^{\oplus i}) \right)^{\mathbbm{1}\{ x_{t, j} = -1\}}}
            {\prod_{t=1}^T \prod_{j=1}^d \left( \frac{1}{2} + \frac{\Delta}{4}(1 + v_j) \right)^{\mathbbm{1}\{ x_{t, j} = 1\}}\left( \frac{1}{2} - \frac{\Delta}{4}(1 + v_j) \right)^{\mathbbm{1}\{ x_{t, j} = -1\}}} \\
            &= \prod_{t=1}^T \left( \frac{1 + \frac{\Delta}{2}(1 - v_i)}{1 + \frac{\Delta}{2}(1 + v_i)} \right)^{\mathbbm{1}\{ x_{t, i} = 1\}} \left( \frac{1 - \frac{\Delta}{2}(1 - v_i)}{1 - \frac{\Delta}{2}(1 + v_i)} \right)^{\mathbbm{1}\{ x_{t, i} = -1\}} \\
            &= \left( \frac{1 - \frac{\Delta}{2}(1 - v_i)}{1 - \frac{\Delta}{2}(1 + v_i)} \right)^T \left( \frac{1 + \frac{\Delta}{2}(1 - v_i)}{1 + \frac{\Delta}{2}(1 + v_i)} \cdot \frac{1 - \frac{\Delta}{2}(1 + v_i)}{1 - \frac{\Delta}{2}(1 - v_i)} \right)^{\#_{+i}(x_{1:T})}= \varphi_{v, j}(x_{1:T}).
        \end{align*}
        We therefore verify \Cref{app:ass:1}.

        \medskip
        \noindent \textbf{Step 2.2: \Cref{app:ass:2}.} Note that for $i, j \in [d]$, $i \neq j$,
        \begin{align*}
            &\mathbb{E}_{P_v}(\varphi_{v, i} \varphi_{v, j})  \\
            =& \mathbb{E}_{P_v}\left[ \left( \frac{1 - \frac{\Delta}{2}(1 - v_i)}{1 - \frac{\Delta}{2}(1 + v_i)} \right)^T \left( \frac{1 + \frac{\Delta}{2}(1 - v_i)}{1 + \frac{\Delta}{2}(1 + v_i)} \cdot \frac{1 - \frac{\Delta}{2}(1 + v_i)}{1 - \frac{\Delta}{2}(1 - v_i)} \right)^{\#_{+i}(x_{1:T})} \right.  \\
            \times& \left. \left( \frac{1 - \frac{\Delta}{2}(1 - v_j)}{1 - \frac{\Delta}{2}(1 + v_j)} \right)^T \left( \frac{1 + \frac{\Delta}{2}(1 - v_j)}{1 + \frac{\Delta}{2}(1 + v_j)} \cdot \frac{1 - \frac{\Delta}{2}(1 + v_j)}{1 - \frac{\Delta}{2}(1 - v_j)} \right)^{\#_{+j}(x_{1:T})} \right]  \\
            =& \left( \frac{1 - \frac{\Delta}{2}(1 - v_i)}{1 - \frac{\Delta}{2}(1 + v_i)} \cdot \frac{1 - \frac{\Delta}{2}(1 - v_j)}{1 - \frac{\Delta}{2}(1 + v_j)} \right)^T \mathbb{E}_{P_v}\left[ \left( \frac{1 + \frac{\Delta}{2}(1 - v_i)}{1 + \frac{\Delta}{2}(1 + v_i)} \cdot \frac{1 - \frac{\Delta}{2}(1 + v_i)}{1 - \frac{\Delta}{2}(1 - v_i)} \right)^{\#_{+i}(x_{1:T})} \right]  \\ 
            \times& \mathbb{E}_{P_v}\left[ \left( \frac{1 + \frac{\Delta}{2}(1 - v_j)}{1 + \frac{\Delta}{2}(1 + v_j)} \cdot \frac{1 - \frac{\Delta}{2}(1 + v_j)}{1 - \frac{\Delta}{2}(1 - v_j)} \right)^{\#_{+j}(x_{1:T})} \right].
        \end{align*}
        In addition it holds that
        \begin{align*}
            \mathbb{E}_{P_v}&\left[ \left( \frac{1 + \frac{\Delta}{2}(1 - v_i)}{1 + \frac{\Delta}{2}(1 + v_i)} \cdot \frac{1 - \frac{\Delta}{2}(1 + v_i)}{1 - \frac{\Delta}{2}(1 - v_i)} \right)^{\#_{+i}(x_{1:T})} \right] \\
            &= \prod_{t=1}^T \mathbb{E}_{P_v} \left[ \left( \frac{1 + \frac{\Delta}{2}(1 - v_i)}{1 + \frac{\Delta}{2}(1 + v_i)} \cdot \frac{1 - \frac{\Delta}{2}(1 + v_i)}{1 - \frac{\Delta}{2}(1 - v_i)} \right) \mathbbm{1}\{ X_{t,i} = 1\} + \mathbbm{1}\{ X_{t,i} = -1\} \right] \\
            &= \left[ \left( \frac{1 + \frac{\Delta}{2}(1 - v_i)}{1 + \frac{\Delta}{2}(1 + v_i)} \cdot \frac{1 - \frac{\Delta}{2}(1 + v_i)}{1 - \frac{\Delta}{2}(1 - v_i)} \right) \left(\frac{1}{2} + \frac{\Delta}{4}(1 + v_i)\right) + \left(\frac{1}{2} - \frac{\Delta}{4}(1 + v_i)\right) \right]^T \\
            &= \left( \frac{1-\frac{\Delta}{2}(1+v_i)}{1-\frac{\Delta}{2}(1-v_i)} \right) ^T.
        \end{align*}
        We then have that $\mathbb{E}_{P_v}(\varphi_{v, i} \varphi_{v, j}) = 1$.
    
        We now consider
        \begin{align*}
            \mathbb{E}_{P_v}(\varphi_{v, i}^2)
            &= \left( \frac{1 - \frac{\Delta}{2}(1 - v_i)}{1 - \frac{\Delta}{2}(1 + v_i)} \right)^{2T} \mathbb{E}_{P_v}\left[ \left( \frac{1 + \frac{\Delta}{2}(1 - v_i)}{1 + \frac{\Delta}{2}(1 + v_i)} \cdot \frac{1 - \frac{\Delta}{2}(1 + v_i)}{1 - \frac{\Delta}{2}(1 - v_i)} \right)^{2\#_{+i}(x_{1:T})} \right].
        \end{align*}
        It holds that
        \begin{align*}
            \mathbb{E}_{P_v}&\left[ \left( \frac{1 + \frac{\Delta}{2}(1 - v_i)}{1 + \frac{\Delta}{2}(1 + v_i)} \cdot \frac{1 - \frac{\Delta}{2}(1 + v_i)}{1 - \frac{\Delta}{2}(1 - v_i)} \right)^{2\#_{+i}(x_{1:T})} \right] \\
            &= \left[ \left( \frac{1 + \frac{\Delta}{2}(1 - v_i)}{1 + \frac{\Delta}{2}(1 + v_i)} \cdot \frac{1 - \frac{\Delta}{2}(1 + v_i)}{1 - \frac{\Delta}{2}(1 - v_i)} \right)^2 \left(\frac{1}{2} + \frac{\Delta}{4}(1 + v_i)\right) + \left(\frac{1}{2} - \frac{\Delta}{4}(1 + v_i)\right) \right]^T \\
            &= \left( \frac{(1+\frac{\Delta}{2}(1+v))(1+\frac{\Delta^2}{4}(3v^2-2v-1))}{(1-\frac{\Delta}{2}(1-v))^2(1+\frac{\Delta}{2}(1+v))} \right)^T
        \end{align*}
        which combined with the fact that $v_i^2 \equiv 1$ for all $i \in [d]$, yields that
        \begin{align*}
            \mathbb{E}_{P_v}(\varphi_{v, i}^2)
            = \left( \frac{1 + \frac{\Delta^2}{2}(1-v)}{1 - \frac{\Delta^2}{2}(1+v)} \right)^T
            =
            \begin{cases}
                \left( \frac{1}{1-\Delta^2} \right)^T, &  v=1\\
                (1+\Delta^2)^T, &  v=-1.
            \end{cases}
        \end{align*}
        As an upper bound suffices, we have that Assumption \ref{app:ass:2} holds with $\eta^2 = (1 + 2\Delta^2)^T - 1$ when $\Delta^2 \leq 1/2$.
    
 \medskip 
        \noindent
        \textbf{Step 3: Obtaining Lower Bounds.}
        Letting $\mathcal{P}$ denote the $T$-fold product distributions induced by the construction \eqref{app:eq:SparseFamilyLower}, we consider the subset $\mathcal{P}^\ast = \{P \in \mathcal{P} : \|\mathbb{E}_P(X)\|_0 \leq s\}$ and denote $\mathcal{V}^\ast = \{v \in \mathcal{V} : \|v\|_0 \leq s\}$.  We have that $v \in \mathcal{V}^\ast$ if and only if $P_v \in \mathcal{P}^\ast$. Lastly, for a fixed user-level $\alpha$-LDP mechanism $Q$, denote the private marginals satisfying the sparsity condition 
        \begin{equation*}
            \mathcal{M}^\ast = \left\{ M(\cdot) = \int Q(\cdot| x_{1:T}^{(1)}, \hdots, x_{1:T}^{(n)}) \diff P(x_{1:T}^{(1)}, \hdots, x_{1:T}^{(n)}) : P \in \mathcal{P}^\ast \right\}.
        \end{equation*}

        Let $V \in \{0, 1\}^d$ be the random variable such that its co-ordinates are mutually independent with $\mathbb{P}(V_j = 1) = \tau = 1 - \mathbb{P}(V_j = -1)$ where $\tau = s/(2d) \leq 1/2$, and denote $V_+ = \sum_{j=1}^d \mathbbm{1} \{ V_j = 1 \}$. In particular, we have that $V_+ \sim \mathrm{Bin}(d, \tau)$. By an application of Bernstein's inequality \citep[e.g.][Proposition 2.14]{Wainwright:2019}, we have that
        \begin{align*}
            \mathbb{P}(V_+ \geq 2d\tau)
            &\leq \exp\left(-\frac{(d\tau)^2}{2\left\{\sum_{j=1}^d\mathbb{P}(V_j = 1) + d\tau/3\right\}}\right)
            \leq e^{-3d\tau/8} \\
            &= e^{-3s/16}
            \leq \frac{1}{d}
            \leq \frac{3\tau}{8\log(ed)},
        \end{align*}
        where the equality is by the value of $\tau$; the penultimate inequality by the requirement $16\log(ed)/3 \leq s$; and the final inequality by the fact that $16\log(ed)/3 \leq s$ and the value of $\tau$.
        
        Thus, we have that $\mathbb{P}(V \notin \mathcal{V}^\ast) \leq \frac{3\tau}{8}$ under the assumption $16\log(ed)/3 \leq s \leq d$. Using \Cref{app:lem:AcharyaSimple} and \Cref{app:cor:GenAssouad} with the private marginals $\mathcal{M}^\ast$, we obtain
        \begin{equation*}
            \mathcal{R}_{n, T, \alpha}(\theta(\mathcal{P}), \Phi \circ \rho)
            \geq 2d \Delta^2 \tau \left[ 1 - \sqrt{\frac{21n\alpha^2\eta^2}{d}}  - \frac{3}{8} \right]. 
        \end{equation*}
        We now set
        \begin{equation}
            \Delta = \left\{\frac{1}{2}\left( 1 + \frac{d}{84n\alpha^2} \right)^{\frac{1}{T}} - \frac{1}{2}\right\}^{1/2} \wedge \left( \frac{1}{2T} \right)^{1/2}, \label{app:eq:sparseLBdelta}
        \end{equation}
        noting the requirement $\Delta^2 \leq 1/2$ is satisfied.
        When $\Delta = \left\{\frac{1}{2}\left( 1 + \frac{d}{84n\alpha^2} \right)^{\frac{1}{T}} - \frac{1}{2}\right\}^{1/2}$, we have that $\eta^2 = d/(84n\alpha^2)$, giving that
        \begin{align*}
            \mathcal{R}_{n, T, \alpha}(\theta(\mathcal{P}), \Phi \circ \rho)
            \gtrsim s \left[ \left( 1 + \frac{d}{84n\alpha^2} \right)^{\frac{1}{T}} - 1 \right].
        \end{align*}
        When $\Delta = \{1/(2T)\}^{1/2}$, we have that $\eta^2 = (1+1/T)^{1/T} - 1 \leq 1$. In this case, the second term in \eqref{app:eq:sparseLBdelta} is smaller than the first, which upon rearrangement gives
        \begin{equation*}
            \frac{21n\alpha^2}{d} \geq \frac{1}{4\{(1 + 1/T)^{T} - 1\}} \geq \frac{1}{4}.
        \end{equation*}
        Hence, we have
        \begin{align*}
            \mathcal{R}_{n, T, \alpha}(\theta(\mathcal{P}), \Phi \circ \rho)
            \gtrsim \frac{s}{T} \left[ 1 - \sqrt{\frac{21n\alpha^2}{d}}  - \frac{3}{8} \right]
            \gtrsim \frac{s}{T}.
        \end{align*}        
        These two choices of the value of $\Delta$ together give a lower bound of
        \begin{equation*}
            \mathcal{R}_{n, T, \alpha}(\theta(\mathcal{P}), \Phi \circ \rho)
            \gtrsim \frac{s}{T} \wedge s\left\{\left(1 + \frac{d}{84n\alpha^2}\right)^{1/T} - 1 \right\}.
        \end{equation*}
        Lastly, combining with the infinite-$T$ lower bound in \eqref{sec2:eq:sparseMeanStatement}, for which the required assumptions are satisfied, completes the proof.
    \end{proof}

\subsection{Proof of Theorem \ref{sec4:thm:main} Lower Bound} \label{app:sec:densityLB}

We first define a H{\"o}lder space of functions denoted by $\mathcal{H}_{\beta, r}$, for $\beta \in \mathbb{N}$ and $r > 0$.  For any $f \,[0,1] \rightarrow \mathbb{R}$, $ f \in \mathcal{H}_{\beta, r}$, if it is $(\beta-1)$-times differentiable and satisfies that
    \begin{align*}
         |f^{(\beta - 1)}(x) - f^{(\beta - 1)}(y)| \leq r|x - y|.
    \end{align*}

\begin{proof}[Proof of \Cref{sec4:thm:main}]
\
\\
\noindent \textbf{Step 1: Constructing a Separated Family.}  Consider the bump function
        \begin{equation*}
            g(x) =
            \begin{cases}
                e^{-1/\{x(1/2-x)\}}, & x \in [0, 1/2], \\
                -e^{-1/\{(x-1/2)(1-x)\}}, & x \in (1/2,1], \\
                0, &\text{otherwise}.
            \end{cases}
        \end{equation*}
        and define the function $g_\beta$ that $g_\beta(x) = c_\beta g(x)$ with a sufficiently small absolute constant $c_\beta > 0$ depending only on $\beta$, so that 
        \[
            \sup_{x\in [0, 1]} \max\{|g_\beta^{(\beta)}(x)|,\,  |g_\beta(x)|\} \leq 1.
        \]
        We remark that the choice of $c_{\beta}$ exists, since $g(\cdot)$ is infinitely differentiable with bounded derivatives. 
        
    For $k \in \mathbb{N}$ to be specified and for each $j \in [k]$, we let
        \begin{align*}
            g_{\beta, j}(x) = \frac{r\pi^\beta}{2k^\beta} g_\beta \left( k \left( x - \frac{j-1}{k}\right) \right), \quad x \in \mathbb{R}.
        \end{align*}
        Note that $g_{\beta, j}(x) \neq 0$ if and only if $x \in [(j-1)/k, j/k]$, and for $x, y \in [0, 1]$,
        \begin{align*}
            |g_{\beta, j}^{(\beta-1)}(x) - g_{\beta, j}^{(\beta-1)}(y)|
            = \frac{r\pi^\beta}{2k} \left|g_{\beta}^{(\beta-1)}(kx - (j-1)) - g_{\beta}^{(\beta-1)}(ky - (j-1)) \right|
            \leq \frac{r\pi^\beta}{2}|x-y|
        \end{align*}
        where the equality follows from the definition of $g_{\beta, j}^{(\beta-1)}(x)$, and the inequality via (i) an application of the mean value theorem, (ii) the fact that the derivative is bounded as $|g_\beta^{(\beta)}(x)| \leq 1$ for all $x$ and $(iii)$ $k \geq 1$.  It then holds that $g_{\beta, j}^{(\beta-1)} \in \mathcal{H}_{\beta, r\pi^\beta/2}$, $j \in [k]$.  We then define the family of densities indexed by the hypercube $\mathcal{V} = \{-1, 1\}^k$ by
        \begin{align} \label{app:eq:DensityFamily}
            \bigg\{f_v = 1 + \sum_{j=1}^k v_j g_{\beta, j} : \, v \in \mathcal{V} \bigg\} \subseteq \mathcal{F}_{\beta, r},
        \end{align}
        where the inclusion in the Sobolev space of densities $\mathcal{F}_{\beta, r}$ follows from
        \begin{itemize}
            \item the fact that each $f_v$ is a sum of functions in $\mathcal{H}_{\beta, r\pi^\beta/2}$ with disjoint support and is thus itself is in $\mathcal{H}_{\beta, r\pi^\beta}$ ~\citep[see e.g.][Section~2.6.1]{Tsybakov:2009};
            
            \item the fact that (for integer $\beta$) $\mathcal{H}_{\beta, r\pi^{\beta}} \subseteq \mathcal{S}_{\beta, r}$ by Definition 1.11 and Proposition 1.14 in \cite{Tsybakov:2009}; and
            
            \item the fact that $g_{\beta, j}$ is small enough of a perturbation that $f_v \geq 0$, and that, since $\int_0^1 g_\beta(x) \diff x = 0$, we also have that $\int f_v(x) \diff x = 1$.
        \end{itemize}
        We also have that
        \begin{equation}
            \|f_v - f_{v'}\|_2^2
            = 2 \sum_{j=1}^k \mathbbm{1}\{v_j \neq v'_j \} \int_{\frac{j-1}{k}}^{\frac{j}{k}} \{g_{\beta, j}(x)\}^2 \diff x
            \geq \frac{c_\beta' r^2\pi^{2\beta}}{2k^{2\beta + 1}} \sum_{j=1}^k \mathbbm{1}\{v_j \neq v'_j \}, \label{app:eq:densityseparation}
        \end{equation}
        where the inequality comes from integrating by substitution and the fact that $\int \{g_\beta(x)\}^2 \diff x > c_\beta'$ for some constant $c'_\beta > 0$ depending on $\beta$.  The family of distributions is $c_\beta'r^2\pi^{2\beta}/(4k^{2\beta + 1})$-separated under squared-$L_2$ loss.  

        \medskip 
        \noindent
        \textbf{Step 2: Verifying Assumptions~\ref{app:ass:1} and \ref{app:ass:2}.}  We claim that, for the family of distributions obtained by taking the $T$-fold product of the distributions defined via the densities in \eqref{app:eq:DensityFamily}, Assumptions~\ref{app:ass:1} and \ref{app:ass:2} hold with 
        \begin{align*}
                \eta^2 = \left( 1 + \frac{2r^2\pi^{2\beta}}{k^{2\beta + 1}} \right)^T - 1.
            \end{align*}

\medskip 
        \noindent \textbf{Step 2.1: \Cref{app:ass:1}.}  For any $v,v' \in \mathcal{V}$ and a fixed $x \in [0,1]$, we have by the construction in \eqref{app:eq:DensityFamily} that $f_v(x) = 1 + v_j g_{\beta, l}(x)$, $f_{v'}(x) = 1 + v_j' g_{\beta, l}(x)$ for some $l \in [k]$. Further, for any $j \in [k]$
        \begin{equation}
            \sup_x |g_{\beta, j}(x)| \leq \frac{r\pi^\beta}{2k^\beta} \sup_x |g_\beta(x)| \leq \frac{r\pi^\beta c_\beta}{2k^\beta} \sup_x |g(x)| \leq \frac{1}{2} \label{app:eq:boundeddensityperturbation},
        \end{equation}
        where the first two inequalities are from the definitions of the $g_\beta$ and $g$, and the final inequality comes from taking $c_\beta > 0$ small enough. We then see that, for $i \in [k]$,
        \begin{align*}
            \frac{\diff P_{v^{\oplus i}}^{\otimes T}}{\diff P_v^{\otimes T}}(x_{1:T})
            &= \frac{\prod_{t=1}^T f_{v^{\oplus i}}(x_t)}{\prod_{t=1}^T f_v(x_t)}
            = \frac{\prod_{t=1}^T [f_v(x_t) - 2v_i g_{\beta, i}(x_t)]}{\prod_{t=1}^T f_v(x_t)} \\
            &= \prod_{t=1}^T \left[ 1 - \frac{2v_i g_{\beta, i}(x_t)}{1 + v_i g_{\beta, i}(x_t)}\right]
            =: \varphi_{v, i}(x_{1:T}),
        \end{align*}
        where we use \eqref{app:eq:boundeddensityperturbation} to ensure that the denominators are bounded away from zero.  We therefore verify \Cref{app:ass:1}.

    \medskip 
        \noindent \textbf{Step 2.1: \Cref{app:ass:2}.}  Note that for $i, j \in [k]$, $i \neq j$, 
        \begin{align*}
            \mathbb{E}_{P_v}(\varphi_{v, i} \varphi_{v, j})
            &= \prod_{t=1}^T \mathbb{E}_{P_v}\left[ \left(1 - \frac{2v_i g_{\beta, i}(X_t)}{1 + v_i g_{\beta, i}(X_t)} \right) \left(1 - \frac{2v_j g_{\beta, j}(X_t)}{1 + v_j g_{\beta, j}(X_t)} \right) \right] \\
            &= \left\{\mathbb{E}_{P_v}\left[ \left(1 - \frac{2v_i g_{\beta, i}(X_1)}{1 + v_i g_{\beta, i}(X_1)} \right) \left(1 - \frac{2v_j g_{\beta, j}(X_1)}{1 + v_j g_{\beta, j}(X_1)} \right) \right]\right\}^T \\
            &= \left\{\mathbb{E}_{P_v}\left[ 1 - \frac{2v_i g_{\beta, i}(X_1)}{1 + v_i g_{\beta, i}(X_1)} - \frac{2v_j g_{\beta, j}(X_1)}{1 + v_j g_{\beta, j}(X_1)} \right]\right\}^T \\
            &= \left[ 1 - \int_0^1 \left( \frac{2v_i g_{\beta, i}(x)}{1 + v_i g_{\beta, i}(x)} + \frac{2v_j g_{\beta, j}(x)}{1 + v_j g_{\beta, j}(x)} \right) f_v(x) \diff x \right]^T \\
            &= \left[ 1 - 2v_i \int_0^1 g_{\beta, i}(x) \diff x - 2v_j \int_0^1 g_{\beta, j}(x) \diff x \right]^T
            = 1,
        \end{align*}
        where the second equality uses the fact that the $X_{t}$ are i.i.d.~across $t \in [T]$; the third and the penultimate are due to the fact that $g_{\beta, i}$ and $g_{\beta, j}$ have disjoint support, $i \neq j$; and the last is by the fact that $\int g_{\beta, l}(x) \diff x = 0$ for all $l \in [k]$. Hence, we have that $\mathbb{E}_{P_v}(\varphi_{v, i} \varphi_{v, j}) = 1$ for all $i \neq j$.
        
        We then consider that
        \begin{align*}
            \mathbb{E}_{P_v}(\varphi_{v, i}^2)
            &= \prod_{t=1}^T \mathbb{E}_{P_v}\left[ \left(1 - \frac{2v_i g_{\beta, i}(X_t)}{1 + v_i g_{\beta, i}(X_t)} \right)^2 \right]
            = \left\{\mathbb{E}_{P_v}\left[ \left(1 - \frac{2v_i g_{\beta, i}(X_1)}{1 + v_i g_{\beta, i}(X_1)} \right)^2 \right]\right\}^T \\
            &= \left\{\mathbb{E}_{P_v}\left[ 1 - \frac{4v_i g_{\beta, i}(X_1)}{1 + v_i g_{\beta, i}(X_1)} + \left( \frac{2v_i g_{\beta, i}(X_1)}{1 + v_i g_{\beta, i}(X_1)} \right)^2 \right]\right\}^T.
        \end{align*}
        We have that 
        \[
            \mathbb{E}_{P_v}\left[\frac{4v_i g_{\beta, i}(X_1)}{1 + v_i g_{\beta, i}(X_1)}\right] = 0
        \]
        due to the fact that $g_{\beta, i}$ and $g_{\beta, j}$ have disjoint support, $i \neq j$.  It also holds that 
        \begin{align*}
            \mathbb{E}_{P_v}\left[ \frac{4 g_{\beta, i}(X_1)^2}{\{1 + v_i g_{\beta, i}(X_1)\}^2} \right]
            &= \int_0^1 \frac{4 \{g_{\beta, i}(x)\}^2}{\{1 + v_i g_{\beta, i}(x)\}^2} f_v(x) \diff x 
            = \int_{\frac{i-1}{k}}^{\frac{i}{k}} \frac{4 \{g_{\beta, i}(x)\}^2}{1 + v_i g_{\beta, i}(x)} \diff x , \\
            &\leq \frac{2r^2\pi^{2\beta}}{k^{2\beta + 1}}\int_0^1 \{g_{\beta}(y)\}^2 \diff y
            \leq \frac{2r^2\pi^{2\beta}}{k^{2\beta + 1}}, 
        \end{align*}
        where in the first inequality use the substitution $y = kx - (i-1)$ and use \eqref{app:eq:boundeddensityperturbation} to bound the denominator away from $0$ and in the last inequality we use the fact that $|g_\beta(x)| \leq 1$ for all $x \in [0,1]$. Hence, we have the upper bound
        \begin{align*}
             \mathbb{E}_{P_v}(\varphi_{v, i}^2)
            &\leq \left[ 1 +  \frac{2r^2\pi^{2\beta}}{k^{2\beta + 1}} \right]^T,
        \end{align*}
        which yields $\eta^2 = [1 + 2r^2\pi^{2\beta}/k^{2\beta + 1}]^T - 1$. We have now verified \Cref{app:ass:2}.
    
        \medskip 
        \noindent
        \textbf{Step 3: Obtaining Lower Bounds.}
        We set 
        \begin{equation*}
            k = (168\pi^{2\beta})^{1/(2\beta+1)}(nT\alpha^2)^{\frac{1}{2\beta + 2}}.
        \end{equation*}
        We then have, setting $r=1$, that
        \begin{equation}
            \eta^2
            = \left( 1 + \frac{2\pi^{2\beta}}{k^{2\beta + 1}} \right)^T - 1 \leq \frac{4\pi^{2\beta}T}{k^{2\beta + 1}}, \label{app:eq:densityetabound}
        \end{equation}
        where the inequality holds under the condition
        \begin{equation}
            \frac{2\pi^{2\beta}}{k^{2\beta + 1}} \leq \frac{1}{T}, \quad \text{or, equivalently,} \quad T \leq 84(n\alpha^2)^{2\beta + 1}. \label{app:eq:phaseTransition}
        \end{equation}
        As Assumptions \ref{app:ass:1} and \ref{app:ass:2} hold, we have, for $T$ satisfying \eqref{app:eq:phaseTransition}, that
        \begin{align*}
            \mathcal{R}_{n, T, \alpha}(\theta(\mathcal{P}_{d}'), \|\cdot\|_2^2)
            &\gtrsim \frac{c_\beta'r^2\pi^{2\beta}}{4k^{2\beta}} \left(1 - \sqrt{\frac{21}{k}n\alpha^2\eta^2}\right)
            \gtrsim \frac{c_\beta''}{(nT\alpha^2)^{2\beta/(2\beta + 2)}}\left(1 - \sqrt{\frac{84\pi^{2\beta}}{(168\pi^{2\beta})^{2\beta+2/(2\beta + 1)}}} \right), \\
            &\gtrsim c_\beta'''(nT\alpha^2)^{-2\beta/(2\beta + 2)},
        \end{align*}
        where $c_\beta''$ and $c_\beta'''$ are constants depending only on $\beta$; the first inequality is from \eqref{app:eq:AcharyaAssouad} and the fact that the distributions are $c_\beta'r^2\pi^{2\beta}/(4k^{2\beta + 1})$-separated; the second from \eqref{app:eq:densityetabound} and the value of $k$; and the last as $C_\beta$ can be taken sufficiently large.

        We observe that this restriction on the range of values of $T$ is not problematic as when $T > (n\alpha^2)^{2\beta + 1}$, we have $(n\alpha^2)^{-2\beta} > (nT\alpha^2)^{-2\beta/(2\beta + 2)}$ and so the infinite-$T$ lower bound in \eqref{sec2:eq:densityStatement} (for which the required assumptions are satisfied) dominates in this case, and taking the maximum of the two lower bounds completes the proof. 
    \end{proof}

\subsection{Proof of Theorem \ref{secmix:thm:main} Lower Bound}
    In this sub-section we consider the lower bound the problem of estimating the multiple means from a mixture distribution. We first prove the $T$-independent lower bound.
    \begin{proof}[Proof of \Cref{secmix:thm:main} ($T$-independent Bound)]
        Consider the two distributions
        \begin{align*}
            P_{1, \mathrm{mix}}^{T} &= m^{-1} \delta_{1/m}^{\otimes T} + m^{-1} \delta_{2/m}^{\otimes T} + \dots + m^{-1} \delta_{1}^{\otimes T}, \\
            P_{2, \mathrm{mix}}^{T} &= m^{-1} \delta_0^{\otimes T} + m^{-1} \delta_{1/m}^{\otimes T} + \dots + m^{-1} \delta_{(m-1)/m}^{\otimes T},
        \end{align*}
        where $\delta_x$ denotes a Dirac mass centred at $x$. We note that both the mixture proportions and the means of the distributions satisfy their respective separation criteria.

        Proceeding similarly to the proof of \Cref{sec2:ex:example} in \Cref{Appendix_sec2}, we obtain
        \begin{align*}
            \mathcal{R}_{n, T, \alpha}^{\mathrm{mix}}(\theta(\mathcal{P}_{m, \pi_{0}, \theta_{\mathrm{sep}}}), \|\cdot\|_2^2)
            &\geq \frac{1}{4} e^{-12n\alpha^2 D_{\mathrm{TV}}(P_{1, \mathrm{mix}}^{T}, P_{2, \mathrm{mix}}^{T})^2} \\
            &= \frac{1}{4} e^{-12n\alpha^2D_{\mathrm{TV}}(m^{-1}\delta_1^{\otimes T}, m^{-1}\delta_0^{\otimes T})^2} \\
            &= \frac{1}{4}e^{-12n\alpha^2/m^2}
        \end{align*}         
        where the inequality follows from Le Cam's lemma \citep[e.g.][Lemma~1]{Yu:1997}; the Bretagnolle–Huber inequality \citep[e.g.][Equation~(2.25)]{Tsybakov:2009}; and Corollary~3 in \cite{Duchi:2018} where we note that $(e^\alpha-1)^2 \leq 3\alpha^2$ for $\alpha \in (0,1]$. This proves the $T$-independent lower bound of \Cref{secmix:thm:main}.
    \end{proof}

    We now consider the $T$-dependent lower bound. We will require a modification to the proof techniques of \cite{Acharya:2022} to obtain a lower bound, but first explain the cause of failure of the previous techniques if applied naively: Suppose we wish to apply the same ideas by considering a family of distributions wherein the means of differing mixture components are perturbed. A crucial difference between our setting with mixtures and those of the previous lower bounds is the role of \Cref{app:ass:2}, which calls for the fluctuations induced by perturbing the mean of a component to be uncorrelated with those of another component. Whereas this holds in the case of the different dimensions of the multi-dimensional distributions considered earlier, with a mixture distribution the perturbations all affect the same single dimension of the output space, violating \Cref{app:ass:2}.

    To circumvent this issue, we modify the results of \cite{Acharya:2022}, instead conditioning on a single component of the mixture, which forces \Cref{app:ass:2} trivially by construction. This shifts the difficulty of the problem from verifying \Cref{app:ass:2}, to instead controlling the relevant quantities once we take the expectation over which mixture component was chosen.

    We now formalise the modifications we require. Recall we consider a family of distributions $\{P_v : v \in \mathcal{V}\}$ taking values in some space $\mathcal{X}$, where $\mathcal{V} = \{-1, 1\}^k$ for some $k \in \mathbb{N}$. We consider a random variable $I$ taking values in some finite set $\mathcal{I}$ and for a random variable $X \sim P_v$, we denote the distribution of $X \mid I$ by $P_{I, v}$. We note by construction that $\mathbb{E}_I[P_{I, v}] = P_v$.

    We start with an analogue to \Cref{app:ass:1}, which denotes local perturbations around unity for the Radon--Nikodym derivative of $P_{I, v^{\oplus i}}$ and $P_{I, v}$ in terms of fluctuation terms which depend on $I$.
    
    \begin{assumption} \label{app:ass:4}
        For every $v \in \mathcal{V}$ and $j \in [k]$, assume that $P_{i, v^{\oplus j}} \ll P_{i, v}$ for all $i \in \mathcal{I}$ and there exist measurable functions $\phi_{i, v,j} : \mathcal{Z} \rightarrow \mathbb{R}$ such that
        \begin{align*}
            \frac{\diff P_{I, v^{\oplus j}}}{\diff P_{I, v}} = 1 + \phi_{I, v,j}.
        \end{align*}        
    \end{assumption}
    \noindent
    It follows from \Cref{app:ass:4} that $\mathbb{E}_{P_{I, v}}(\phi_{I, v, j} \mid I) = 0$, almost surely, for all $j \in [k]$.

    We then obtain the following analogous result to \Cref{app:thm:Acharya}.
    \begin{proposition} \label{app:prop:AcharyaThmMix}
        Under \Cref{app:ass:4}, we have that
        \begin{equation} \label{app:eq:AcharyaThmMixEq}
            \bigg(\frac{1}{k} \sum_{j=1}^k D_{\mathrm{TV}}(M_{+j}^n, M_{-j}^n) \bigg)^2
            \leq \frac{7n}{k} \max_{v \in \mathcal{V}} \sup_{Q \in \mathcal{Q}_\alpha} \sum_{j = 1}^k \int_{ \mathcal{Z}} \frac{\mathbb{E}_{P_v}[\phi_{I, v, j}(X) q(z \mid X)]^2}{\mathbb{E}_I[\mathbb{E}_{P_{I,v}}[q(z \mid X) \mid I]]} \diff z
        \end{equation}
    \end{proposition}
    \begin{proof}
        For $v \in \mathcal{V}$, we start with
        \begin{align*}
            D_\mathrm{H} (M_{v}, M_{v^{\oplus j}})^2
            &= \frac{1}{2} \int_{ \mathcal{Z}} \big( \mathbb{E}_{P_v}[q(z \mid X)]^{1/2} - \mathbb{E}_{P_{v^{\oplus j}}}[q(z \mid X)]^{1/2} \big)^2 \diff z \\
            &= \frac{1}{2} \int_{ \mathcal{Z}} \bigg( \frac{\mathbb{E}_{P_v}[q(z \mid X)] - \mathbb{E}_{P_{v^{\oplus j}}}[q(z \mid X)]}{\mathbb{E}_{P_v}[q(z \mid X)]^{1/2} + \mathbb{E}_{P_{v^{\oplus j}}}[q(z \mid X)]^{1/2}} \bigg)^2 \diff z \\
            &\leq \frac{1}{2} \int_{ \mathcal{Z}} \frac{\big( \mathbb{E}_{P_v}[q(z \mid X)] - \mathbb{E}_{P_{v^{\oplus j}}}[q(z \mid X)]\big)^2}{\mathbb{E}_{P_v}[q(z \mid X)]} \diff z \\
            &= \frac{1}{2} \int_{ \mathcal{Z}} \frac{\mathbb{E}_I[\mathbb{E}_{P_{I, v}}[\phi_{I, v, j}(X) q(z \mid X) \mid I]]^2}{\mathbb{E}_I[\mathbb{E}_{P_{I, v}}[q(z \mid X)]]} \diff z
        \end{align*}
        where the last line follows from the tower property and the fact, by \Cref{app:ass:4}, that
        \begin{align*}
            \mathbb{E}_{P_{v^{\oplus j}}}[q(z \mid X)]
            = \mathbb{E}_I[\mathbb{E}_{P_{I, v^{\oplus j}}}[q(z \mid X) \mid I] ]
            &= \mathbb{E}_I \bigg[\mathbb{E}_{P_{I, v}}\bigg[ \frac{\diff P_{I, v^{\oplus j}}}{\diff P_{I, v}} q(z \mid X) \biggm| I \bigg] \bigg] \\
            &= \mathbb{E}_I [\mathbb{E}_{P_{I, v}}[ (1 + \phi_{I, v, j}) q(z \mid X) \mid I ] ] \\
            &= \mathbb{E}_{P_v}[q(z \mid X)] + \mathbb{E}_I [\mathbb{E}_{P_{I, v}}[\phi_{I, v, j} q(z \mid X) \mid I]].
        \end{align*}
        Combining \Cref{app:lem:Acharya} with the obtained bound on $D_\mathrm{H} (M_{v}, M_{v^{\oplus j}})^2$ completes the proof. 
    \end{proof}
    
    Analogously to \Cref{app:ass:2}, if we assume the \emph{conditional} fluctuations are (almost surely) orthogonal, we can similarly simplify the right-hand-side of \eqref{app:eq:AcharyaThmMixEq}.
    \begin{assumption}\label{app:ass:5}
        Assume that for all $v \in \mathcal{V}$ and $i,j \in [k]$ with $i \neq j$, there exists $\eta_{I, i} \geq 0$ such that, almost surely,
        \begin{align*}
            \mathbb{E}_{P_{I, v}}[\phi_{I, v, i}\phi_{I, v, j} \mid I] = 0 \quad \mbox{and} \quad \mathbb{E}_{P_{I, v}}[\phi_{I, v, i}^2 \mid I] \leq \eta_{I, i}^2.
        \end{align*}
    \end{assumption}
    
    \begin{lemma} \label{app:prop:AcharyaLemMix}
        When Assumptions \ref{app:ass:4} and \ref{app:ass:5} hold, we have for any $\alpha \in (0,1]$ and any $Q \in \mathcal{Q}_\alpha$ that
        \begin{align*}
            \frac{1}{k} \sum_{j=1}^k D_{\mathrm{TV}}(M_{+j}^n, M_{-j}^n) \leq \sqrt{\frac{21n\alpha^2}{k} \max_{i \in [k]}\{ \mathbb{E}_I[\eta_{I,i}^2] \} \max_{v \in \mathcal{V}} \sup_{Q \in \mathcal{Q}_\alpha} \int_{ \mathcal{Z}} \frac{\mathbb{E}_I[\mathbb{E}_{P_{I, v}}[q(z \mid X) \mid I]^2]}{\mathbb{E}_I[\mathbb{E}_{P_{I,v}}[q(z \mid X) \mid I]]} \diff z}.
        \end{align*}
    \end{lemma}
    \begin{proof}
        We proceed as in the proof of Theorem~2 in \citet{Acharya:2022}. Denoting the normalised fluctuations $\psi_{I, v, k} = \phi_{I, v, k}/\mathbb{E}_{P_{I,v}}[ \phi_{I, v, k}^2 \mid I]^{1/2}$, we note conditional on $I$ and for a fixed $v \in \mathcal{V}$, by Assumptions \ref{app:ass:4} and \ref{app:ass:5}, that we may extend $\{1, \psi_{I, v, 1}, \hdots, \psi_{I, v, k}\}$ to an orthonormal basis of $L_2(\mathcal{X}, P_{I, v})$, the Hilbert space of measurable functions with finite second moment under $P_{I, v}$. We thus obtain
        \begin{align*}
            \sum_{j = 1}^k \mathbb{E}_{P_v}[\phi_{I, v, j}(X) q(z \mid X)]^2
            &= \sum_{j = 1}^k \mathbb{E}_I[\mathbb{E}_{P_{I, v}}[\phi_{I, v, j}(X) q(z \mid X) \mid I]]^2 \\
            &= \sum_{j = 1}^k \mathbb{E}_I[\mathbb{E}_{P_{I,v}}[ \phi_{I, v, k}^2 \mid I]^{1/2} \mathbb{E}_{P_{I, v}}[ \psi_{I, v, j}(X) q(z \mid X) \mid I]]^2 \\
            &\leq \sum_{j = 1}^k \big\{ \mathbb{E}_I[\mathbb{E}_{P_{I,v}}[ \phi_{I, v, k}^2 \mid I]]^{1/2} \mathbb{E}_I [\mathbb{E}_{P_{I, v}}[\psi_{I, v, j}(X) q(z \mid X) \mid I]^2]^{1/2} \big\}^2 \\
            &\leq \max_{j \in [k]}\{ \mathbb{E}_I[\eta_{I,j}^2] \} \mathbb{E}_I \bigg[\sum_{j = 1}^k \mathbb{E}_{P_{I, v}}[\psi_{I, v, j}(X) q(z \mid X) \mid I]^2\bigg] \\
            &= \max_{j \in [k]}\{ \mathbb{E}_I[\eta_{I,j}^2] \} \mathbb{E}_I \bigg[ \sum_{j = 1}^k \langle q(z \mid \cdot), \psi_{I, v, j} \rangle_{L_2(\mathcal{X}, P_{I, v})}^2 \bigg] \\
            &= \max_{j \in [k]}\{ \mathbb{E}_I[\eta_{I,j}^2] \} \mathbb{E}_I \bigg[ \sum_{j = 1}^k \langle q(z \mid \cdot) - \mathbb{E}_{P_{I, v}}[q(z \mid X) \mid I], \psi_{I, v, j} \rangle_{L_2(\mathcal{X}, P_{I, v})}^2 \bigg] \\
            &\leq \max_{j \in [k]}\{ \mathbb{E}_I[\eta_{I,j}^2] \} \mathbb{E}_I[\mathrm{Var}_{P_{I, v}}(q(z \mid X) \mid I)],
        \end{align*}
        where the third line is by the Cauchy--Schwarz inequality; the penultimate line by the orthogonality of $1$ and $\psi_{I, v, i}$; and the final by Bessel's inequality \citep[e.g.][p.~316]{Royden:2010}.

        Hence, we have that
        \begin{align*}
            \sum_{j = 1}^k \int_{ \mathcal{Z}} \frac{[\mathbb{E}_{P_{I, v}}[\phi_{I, v, j}(X) q(z \mid X) ]^2}{\mathbb{E}_{P_{I,v}}[q(z \mid X)]} \diff z
            &\leq \max_{i \in [k]}\{ \mathbb{E}_I[\eta_{I,i}^2] \} \int_{ \mathcal{Z}} \frac{\mathbb{E}_I[\mathrm{Var}_{P_{I, v}}(q(z \mid X) \mid I)]}{\mathbb{E}_I[\mathbb{E}_{P_{I,v}}[q(z \mid X) \mid I]]} \diff z \\
            &\leq 3\alpha^2 \max_{i \in [k]}\{ \mathbb{E}_I[\eta_{I,i}^2] \} \int_{ \mathcal{Z}} \frac{\mathbb{E}_I[\mathbb{E}_{P_{I, v}}[q(z \mid X) \mid I]^2]}{\mathbb{E}_I[\mathbb{E}_{P_{I,v}}[q(z \mid X) \mid I]]} \diff z,
        \end{align*}
        where the final inequality is by the fact that $\mathrm{Var}_{P_{I, v}}(q(z \mid X) \mid I) \leq (e^\alpha - 1)^2\mathbb{E}_{P_{I, v}}[q(z \mid X) \mid I]^2$ (see \citealt[Corollary~1]{Acharya:2022}), and that $(e^\alpha - 1)^2 \leq 3\alpha^2$ for $\alpha \in (0,1]$. Finally, combining the above bound with \Cref{app:prop:AcharyaThmMix} completes the proof.
    \end{proof}

    With these results, we now prove the lower bound for estimating the means of a mixture distribution.

    \begin{proof}[Proof of \Cref{secmix:thm:main} ($T$-dependent Bound)]
        \
        \\
        \textbf{Step 1: Constructing a Separated Family.} We consider the hypercube $\mathcal{V} = \{-1, 1\}^m$ for $m$ the number of mixture components. For $v \in \mathcal{V}$, let $P_v$ be the mixture distribution
        \begin{equation} \label{app:eq:MixtureFamily}
            P_v = \sum_{j = 1}^m m^{-1} \mathrm{Ber}(\theta_j + v_j \Delta/2)^{\otimes T},
        \end{equation}
        where $\theta_j = 1/4 + j/(2m)$ and $\Delta = \{cm^2/(nT\alpha^2)\}^{1/2}$ for $c > 0$ a sufficiently small absolute constant. In particular, we note that by the theorem statement that $n\alpha^2 \gtrsim m^2\log(n\alpha^2)$ and $T \gtrsim m^2\log(nT\alpha^2)$, and hence for $c$ sufficiently small, the separation condition on the component means is satisfied for all $P_v$. Hence, have that $P_v \in \mathcal{P}_{m, \pi_{0}, \theta_{\mathrm{sep}}}$ for $\mathcal{P}_{m, \pi_{0}, \theta_{\mathrm{sep}}}$ as defined in \eqref{secmix:eq:MixClass}.
        
        Observing $\theta(P_v) = \theta + \Delta v/2$, for the Hamming separation condition \eqref{app:eq:HammingSeparation} it then holds that
        \begin{align*}
            \frac{1}{4}\|\theta(P_v) - \theta(P_{v'})\|_2^2 \geq \Delta^2 \sum_{j=1}^d \mathbbm{1} \{ v_j \neq v_j'\},
        \end{align*}
        which shows that the distributions are $\Delta^2$-separated.
    
        Denote the $j$-th component of the mixture distribution $P_{j, v} = \mathrm{Ber}(\theta_j + v_j \Delta/2)^{\otimes T}$. We consider the random variable $I = \mathrm{Unif}(\{1, 2, \hdots, m\})$ such that for $X \mid I \sim P_{I, v}$. We then have $P_v = \mathbb{E}_I[P_{I, v}]$.

        \medskip 
        \noindent
        \textbf{Step 2: Verifying Assumptions~\ref{app:ass:4} and \ref{app:ass:5}.}
        We claim that, for the family of distributions \eqref{app:eq:MixtureFamily} and the random variable $I$, we have that Assumptions \ref{app:ass:4} and \ref{app:ass:5} hold with
        \begin{align*}
            \eta_{I, i}^2 = \mathbbm{1}\{I = i\}CT\Delta^2,
        \end{align*} 
        for $C > 0$ a sufficiently large absolute constant.
        
        Starting with \Cref{app:ass:4}, we see that, for $i \in [m]$,
        \begin{align*}
            \frac{\diff P_{I, v^{\oplus i}}}{\diff P_{I, v}}(x_{1:T})
            &= \frac{ \prod_{t=1}^T(\theta_I + v_I^{\oplus i} \Delta/2)^{\mathbbm{1}\{x_t = 1\}}(1 - \theta_I - v_I^{\oplus i} \Delta/2)^{\mathbbm{1}\{x_t = -1\}}}{\prod_{t=1}^T(\theta_I + v_I \Delta/2)^{\mathbbm{1}\{x_t = 1\}}(1 - \theta_I - v_I \Delta/2)^{\mathbbm{1}\{x_t = -1\}}} \\
            &= 1 + \mathbbm{1}\{I = i\} \bigg\{ \bigg( \frac{1 - \theta_i + v_i \Delta/2}{1 - \theta_i - v_i \Delta/2} \bigg)^T \bigg(\frac{(\theta_i - v_i \Delta/2)(1 - \theta_i - v_i \Delta/2)}{(\theta_i + v_i \Delta/2)(1 - \theta_i + v_i \Delta/2)}\bigg)^S - 1 \bigg\} \\
            &=: 1 + \phi_{I, v, i}(x_{1:T}),
        \end{align*}
        where $S = \sum_{t = 1}^T \mathbbm{1}\{x_t = 1\}$.
        This verifies \Cref{app:ass:4}.
        
        We now consider \Cref{app:ass:5}. By the form of $\phi_{I, v, i}$, we immediately have for $i \neq j$ that $\phi_{I, v, i}\phi_{I, v, j} \equiv 0$. Hence, it remains to consider $\mathbb{E}_{P_{I, v}}[\phi_{I, v, i}^2 \mid I]$. In particular, as $\phi_{I, v, i} = \diff P_{I, v^{\oplus i}}/\diff P_{I, v} - 1$, we have the almost sure bound
        \begin{align*}
            \mathbb{E}_{P_{I, v}}[\phi_{I, v, i}^2 \mid I]
            &= \mathbbm{1}\{I = i\} \chi^2(\mathrm{Ber}(\theta_i - v_i \Delta/2)^{\otimes T}, \mathrm{Ber}(\theta_i + v_i \Delta/2)^{\otimes T}) \\
            &= \mathbbm{1}\{I = i\} \{1 + \chi^2(\mathrm{Ber}(\theta_i - v_i \Delta/2), \mathrm{Ber}(\theta_i + v_i \Delta/2))\}^T - \mathbbm{1}\{I = i\} \\
            &= \mathbbm{1}\{I = i\} \bigg\{1 + \frac{\Delta^2}{(\theta_i + v_i\Delta/2)(1 - \theta_i - v_i\Delta/2)} \bigg\}^T - \mathbbm{1}\{I = i\} \\
            &\leq \mathbbm{1}\{I = i\} CT\Delta^2
        \end{align*}
        where the first equality is by the definition of $\chi^2$-divergence; the second by the tensorisation identity for $\chi^2$-divergence \cite[e.g.][p.~86]{Tsybakov:2009}; and the inequality by noting that the values of $\theta_i$ and $\Delta$ are such that that the denominator in the third line is bounded away from zero by a constant, and $\Delta$ is sufficiently small so that there exists an absolute constant $C > 0$ satisfying
        \begin{equation*}
            \bigg\{1 + \frac{\Delta^2}{(\theta_i + v_i\Delta/2)(1 - \theta_i - v_i\Delta/2)} \bigg\}^T
            \leq 1 + CT\Delta^2.
        \end{equation*}
        Hence, we have that \Cref{app:ass:5} is satisfied with $\eta_{I, i}^2 = \mathbbm{1}\{I = i\}CT\Delta^2$.

        \medskip        
        \noindent
        \textbf{Step 3: Obtaining Lower Bounds.}
        By \eqref{app:eq:userlevelriskassouad}, the fact that the family $\{P_v\}_{v \in \mathcal{V}}$ is $\Delta^2$-separated, and \Cref{app:prop:AcharyaLemMix}, we have that
        \begin{equation}
            \begin{aligned}
                \mathcal{R}_{n, T, \alpha}^{\mathrm{mix}}(\theta(&\mathcal{P}_{m, \pi_{0}, \theta_{\mathrm{sep}}}), \|\cdot\|_2^2) \\
                &\geq \frac{m\Delta^2}{2}\bigg( 1 - \sqrt{\frac{21n\alpha^2}{m} \max_{i \in [k]}\{ \mathbb{E}_I[\eta_{I,i}^2] \} \max_{v \in \mathcal{V}} \sup_{Q \in \mathcal{Q}_\alpha} \int_{ \mathcal{Z}} \frac{\mathbb{E}_I[\mathbb{E}_{P_{I, v}}[q(z \mid X) \mid I]^2]}{\mathbb{E}_I[\mathbb{E}_{P_{I,v}}[q(z \mid X) \mid I]]} \diff z} \bigg)
            \end{aligned} \label{app:eq:mixLBoverall}
        \end{equation}
        
        Hence, it remains to control the quantities in the square root term above. We first control the quantity
        \begin{align}
            \int_{ \mathcal{Z}} \frac{\mathbb{E}_I[\mathbb{E}_{P_{I, v}}[q(z \mid X) \mid I]^2]}{\mathbb{E}_I[\mathbb{E}_{P_{I,v}}[q(z \mid X) \mid I]]} \diff z
            &= \int_{ \mathcal{Z}} \frac{\sum_{i = 1}^m \mathbb{P}(I = i) \mathbb{E}_{P_{i, v}}[q(z \mid X)]^2}{\sum_{j = 1}^m \mathbb{P}(I = j) \mathbb{E}_{P_{j,v}}[q(z \mid X)]} \diff z \nonumber \\
            &= \sum_{i = 1}^m \int_{ \mathcal{Z}} \frac{  \mathbb{E}_{P_{i, v}}[q(z \mid X)]^2}{\sum_{j = 1}^m \mathbb{E}_{P_{j,v}}[q(z \mid X)]} \diff z \nonumber \\
            &= \sum_{i = 1}^m \int_{ \mathcal{Z}} \frac{  \mathbb{E}_{P_{i, v}}[q(z \mid X)]}{1 + \sum_{j \neq i} \mathbb{E}_{P_{j,v}}[q(z \mid X)]/\mathbb{E}_{P_{i,v}}[q(z \mid X)]} \diff z \nonumber \\
            &\leq \sum_{i = 1}^m \int_{ \mathcal{Z}} \frac{  \mathbb{E}_{P_{i, v}}[q(z \mid X)]}{1 + (m-1)e^{-\alpha}} \diff z = \frac{m}{1 + (m-1)e^{-\alpha}} \leq e \label{app:eq:mixLB1}
        \end{align}
        where the second equality uses the fact that $I \sim \mathrm{Unif}(\{1, 2, \hdots, m\})$; the first inequality by \Cref{app:lem:marginalratio}; and the final equality by the fact that $\alpha \in (0,1]$.

        We also have
        \begin{align}
            \max_{i \in [m]}\{ \mathbb{E}_I[\eta_{I,i}^2] \}
            = \max_{i \in [m]}\{ CT\Delta^2 \mathbb{P}(I = i) \}
            = \frac{CT\Delta^2}{m}. \label{app:eq:mixLB2}
        \end{align}

        Hence, combining \eqref{app:eq:mixLBoverall}, \eqref{app:eq:mixLB1} and \eqref{app:eq:mixLB2}, we obtain
        \begin{align*}
            \mathcal{R}_{n, T, \alpha}^{\mathrm{mix}}(\theta(&\mathcal{P}_{m, \pi_{0}, \theta_{\mathrm{sep}}}), \|\cdot\|_2^2)
            \geq \frac{m\Delta^2}{2} \bigg(1 - \sqrt{\frac{21CenT\alpha^2\Delta^2}{m^2} }\bigg)
            \gtrsim \frac{m^3}{nT\alpha^2},
        \end{align*}
        where the final inequality uses the value $\Delta = \{cm^2/(nT\alpha^2)\}^{1/2}$ with $c$ taken small enough. This completes the proof.
    \end{proof}

\section{Auxiliary Technical Details} \label{app:sec:misc}
 
    We have the following lemma which follows from \citet[][Lemma~2]{Levy:2021}.
    \begin{lemma}[{\citealp[][Lemma~2]{Levy:2021}}] \label{sec3:lem:RandomRotation}
    Assume that the dimension $d \in \mathbb{N}$ is a power of $2$ and fix any $j \in [d]$. For the rotation matrix $R_d$ in \eqref{eq-rotation-matrix}, vectors $x, x_0 \in \mathbb{R}^d$ and $\gamma \in (0,1)$, we have, with probability at least $1 - \gamma$, that
        \begin{equation*}
            |(R_d x)_j - (R_d x_0)_j| \leq \frac{10 \|x - x_0\|_2 \{\log(1/\gamma)\}^{1/2}}{d^{1/2}}.
        \end{equation*}        
    \end{lemma}

    We now prove \Cref{sec5:cor:mainthmsimplified}.
    \begin{proof}[Proof of \Cref{sec5:cor:mainthmsimplified}]
        \noindent
        \begin{enumerate}[{Regime} 1.]
            \item Consider $n\alpha^2 \lesssim d^{\gamma}$ for some $0 < \gamma \leq 1$. This is the usual regime of interest for sparse problems, where the effective sample size is smaller than the ambient dimension. For the form of the lower bound in \eqref{sec5:eq:SparseMeanStatement}, we first focus on the case $\gamma < 1$. We have
            \begin{equation*}
                \left( 1 + \frac{d}{n\alpha^2} \right)^{1/T} - 1
                \asymp \frac{\log\{1 + d/(n\alpha^2)\}}{T}
                \gtrsim \frac{\log(1+d^{1-\gamma})}{T}
                \asymp \frac{\log(d)}{T},
            \end{equation*}
            where the first equality up to constants holds as the expression is bounded due to the assumption $T \gtrsim \log\{d/(n\alpha^2)\}$. By the condition that $s \gtrsim \log(d)$, the lower bound takes the form $s/T \vee e^{-Cn\alpha^2/s}$.

            When $\gamma = 1$, we have $n\alpha^2 \lesssim d$, and so write write $x = d/(n\alpha)^2 > \tilde{c}$ for some absolute constant $\tilde{c} > 0$. We then have
            \begin{align*}
                \bigg(1 + \frac{d}{n\alpha^2} \bigg)^{1/T} - 1
                = \exp\bigg( \frac{\log(1 + x)}{T}\bigg) - 1.
            \end{align*}
            In particular, we have the assumption $T \gtrsim \log\{d/(n\alpha^2)\}$ and so may take $T \geq 2\log(x) \vee 1$. Hence, $\log(1 + x)/T \leq 1$, and so we have
            \begin{align*}
                \exp\bigg( \frac{\log(1 + x)}{T}\bigg) - 1
                \geq \frac{\log(1+x)}{T} \geq \frac{\log(1+c)}{T} \gtrsim \frac{1}{T}.
            \end{align*}
            Hence, the lower bound takes the form $s/T$ as in the case $\gamma \in (0, 1)$.
            
            For the upper bound, we note that the second term in the minimum in \eqref{sec5:eq:SparseMeanStatement} is smaller than the third in this regime, giving an upper bound of $\min\{s, s\log(dnT\alpha^2)/T + e^{-cn\alpha^2/s}\}$. Under the condition $s\log(dnT\alpha^2)/T + e^{-cn\alpha^2/s} \lesssim s$, we obtain an upper bound of the order $s\log(dnT\alpha^2)/T + e^{-cn\alpha^2/s}$, completing the proof for this regime.
        
            \item
            Consider $n\alpha^2 \geq \widetilde{C}d\log(dnT\alpha^2)$ where $\widetilde{C}$ is as in \Cref{sec5:thm:main}. We first note in this regime that
            \begin{equation*}
                e^{-cn\alpha^2/s} \leq e^{-cn\alpha^2/d} \leq e^{-c\widetilde{C}\log(dnT\alpha^2)} \leq \frac{1}{dnT\alpha^2},
            \end{equation*}
            for $\widetilde{C}$ sufficiently large. Hence, we have that the polynomial term dominates in the lower bound of \eqref{sec5:eq:SparseMeanStatement}, and the polynomial terms dominate in both the second and third terms in the minimum in the upper bound of \eqref{sec5:eq:SparseMeanStatement}.
            
            We then note that
            \begin{equation*}
                \left( 1 + \frac{d}{n\alpha^2} \right)^{1/T} - 1 \asymp \frac{d}{nT\alpha^2},
            \end{equation*}
            where the equality up to constants holds as $d/(n\alpha^2)$ is bounded in this regime. Hence, the lower bound of \eqref{sec5:eq:SparseMeanStatement} takes the form $sd/(nT\alpha^2)$.
            
            For the upper bound, in this regime we have that $s/T \geq \widetilde{C}sd\log(dnT\alpha^2)/(nT\alpha^2)$ and so the third term in the minimum is less than the second term. This gives an upper bound of $\min\{s$, $sd\{\log(dnT\alpha^2)\}^2/(nT\alpha^2)\}$. Under the condition $sd\{\log(dnT\alpha^2)\}^2/(nT\alpha^2) \lesssim s$, we obtain an upper bound of the order $sd\{\log(dnT\alpha^2)\}^2/(nT\alpha^2)$. This completes the proof.
        \end{enumerate}
    \end{proof}

    \begin{lemma} \label{app:lem:marginalratio}
        Denote two probability distributions $P_1, P_2$ on some space $\mathcal{X}$, and denote by $Q$ an (interactive) $\alpha$-LDP channel which is assumed to have density $q(\cdot \mid x, z')$ for $z' \in \mathcal{Z}^{m}$ for some $m \in \mathbb{N}$. Let the private marginals densities $m_i(\cdot \mid z') = \int_{\mathcal{X}} q(\cdot \mid x, z') \diff P_i(x)$ for $i \in \{1,2\}$. It holds that
        \begin{equation*}
            e^{-\alpha} \leq \frac{m_1(z \mid z')}{m_2(z \mid z')} \leq e^\alpha.
        \end{equation*}
    \end{lemma}
    \begin{proof}
        For the upper bound
        \begin{equation*}
            \frac{m_1(z \mid z')}{m_2(z \mid z')}
            = \frac{\int_\mathcal{X} q(z \mid x, z')\diff P_1(x)}{\int_\mathcal{X} q(z \mid x, z')\diff P_2(x)}
            \leq \frac{\int_\mathcal{X}\int_\mathcal{X} e^{\alpha} q(z \mid y, z')\diff P_2(y) \diff P_1(x)}{\int_\mathcal{X} q(z \mid x, z')\diff P_2(x)}
            \leq e^{\alpha},
        \end{equation*}
        where the inequality is by Tonelli's theorem and the fact $Q$ is an $\alpha$-LDP channel. The lower bound follows similarly, which completes the proof.
    \end{proof}

    \begin{lemma} \label{app:lem:KLboundonevent}
        Recall the notation and setting of the proof of \Cref{sec2:thm:GeneralBoundApproximateLDP}. It holds for all $v \in \mathcal{V}$ that, almost surely,
        \begin{equation} \label{app:eq:KLboundonevent}
             \sum_{z_i \in [k]} \{m_{v, B_i = 0}^{(i)}(z_i \mid Z_{1:i-1}) - \overline{m}_{B_i = 0}^{(i)}(z_i \mid Z_{1:i-1})\} \log\bigg( \frac{m_{v, B_i = 0}^{(i)}(z_i \mid Z_{1:i-1})}{\overline{m}_{B_i = 0}^{(i)}(z_i \mid Z_{1:i-1})}\bigg)
            \leq C\min\{\alpha, \alpha^2\},
        \end{equation}
        where $C > 0$ is some absolute constant.
    \end{lemma}
    \begin{proof}
        We first prove a bound on the ratio of $m_{v, B_i = 0}^{(i)}(z_i \mid Z_{1:i-1})$ and $\overline{m}_{B_i = 0}^{(i)}(z_i \mid Z_{1:i-1})$. Importantly, we note $m_{v, B_i = 0}^{(i)}(\cdot \mid Z_{1:i-1})$ is absolutely continuous with respect to $\overline{m}_{B_i = 0}^{(i)}(z_i \mid Z_{1:i-1})$, and that $m_{v, B_i = 0}^{(i)}(\cdot \mid Z_{1:i-1})$ is non-zero on $B_i = 0$.
        
        Assuming $n > 2$ via the assumption $n\min\{1, \alpha\} \geq C$ for some $C > 0$, we have that almost surely
        \begin{align}
            &\frac{m_{v, B_i = 0}^{(i)}(z_i \mid Z_{1:i-1})}{\overline{m}_{B_i = 0}^{(i)}(z_i \mid Z_{1:i-1})} \nonumber \\
            &= \bigg( \sum_{v' \in \mathcal{V}} \frac{m_{v'}^{(i)}(z_i \mid Z_{1:i-1}, B_i(V, Z_{1:i}) = 0)}{m_{v}^{(i)}(z_i \mid Z_{1:i-1}, B_i(V, Z_{1:i}) = 0)} \mathbb{P}(V = v' \mid Z_{1:i-1}, B_i(V, Z_{1:i}) = 0) \bigg)^{-1} \nonumber \\
            &= \bigg(\sum_{v' \in \mathcal{V}} \frac{\mathbbm{1}\{B_i(v', Z_{1:i-1}, z_i) = 0\} m_{v'}^{(i)}(z_i \mid Z_{1:i-1})\mathbb{P}(B_i(V, Z_{1:i}) = 0 \mid V = v, Z_{1:i-1})}{\mathbbm{1}\{B_i(v, Z_{1:i-1}, z_i) = 0\} m_{v}^{(i)}(z_i \mid Z_{1:i-1})\mathbb{P}(B_i(V, Z_{1:i}) = 0 \mid V = v', Z_{1:i-1})} \nonumber \\
            &\hspace{9cm} \times \mathbb{P}(V = v' \mid Z_{1:i-1}, B_i(V, Z_{1:i}) = 0)\bigg)^{-1} \nonumber \\
            &\leq \bigg(e^{-4\alpha}\sum_{v' \in \mathcal{V}} \frac{\mathbbm{1}\{B_i(v', Z_{1:i-1}, z_i) = 0\} \mathbb{P}(B_i(V, Z_{1:i}) = 0 \mid V = v, Z_{1:i-1})}{\mathbb{P}(B_i(V, Z_{1:i}) = 0 \mid V = v', Z_{1:i-1})} \nonumber \\ 
            &\hspace{9cm} \times \mathbb{P}(V = v' \mid Z_{1:i-1}, B_i(V, Z_{1:i}) = 0)\bigg)^{-1} \nonumber \\
            &\leq \bigg(e^{-4\alpha}(1 - 2/n) \sum_{v' \in \mathcal{V}} \mathbbm{1}\{B_i(v', Z_{1:i-1}, z_i) = 0\} \mathbb{P}(V = v' \mid Z_{1:i-1}, B_i(V, Z_{1:i}) = 0)\bigg)^{-1}
            = \frac{e^{4\alpha}}{1 - 2/n} \nonumber
        \end{align}
        where the first inequality is by the conditions defining $B_i$ in \eqref{app:eq:Bvardefs}, and the second by \eqref{app:eq:BiEntropProbFinalControl} and the fact $\delta \leq c\min\{1, \alpha\}/\{n\log(n)\}$ for $c > 0$ some absolute constant. Arguing similarly for a lower bound, we obtain
        \begin{equation} \label{app:eq:approxLDPlike}
            \bigg( \frac{e^{4\alpha}}{1 - 2/n} \bigg)^{-1}
            \leq \frac{m_{v, B_i = 0}^{(i)}(z_i \mid Z_{1:i-1})}{\overline{m}_{B_i = 0}^{(i)}(z_i \mid Z_{1:i-1})}
            \leq \frac{e^{4\alpha}}{1 - 2/n}
        \end{equation}

        With the previous bound in hand, we now prove \eqref{app:eq:KLboundonevent} two cases depending on the value of $\alpha$. We first consider the case $\alpha \in (0,1]$. In particular, we have by assumption that $n\alpha \geq C$ for some absolute constant $C > 0$, and hence
        \begin{align*}
            &\sum_{z_i \in [k]} \{m_{v, B_i = 0}^{(i)}(z_i \mid Z_{1:i-1}) - \overline{m}_{B_i = 0}^{(i)}(z_i \mid Z_{1:i-1})\} \log\bigg( \frac{m_{v, B_i = 0}^{(i)}(z_i \mid Z_{1:i-1})}{\overline{m}_{B_i = 0}^{(i)}(z_i \mid Z_{1:i-1})}\bigg) \\
            &\leq \sum_{z_i \in [k]}  \overline{m}^{(i)}(z_i \mid Z_{1:i-1}) \bigg| \frac{m_v^{(i)}(z_i \mid Z_{1:i-1})}{\overline{m}^{(i)}(z_i \mid Z_{1:i-1})} - 1 \bigg| \cdot \bigg| \log\bigg( \frac{m_{v, B_i = 0}^{(i)}(z_i \mid Z_{1:i-1})}{\overline{m}_{B_i = 0}^{(i)}(z_i \mid Z_{1:i-1})}\bigg) \bigg| \\
            &\leq \bigg(\frac{e^{4\alpha}}{1-2/n} - 1 \bigg)\{4\alpha - \log(1-2/n)\}
            \leq \bigg(\frac{1}{1-2/n} - 1 + C'\alpha \bigg)\{4\alpha - \log(1-2/n)\} \\
            &\leq \bigg(\frac{1}{1-2/n} - 1 + C'\alpha \bigg)2\alpha
            \leq C''\alpha^2
        \end{align*}
        for $C', C'' > 0$ some absolute constants, where the second inequality is by \eqref{app:eq:approxLDPlike} and the fact $|x - 1| \leq a - 1$ for $x \in [a^{-1}, a]$ with $a > 1$; the third by the fact that $e^{4\alpha} \leq 1 + C'\alpha$ for $\alpha \in (0,1]$; and the final two inequalities by the fact that $1/n \leq \alpha/C$ and taking $C$ sufficiently large.

        In the case $\alpha > 1$, we instead have
        \begin{align*}
            &\sum_{z_i \in [k]} \{m_{v, B_i = 0}^{(i)}(z_i \mid Z_{1:i-1}) - \overline{m}_{B_i = 0}^{(i)}(z_i \mid Z_{1:i-1})\} \log\bigg( \frac{m_{v, B_i = 0}^{(i)}(z_i \mid Z_{1:i-1})}{\overline{m}_{B_i = 0}^{(i)}(z_i \mid Z_{1:i-1})}\bigg) \\
            &\leq \sum_{z_i \in [k]} |m_{v, B_i = 0}^{(i)}(z_i \mid Z_{1:i-1}) - \overline{m}_{B_i = 0}^{(i)}(z_i \mid Z_{1:i-1})| \{4\alpha - \log(1-2/n)\} \\
            &\leq 2\{4\alpha - \log(1-2/n)\}
            \leq C'''\alpha
        \end{align*}
        for $C''' > 0$ some absolute constant, where the first inequality is by \eqref{app:eq:approxLDPlike}, and the last by the fact that $\alpha > 1 > -\log(1 - 2/n)$. Combining the results of these two cases completes the proof.
    \end{proof}

    \begin{proof}[Proof of \Cref{app:lem:GenAssouadPrelim}] 
        For an estimator $\hat{\theta}$, we define $\hat{v}$ to be any element of the set $\argmin_{v \in \mathcal{V}} \rho(\theta(P_v), \hat{\theta})$.  For any $v \in \mathcal{V}$, letting $\theta_v = \theta(P_v)$, it holds that
        \begin{equation}
                2\rho(\hat{\theta}, \theta_v)
                \geq \rho(\hat{\theta}, \theta_v) + \rho(\hat{\theta}, \theta_{\hat{v}})
                \geq \rho(\theta_v, \theta_{\hat{v}}), \label{app:eq:AssouadTriangle}        
        \end{equation}
        where the first inequality is by the definition of $\hat{v}$ and the second by the triangle inequality for the metric~$\rho$.
        
        For $\mathcal{V} = \{-1,1\}^k$ the hypercube which parameterises a family of distributions in $\mathcal{P}$, and for the subfamily of distributions $\mathcal{P}^\ast \subseteq \mathcal{P}$ under consideration, recall that $\mathcal{V}^\ast = \{v \in \mathcal{V}: P_v \in \mathcal{P}^\ast\}$ denotes the subset of~$\mathcal{V}$ which consists of $v$ corresponding to distributions in $\mathcal{P}^\ast$.

        Firstly, by \eqref{app:eq:AssouadTriangle} and fact that $\Phi(\cdot)$ is increasing, we see that
        \begin{equation}
            \sup_{P \in \mathcal{P}^\ast} \mathbb{E}_{P} \left[ \Phi \circ \rho (\hat{\theta}, \theta(P)) \right]
            \geq \sup_{P \in \mathcal{P}^\ast} \mathbb{E}_{P} \left[ \Phi\left( \frac{1}{2} \rho ( \theta_{\hat{v}}, \theta(P) )\right) \right]. \label{app:eq:AssouadTriangleCons}
        \end{equation}
        Let $V$ be any random variable on the set $\mathcal{V}$ and let $V^\ast$ be the random variable arising from conditioning~$V$ on the event $\{V \in \mathcal{V}^\ast \}$. Then, as any average is less than or equal to a supremum, we have
        \begin{align*}
            \sup_{P \in \mathcal{P}^\ast} \mathbb{E}_{P} \left[ \Phi\left( \frac{1}{2} \rho ( \theta_{\hat{v}}, \theta(P) ) \right) \right]
            \geq \mathbb{E}_{V^\ast} \left[ \mathbb{E}_{P_{V^\ast}} \left[ \Phi\left( \frac{1}{2} \rho ( \theta_{\hat{v}}, \theta(P_{V^\ast}) ) \right) \right] \right].
        \end{align*}
        Writing $\theta(P_{V}) = \theta_{V}$ and using the law of total expectation, we also have
        \begin{align*}
            \mathbb{E}_{V} \left[ \mathbb{E}_{P_{V}} \left[ \Phi\left( \frac{1}{2} \rho ( \theta_{\hat{v}}, \theta(P_V) ) \right) \right] \right]
            = \mathbb{E}_{V} &\left[ \mathbb{E}_{P_{V}} \left[ \Phi\left( \frac{1}{2} \rho ( \theta_{\hat{v}}, \theta(P_V) ) \right) \right] \mid V \in \mathcal{V}^\ast \right] \mathbb{P}(V \in \mathcal{V}^\ast) \\
            &+ \mathbb{E}_{V} \left[ \mathbb{E}_{P_{V}} \left[ \Phi\left( \frac{1}{2} \rho (\theta_{\hat{v}}, \theta(P_{V}) \right) \right] \mid V \notin \mathcal{V}^\ast \right] \mathbb{P}(V \notin \mathcal{V}^\ast).
        \end{align*}
        Trivially bounding the first probability term by $1$ and bounding the risk in the second line using the worst case under the separation condition \eqref{app:eq:HammingSeparationEquality}, we have
        \begin{align*}
            \mathbb{E}_{V} \left[ \mathbb{E}_{P_{V}} \left[ \Phi\left( \frac{1}{2} \rho ( \theta_{\hat{v}}, \theta(P_V) ) \right) \right] \right]
            &\leq \mathbb{E}_{V} \left[ \mathbb{E}_{P_{V}} \left[ \Phi\left( \frac{1}{2} \rho ( \theta_{\hat{v}}, \theta(P_V) ) \right) \right] \mid V \in \mathcal{V}^\ast \right] + 2k \varrho \mathbb{P}(V \notin \mathcal{V}^\ast), \\
            &= \mathbb{E}_{V^\ast} \left[ \mathbb{E}_{P_{V^\ast}} \left[ \Phi\left( \frac{1}{2} \rho ( \theta_{\hat{v}}, \theta(P_{V^\ast}) ) \right) \right] \right] + 2k \varrho \mathbb{P}(V \notin \mathcal{V}^\ast).
        \end{align*}
        Hence, we have that
        \begin{align} \label{app:eq:GenAssouad}
            \sup_{P \in \mathcal{P}^\ast} \mathbb{E}_{P} \left[ \Phi \circ \rho (\hat{\theta}, \theta(P)) \right]
            \geq \mathbb{E}_{V} \left[ \mathbb{E}_{P_{V}} \left[ \Phi \circ \rho (\hat{\theta}, \theta(P_V)) \right] \right] - 2k \varrho \mathbb{P}(V \notin \mathcal{V}^\ast),
        \end{align}
        which completes the proof.
    \end{proof}

    \begin{proof}[Proof of \Cref{app:cor:GenAssouad}]
        Given the random variable $V$, we denote for $j \in[k]$ the mixture distributions
        \begin{equation*}
            P_{+j} = \sum_{v \in \mathcal{V} : v_j = 1} \mathbb{P}(V = v | V_j = 1) P_{v} \quad \mbox{and} \quad P_{-j} = \sum_{v \in \mathcal{V} : v_j = -1} \mathbb{P}(V = v | V_j = -1) P_{v}.
        \end{equation*}
        We then have that
        \begin{align*}
            \mathbb{E}_{V} \left[ \mathbb{E}_{P_{V}} \left[ \Phi \circ \rho (\hat{\theta}, \theta(P)) \right] \right]
            &\geq \mathbb{E}_{V} \left\{ \mathbb{E}_{P_{V}} \left[ 2\varrho \sum_{j=1}^k \mathbbm{1}\{ \hat{v}_j \neq V_j \} \right] \right\}\\
            & = 2\varrho \sum_{j=1}^k \mathbb{E}_{V} \left[\mathbb{E}_{P_{V}} \left[\mathbbm{1}\{ \hat{v}_j \neq 1 \} \mid V_j = 1 \right]\right]\mathbb{P}(V_j = 1) \\
            & \hspace{1cm} + 2\varrho \sum_{j=1}^k \mathbb{E}_{V} \left[\mathbb{E}_{P_{V}} \left[  \mathbbm{1}\{ \hat{v}_j \neq -1 \}  \mid V_j = -1 \right] \right]\mathbb{P}(V_j = -1) \\
            &\geq  2\varrho \sum_{j=1}^k \left\{ P_{+j}\left( \hat{v}_j \neq 1 \right) \tau + P_{-j}\left( \hat{v}_j \neq -1 \right) (1-\tau) \right\} \\
            &\geq 2\varrho \tau \sum_{j=1}^k \left\{ P_{+j}\left( \hat{v}_j \neq 1 \right) + P_{-j}\left( \hat{v}_j \neq -1 \right) \right\} \\
            &\geq 2\varrho \tau \sum_{j=1}^k \left\{1 - D_{\mathrm{TV}}(P_{+j}, P_{-j}) \right\},
        \end{align*}
        where the first inequality is by \eqref{app:eq:HammingSeparationEquality} and \eqref{app:eq:AssouadTriangleCons}, and in the penultimate we used the fact that $\tau \leq 1/2$ implies $1 - \tau \geq \tau$. Note that when $\tau = 1/2$, we recover the usual proof of Assouad's method.

        Plugging in the above, and the assuming that $\mathbb{P}(V \notin \mathcal{V}^\ast) \leq c\tau$, into \eqref{app:eq:GenAssouad}, we obtain that
        \begin{align*}
            \sup_{P \in \mathcal{P}^\ast} \mathbb{E}_{P} \left[ \Phi \circ \rho (\hat{\theta}, \theta(P)) \right]
            \geq 2\varrho \tau \left[ \sum_{j=1}^k \left\{1 - D_{\mathrm{TV}}(P_{+j}, P_{-j}) \right\} - ck \right].            
        \end{align*}
    \end{proof}
    
    \begin{lemma} \label{app:lem:PGFArgument}
        Let $X \sim \mathrm{Binomial}(n, p)$ for $n \in \mathbb{N}$, $p \in (0,1)$. It holds that $\mathbb{E}[Xt^X] = ntp\{(1-p) + pt\}^{n-1}$ for $t \in \mathbb{R}$.
    \end{lemma}
    \begin{proof}
        We first note that the probability generating function $G_X(t) = \mathbb{E}[t^X] = \{(1-p) + pt\}^n$. As the expectation defining the probability generating function is a finite sum, we can swap the order of differentiation and the expectation, yielding
        \begin{align*}
            \mathbb{E}[X t^X]
            = t\mathbb{E}[X t^{X-1}]
            = t G_X'(t)
            = ntp\{(1-p) + pt\}^{n-1}.
        \end{align*}
    \end{proof}

    \begin{lemma} \label{app:lem:sechbound}
        For $x > 0$, it holds that
        \begin{equation*}
            \bigg( \frac{4e^{x/2}}{(1+e^{x/2})^2}\bigg)^n
            \leq e^{-n\min\{x, x^2\}/20}.
        \end{equation*}
    \end{lemma}
    \begin{proof}
        We start by noting that
        \begin{equation*}
            \frac{4e^{x/2}}{(1+e^{x/2})^2}
            = \mathrm{sech}^{2}(x/4)
            = \mathrm{cosh}^{-2}(x/4).
        \end{equation*}
        Hence, we proceed to show that $\mathrm{cosh}(x/4) \geq e^{\min\{x, x^2\}/40}$, considering $x \in (0,1]$, $1 < x \leq 4$ and $x > 4$ separately.

        In the case that $x \in (0,1]$, we note that
        \begin{equation*}
            \mathrm{cosh}(x/4)
            \geq 1 + x^2/32
            \geq e^{x^2/40},
        \end{equation*}
        where the second inequality can be verified by, for example, considering the derivative of the differences.

        In the case that $1 < x \leq 4$,
        \begin{equation*}
            \mathrm{cosh}(x/4)
            \geq 1 + x^2/32
            \geq 1 + x/40 + x^2/160
            \geq e^{x/40},
        \end{equation*}
        where the second inequality holds as $x > 1$, and the final inequality by considering, for example, the expansion of the exponential.

        In the case that $x > 4$,
        \begin{equation*}
            \mathrm{cosh}(x/4)
            = (e^{x/4} + e^{-x/4})/2
            \geq e^{x/4}/2
            \geq e^{x/40}.
        \end{equation*}

        Combining these three cases completes the proof.
    \end{proof}
    
\begin{lemma} \label{app:lem:nonprivateLB}
    Recalling the families~$\mathcal{P}_{d}$ and $\mathcal{P}_{d}'$ defined in \eqref{sec2:eq:MeanClasses}, the corresponding non-private minimax lower bounds satisfy
    \begin{align*}
        &\mathcal{R}_{n, T, \infty}(\theta(\mathcal{P}_{d}), \|\cdot\|_2^2)
        = \mathcal{R}_{nT, 1, \infty}(\theta(\mathcal{P}_{d}), \|\cdot\|_2^2)
        \gtrsim \frac{d}{nT}, \mbox{ and} \\
        &\mathcal{R}_{n, T, \infty}(\theta(\mathcal{P}_{d}'), \|\cdot\|_2^2)
        = \mathcal{R}_{nT, 1, \infty}(\theta(\mathcal{P}_{d}'), \|\cdot\|_2^2)
        \gtrsim \frac{1}{nT}.
    \end{align*}
\end{lemma}

\begin{proof}[Proof of \Cref{app:lem:nonprivateLB}]
    We start as in the proof of the lower bound of \Cref{sec3:thm:main}. Consider the constructions \eqref{app:eq:MeanFamily2} and \eqref{app:eq:MeanFamily} which are $\Delta^2$- and $\Delta^2/d$-separated respectively. We first focus on the construction \eqref{app:eq:MeanFamily2}. Setting $\Delta = 1/(6nT)$, by \Cref{app:lem:Assouad} it holds that
    \begin{align*}
        \mathcal{R}_{nT, 1, \infty}(\theta(\mathcal{P}_{d}), \|\cdot\|_2^2)
        &\geq \Delta^2 \sum_{j=1}^k \{1 - D_{\mathrm{TV}}(P_{+j}^{nT}, P_{-j}^{nT})\} \\
        &\geq \Delta^2 \sum_{j=1}^k \bigg\{1 - \sqrt{\frac{1}{2}D_{\mathrm{KL}}(P_{+j}^{nT}, P_{-j}^{nT})}\bigg\} \\
        &= \Delta^2 \sum_{j=1}^k \bigg\{1 - \sqrt{\frac{1}{2}D_{\mathrm{KL}}\bigg(\frac{1}{2^{d-1}}\sum_{v : v_j = 1 }P_v^{\otimes nT}, \frac{1}{2^{d-1}} \sum_{v : v_j = 1 }P_v^{\otimes nT}\bigg)}\bigg\} \\
        &\geq \Delta^2 \sum_{j=1}^k \bigg\{1 - \sqrt{\frac{1}{2}\sum_{v : v_j = 1 }\frac{nT}{2^{d-1}}D_{\mathrm{KL}}(P_v, P_{v^{\oplus j}})}\bigg\} \\
        &= \Delta^2 \sum_{j=1}^k \bigg\{1 - \sqrt{\frac{nT}{2^d}\sum_{v : v_j = 1 }D_{\mathrm{KL}}(\mathrm{Bern}(1/2 + \Delta v_j/2), \mathrm{Bern}(1/2 - \Delta v_j/2)}\bigg\} \\
        &= \Delta^2 \sum_{j=1}^k \bigg\{1 - \sqrt{nT\Delta v_j\log\{(1 + \Delta v_j)/(1 - \Delta v_j)\}} \bigg\} \\
        &\geq \Delta^2 \sum_{j=1}^k \bigg\{1 - \sqrt{3nT\Delta^2} \bigg\}
        \gtrsim \frac{d}{nT},
    \end{align*}
    where the second line is by Pinsker's inequality \citep[e.g.][Lemma~2.5]{Tsybakov:2009}; the fourth by the joint convexity of Kullback--Leibler divergence which follows from the joint convexity of $f$-divergences \citep[e.g.][Proposition~6.1]{Goldfield:2020}; and the seventh by the bound $x\log\{(1+x)/(1-x)\}$ for $|x| < 1/2$ the value of $\Delta$. This completes the proof for the construction \eqref{app:eq:MeanFamily2} corresponding to the family $\mathcal{P}_{d}$, and the result for $\mathcal{P}_{d}'$ follows similarly noting the $\Delta^2/d$-separation.
\end{proof}
    
\begin{lemma} \label{app:lem:sparsebinningevent}
    For the event $A$ defined in \eqref{app:eq:sparsebinningevent}, it holds that $\mathbb{P}(A^c) \leq 1/T^\ast$.
\end{lemma}

\begin{proof}[Proof of \Cref{app:lem:sparsebinningevent}]
        Recall the definition of $\hat{\omega}_j^{(i)}$  in \eqref{sec5:eq:HTestimator}. We denote
        \begin{align*}
            U_i = \sum_{j=1}^d r_{i,j} \hat{\omega}_j^{(i)}, \qquad
            \bar{U}_i = \Pi_{[-\eta, \eta]}(U_i),
        \end{align*}
        noting that with this notation, we have that $I_j$ as in \eqref{app:eq:sparseselectionvoteagg} satisfies $I_j = (n/2)^{-1} \sum_{i=1}^{n/2} r_{i,j} \{\bar{U}_i  + (2\eta/\alpha)\ell_i\}$, and that $\bar{U}_i + (2\eta/\alpha)\ell_i$ is an $\alpha$-LDP view of $U_i$. We also recall that the Rademacher random variables $\{r_{i,j}\}_{i \in [n/2], j \in [d]}$ are mutually independent and independent of the collection $\{\hat{\omega}_j^{(i)}\}_{i \in [n/2], j \in [d]}$.

        We now focus on $I_j$ in the two cases where $j \in S_0$ and $j \in S_1$.
        \begin{enumerate}[{Case} 1.]
            \item We start by considering the expectation of $r_{i,j} \bar{U}_i$ for $j \in S_0$. Define the event $B_{i,j} = \{\hat{\omega}_j^{(i)} = 0 \}$ where by the construction \eqref{sec5:eq:HTestimator}, we have that
            \begin{equation}
                \mathbb{P}(B_{i,j}^c) = \mathbb{P}(\hat{\omega}_j^{(i)} = 1) = \mathbb{P}\left(\left| \frac{1}{T^\ast}\sum_{t=1}^{T^*} X_{t,j}^{(i)} \right| \geq \varepsilon\right) \leq 2e^{-\frac{T^\ast\varepsilon^2}{2}} \leq \frac{2}{dnT^\ast\alpha^2} \leq 1/5, \label{app:eq:sparsecase1eventbound}
            \end{equation}
            where the first inequality is by Hoeffding's inequality; the second by the definition of $\varepsilon$ as in \eqref{sec5:eq:thresholdboundary}; and the last by the assumption $n\alpha^2 \geq \widetilde{C}s\log(ed)$ for some sufficiently large $\widetilde{C}$.
            
            We also have that
            \begin{equation}
                \mathbb{E}[r_{i,j} \bar{U}_i] = \mathbb{E}[r_{i,j} \bar{U}_i \mathbbm{1}(B_{i,j}^c)], \label{app:eq:sparsecase1expectationiden}
            \end{equation}
            as on $B_{i,j}$ it holds that $r_{i,j}$ and $\bar{U}_i$ are independent and so the expectation of their product is zero. Hence, we have that
            \begin{align}
                \mathbb{E}[r_{i,j} \bar{U}_i \mathbbm{1}(B_{i,j}^c)] &= \mathbb{E}[r_{i,j} (\bar{U}_i - U_i)\mathbbm{1}(B_{i,j}^c)] + \mathbb{E}[r_{i,j}U_i\mathbbm{1}(B_{i,j}^c)] \nonumber \\
                &= \mathbb{E}[r_{i,j} (\bar{U}_i - U_i)\mathbbm{1}(B_{i,j}^c)] + \mathbb{E}\left[\left(1 + \sum_{k \neq j} (r_{i,j}r_{i,k}\hat{\omega}_k^{(i)})\right)\mathbbm{1}(B_{i,j}^c)\right] \nonumber \\
                &= \mathbb{E}[r_{i,j} (\bar{U}_i - U_i)\mathbbm{1}(B_{i,j}^c)] + \mathbb{P}(B_{i,j}^c), \label{app:eq:hashingtrunccase1}
            \end{align}
            where the final line comes from the fact that $\mathbb{E}(r_{i,j}r_{i,k}) = 0$ for $j \neq k$. Then, writing $\tilde{U} = U_1 - r_{1, j}$, we have that
            \begin{align}
                \mathbb{E}[&r_{1,j} (\bar{U}_1 - U_1)\mathbbm{1}(B_{1,j}^c)] \nonumber \\
                &= \mathbb{E}[r_{1,j} (\eta - U_1) \mathbbm{1}(\{U_1 \geq \eta\} \cap B_{1,j}^c)] + \mathbb{E}[r_{1,j} (-\eta - U_1)\mathbbm{1}(\{U_1 \leq - \eta\} \cap B_{1,j}^c)] \nonumber \\
                &= 2\mathbb{E}[r_{1,j} (\eta - U_1) \mathbbm{1}(\{U_1 \geq \eta\} \cap B_{1,j}^c)] \nonumber \\
                &= \mathbb{E}[(\eta - 1 - \tilde{U}) \mathbbm{1}(\{\tilde{U} \geq \eta - 1\} \cap B_{1,j}^c)] - \mathbb{E}[(\eta + 1 - \tilde{U}) \mathbbm{1}(\{\tilde{U} \geq \eta + 1\} \cap B_{1,j}^c)] \nonumber \\
                &= -2 \mathbb{P}(\{\tilde{U} \geq \eta + 1\} \cap B_{1,j}^c) + \mathbb{E}[(\eta - 1 - \tilde{U})\mathbbm{1}(\{\eta - 1 \leq \tilde{U} < \eta + 1\} \cap B_{1,j}^c)] \nonumber \\
                &\leq 0. \label{app:eq:rademacherseries}
            \end{align}

            Combining \eqref{app:eq:sparsecase1eventbound}, \eqref{app:eq:sparsecase1expectationiden}, \eqref{app:eq:hashingtrunccase1} and \eqref{app:eq:rademacherseries} we obtain
            \begin{equation}
                \mathbb{E}[r_{i,j} \bar{U}_i] \leq \mathbb{P}(B_{i,j}^c) \leq 1/5. \label{app:eq:sparsecase1probbound}
            \end{equation}

            We also note for any fixed $j \in [d]$ that $r_{i,j}\bar{U}_i \in [-\eta, \eta]$ for $i \in [n/2]$ and that $\{r_{i,j}\bar{U}_i\}_{i \in [n/2]}$ are mutually independent. Hence, by Hoeffding's inequality, we have for $x > 0$ that
            \begin{equation}
                \mathbb{P}\left(\frac{1}{n/2} \sum_{i=1}^{n/2} r_{i,j}\bar{U}_i - \mathbb{E}[r_{i,j}\bar{U}_i] \geq x \right) \leq e^{-nx^2/(4\eta^2)}. \label{app:eq:sparsecase1probcontroll}
            \end{equation}

            This gives that
            \begin{align}
                \mathbb{P}(&I_j \geq 1/2) \nonumber \\
                &= \mathbb{P}\left( \frac{1}{n/2} \sum_{i=1}^{n/2} r_{i,j} (\bar{U}_i  + (2\eta/\alpha)\ell_i) \geq 1/2 \right) \nonumber \\
                &= \mathbb{P}\left( \frac{1}{n/2} \sum_{i=1}^{n/2} r_{i,j} \bar{U}_i - \mathbb{E}[r_{i,j} \bar{U}_i] + (2\eta/\alpha)\ell_i \geq 1/2 - \mathbb{E}[r_{1,j} \bar{U}_1] \right) \nonumber \\
                &\leq \mathbb{P}\left( \frac{1}{n/2} \sum_{i=1}^{n/2} r_{i,j} \bar{U}_i - \mathbb{E}[r_{i,j} \bar{U}_i] + (2\eta/\alpha)\ell_i \geq 3/10 \right) \nonumber \\
                &\leq \mathbb{P}\left(  \frac{1}{n/2} \sum_{i=1}^{n/2} r_{i,j}\bar{U}_i - \mathbb{E}[r_{i,j} \bar{U}_i] \geq 3/20 \right) + \mathbb{P}\left( \frac{1}{n/2} \sum_{i=1}^{n/2} (2\eta/\alpha)\ell_i \geq 3/20 \right) \nonumber \\
                &\leq \exp\left( -\frac{9n}{1600\eta^2} \right) + \exp\left( -\frac{9n\alpha^2}{25600\eta^2} \right) \leq 2 \exp\left( -\frac{9n\alpha^2}{25600\eta^2} \right), \label{app:eq:sparseselectioncase1bound}
            \end{align}
            where the second equality uses the independence and symmetry of $r_{i,j}$ and $\ell_i$ for any fixed $i$ and $j$; the second equality uses \eqref{app:eq:sparsecase1probbound}; and in the third inequality the first term is via \eqref{app:eq:sparsecase1probcontroll}, and the second term by e.g.~\citet[][Equation~2.18]{Wainwright:2019} where it is easy to verify that the Laplace random variables are sub-Exponential with parameters $(2,2)$, \citep[e.g.][~Definition~2.7]{Wainwright:2019}.
            
            Hence, by \eqref{app:eq:sparseselectioncase1bound} and the union bound, we have that
            \begin{equation}
                \mathbb{P}\left( \bigcup_{j \in S_0} \{ I_j \geq 1/2 \} \right) \leq 2d \exp\left( -\frac{9n\alpha^2}{25600\eta^2} \right), \label{app:eq:S0votesbound}
            \end{equation}
            where we use the fact that $|S_0| \leq d$.

            \item We start by considering the expectation of $r_{i,j} \bar{U}_i$ for $j \in S_1$. Define the event $D_{i,j} = \{\hat{\omega}_j^{(i)} = 1 \}$ where by the construction \eqref{sec5:eq:HTestimator}, we have that
            \begin{align}
                \mathbb{P}(D_{i,j}^c) &= \mathbb{P}(\hat{\omega}_j^{(i)} = 0) = \mathbb{P}\left(\left| \frac{1}{T^\ast}\sum_{t=1}^T X_{t,j}^{(i)} \right| < \varepsilon\right) \nonumber \\
                &\leq \mathbb{P}\left(\left| \frac{1}{T^\ast}\sum_{t=1}^T (X_{t,j}^{(i)} - \theta_j) \right| \geq \varepsilon\right) \leq 2e^{-\frac{T^\ast\varepsilon^2}{2}} \leq \frac{2}{nT^\ast\alpha^2} \leq 1/10, \label{app:eq:sparsecase2eventbound}
            \end{align}
            where the first inequality uses the fact that for $j \in S_1$, $|\theta_j| > 2\varepsilon$; the second inequality is by Hoeffding's inequality; the third by the definition of $\varepsilon$ as in \eqref{sec5:eq:thresholdboundary}; and the last by the assumption $n\alpha^2 \geq \widetilde{C}s\log(ed)$ for some sufficiently large $\widetilde{C}$. We then have
            \begin{equation}
                \mathbb{E}[r_{i,j} \bar{U}_i] = \mathbb{E}[r_{i,j} \bar{U}_i \mathbbm{1}(D_{i,j})] \label{app:eq:sparsecase2expectationiden}
            \end{equation}
            as on $D_{i,j}^c$ it holds that $r_{i,j}$ and $\bar{U}_i$ are independent and so the expectation of their product is zero. Hence, we have that 
            \begin{align}
                \mathbb{E}[r_{i,j} \bar{U}_i \mathbbm{1}(D_{i,j})] 
                = \mathbb{E}[r_{i,j} (\bar{U}_i - U_i)\mathbbm{1}(D_{i,j})] + \mathbb{P}(D_{i,j}) \label{app:eq:hashingtrunccase2}
            \end{align}
            by the same argument as \eqref{app:eq:hashingtrunccase1}. Then, writing $\tilde{U} = U_1 - r_{1, j}$, we have that
            \begin{align}
                \mathbb{E}[&r_{1,j} (\bar{U}_1 - U_1)\mathbbm{1}(D_{1,j})] \nonumber \\
                &= -2 \mathbb{P}(\{\tilde{U} \geq \eta + 1\} \cap D_{1,j}) + \mathbb{E}[(\eta - 1 - \tilde{U})\mathbbm{1}(\{\eta - 1 \leq \tilde{U} < \eta + 1\} \cap D_{1,j})] \nonumber \\
                &\geq -2 \mathbb{P}(\{\tilde{U} \geq \eta + 1\} \cap D_{1,j}) - 2\mathbb{P}(\{\eta - 1 \leq \tilde{U} < \eta + 1\} \cap D_{1,j}) \nonumber \\
                &= -2 \mathbb{P}(\{\tilde{U} \geq \eta - 1\} \cap D_{1,j}) \geq -\frac{2}{(\eta-1)^2}\mathbb{E}\left[\left(\sum_{k \neq j} r_{1,k} \hat{\omega}_k^{(1)} \right)^2 \right] \nonumber \\
                &\geq -\frac{2}{(\eta - 1)^2}\left(\sum_{k \in S_0} \mathbb{P}(\hat{\omega}_k^{(1)} = 1) + \sum_{k \in S_1 \cup S_2} \mathbb{P}(\hat{\omega}_k^{(1)} = 1)\right) \nonumber \\
                &\geq -\frac{2}{(\eta - 1)^2}\left(\frac{2}{nT\alpha^2} + s\right)
                =: \zeta > -1/10, \label{app:eq:sparsecase2zetaval}
            \end{align}
            where first equality is by the same calculations leading to \eqref{app:eq:rademacherseries}; the second inequality is by Markov's inequality; the fourth inequality is by the same argument as \eqref{app:eq:sparsecase1eventbound} and the bounds $|S_0| \leq d$ and $|S_1 \cup S_2| \leq s$; and the final inequality is by the value of $\eta$ and the assumption $n\alpha^2 \geq \widetilde{C}s\log(ed)$ for some sufficiently large $\widetilde{C}$. Combining \eqref{app:eq:sparsecase2eventbound}, \eqref{app:eq:sparsecase2expectationiden} \eqref{app:eq:hashingtrunccase2} and the value $\zeta$ in \eqref{app:eq:sparsecase2zetaval}, we obtain
            \begin{equation}
                \mathbb{E}[r_{i,j} \bar{U}_i \mathbbm{1}(D_{i,j})] \geq \mathbb{P}(D_{i,j}) + \zeta > 1- 1/10 - 1/10 \geq 4/5. \label{app:eq:sparsecase2probbound}
            \end{equation}

            We then calculate
            \begin{align}
                \mathbb{P}(I_j < 1/2) &= \mathbb{P}\left( \frac{1}{n/2} \sum_{i=1}^{n/2} r_{i,j} (\bar{U}_i  + (2\eta/\alpha)\ell_i) < 1/2 \right) \nonumber \\
                &= \mathbb{P}\left( \frac{1}{n/2} \sum_{i=1}^{n/2} r_{i,j} \bar{U}_i - \mathbb{E}[r_{i,j} \bar{U}_i] + (2\eta/\alpha)\ell_i < 1/2 - \mathbb{E}[r_{1,j} \bar{U}_1] \right) \nonumber \\
                &\leq \mathbb{P}\left( \frac{1}{n/2} \sum_{i=1}^{n/2} r_{i,j} \bar{U}_i - \mathbb{E}[r_{i,j} \bar{U}_i] + (2\eta/\alpha)\ell_i < -3/10 \right) \nonumber \\
                &\leq \mathbb{P}\left(  \frac{1}{n/2} \sum_{i=1}^{n/2} r_{i,j}\bar{U}_i - \mathbb{E}[r_{i,j} \bar{U}_i] < -3/20 \right) + \mathbb{P}\left( \frac{1}{n/2} \sum_{i=1}^{n/2} (2\eta/\alpha)\ell_i < -3/20 \right) \nonumber \\
                &\leq \exp\left( -\frac{9n}{1600\eta^2} \right) + \exp\left( -\frac{9n\alpha^2}{25600\eta^2} \right) \leq 2 \exp\left( -\frac{9n\alpha^2}{25600\eta^2} \right), \label{app:eq:sparseselectioncase2bound}
            \end{align}
            where the second equality uses the independence and symmetry of $r_{i,j}$ and $\ell_i$ for fixed $i$ and $j$; the first inequality uses \eqref{app:eq:sparsecase2probbound}; and in the third inequality the first term is via \eqref{app:eq:sparsecase1probcontroll} noting Hoeffding's inequality still holds for the negative $-r_{i,j}\bar{U}_i$, and the second term by e.g.~\citet[][Equation~2.18]{Wainwright:2019} where it is easy to verify that the Laplace random variables are sub-Exponential with parameters $(2,2)$, \citep[e.g.][~Definition~2.7]{Wainwright:2019}.
            
            Hence, by \eqref{app:eq:sparseselectioncase2bound} and the union bound, we have that
            \begin{equation}
                \mathbb{P}\left( \bigcup_{j \in S_1} \left\{ I_j < 1/2 \right\} \right) \leq 2s \exp\left( -\frac{9n\alpha^2}{25600\eta^2} \right),\label{app:eq:S1votesbound}
            \end{equation}
            where we use the fact that $|S_1| \leq s$ by the sparsity assumption.
        \end{enumerate}
        Combining \eqref{app:eq:S0votesbound} and \eqref{app:eq:S1votesbound} gives
        \begin{align*}
            \mathbb{P}(A^c)
            &= \mathbb{P}\left( \left\{ \bigcup_{j \in S_0} \left\{ I_j \geq 1/2 \right\}\right\} \cup \left\{ \bigcup_{j \in S_1} \left\{ I_j < 1/2 \right\}\right\} \right) \\
            &\leq  4d \exp\left( -\frac{9n\alpha^2}{25600\eta^2} \right) = 4d \exp\left( -\frac{cn\alpha^2}{s} \right),
        \end{align*}
        for $c > 0$ an absolute constant.
        Hence, we have that
        \begin{align*}
            \mathbb{P}(A^c)
            &\leq 4d \exp\left( -\frac{cn\alpha^2}{s} \right)
            \leq e^{-n\alpha^2/(Ks)} = 1/T^\ast,
        \end{align*}
        where the second inequality comes from the fact that $n\alpha^2 \geq \widetilde{C}s\log(ed)$ for a sufficiently large constant $\widetilde{C}$ and taking $K$ sufficiently large, and the equality comes from the value of $T^\ast$. Hence, the proof is completed.
    \end{proof}

\begin{lemma}[Corollary 7 in \citealt{Jin:2019}] \label{app:lem:SGnorm} 
            Let $X_i$ for $i \in [n]$ be a collection of i.i.d.~mean-zero random variables taking values in $\mathbb{R}^d$ such that, for some $\sigma^2 > 0$ and for all $i \in [n]$ and any $\varepsilon > 0$,
            \begin{align*}
                \mathbb{P}(\|X_i\|_2 \geq \varepsilon) \leq 2e^{-\varepsilon^2/(2\sigma^2)}.
            \end{align*}
            Then, writing $\overline{X} = n^{-1}\sum_{i=1}^n X_i$, there exists an absolute constant $c>0$ such that for any $\gamma \in (0,1)$,
            \begin{align*}
                \|\overline{X}\|_2 \leq c\left\{\frac{\sigma^2\log(2d/\gamma)}{n}\right\}^{1/2}
            \end{align*}
            holds with probability at least $1 - \gamma$. We equivalently have that
            \begin{align*}
                \mathbb{P}(\|\overline{X}\|_2 \geq \varepsilon) \leq 2de^{-n\varepsilon^2/(c'\sigma^2)}.
            \end{align*}
            where $c'>0$ is some absolute constant.
        \end{lemma}
        
    \begin{lemma} \label{app:lem:RR} 
        For $d \in \mathbb{N}$, let $X \in \{0, 1\}^d$ be a binary vector with at most $m \in \mathbb{N}$ non-zero entries, i.e.~$\|X\|_0 \leq m$. For any $x > 0$, write $\pi_x := e^x/(1+e^x)$ and let $\{U_j\}_{j \in [d]}$ be a set of i.i.d.~$\mathrm{Unif}(0, 1)$ independent of $X$.  The privatised vector $Z \in \{0, 1\}^d$ with the co-ordinates
        \begin{align}\label{ap:eq:RR}
            Z_j = \begin{cases}
                X_j, & U_j \leq \pi_{\alpha / (2m)}, \\
                1 - X_j, &\text{otherwise}, 
            \end{cases} \quad j \in [d],
        \end{align}
        is an $\alpha$-LDP view of the vector $X$.
    \end{lemma}
    \begin{proof}[Proof of \Cref{app:lem:RR}]
        Let $v \in \{0, 1\}^d$ be an arbitrary binary vector.  For $x,y \in \{0, 1\}^d$ satisfying that $\|x\|_0, \|y\|_0 \leq m$, the privatised vector given by the mechanism \eqref{ap:eq:RR} satisfies
        \begin{align*}
            & \mathbb{P}(Z = v \mid X = x)
            = \prod_{j=1}^d \pi_{\alpha/(2m)}^{\mathbbm{1}\{ v_j = x_j \} } (1-\pi_{\alpha/(2m)})^{\mathbbm{1}\{ v_j \neq x_j \} } = \pi_{\alpha/(2m)}^{\# \{j : v_j = x_j \} } (1 - \pi_{\alpha/(2m)})^{\# \{j : v_j \neq x_j \} } \\
            = & \left( \frac{\pi_{\alpha/(2m)}}{1-\pi_{\alpha/(2m)}} \right)^{n_x} (1 - \pi_{\alpha/(2m)})^{d} = e^{\alpha n_x/(2m)} (1 - \pi_{\alpha/(2m)})^{d}.
        \end{align*}
        We then have that 
        \begin{align}
            \frac{\mathbb{P}(Z = v \mid X = x)}{\mathbb{P}(Z = v \mid X = y)}
            = e^{(n_x - n_y)\alpha/(2m)}. \label{ap:eq:RRbound}
        \end{align}
        Since $x$ and $y$ are non-zero on at most $m$ co-ordinates each, they can differ from each other on at most $2m$ co-ordinates which implies $|n_x - n_y| \leq 2m$. Recalling that for a set $A$, $\sigma(A)$ denotes the sigma-algebra generated by $A$, we have that (\ref{ap:eq:RRbound}) holds for arbitrary $v$ and so
        \begin{align*}
            \sup_{V \in \sigma(\{0, 1\}^d)} \frac{\mathbb{P}(Z \in V \mid X = x)}{\mathbb{P}(Z \in V \mid X = y)} \leq e^{\alpha},
        \end{align*}
        which shows the $\alpha$-LDP condition as in \eqref{sec1:eq:alphaLDP} is satisfied.
    \end{proof}

\begin{lemma} \label{app:lem:densitycoverlower}  
    Given $\beta \in \mathbb{N}$ and $r > 0$, for $\Delta > 0$ sufficiently small, the covering number $N(\Delta)$ of the metric space $(\mathcal{F}_{\beta, r}, \|\cdot\|_2)$ with $\mathcal{F}_{\beta, r}$ defined \eqref{sec2:eq:SobolevDensity}, is lower bounded as,
        \begin{equation*}
            N(\Delta) \geq \exp\left(c_\beta (r/\Delta)^{1/\beta}\right),
        \end{equation*}
    for some constant $c_\beta > 0$ depending only on $\beta$.
\end{lemma}

\begin{proof}[Proof of \Cref{app:lem:densitycoverlower}]
    To construct an explicit packing, we consider the construction in \Cref{app:sec:densityLB} to obtain the family $\{f_v\}_{v \in \mathcal{V}}$ defined in \eqref{app:eq:DensityFamily} which is indexed by $\mathcal{V} = \{\pm 1\}^k$ for some $k \in \mathbb{N}$.

    For $v, v' \in \mathcal{V}$, it follows from \eqref{app:eq:densityseparation} that
    \begin{equation*}
        \|f_v - f_{v'}\|_2^2 \geq \frac{c_\beta' r^2\pi^{2\beta}}{2k^{2\beta + 1}} \sum_{j=1}^k \mathbbm{1}\{v_j \neq v'_j \},
    \end{equation*}
    where $c_\beta' > 0$ is some constant depending on $\beta$.
        
    By the Varshamov--Gilbert Lemma \citep[e.g.][Lemma~2.9]{Tsybakov:2009}, there exists a subset $\mathcal{V}' \subseteq \mathcal{V}$ such that $|\mathcal{V}'| \geq 2^{k/8}$ and $\mathrm{Hamm}(u, u') \geq k/8$ for $u, u' \in \mathcal{V}'$, where $\mathrm{Hamm}(\cdot, \,\cdot)$ denotes the Hamming distance.  For $u,u' \in \mathcal{V}'$, we therefore have that
        \begin{equation*}
            \|f_u - f_{u'}\|_2^2
            \geq \frac{c_\beta' r^2\pi^{2\beta}}{16k^{2\beta}},
        \end{equation*}        
        and hence we have $2^{k/8}$-many distributions that are $\Delta = (c_\beta')^{1/2}r\pi^\beta / (16k^{\beta})$-separated.

        Rearranging to get $k$ in terms of $\Delta$, we have
        \begin{equation*}
            2^{k/8} = \exp\left(\frac{\log(2)}{8}\left( \frac{(c_\beta')^{1/2} r \pi^\beta}{16\Delta}\right)^{1/\beta}\right) \geq e^{c_\beta'' (r/\Delta)^{1/\beta}},
        \end{equation*}
        for some absolute constant $c_\beta'' > 0$. This provides a lower bound on the $\Delta$-packing number whence a lower bound on the $\Delta/2$-covering number can be obtained, completing the proof.
    \end{proof}

\begin{lemma} \label{app:lem:sparsecoverlower}  
    Given $r > 0$ and $d \in \mathbb{N}$, for any integer $s \in [d]$ and $0 < \Delta < r/4$, the covering number $N(\Delta)$ of the metric space $(\{ \theta \in [-r, r]^d :\, \|\theta\|_0 \leq s \}, \|\cdot\|_2)$ is lower bounded as,
    \begin{equation*}
        N(\Delta) \geq \left\{\frac{d}{s \wedge (d-s)} \right\}^{s \wedge (d-s)} \left\{\frac{s(r-2\Delta)^2}{20\Delta^2}\right\}^{\frac{s}{2}}.
    \end{equation*}
\end{lemma}
\begin{proof}[Proof of \Cref{app:lem:sparsecoverlower}]
    Note that when $s = d$, we can use the lower bound in \eqref{app:eq:linfcover} which is larger than, and thus implies, the lower bound \eqref{app:eq:SparseCoveringBoundLower} with $s$ set to $d$ therein.  We hence assume $s < d$ in what follows.
    
    Writing $\Theta = \{ \theta \in [-r, r]^d : \|\theta\|_0 \leq s \}$, we proceed by considering the packing number of the space $\Theta' = \{ \theta \in [-r, r]^d : \|\theta\|_0 = s \}$, noting that since $\Theta' \subseteq \Theta$, this will give a lower bound on the packing number of $\Theta$, and hence its covering number. We now decompose $\Theta'$ into a union of $\binom{d}{s}$-many $s$-dimensional subspaces of $[-r, r]^d$. For a given indexing vector $\iota \in \{ 0 ,1\}^d$, we denote the subspace
    \begin{equation*}
        A_\iota = \{ x \in [-r, r]^d : x_j = 0 \text{ for } j \in [d] \text{ such that } \iota_j = 0\}.
    \end{equation*}
    We thus have that
    \begin{equation*}
        \Theta' = \bigcup_{\iota \in \{0, 1\}^d : \, \|\iota\|_0 = s} A_\iota.
    \end{equation*}
    For any $\iota \in \{ 0 ,1\}^d$, defining 
    \begin{equation*}
        B_\iota = \{ x \in A_\iota: \, |x_j| > \Delta \text{ for } j \in [d] \text{ such that } \iota_j = 1\} \subseteq A_\iota,
    \end{equation*}
    we construct
    \begin{equation*}
        \Theta'' = \bigcup_{\iota \in \{0, 1\}^d : \, \|\iota\|_0 = s} B_\iota.
    \end{equation*}
    Denoting the $\Delta$-packing numbers of $\Theta$ and $\Theta''$ as $M(\Delta)$ and $M''(\Delta)$ respectively, we have the key observation that $M''(\Delta)$ is precisely equal to the sum of the $\Delta$-packing numbers of the $B_\iota$.  This is because that, for two distinct index vectors $\iota^{(1)}$ and $\iota^{(2)}$, two points $x \in B_{\iota^{(1)}}$ and $y \in B_{\iota^{(2)}}$ necessarily satisfy $\| x - y \|_2 > 2^{1/2}\Delta$.

    For each $\iota \in \{ 0 ,1\}^d$ such that $\|\iota\|_0 = s$, define the set of $s$-sparse sign vectors as $S_\iota = \{ \varsigma \in \{-1,0,1\}^d : \varsigma_j = 0 \text{ for } j \in [d] \text{ such that } \iota_j = 0\}$. For $\iota  \in \{ 0 ,1\}^d$ and $\varsigma \in S_{\iota}$, we consider 
    \begin{equation*}
        C_{\iota, \varsigma} = \{ x \in A_\iota : \varsigma_j x_j >  \Delta \text{ for } j \in [d] \text{ such that } \iota_j = 1\} \subseteq B_\iota,
    \end{equation*}
    which leads to that 
    \begin{equation*}
        B_\iota = \bigcup_{\varsigma \in S_\iota} C_{\iota, \varsigma}.
    \end{equation*}
    For any $\iota \in \{ 0 ,1\}^d$, $\varsigma^{(1)}, \varsigma^{(2)} \in S_\iota$ with $\varsigma^{(1)} \neq \varsigma^{(2)}$, $x \in C_{\iota, \varsigma^{(1)}}$ and $y \in C_{\iota, \varsigma^{(2)}}$, we have that $\| x - y \|_2 > 2\Delta$. Hence, the $\Delta$-packing number of $B_\iota$ is exactly equal to the sum of the $\Delta$-packing numbers of the $C_{\iota, \varsigma}$.

    It remains to lower bound the packing number of $C_{\iota, \varsigma}$ for arbitrary $\iota$ and $\varsigma$. By symmetry, we can assume without loss of generality that $\iota = \varsigma = \{1, \hdots, 1, 0 \hdots, 0\}$, that is, both the index vector $\iota$ and the sign vector $\varsigma$ have the first $s$ entries as $1$ and the rest as $0$.  It then holds that $C_{\iota, \varsigma} = \{x \in [-r,r]^d : (x_1, \hdots, x_s)^{\top} \in [\Delta, r]^s,~ x_j = 0, \, s < j \leq d\}$. Restricted to the first $s$ co-ordinates, we see that $C_{\iota, \varsigma}$ is a shifted $s$-dimensional $\ell_\infty$-ball of radius $(r-\Delta)/2$ and hence we can bound the packing number of $C_{\iota, \varsigma}$ by treating it as an $s$-dimensional $\ell_\infty$-ball and considering its covering number.
        
    We can lower bound $N_{\mathbb{B}_\infty(r)}(\Delta)$, the $\Delta$-covering number of an $s$-dimensional $\ell_\infty$-ball of radius $r$ with respect to the metric $\|\cdot\|_2$, by a standard volumetric argument. Letting $\{\theta^{(j)}, \, j \in [N_{\mathbb{B}_\infty(r)}(\Delta)]\}$ denote a cover for $\mathbb{B}_\infty(r)$, we have that
    \begin{equation*}
        \mathbb{B}_\infty(r) \subseteq \bigcup_{j = 1}^{N_{\mathbb{B}_\infty(r)}(\Delta)} \{\theta^{(j)} + \mathbb{B}_2(\Delta)\},
    \end{equation*}
    which implies that
    \begin{equation*}
        \mathrm{Vol}\left( \mathbb{B}_\infty(r) \right) \leq \mathrm{Vol}\left( \bigcup_{j = 1}^{N_{\mathbb{B}_\infty(r)}(\Delta)} \{\theta^{(j)} + \mathbb{B}_2(\Delta)\} \right) \leq N_{\mathbb{B}_\infty(r)}(\Delta) \mathrm{Vol}(\mathbb{B}_2(\Delta)),
    \end{equation*}
    giving that
    \begin{equation}
        N_{\mathbb{B}_\infty(r)}(\Delta)
            \geq \frac{\mathrm{Vol}(\mathbb{B}_\infty(r))}{\mathrm{Vol}(\mathbb{B}_2(\Delta))}
            = \left( \frac{r}{\Delta} \right)^s \frac{2^s \Gamma(1 +s/2)}{\pi^{s/2}}
            \geq \left( \frac{dr^2}{5 \Delta^2}\right)^{\frac{s}{2}}. \label{app:eq:generallinflbcover}
    \end{equation}

    Writing $M_{C_{\iota, \varsigma}}(\Delta)$ and $N_{C_{\iota, \varsigma}}(\Delta)$ for the $\Delta$-packing and $\Delta$-covering numbers of $C_{\iota, \varsigma}$ respectively, this gives the bound
    \begin{equation}
        M_{C_{\iota, \varsigma}}(\Delta)
            \geq N_{\mathbb{B}_\infty\{(r-\Delta)/2\}}(\Delta)
            \geq \left\{\frac{s(r-\Delta)^2}{20\Delta^2}\right\}^{\frac{s}{2}}, \label{app:eq:packinglb}
    \end{equation}
    the first inequality follows from the fact that the packing number of $C_{\iota, \varsigma}$ is lower bounded by its covering number and that $C_{\iota, \varsigma}$ can be viewed as an $s$-dimensional $\ell_\infty$-ball of radius $(r-\Delta)/2$, and the second inequality by \eqref{app:eq:generallinflbcover}.
        
        Finally, writing $M_{B_{\iota}}(\Delta)$ for the packing number of $B_{\iota}$, we have
        \begin{equation*}
            M''(\Delta)
            = \sum_{\iota \in \{0, 1\}^d: \,\|\iota\|_0 = s} M_{B_\iota}(\Delta)
            = \sum_{\iota \in \{0, 1\}^d: \, \|\iota\|_0 = s} \sum_{\varsigma \in S_\iota} M_{C_{\iota, \varsigma}}(\Delta)
            \geq \binom{d}{s} 2^s \left\{\frac{s(r-\Delta)^2}{20\Delta^2}\right\}^{\frac{s}{2}},
        \end{equation*}
        where in the inequality we use the lower bound \eqref{app:eq:packinglb} and the fact that $| \{\iota : \|\iota\|_0 = s\} | = \binom{d}{s}$ and $|I_\iota| = 2^s$. This gives a lower bound on $N(\Delta)$ of
        \begin{equation*}
            N(\Delta)
            \geq M(2\Delta)
            \geq M''(2\Delta)
            \geq \left\{\frac{d}{s \wedge (d-s)} \right\}^{s \wedge (d-s)} \left\{\frac{s(r-2\Delta)^2}{20\Delta^2}\right\}^{\frac{s}{2}},
        \end{equation*}
        where in the first inequality we use the fact that $N(\Delta) \geq M(2\Delta)$, in the second the fact that $\Theta'' \subseteq \Theta$, and the bound $\binom{d}{s} \geq (d/s)^s$ for $1 \leq s \leq d/2$ for the final inequality.
    \end{proof}
\end{document}